\pdfoutput=1
\documentclass[
  a4paper,ngerman, british,
]{article}
\usepackage[left=3cm, right=2.5cm, top=2.5cm, bottom=2.5cm]{geometry}
\usepackage[T1]{fontenc}
\usepackage[utf8]{inputenc}

\usepackage{babel}
\usepackage[autostyle=true]{csquotes}

\usepackage{lmodern}
\usepackage{mathtools}
\usepackage{amssymb}
\usepackage{dsfont}
\usepackage{amsthm}
\usepackage{bm}
\usepackage{enumitem}

\usepackage{microtype}

\usepackage{caption}
\usepackage{subcaption}
\usepackage{array}
\usepackage{booktabs}
\usepackage[svgnames]{xcolor}
\usepackage{tikz}
\usepackage[theorems]{tcolorbox}

\usepackage{centernot}

\usepackage{natbib}

\usepackage{aliascnt}

\usepackage[colorlinks]{hyperref}
\usepackage{orcidlink}
\newcommand*{\doi}{}
\makeatletter
\newcommand{\doi@}[1]{\textsc{doi}:~\href{https://doi.org/#1}{\texttt{#1}}}
\DeclareRobustCommand{\doi}{\hyper@normalise\doi@}
\makeatother
\usepackage{bookmark}

\usetikzlibrary{calc,decorations.pathreplacing,positioning,shapes.multipart}

\setlist[enumerate]{label=\textup{(\roman*)}}

\makeatletter
\newtheorem*{rep@theorem}{\rep@title}
\newcommand{\newreptheorem}[2]{%
\newenvironment{rep#1}[1]{%
 \def\rep@title{#2 \ref{##1}}%
 \begin{rep@theorem}}%
 {\end{rep@theorem}}}
\makeatother

\newtheorem{theorem}{Theorem}[section]

\newcommand*{\newtheoremwithcounter}[2]{%
  \newaliascnt{#1}{theorem}%
  \newtheorem{#1}[#1]{#2}%
  \aliascntresetthe{#1}%
  \csdef{#1autorefname}{#2}%
}

\newtheoremwithcounter{corollary}{Corollary}
\newtheoremwithcounter{lemma}{Lemma}
\newtheoremwithcounter{prop}{Proposition}
\newreptheorem{prop}{Proposition}
\newtheoremwithcounter{claim}{Claim}

\theoremstyle{remark}
\newtheoremwithcounter{note}{note}
\newtheorem*{remark}{Remark}
\newtheoremwithcounter{example}{Example}

\theoremstyle{definition}
\newtheoremwithcounter{notation}{Notation}
\newtheoremwithcounter{definition}{Definition}
\newtheoremwithcounter{algorithm}{Algorithm}
\newtheoremwithcounter{assumption}{Assumption}

\newcounter{descriptcount}
\newlist{enumdescript}{description}{2}
\setlist[enumdescript]{%
  before={\setcounter{descriptcount}{0}%
          \renewcommand*\thedescriptcount{Case~(\roman{descriptcount})}}
  ,font=\normalfont\itshape\stepcounter{descriptcount}\thedescriptcount\enskip
}

\newcommand*{\qqquad}{\qquad\quad}

\newcommand*{\treeroot}{\varnothing}

\newcommand*{\setcard}{\abs}
\newcommand*{\setweight}{\norm}
\newcommand*{\setunion}{\cup}
\newcommand*{\setintersect}{\cap}
\newcommand*{\bigsetunion}{\bigcup}
\newcommand*{\bigsetintersect}{\bigcap}
\newcommand*{\setcomplement}[1]{#1^{c}}

\newcommand*{\indfunc}{\mathds{1}}

\newcommand*{\prob}{\mathbb{P}}
\newcommand*{\expe}{\mathbb{E}}
\DeclareMathOperator{\var}{Var}
\DeclareMathOperator{\cov}{Cov}

\newcommand*{\reals}{\mathbb{R}}
\newcommand*{\naturals}{\mathbb{N}}
\newcommand*{\integers}{\mathbb{Z}}

\newcommand*{\wasserstein}{\mathcal{W}}

\newcommand*{\placeholder}{\,\cdot\,}

\newcommand*{\eqdist}{\overset{\mathcal{D}}{=}}

\newcommand*{\pathbetw}{\leftrightsquigarrow}

\newcommand*{\sizebias}{\widehat}
\newcommand*{\imt}{\widetilde}

\DeclarePairedDelimiter{\abs}{\lvert}{\rvert}
\DeclarePairedDelimiter{\norm}{\lVert}{\rVert}

\newcommand{\downto}{\searrow}

\providecommand*{\notcong}{\ncong}

\newcommand*{\edge}[2]{\set{#1,#2}}

\newcommand*{\imin}{\wedge}

\newcommand*{\ulamharris}{\mathcal{U}}

\newcommand{\conn}[2]{X_{#1#2}}

\newcommand{\mathbox}[2][red]{%
  \tcboxmath[colframe=#1!20, colback=#1!20,
    boxsep=0pt,
    left=0pt,right=0pt,top=1pt,bottom=1pt]{#2}}

\DeclarePairedDelimiter{\set}{\lbrace}{\rbrace}
\providecommand{\given}[1][]{%
  \nonscript\:#1\vert\nonscript\:\mathopen{}\allowbreak}
\newcommand*{\setgiven}{\given}
\newcommand{\SetSymbol}[1][]{\nonscript:\nonscript\mathopen{}\allowbreak}
\DeclarePairedDelimiterX{\sset}[2]\{\}{%
  \renewcommand\setgiven{\SetSymbol[\delimsize]}
  #1\setgiven#2
}

\newcommand{\intsep}{, }
\newcommand{\intol}{\lparen}
\newcommand{\intor}{\rparen}
\newcommand{\intcl}{\lbrack}
\newcommand{\intcr}{\rbrack}
\DeclarePairedDelimiterX{\intervalcc}[2]{\intcl}{\intcr}{#1\intsep#2}
\DeclarePairedDelimiterX{\intervalco}[2]{\intcl}{\intor}{#1\intsep#2}
\DeclarePairedDelimiterX{\intervaloc}[2]{\intol}{\intcr}{#1\intsep#2}
\DeclarePairedDelimiterX{\intervaloo}[2]{\intol}{\intor}{#1\intsep#2}

\makeatletter
\let\save@mathaccent\mathaccent
\newcommand*\if@single[3]{%
  \setbox0\hbox{${\mathaccent"0362{#1}}^H$}%
  \setbox2\hbox{${\mathaccent"0362{\kern0pt#1}}^H$}%
  \ifdim\ht0=\ht2 #3\else #2\fi
  }
\newcommand*\rel@kern[1]{\kern#1\dimexpr\macc@kerna}
\newcommand*\widebar[1]{\@ifnextchar^{{\wide@bar{#1}{0}}}{\wide@bar{#1}{1}}}
\newcommand*\wide@bar[2]{\if@single{#1}{\wide@bar@{#1}{#2}{1}}{\wide@bar@{#1}{#2}{2}}}
\newcommand*\wide@bar@[3]{%
  \begingroup
  \def\mathaccent##1##2{%
    \let\mathaccent\save@mathaccent
    \if#32 \let\macc@nucleus\first@char \fi
    \setbox\z@\hbox{$\macc@style{\macc@nucleus}_{}$}%
    \setbox\tw@\hbox{$\macc@style{\macc@nucleus}{}_{}$}%
    \dimen@\wd\tw@
    \advance\dimen@-\wd\z@
    \divide\dimen@ 3
    \@tempdima\wd\tw@
    \advance\@tempdima-\scriptspace
    \divide\@tempdima 10
    \advance\dimen@-\@tempdima
    \ifdim\dimen@>\z@ \dimen@0pt\fi
    \rel@kern{0.6}\kern-\dimen@
    \if#31
      \overline{\rel@kern{-0.6}\kern\dimen@\macc@nucleus\rel@kern{0.4}\kern\dimen@}%
      \advance\dimen@0.4\dimexpr\macc@kerna
      \let\final@kern#2%
      \ifdim\dimen@<\z@ \let\final@kern1\fi
      \if\final@kern1 \kern-\dimen@\fi
    \else
      \overline{\rel@kern{-0.6}\kern\dimen@#1}%
    \fi
  }%
  \macc@depth\@ne
  \let\math@bgroup\@empty \let\math@egroup\macc@set@skewchar
  \mathsurround\z@ \frozen@everymath{\mathgroup\macc@group\relax}%
  \macc@set@skewchar\relax
  \let\mathaccentV\macc@nested@a
  \if#31
    \macc@nested@a\relax111{#1}%
  \else
    \def\gobble@till@marker##1\endmarker{}%
    \futurelet\first@char\gobble@till@marker#1\endmarker
    \ifcat\noexpand\first@char A\else
      \def\first@char{}%
    \fi
    \macc@nested@a\relax111{\first@char}%
  \fi
  \endgroup
}

\NewDocumentCommand{\vdegbound}{oO{v}}{%
  \IfNoValueTF{#1}
    {h(\setcard{D_{1}(#2)})}
    {h(\setcard{D_{1}^{#1}(#2)})}}

\appto\HyLang@english{%
}

\newcommand\blfootnote{\gdef\@thefnmark{}\@footnotetext}
\makeatother

\newcommand*{\email}[1]{\href{mailto:#1}{\texttt{#1}}}

\begin{document}
\date{}
\title{A Central Limit Theorem for Functions
  on Weighted Sparse Inhomogeneous Random Graphs}
\author{Anja Sturm\thanks{%
    Georg-August-Universität Göttingen,
    Institute for Mathematical Stochastics,
    Goldschmidtstr.~7,
    37077 Göttingen, Germany.
    \email{anja.sturm@mathematik.uni-goettingen.de}}
  \and Moritz Wemheuer\thanks{%
    Georg-August-Universität Göttingen,
    Institute for Mathematical Stochastics,
    Goldschmidtstr.~7,
    37077 Göttingen, Germany.
    \email{moritz.wemheuer@uni-goettingen.de}
    \orcidlink{0009-0008-0032-6015}}}
\hypersetup{
  pdftitle={A Central Limit Theorem for Functions
            on Weighted Sparse Inhomogeneous Random Graphs},
  pdfauthor={Anja Sturm and Moritz Wemheuer},
}
\maketitle
\blfootnote{%
  The authors gratefully acknowledge financial support by the
  Deutsche Forschungsgemeinschaft (DFG, German Research Foundation)
  via the Research Training Group (RTG) 2088.
}

\begin{abstract}
    \noindent
    We prove a central limit theorem for a certain class of functions on sparse rank-one inhomogeneous random graphs endowed with additional
    i.i.d.~edge and vertex weights.
    Our proof of the central limit theorem uses a perturbative form of Stein's method and relies on a careful analysis of the local structure of the underlying sparse inhomogeneous random graphs
    (as the number of vertices in the graph tends to infinity),
    which may be of independent interest, as well as a local approximation property of the function, which is satisfied for a number of combinatorial optimisation problems.
    These results extend recent work by
    Cao (2021) for Erdős--Rényi random graphs and additional i.i.d.~weights only on the edges.
\end{abstract}

\tableofcontents

\section{Introduction}

The main aim of this work
is to establish a central limit theorem
for functions on weighted sparse rank-one inhomogeneous random graphs.
This result extends recent work by \citet{cao}
for sparse Erdős--Rényi random graphs with weights only on edges.

In the classical Erdős--Rényi random graph model,
edges between pairs of vertices are present independently of each other
with the same fixed probability.
Many real-world graphs and networks are more accurately modelled with
random graphs in which the probability that an edge forms
between two vertices is not the same for all pairs of vertices
and may depend on additional properties of the vertices.
These models allow for a greater irregularity of the resulting
graph and are therefore called \emph{inhomogeneous random graphs}
\citep{soderberg,bollobas,vdh,vdh:2}.
Our work focuses on sparse rank-one inhomogeneous random graphs,
in which the edge probability is proportional
to the product of the \emph{connectivity weights}
of its end vertices divided by the total number of vertices.
This model includes sparse Erdős--Rényi random graphs
and variations of the Chung--Lu and Norros--Reittu models
\citep{chung-lu,norrosreittu}.
Additionally,
we will endow our random graph model with independent
edge and vertex weights.

While we formulate our central limit theorem
generally for functions on weighted graphs
satisfying a certain \emph{good local approximation property},
we will note that interesting applications of this theorem
will often be related to combinatorial optimisation
problems on weighted graphs (formalised as
real-valued functions).
Verification of the good local approximation property
will often involve considering the recursive structure
of the function involved.
The good local approximation property is related
to concepts that have been called \emph{long-range independence}
\citep{gamarnik}, \emph{endogeny} \citep{aldous:rtp}
and \emph{replica symmetry} \citep{waestlund}
in the literature.

In the Erdős--Rényi setting law of large numbers-like results
for certain quantities of the graph
and a number of combinatorial optimisation problems
have been well studied \citep{er:evo,karp,gamarnik,bayati}.
There are also results for much more general sparse graph
settings, but especially for combinatorial optimisation problems
those are not as abundant \citep{bollobas,bord}.
A number of results are also known for the mean field model,
where the underlying graph is a complete graph
with i.i.d.~edge weights,
which reduces to an Erdős--Rényi setting
for certain (relaxed) optimisation problems
\citep{waestlund,waestlund:tsp,aldous:zeta}.

For some of the graph quantities
central limit theorems have been established
in the sparse Erdős--Rényi random graphs
\citep{stepanov,pittel,pittelwormald,bollobasriordan,barraez}.
In that setting a central limit theorem
has also been shown for the maximal matching problem (without weights)
\citep{glasgow,kreacic,pittel,aronson}.
As far as we are aware Cao's work \cite{cao} provided the first results regarding central limit theorems for  general classes of functions in
the sparse Erdős--Rényi setting with additional edge weights.

Elsewhere, \citet{barbour} recently proved a general central limit theorem
for \emph{local} graph statistics in the configuration model.
The configuration model generates a random graph with a given
degree sequence.
In fact, conditional on its degrees the inhomogeneous
graph model we consider here (and more general inhomogeneous graph models)
have the same distribution as a configuration model conditioned
on producing no loops or multiple edges \citep[Thm.~7.18]{vdh}.
\citet{rucinski} established a central limit theorem for subgraph counts
in non-sparse Erdős--Rényi settings.
\citet{maugis} extended this to
a related class of inhomogeneous random graphs.

In \autoref{chap:result} we introduce some basic notions,
present the general setting in mathematical detail
and state and discuss our central limit theorem
(\autoref{thm:fgnconv}).
\autoref{chap:structure}
is dedicated to the analysis of the local structure of
sparse rank-one inhomogeneous random graphs.
These results
are required to establish the central
limit theorem in our setting,
but they may also be of independent interest.
In particular we investigate the sizes of the neighbourhoods of vertices
in some detail (see \autoref{sec:nbhdsize})
and show that the neighbourhoods are generally only
weakly correlated (see \autoref{sec:corr}).
We also establish explicit coupling results between
the local neighbourhood of a vertex in the graph
and a limiting Galton--Watson tree
(\autoref{prop:maincoup}).
In \autoref{chap:proof}
we will follow Cao's strategy and prove the central limit
theorem via the (generalised) perturbative Stein's method
introduced by \citet{chatterjee:new,chatterjee}.

\section{Setting and statement of the main results}\label{chap:result}

\subsection{Rank-one sparse inhomogeneous random graphs}\label{sec:setting}

We now  describe our setting in more detail.
First we define the rank-one inhomogeneous graph model
that we will focus on.
All of this happens on an
underlying probability space~\((\Omega,\mathcal{F},\prob)\).

\begin{definition}[Rank-one inhomogeneous random graph]\label{def:irg}
For~\(n \in \naturals_{+} = \set{1,2,\dots}\)
let~\(G_{n} = (V_{n},E_{n})\)
be a graph with vertex set~\(V_{n} = [n] = \set{1,\dots,n}\).
The edge set~\(E_{n}\) is generated as follows.

We assign a possibly random \emph{connectivity
weight}~\(W^{n}_{v} \in \intervaloo{0}{\infty}\)
to each vertex~\(v \in V_{n}\).
Given these connectivity weights
we realise independent edges
between all (unordered) pairs of vertices~\(u\) and~\(v\)
with probability
\begin{equation}\label{eq:puvdef}
  p^{n}_{uv} = \frac{W^{n}_{u} W^{n}_{v}}{n\vartheta} \imin 1,
\end{equation}
where~\(\Lambda_{n} = \sum_{u \in V_{n}} W^{n}_{u}\)
and we assume that~\(\vartheta \in \intervaloo{0}{\infty}\) satisfies
\begin{equation}\label{eq:thetaconv}
  \frac{1}{n}\Lambda_{n}
  = \frac{1}{n}\sum_{u' \in V_{n}} W ^{n}_{u'}
  \overset{\prob}{\to} \vartheta
  \quad\text{as~\(n \to \infty\)}
\end{equation}
with~\(\overset{\prob}{\to}\) denoting convergence in probability.
We will generally make stronger assumptions about the distribution
of the connectivity weights (which will be detailed in \autoref{ass:coupling}
in just a moment),
so we will not highlight this assumption here in more detail.

\begin{notation}
  Let~\(\mathbf{W}^{n} = (W^{n}_{v})_{v\in V_{n}}\)
  be the collection of all connectivity weights
  for vertices in~\(V_{n}\)
  and let~\(\mathcal{F}_{n} = \sigma(\mathbf{W}^{n})
  = \sigma((W^{n}_{v})_{v \in V_{n}})\)
  be the~\(\sigma\)-algebra generated by all connectivity weights
  for vertices in~\(V_{n}\).
  From now on we will write~\(
  \prob_{n}(\placeholder) = \prob(\;\placeholder \given \mathcal{F}_{n})
  \)
  and~\(
    \expe_{n}[\placeholder] = \expe[\;\placeholder \given \mathcal{F}_{n}]
  \)
  for the probability measure and expectation
  conditioned on the connectivity weights~\(W^{n}_{1},\dots,W^{n}_{n}\).

  We will also drop the superscript~\(n\)
  from~\(W^{n}_{v}\) and~\(p^{n}_{uv}\)
  to make formulas slightly easier on the eye.
  The~\(n\) will be clear from the context.
\end{notation}

Formally, let~\(V_{n}^{(2)} = \sset{\edge{u}{v}}{u,v \in V_{n}}\)
be the maximal set of edges that~\(G_{n}\)
could possibly have,
i.e.~the set of edges of the complete graph on~\(V_{n}\).
Conditional on~\(\mathcal{F}_{n}\)
let~\(X_{uv} \sim \mathrm{Bin}(1,p_{uv})\)
for~\(1 \leq u < v \leq n\) be independent indicator functions
(the \emph{edge indicators}).
We will write~\(X_{uv} = X_{vu}\)
whenever~\(u, v \in V_{n}\), ~\(u \neq v\),
and set~\(X_{vv}=0\) for all~\(v \in V_{n}\).
The set of edges of~\(G_{n}\) is then given by~\(
  E_{n}
  = \sset{\edge{u}{v} \in V_{n}^{(2)}}{X_{uv}=1}
\).
\end{definition}

\begin{remark}
  This model for~\(G_{n}\) is related~--~but in this formulation not exactly equal~--~to
  the Chung--Lu model \citep{chung-lu},
  where vertices are connected independently with probability~\(
    \frac{W_{u}W_{v}}{\Lambda_{n}} \imin 1
  \),
  and the Norros--Reittu model \citep{norrosreittu},
  where the edge probability is~\(
    1-\exp(-W_{u}W_{v}/\Lambda_{n})
  \).
  By the assumed convergence of~\(n^{-1} \Lambda_{n}\)
  to~\(\vartheta\) and since~\(1-\exp(-x) \approx x\)
  for~\(x \to 0\), however,
  the edge probabilities will be very similar for large~\(n\).

  The classical Erdős--Rényi model
  with~\(p_{n} = n^{-1} \lambda_{n}\)
  for a sequence~\((\lambda_{n}) \subseteq \intervaloo{0}{\infty}\)
  with~\(\lambda_{n} \to \lambda \in \intervaloo{0}{\infty}\)
  as~\(n\to\infty\)
  can
  be obtained by setting~\(W_{v} = (\lambda_{n}\lambda)^{1/2}\)
  for all~\(v \in V_{n}\),
  so that~\(\vartheta = \lambda\)
  and~\(p_{uv}=n^{-1}\lambda_{n} = p_{n}\)
  as desired.

In the framework of inhomogeneous sparse random graphs
by \citet{bollobas}
our graph is a so-called \emph{rank-one model},
since its connection kernel~\(\kappa(x,y) = xy/\vartheta\)
has a simple product form.
\end{remark}

In order to allow us to
identify a limiting object for the graph~\(G_{n}\)
we will have to impose some conditions
on the connectivity weight distribution.

\begin{assumption}\label{ass:coupling}
  Let~\((G_{n})_{n \in \naturals}\) be a sequence of graphs as defined above.

  Given~\(\mathcal{F}_{n}\) let
  \begin{equation}\label{eq:nunnunhat}
    \nu_{n}(\placeholder)
    = \frac{1}{n}\sum_{v \in V_{n}} \indfunc_{\set{W_{v} \in
        \placeholder}}
    \quad
    \text{and}
    \quad
    \sizebias{\nu}_{n}(\placeholder)
    = \frac{1}{\Lambda_{n}} \sum_{v \in V_{n}} W_{v}
    \indfunc_{\set{W_{v} \in \placeholder}}
  \end{equation}
  be the empirical distribution function
  of the connectivity weights and its
  size-biased version.

  Assume that there is a measure~\(\nu\) on~\(\intervaloo{0}{\infty}\)
  with mean in~\(\intervaloo{0}{\infty}\)
  that satisfies the following properties.
  \begin{enumerate}
  \item\label{eq:wnun} There exists a sequence~\((\alpha_{n})_{n \in \naturals}\)
    that converges to zero in probability such that~\(
    \wasserstein(\nu_{n},\nu) \leq \alpha_{n}\),
  where~\(\wasserstein(\mu,\nu)\) denotes the~\(1\)-Wasserstein distance between
  the measures~\(\mu\) and~\(\nu\) on~\(\reals\)
  \begin{gather*}
    \wasserstein(\mu,\nu)
    = \inf\sset[\big]{\expe_{n}[\abs{X-Y}]}
    {X\sim \mu, Y \sim \nu, \;\text{\(X, Y\)
        defined on~\((\Omega,\mathcal{F}_{n},\prob_{n})\)}}.
  \end{gather*}
  \item\label{itm:nuhat} Let~\(\sizebias{\nu}\) be the size-biased distribution
  of~\(W \sim \nu\) given by~\(
    \sizebias{\nu}(A)
    = (\expe[W])^{-1}
    \expe[W \indfunc_{\set{W \in A}}]
  \).
  Then with the same sequence~\((\alpha_{n})_{n \in \naturals}\) as
  in~\ref{eq:wnun} we also have~\(
    \wasserstein(\sizebias{\nu}_{n},\sizebias{\nu}) \leq \alpha_{n}
  \).
  \item\label{itm:thirdmoment} Furthermore,
   we assume that the third moment of~\(W^{(n)} \sim \nu_{n}\),
  i.e.~\(
    \expe_{n}[(W^{(n)})^{3}]
    = \frac{1}{n} \sum_{u \in W_{u}} W_{u}^{3}
  \),
  is bounded in probability.
  \end{enumerate}
\end{assumption}

Assumption \ref{eq:wnun}
is in particular satisfied with an~\(\alpha_{n}\)
of rate~\(n^{-1/2}\)
if the weights~\(W_{v}\) are drawn i.i.d.~from the distribution~\(\nu\)
assuming that~\(\nu\) has second moments
\citep[Thm.~1]{fournier}.
It is tempting to conjecture that the rate of convergence
for~\(\wasserstein(\sizebias{\nu}_{n},\sizebias{\nu})\)
should be similar under the same moment conditions
for~\(\sizebias{\nu}\),
which would translate into the existence of third
moments for~\(\nu\).
We do not attempt to address this question further,
we will just mention that \citet[in a slightly different setting in
proof of Lem.~4.8]{olveracravioto}
briefly argues that~\(\wasserstein(\nu_{n},\nu) \overset{\prob}{\to} 0\)
implies~\(\wasserstein(\sizebias{\nu}_{n},\sizebias{\nu})
  \overset{\prob}{\to} 0\)
without claims on the rate of the latter convergence.
Intuitively, this is true because size-biasing respects
convergence in distribution \citep[Thm.~2.3]{sizebias}
if the means converge as well
and we can then use Skorohod's representation theorem
to obtain coupled random variables with the desired distributions
(possibly on a new probability space).

\citet{olveracravioto} constructed the couplings between
the graph and the limiting object
assuming only existence of the first moments.
We decided to work under the stronger assumptions presented here
because they give us more explicit control over the rate of convergence
and make the construction of the coupling slightly more natural.

We will use the notation~\(W^{(n)} \sim \nu_{n}\)
and~\(W \sim \nu\)
to recall the definition of~\(\Lambda_{n}
= \sum_{u \in V_{n}}W_{u} = n\expe_{n}[W^{(n)}]
\).
If~\(W^{(n)}\) and~\(W\)
are constructed via the optimal coupling guaranteed
by the Wasserstein distance \citep[Thm.~1.7]{santambrogio},
we have~\(
  \abs{n^{-1}\Lambda_{n}-\expe_{n}[W]}
  \leq \expe_{n}[\abs{W^{(n)}-W}]
  \leq \alpha_{n}
\).
This implies that~\(n^{-1}\Lambda_{n}\) converges in probability
to~\(\expe_{n}[W]\),
so that we can set~\(\vartheta = \expe_{n}[W]\) for~\eqref{eq:puvdef}.

Fix~\(p \in \intervaloo{0}{\infty}\). Now define
\begin{equation}\label{eq:gamman}
  \Gamma_{p,n}
  = \frac{1}{n\vartheta} \sum_{u \in V_{n}} W_{u}^{p}
  = \frac{\expe_{n}[(W^{(n)})^{p}]}{\vartheta}
\end{equation}
for the average~\(p\)-th power of the connectivity weights
normalised with~\(\vartheta\)
and
\begin{equation}\label{eq:kappan}
  \kappa_{p,n}
  = \frac{1}{n\vartheta} \sum_{i = 1}^{n}
  W_{i}^{p}\indfunc_{\set{W_{i} > \sqrt{n\vartheta}}}
  = \frac{1}{\vartheta}\expe_{n}[(W^{(n)})^{p}\indfunc_{\set{W^{(n)} > \sqrt{n\vartheta}}}]
\end{equation}
for the average excess of~\(p\)-th power of the connectivity weights
above~\(\sqrt{n\vartheta}\) normalised with~\(\vartheta\).

By \ref{itm:thirdmoment} we immediately have that~\(\Gamma_{p,n}\)
is bounded in probability for all~\(p \in \intervalcc{0}{3}\).
For~\(\kappa_{p,n}\) observe that if~\(p \in \intervalco{0}{3}\)
by Hölder's and Markov's inequality~\(
  \kappa_{p,n}
  \leq \Gamma_{3,n} (n\vartheta)^{-(3-p)/2}
\).
This term goes to zero in probability as~\(n\to\infty\)
if~\(p \in \intervalco{0}{3}\)
since~\(\Gamma_{3,n}\) is bounded in probability.
Note that if~\(W_{u} \leq \sqrt{n\vartheta}\)
and~\(W_{v} \leq \sqrt{n\vartheta}\)
it follows that~\(W_{u}W_{v} \leq n\vartheta\).
This implies that the minimum with~\(1\) in the
definition~\eqref{eq:puvdef} of~\(p_{uv}\) is not needed
in this case.
Hence,~\(\kappa_{p,n}\) measures the~\(p\)-th moment of
the connectivity weight of the vertices
exceed this \enquote{safe} threshold.

Analogous to~\(\Gamma_{p,n}\) we also define~\(\Gamma_{p}\)
as the \enquote{normalised}~\(p\)-th moment of~\(\nu\).
Fix~\(p\), let~\(W \sim \nu\)
and set
\(
\Gamma_{p} = \vartheta^{-1}\expe[W^{p}]
= (\expe[W])^{-1}\expe[W^{p}]
\).

We now add vertex and edge weights to our graph model
to obtain a weighted graph.
For us a weighted graph is a graph with additional values
associated to vertices and edges.
\begin{definition}[Weighted graph]
  A \emph{weighted graph}~\(\mathbf{G}\) is an ordered
  tuple~\(\mathbf{G}=(V,E,\mathbf{w})\),
  where~\(G=(V,E)\) is a graph (the underlying graph)
  and~\(\mathbf{w} \colon V \setunion E \to \reals\)
  is a function that assigns a real-valued weight to
  each vertex~\(v \in V\) and each edge~\(e \in E\).
\end{definition}

In our model
these weights are added independently of the underlying graph
structure in an i.i.d.\ manner.

\begin{definition}[Weighted rank-one inhomogeneous
random graph]\label{def:wirg}
Fix a graph~\(G_{n}\) as defined in \autoref{def:irg}
and fix two distributions on the non-negative real numbers~\(\mu_{V,n}\)
and~\(\mu_{E,n}\).
Then assign i.i.d.\ vertex weights~\(w^{(n)}_{v} \sim \mu_{V,n}\)
to each~\(v \in V_{n}\)
and i.i.d.\ edge weights~\(w^{(n)}_{e} \sim \mu_{E,n}\)
to each~\(e \in V_{n}^{(2)}\).%
\footnote{Technically, we would only need to assign weights to
  edges~\(e \in E_{n}\) that are actually present in~\(G_{n}\),
  but it is more convenient to assign a weight to all \enquote{possible}
  edges~\(e \in V_{n}^{(2)}\) \enquote{just in case}.}
\end{definition}

The distributions of the vertex and edge weights
may depend on~\(n\),
so in order to identify a limiting object we will
assume that the weight distributions converge to a limiting
distribution.
In particular we will assume that~\(\mu_{V,n}\) converges to some~\(\mu_{V}\)
and~\(\mu_{E,n}\) to some~\(\mu_{E}\) in total variation distance,
i.e. that
\[
  d_{\mathrm{TV}}(\mu_{V,n},\mu_{V}) \to 0
  \quad\text{and}\quad
  d_{\mathrm{TV}}(\mu_{E,n}, \mu_{E}) \to 0,
\]
where the total variation distance for two probability
measures~\(\mu\) and~\(\lambda\)
on a measurable space~\((\Omega', \mathcal{F}')\) is given by~\(
  d_{\mathrm{TV}}(\mu, \nu) =
  \inf\sset{\prob(X \neq Y)}{X \sim \mu, Y \sim \nu}
\).

\subsection{Local structure and limit of inhomogeneous random graphs}
We now describe the local limiting behaviour of the rank-one
inhomogeneous graph model without weights.
This is done by showing that the local neighbourhood
of a vertex~\(v\) in~\(G_{n}\) can be coupled with
high probability to a \enquote{delayed} Galton--Watson tree.
Since the weights are added to the model independently of the underlying
structure the extension to the weighted case will be straightforward.

\begin{definition}[Local neighbourhood]\label{def:nbhd}
  Let~\(G = (V,E)\) be a graph.
  For a vertex~\(v \in  V\)
  and level~\(\ell \in \naturals\)
  denote by~\(B_{\ell}(v, G)\)
  the \emph{(local) neighbourhood of the vertex~\(v\)}
  up to level~\(\ell\)
  in the graph~\(G\).
  Formally, we define~\(B_{\ell}(v, G)\) as the subgraph of~\(G\)
  induced by the union of all paths starting in~\(v\)
  that are no longer than~\(\ell\) steps.
\end{definition}

This definition can be extended trivially to weighted graphs~\(\mathbf{G}\)
by including a restriction of the weight function~\(\mathbf{w}\)
to the relevant vertices and edges,
in which case we write~\(B_{\ell}(v, \mathbf{G})\).
When the context is clear, we
will sometimes drop the reference to the underlying graph~\(G\)
completely and will just write~\(B_{\ell}(v)\) instead of~\(B_{\ell}(v,G)\).

\begin{definition}[Isomorphism of rooted graphs]
  Let~\(G = (V,E)\)
  and~\(H = (U,F)\)
  be two weighted rooted graphs with root~\(v^{*} \in V\)
  and~\(u^{*} \in U\), respectively.
  We write~\(G \cong H\)
  if there exists a bijection~\(\varphi \colon V \to U\)
  that satisfies~\(\varphi(v^{*})=u^{*}\)
  and preserves edges,
  i.e.~\(\edge{\varphi(v_{1})}{\varphi(v_{2})} \in F\)
 if and only if~\(\edge{v_{1}}{v_{2}} \in E\).
\end{definition}
This notion naturally extends to weighted rooted
graphs~\(\mathbf{G} = (V,E,\mathbf{w})\)
and~\(\mathbf{H} = (U,F,\mathbf{x})\),
where we additionally require~\(\mathbf{x}(\varphi(v)) = \mathbf{w}(v)\)
for all~\(v \in V\)
and~\(\mathbf{x}(\varphi(e)) = \mathbf{w}(e)\)
for all~\(v \in E\).

For technical reasons we also need the concept of isomorphism
for pairs of graphs with the same vertex sets.
\begin{definition}[Isomorphism for pairs of rooted graphs]
  Let~\(G= (V,E)\) and~\(G' = (V', E')\) be two weighted rooted graphs
  with same root~\(v^{*} \in V \setintersect V'\).
  Let similarly~\(H = (U,F)\)
  and~\(H' = (U',F')\)
  be two further weighted rooted graphs
  with the same root~\(u^{*} \in U \setintersect U'\).
  We write~\((G,G') \cong (H,H')\)
  if there exists a bijection~\(\varphi \colon V \to U\)
  that satisfies~\(\varphi(v^{*})=u^{*}\)
  and preserves edges,
  i.e.~\(\edge{\varphi(v_{1})}{\varphi(v_{2})} \in F\)
  if and only if~\(\edge{v_{1}}{v_{2}} \in E\),
  and a bijection~\(\varphi' \colon V' \to U'\)
  that also maps~\(v^{*}\) to~\(u^{*}\)
  and preserves edges in the way just described.
\end{definition}
Again, the notion naturally extends to weighted graphs
by requiring that weights be preserved by~\(\varphi\)
and~\(\varphi'\).

\begin{definition}\label{def:limtree:noweight}
  For a probability measure~\(\nu\) on~\(\intervaloo{0}{\infty}\)
  and a connectivity weight~\(W \in \intervaloo{0}{\infty}\)
  let~\(T(W,\nu)\) be a Galton--Watson
  tree
  in which the root has~\(\mathrm{Poi}(W)\) children
  and all other levels have offspring
  distribution~\(\mathrm{MPoi}(\sizebias{\nu})\),
  where~\(\mathrm{MPoi}(\mu)\) denotes a mixed Poisson
  distribution with mixing distribution~\(\mu\)
  (where~\(X \sim \mathrm{MPoi}(\mu)\)
   means
   \(\prob(X=k) = \expe[\exp(-\Lambda)  \Lambda^{k}/k!]\)
   for all~\(k \in \naturals\)
   when~\(\Lambda \sim \mu\))
  and~\(\sizebias{\nu}\) denotes the size-biased
  distribution of~\(\nu\).

  For an integer~\(\ell \in\naturals\) let~\(T_{\ell}(W,\nu)\)
  be the subtree of~\(T(W,\nu)\) cut at height~\(\ell\)
  (or alternatively the~\(\ell\)-neighbourhood of the
   root~\(B_{\ell}(\treeroot, T(W,\nu))\)).
\end{definition}

\begin{remark}
  Note that~\(T(W,\nu)\)
  can be constructed
  by joining~\(N \sim \mathrm{Poi}(W)\)
  independent Galton--Watson trees~\(T^{(i)}\) for~\(i \in \set{1,\dots,n}\)
  with root~\(\treeroot_{i}\)
  and offspring distribution~\(\mathrm{MPoi}(\sizebias{\nu})\) for all levels
  together at a root~\(\treeroot\)
  with edges~\(\edge{\treeroot}{\treeroot_{1}},
  \dots,\edge{\treeroot}{\treeroot_{N}}\).
\end{remark}

\citet{olveracravioto} calls such a tree process \enquote{delayed},
because the root has a different offspring distribution
than all other individuals \citep[see also][]{esker}.

This limiting tree is closely connected to the local weak
limit of the graph~\(G_{n}\) \citep[Chap.~2]{vdh}.
The notion of local limits for graphs
was introduced by \citet{benjamini}
and later by \citet{aldoussteele},
who used it extensively to develop the so-called objective method,
in which the limiting properties of a sequence of finite problems
are analysed in terms of local properties of a new infinite
object.
In our treatment we keep the vertex whose neighbourhood we explore
fixed, whereas in the context of local weak limits
this vertex is chosen uniformly at random.
We can think of our setup as conditioning on the type of the
root vertex, so that the
usual local weak limit can then be recovered from our results
by averaging over all vertices (and possibly
adjusting the coupling of the root vertex).
Specifically, the resulting tree would
have a root with~\(\mathrm{MPoi}(\nu)\)
children, while all other individuals have offspring
distribution~\(\mathrm{MPoi}(\sizebias{\nu})\).
Such a tree process is called \emph{unimodular} \citep{vdh:2}.

One of the main results about the structure of the inhomogeneous
random graph considered here is
an explicit  coupling construction that yields:
\begin{prop}\label{prop:maincoup}
  Let~\((G_{n})_{n \in \naturals}\) be a sequence of
  rank-one inhomogeneous random graph
  that satisfies \autoref{ass:coupling}
  for some measure~\(\nu\) on~\(\intervaloo{0}{\infty}\).

  Let~\(\mathcal{V} \subseteq V_{n}\) be a set of vertices.
  Then for all~\(\ell \in \naturals\)
  the neighbourhoods around~\(v \in \mathcal{V}\)
  can be coupled to independent limiting
  trees~\(\mathcal{T}(v) \sim T(W_{v},\nu) \)
  such that for all~\(n \in \naturals\)
  \begin{align*}
    \prob_{n}\biggl(
    \bigsetunion_{v \in \mathcal{V}} \set[\big]{
    B_{\ell}(v) \ncong \mathcal{T}_{\ell}(v)
    }
    \biggr)
  &\leq
    \frac{\Gamma_{2,n}}{n\vartheta} \sum_{v\in \mathcal{V}} W_{v}^{2}
    +\Gamma_{1,n} \sum_{v \in \mathcal{V}} W_{v} \indfunc_{\set{W_{v} > \sqrt{n\vartheta}}}\\
  &\quad  +(\Gamma_{2,n}+1)^{\ell}
    \biggl(
      \frac{\Gamma_{3,n}}{n\vartheta}+\kappa_{1,n}+\kappa_{2,n}+
    \frac{2+\Gamma_{1,n}}{k_{n}}+\frac{k_{n}}{n\vartheta}\biggr)
    \sum_{v \in \mathcal{V}} W_{v}\\
  &\quad + \setcard{\mathcal{V}}\frac{1}{k_{n}}
    +\frac{k_{n}^{2}}{n\vartheta\Gamma_{1,n}}
    +\sum_{v\in \mathcal{V}} W_{v}\alpha_{n}
    \biggl(
      \frac{1}{\vartheta}
      + (\Gamma_{2}+1)^{\ell-1}
        \biggl(\frac{\Gamma_{2,n}}{\vartheta\Gamma_{1,n}}+1\biggr)
    \biggr),
  \end{align*}
  where~\((k_{n})_{n \in\naturals} \subseteq \intervaloo{0}{\infty}\)
  is an arbitrary sequence of positive real numbers.
\end{prop}

By \autoref{ass:coupling}~\(\Gamma_{1,n}\), \(\Gamma_{2,n}\)
and~\(\Gamma_{3,n}\) are bounded in probability
and~\(\alpha_{n}\), \(\kappa_{1,n}\) and~\(\kappa_{2,n}\) converge to zero
in probability.
Additionally,~\(\sum_{v \in \mathcal{V}} W_{v}\indfunc_{\set{W_{v} > \sqrt{n\vartheta}}}\)
is zero if~\(n\) is large enough for any finite set~\(\mathcal{V}\).
Then the probability that the coupling does not
hold goes to zero in probability if~\((k_{n})_{n \in \naturals}\)
is chosen appropriately.
In particular the sequence needs to satisfy~\(k_{n} \to \infty\)
as well as~\(k_{n}^{2}/n \to 0\).
The choice~\(k_{n} \approx n^{1/3}\)
balances the rate of~\(1/k_{n}\) and~\(k_{n}^{2}/n\)
so that both are of order~\(n^{-1/3}\).

This lemma can be used to obtain a coupling
for neighbourhoods in the Erdős--Rényi model
to a Galton--Watson tree with simple Poisson offspring distribution.
\begin{example}
  Consider the Erdős--Rényi model with
  edge probability~\(p_{n} = \lambda_{n}/n\)
  (and~\(\lambda_{n} \to \lambda\)).
  Recall that we had to set~\(W_{v} = (\lambda_{n}\lambda)^{1/2}\)
  to obtain the desired edge probabilities.

  Then~\(\nu_{n} = \delta_{(\lambda_{n}\lambda)^{1/2}}\)
  and~\(\nu = \delta_{\lambda}\)
  and by basic properties of the Wasserstein distance
  and the square root function
  \[
    \wasserstein(\nu_{n},\nu) = \abs{\lambda_{n}^{1/2}\lambda^{1/2}-\lambda}
    = \lambda^{1/2} \abs{\lambda_{n}^{1/2}-\lambda^{1/2}}
    \leq \abs{\lambda_{n}-\lambda}
  \]
  The size-biased measures coincide
  with the original measures so that also
  \[
    \wasserstein(\sizebias{\nu}_{n},\sizebias{\nu})
    = \wasserstein(\nu_{n},\nu)
    = \abs{\lambda_{n}-\lambda}.
  \]
  We thus set~\(\alpha_{n} = \abs{\lambda_{n}-\lambda}\).
  Furthermore, we may assume that~\(\kappa_{p,n} = 0\)
  and~\(\sum_{v \in \mathcal{V}} W_{v} \indfunc_{\set{W_{v} > \sqrt{n\vartheta}}} = 0\),
  because~\(W_{v} > \sqrt{n\vartheta}\)
  if and only if~\(\lambda_{n} > n\)
  which is not the case for~\(n\)
  large enough as~\(\lambda_{n} \leq 2\lambda < n\)
  for all~\(n\) large enough.
  Finally~\(\Gamma_{p,n} = \lambda_{n}^{1/2p}\lambda^{1/2p-1}\)
  and~\(\Gamma_{p} = \lambda^{p-1}\)
  in particular~\(\Gamma_{2,n}=\lambda_{n}\) and~\(\Gamma_{2}=\lambda\).

  The trees~\(\mathcal{T}(v)\) are just independent
  Galton--Watson trees with
  \(\mathrm{Poi}(\lambda_{n}^{1/2}\lambda^{1/2})\) children at the root
  and offspring distribution~\(\mathrm{Poi}(\lambda)\)
  for all other individuals.
  For~\(\mathcal{V} = \set{v,u}\) \autoref{prop:maincoup}
  then reduces to
  \begin{align*}
    &\prob_{n}(
      \set{B_{\ell}(v) \neq \mathcal{T}_{\ell}(v)}
      \setunion
      \set{B_{\ell}(u) \neq \mathcal{T}_{\ell}(u)}
    \\
    &\quad\leq 2\frac{\lambda_{n}^{2}}{n}
      + 2\lambda_{n}^{1/2}\lambda^{1/2}
      (\lambda_{n}+1)^{\ell}
      \biggl(\frac{\lambda_{n}^{3/2}}{n\lambda^{1/2}}
      +\frac{2+\lambda_{n}^{1/2}/\lambda^{1/2}}{k_{n}}+
    \frac{k_{n}}{n\lambda}\biggr)\\
    &\qquad+2\frac{1}{k_{n}}+2\frac{k_{n}^{2}}{n\lambda_{n}^{1/2}\lambda^{1/2}}
      +2\lambda_{n}^{1/2}\lambda^{1/2}\abs{\lambda_{n}-\lambda}
      \biggl(
      \frac{1}{\lambda}
      +(\lambda+1)^{\ell-1}\biggl(\frac{\lambda_{n}^{1/2}}{\lambda^{1/2}}+1\biggr)
      \biggr)
  \end{align*}
  for~\(n\) large enough.
  Choose~\(k_{n} = n^{1/3}\),
  then this bound can be estimated by
  \[
    \prob_{n}(
      \set{B_{\ell}(v) \neq \mathcal{T}_{\ell}(v)}
      \setunion
      \set{B_{\ell}(u) \neq \mathcal{T}_{\ell}(u)}
      )
    \leq C \frac{(\lambda_{n}+1)^{\ell+2}}{\min\set{1,\lambda}n^{1/3}}
    +C \frac{(\lambda+1)^{\ell}}{\min\set{1,\lambda}}
    \abs{\lambda_{n}-\lambda},
  \]
  which is of the same order as the coupling
  probability that \citeauthor{cao}
  established for Erdős--Rényi random graphs \citep[Lem.~6.1]{cao}.

  Note that we coupled the neighbourhoods to Galton--Watson trees
  whose offspring distribution~\(\mathrm{Poi}((\lambda_{n}\lambda)^{1/2})\)
  at the root
  differs from the offspring distribution~\(\mathrm{Poi}(\lambda)\)
  of all other individuals.
  The classical coupling for Erdős--Rényi random graphs
  that \citeauthor{cao} established
  couples the neighbourhood to Galton--Watson trees
  with offspring distribution~\(\mathrm{Poi}(\lambda)\) for all individuals.
  If we wanted to obtain this classical coupling,
  we would have to modify (or re-couple) the offspring distribution
  of the root at a cost
  of~\(\wasserstein(\delta_{\lambda_{n}\lambda},\delta_{\lambda})
  \leq \abs{\lambda_{n}-\lambda}\).
  This additional cost does not change the rate estimate.
\end{example}

In order to describe the local limit of
the weighted graph
we just need to add vertex and edge weight
to the limiting object identified for \autoref{prop:maincoup}.
The limiting object will be the same as in \autoref{def:limtree:noweight}
just with added weights.
\begin{definition}\label{def:limtree:weights}
  Given two weight distributions~\(\mu_{E}\)
  and~\(\mu_{V}\) let~\(\mathbf{T}(W,\nu,\mu_{E},\mu_{V})\)
  be the Galton--Watson tree~\(T(W,\nu)\)
  endowed with i.i.d.\ edge weights drawn from~\(\mu_{E}\)
  and i.i.d\ vertex weights drawn from~\(\mu_{V}\).
  For~\(\ell \in \naturals\)
  let~\(\mathbf{T}_{\ell}(W,\nu,\mu_{E},\mu_{V})\) denote the~\(\ell\)-level
  subtree of~\(\mathbf{T}(W,\nu,\mu_{E},\mu_{V})\).
\end{definition}

The following object arises from the limiting
object by conditioning on the presence of a certain edge.
\begin{definition}\label{def:ttilde}
  Let~\(\nu\) be a probability measure on~\(\intervaloo{0}{\infty}\)
  and~\(W, W' \in \intervaloo{0}{\infty}\) two connectivity weights.
  Fix an edge weight distribution~\(\mu_{E}\)
  and a vertex weight distribution~\(\mu_{V}\).

  Let~\(\mathbf{T}\sim\mathbf{T}(W,\nu,\mu_{E},\mu_{V})\)
  with root~\(\treeroot\)
  and~\(\mathbf{T}'\sim\mathbf{T}(W',\nu,\mu_{E},\mu_{V})\)
  with root~\(\treeroot'\)
  be independent.
  Construct~\(\mathbf{\tilde{T}}(W,W',\nu,\mu_{E},\mu_{V})\)
  by grafting~\(\mathbf{T}'\) onto~\(\mathbf{T}\)
  via an edge
  of weight~\(w \sim \mu_{E}\) (independent of everything else)
  between~\(\treeroot\) and~\(\treeroot'\).
  (In particular~\(\treeroot\) is the root
  of~\(\mathbf{\tilde{T}}(W,W',\nu,\mu_{E},\mu_{V})\).)
  Let~\(\mathbf{\tilde{T}}_{\ell}(W,W',\nu,\mu_{E},\mu_{V})\) be the
  depth-\(\ell\)
  subtree of~\(\mathbf{\tilde{T}}(W,W',\nu,\mu_{E},\mu_{V})\).

  Alternatively,~\(\mathbf{\tilde{T}}_{\ell} \sim
  \mathbf{\tilde{T}}_{\ell}(W,W',\nu,\mu_{E},\mu_{V})\) can directly be
  constructed
  from independent
  trees~\(\mathbf{T}_{\ell} \sim \mathbf{T}_{\ell}(W,\nu,\mu_{E},\mu_{V})\)
  and~\(\mathbf{T}'_{\ell-1} \sim \mathbf{T}_{\ell-1}(W',\nu,\mu_{E},\mu_{V})\)
  with roots~\(\treeroot\) and~\(\treeroot'\), respectively,
  by
  grafting~\(\mathbf{T}'_{\ell-1}\) onto~\(\mathbf{T}_{\ell}\)
  via an edge between~\(\treeroot'\) and~\(\treeroot\)
  of weight~\(w \sim \mu_{E}\)
  (independent of everything else).
  Whenever~\(\mathbf{\tilde{T}}_{\ell}\) is defined via this procedure,
  we say it is
  \emph{constructed
  from~\((\mathbf{T}_{\ell},\mathbf{T}'_{\ell-1},\treeroot,\treeroot',w)\)}.
\end{definition}
As alluded to above, this object can be thought of as
the limiting object of the neighbourhood of~\(v\)
if we condition on the presence of an edge between~\(v\) and~\(u\).

We will not state and discuss the coupling results in the weighted setting here,
because they are structurally similar
to \autoref{prop:maincoup}.
The intuition should be that in a first step the underlying
graph structure is coupled as in the unweighted case
and then edge and vertex weights are added.
Since the weight distributions converge in total variation distance,
the weights can be coupled so that they are equal with high
probability, and because the number of vertices and edges in
the neighbourhood can be estimated,
the probability that the coupled weights are different
can be controlled.
We refer the reader to \autoref{sec:compcoup} for more details.

\subsection{Statement of the Central Limit Theorem}

The proof of our central limit theorem relies on
the analysis of
the effect of a small perturbation to the weighted graph~\(\mathbf{G}_{n}\)
on a function~\(f\).
We introduce some notation to refer to the effect of this perturbation.

\begin{notation}
Given~\(\mathcal{F}_{n}\) the entire structure of~\(\mathbf{G}_{n}\)
defined in \autoref{def:wirg}
can be encoded in the following sequences of
independent random variables
\[
  (\mathbf{X}^{(n)},\mathbf{w}^{(n)})
  = ((X^{(n)}_{e})_{e \in V_{n}^{(2)}}, (w^{(n)}_{x})_{x \in V_{n} \setunion V_{n}^{(2)}}),
\]
where~\(X^{(n)}_{\edge{u}{v}} \sim \mathrm{Bin}(1,p_{uv})\)
for~\(\edge{u}{v} \in V_{n}^{(2)}\),
\(w^{(n)}_{v} \sim \mu_{V,n}\) for~\(v \in V_{n}\)
and~\(w^{(n)}_{e} \sim \mu_{E,n}\) for~\(e \in V_{n}^{(2)}\)
are all independent random variables.
We will usually drop the superscript~\((n)\) for all these objects.
Additionally,
we will use the notational convention that~\(X_{uv} = X_{vu} = X_{\edge{u}{v}}\)
and~\(w_{uv} = w_{vu} = w_{\edge{u}{v}}\) for all~\(u \neq v\).
\end{notation}

Let~\(\mathbf{X}'\) be an independent copy of~\(\mathbf{X}\)
and likewise~\(\mathbf{w}'\) be an independent~copy of~\(\mathbf{w}\).
Let~\(F\) be a subset of~\(V_{n} \setunion V_{n}^{(2)}\),
i.e.~sets of vertices and edges alike.

Let~\(\mathbf{G}_{n}^{F}\) be the weighted graph
obtained from~\(\mathbf{G}_{n}\) by replacing
\begin{itemize}
\item \(X_{e}\) with~\(X'_{e}\)
  whenever~\(e \in F\) and
\item \(w_{z}\)
  with~\(w'_{z}\)
  whenever~\(z \in F\).
\end{itemize}
For singleton sets~\(A\) we often omit the curly brackets
and simply write~\(\mathbf{G}_{n}^{e}\)
for~\(\mathbf{G}_{n}^{\set{e}}\)
and~\(\mathbf{G}_{n}^{v}\)
for~\(\mathbf{G}_{n}^{\set{v}}\).
We abuse notation even further to write~\(\mathbf{G}_{n}^{F \setunion
e}\)
for~\(\mathbf{G}_{n}^{F \setunion \set{e}}\)
and~\(\mathbf{G}_{n}^{F \setunion v}\)
for~\(\mathbf{G}_{n}^{F \setunion \set{v}}\).
For brevity we write
\[
  X^{F}_{\edge{u}{v}}
  = \begin{cases*}
    X'_{\edge{u}{v}} & if~\(\edge{u}{v} \in F\),\\
    X_{\edge{u}{v}} & if~\(\edge{u}{v} \notin F\),
  \end{cases*}
  \quad\text{and}\quad
  w^{A}_{z}
  = \begin{cases*}
    w'_{z} & if~\(z \in F\),\\
    w_{z} &  if~\(z \notin F\),
  \end{cases*}
\]
whenever~\(u,v \in V_{n}\)
and~\(z \in V_{n}  \setunion V_{n}^{(2)}\).
With this notation~\(\mathbf{G}_{n}^{F}\)
is the weighted graph based on the sequences~\((X^{F}_{e})_{e \in V_{n}^{(2)}}\)
and~\((w^{F}_{z})_{z \in V_{n} \setunion V_{n}^{(2)}}\)
instead of~\((X_{e})_{e \in V_{n}^{(2)}}\)
and~\((w_{z})_{z \in V_{n} \setunion V_{n}^{(2)}}\).
Note that by construction~\(X^{\emptyset}_{e}=X_{e}\) for~\(e \in V_{n}^{(2)}\)
and~\(w^{\emptyset}_{z}=w_{z}\) for~\(z \in V_{n} \setunion V_{n}^{(2)}\).
It follows that~\(\mathbf{G}_{n}^{\emptyset} = \mathbf{G}_{n}\).

\begin{definition}
  Let~\(\mathbf{G}_{n}\) be a weighted graph and let~\(f\)
  be a function on weighted graphs.
  Recall the definition of the perturbed graph~\(\mathbf{G}_{n}^{e}\)
  and~\(\mathbf{G}_{n}^{v}\) for an edge~\(e \in V_{n}^{(2)}\)
  and a vertex~\(v \in V_{n}\), respectively.
  Then define
  \[
  \Delta_{e}f
  = f(\mathbf{G}_{n})-f(\mathbf{G}_{n}^{e})
  \quad
  \text{and}
  \quad
  \Delta_{v}f
  = f(\mathbf{G}_{n})-f(\mathbf{G}_{n}^{v}).
  \]
\end{definition}

The main assumption of the theorem is that it is possible to
approximate the effect of resampling perturbations
on the function~\(f\) by considering local neighbourhoods
around the perturbed site,
i.e.~that we can find a \emph{good local approximation}
for the effects of the perturbation on~\(f\).
\begin{assumption}[Property~GLA]\label{def:gla}
  Let~\(f\) be a function on weighted graphs
  and let~\((\mathbf{G}_{n})_{n \in \naturals}\) be a sequence of weighted
  inhomogeneous random graphs.
  Then the pair~\((f, (\mathbf{G}_{n})_{n\in\naturals})\) has
  \emph{property~GLA} for~\(\nu\),~\(\mu_{E}\) and~\(\mu_{V}\)
  if
  \begingroup
  \raggedright
  \begin{enumerate}
  \item  the underlying unweighted graph sequence~\((G_{n})_{n\in\naturals}\)
    satisfies \autoref{ass:coupling},
  \item the weight distributions satisfy
  \(d_{\mathrm{TV}}(\mu_{E,n},\mu_{E}) \to 0\)
  and~\(d_{\mathrm{TV}}(\mu_{V,n},\mu_{V}) \to 0\) as~\(n \to \infty\) and
\item the effects of perturbations of~\(\mathbf{G}_{n}\) on the function~\(f\)
  can be approximated locally in the following sense.

  For all~\(k \in \naturals\) there exist
  functions~\(\mathrm{LA}^{E,L}_{k}\),~\(\mathrm{LA}^{E,U}_{k}\),
  \(\mathrm{LA}^{V,L}_{k}\) and~\(\mathrm{LA}^{V,U}_{k}\)
  from pairs of rooted weighted trees to the real numbers
  and
  furthermore there exist two sequences
  of functions~\(m^{E}_{n} \colon V_{n}^{2} \to \reals\)
  and~\(m^{V}_{n} \colon V_{n} \to \reals\)
  such that
  \[
    (M^{E}_{n})_{n\in\naturals}
    = \Bigl(n^{-2}\sum_{v,u \in V_{n}} m^{E}_{n}(v,u)\Bigr)_{n\in\naturals}
    \quad\text{and}\quad
    (M^{V}_{n})_{n\in\naturals}
    =\Bigl(n^{-1}\sum_{v\in V_{n}} m^{V}_{n}(v)\Bigr)_{n\in\naturals}
  \]
  are bounded in probability
  and two
  sequences~\((\delta^{E}_{k})_{k \in \naturals}, (\delta_{k}^{V})_{k \in \naturals}\)
  with~\(\delta^{E}_{k} \to 0\) and~\(\delta^{V}_{k} \to 0\) as~\(k\to \infty\)
  such that the following conditions hold for any~\(k\in\naturals\).
  \end{enumerate}

    \begin{enumerate}[label=(GLA~\arabic*), leftmargin=*]
    \item \label{itm:gla:e:delta}
      For any edge~\(e=\edge{u}{v} \in V_{n}^{(2)}\),
      if~\(B_{k} = B_{k}(v,\mathbf{G}_{n})\)
    and~\(B^{e}_{k} = B_{k}(v,\mathbf{G}_{n}^{e})\) are trees, then
    \[
    \mathrm{LA}^{E,L}_{k}(B_{k},B_{k}^{e})
    \leq \Delta_{e}f
    \leq \mathrm{LA}^{E,U}_{k}(B_{k},B_{k}^{e}).
    \]
    \item \label{itm:gla:e:tree}
      For any edge~\(e=\edge{u}{v} \in V_{n}^{(2)}\),
      if~\((\mathbf{T},\mathbf{T}^{e})\) is a pair of trees
    satisfying~\((B_{k},B_{k}^{e})
    = (B_{k}(v,\mathbf{G}_{n}), B_{k}(v, \mathbf{G}_{n}^{e}))
    \cong (\mathbf{T},\mathbf{T}')\), then
    \begin{equation*}
    \mathrm{LA}^{E,L}_{k}(\mathbf{T},\mathbf{T}')
      = \mathrm{LA}^{E,L}_{k}(B_{k},B_{k}^{e})
      \quad\text{and}\quad
    \mathrm{LA}^{E,U}_{k}(\mathbf{T},\mathbf{T}')
    = \mathrm{LA}^{E,U}_{k}(B_{k},B_{k}^{e}).
    \end{equation*}
    \item \label{itm:gla:e:conv}
      For any two vertices~\(v,u \in V_{n}\)
      let~\(\mathbf{\tilde{T}}_{k}(v,u) \sim
      \mathbf{\tilde{T}}_{k}(W_{v},W_{u},\nu,\mu_{E},\mu_{V})\) be constructed
      from~\((\mathbf{T}_{k}(v),\mathbf{T}_{k-1}(u), \treeroot, \treeroot',
    w)\) (cf.~\autoref{def:ttilde}).
    Then
    \[
    \begin{split}
    &\max\set{
      \expe_{n}[
      (\mathrm{LA}^{E,U}_{k}(\mathbf{\tilde{T}}_{k}(v,u),\mathbf{T}_{k}(v))
      -\mathrm{LA}^{E,L}_{k}(\mathbf{\tilde{T}}_{k}(v,u),\mathbf{T}_{k}(v)))^{2}
      ],\\
      &\qquad\quad\expe_{n}[
      (\mathrm{LA}^{E,U}_{k}(\mathbf{T}_{k}(v),\mathbf{\tilde{T}}_{k}(v,u))
      -\mathrm{LA}^{E,L}_{k}(\mathbf{T}_{k}(v),\mathbf{\tilde{T}}_{k}(v,u)))^{2}
      ]
    }\\
    &\quad\leq m^{E}_{n}(v,u)\delta^{E}_{k}.
    \end{split}
    \]
    \item \label{itm:gla:v:delta}
      For any vertex~\(v \in V_{n}\),
      if~\(B_{k}=B_{k}(v,\mathbf{G}_{n})\)
      and~\(B^{v}_{k}=B_{k}(v,\mathbf{G}_{n}^{v})\) are trees, then
    \[
    \mathrm{LA}^{V,L}_{k}(B_{k},B_{k}^{v})
    \leq \Delta_{v}f
    \leq \mathrm{LA}^{V,U}_{k}(B_{k},B_{k}^{v}).
    \]
    \item \label{itm:gla:v:tree}
    For any vertex~\(v \in V_{n}\),
    if~\((\mathbf{T},\mathbf{T}')\) are a pair of trees that
    satisfy~\((B_{k},B_{k}^{v})
    = (B_{k}(v,\mathbf{G}_{n}), B_{k}(v, \mathbf{G}_{n}^{v}))
    \cong (\mathbf{T},\mathbf{T}')\), then
    \begin{equation*}
    \mathrm{LA}^{V,L}_{k}(\mathbf{T},\mathbf{T}')
      = \mathrm{LA}^{V,L}_{k}(B_{k},B_{k}^{v})
    \quad\text{and}\quad
    \mathrm{LA}^{V,U}_{k}(\mathbf{T},\mathbf{T}')
    = \mathrm{LA}^{V,U}_{k}(B_{k},B_{k}^{v}).
    \end{equation*}
    \item \label{itm:gla:v:conv}
      For vertex~\(v \in V_{n}\)
      let~\(\mathbf{\bar{T}}_{k}(v)\) be the weighted tree obtained
    from~\(\mathbf{T}_{k}(v) \sim \mathbf{T}_{k}(W_{v},\nu,\mu_{E},\mu_{V})\)
    by resampling the weight of the root.
    Then
    \[
    \expe_{n}[
    (\mathrm{LA}^{V,U}_{k}(\mathbf{T}_{k}(v),\mathbf{\bar{T}}_{k}(v))
    -\mathrm{LA}^{V,L}_{k}(\mathbf{T}_{k}(v),\mathbf{\bar{T}}_{k}(v)))^{2}
    ] \leq m_{n}^{V}(v)\delta^{V}_{k}.
    \]
  \end{enumerate}
  \endgroup
\end{assumption}

As discussed in the previous sections
\autoref{ass:coupling} and convergence of the weight
distributions guarantee that
the local neighbourhoods in~\(\mathbf{G}_{n}\)
can be coupled to limiting Galton--Watson trees.

The three assumptions~\ref{itm:gla:v:delta}, \ref{itm:gla:v:tree}
\ref{itm:gla:v:conv} for resampling at a vertex (which involves only
resampling the weight at the vertex)
are structurally analogous to
\ref{itm:gla:e:delta},
\ref{itm:gla:e:tree} and~\ref{itm:gla:e:conv}
for resampling at an edge (which involves resampling
the edge indicator and its weight).
It would have been possible to collect
the conditions for edge and vertex resampling
in a combined condition (even though a combination
of \ref{itm:gla:v:conv} and \ref{itm:gla:e:conv}
would be even more complex)
but since it is more intuitive to think
about the effect of resampling separately,
we decided to present the conditions in this way.
In the discussion of the interpretation of the conditions we will focus
mainly on the first three properties,
since the interpretation of the other three is analogous.

\ref{itm:gla:e:delta} implies
that the effect of the perturbation of~\(\mathbf{G}_{n}\) on~\(f\)
can be approximated by the local quantity~\(\mathrm{LA}_{k}^{E,L}\)
that only takes into account a~\(k\)-neighbourhood
of the perturbed site.
The error of this approximation
is bounded by~\(\mathrm{LA}_{k}^{E,U}-\mathrm{LA}_{k}^{V,L}\).
\ref{itm:gla:e:tree} implies that the values of~\(\mathrm{LA}_{k}^{E,L}\)
and~\(\mathrm{LA}_{k}^{E,U}\) only depend
on properties that are preserved under graph isomorphisms,
which means that we can substitute
the limiting Galton--Watson trees for the local neighbourhoods
in order to analyse the approximation error.
\ref{itm:gla:e:conv} ensures that the approximation error
goes to zero as the level~\(k\) of the considered neighbourhood
increases.

Note that property~GLA does not need the function~\(f\)
to be local in the sense that~\(\Delta_{e}f\) and~\(\Delta_{v}f\)
only depend on a fixed neighbourhood~\(B_{k}(v,\mathbf{G}_{n})\).
All that is required is that there be local approximations
and that the approximation improves as~\(k\) gets large.
We will see the difference in \autoref{sec:appl},
where we present a local function to get started and then
a function for which~\(\Delta_{e}f\) can only be approximated locally.

\ref{itm:gla:e:conv} and \ref{itm:gla:v:conv} might look a bit
daunting at first.
Our proof relies on coupling the neighbourhood of fixed vertices,
which as we remarked when we discussed
the limiting object and the coupling from \autoref{prop:maincoup}
makes for a slightly more complex situation at the root.
The conditions state that the effect of the root
can be separated from the approximation error
that is due to the remaining tree structure.
Essentially we can think of the averaging we apply
as choosing the root uniformly,
which transfers our setup
to the \emph{unimodular} setting.
The boundedness assumptions then guarantee that even in this setting
the approximation error goes to zero.
We will show that the effect of the root can be separated out
in a concrete example in \autoref{sec:appl}.

In applications the function~\(f\) will often be related
to a combinatorial optimisation problem that
has certain recursive properties
so that the local approximation
functions~\(\mathrm{LA}_{k}^{\ast,\ast}\)
can be defined via a recursion on the graph
by essentially cutting off everything that is not in the
local neighbourhood of level~\(k\),
imposing an arbitrary starting value for those vertices
and then passing it down recursively towards the root vertex
of the neighbourhood.
Depending on the properties of the recursion in question
natural lower and upper bounds may be found by selecting
certain extremal values for the vertices that are cut
or by exploiting that the recursion values oscillate for even and odd
levels.

If applied to a Galton--Watson tree
the recursive nature of~\(f\)
gives rise to a \emph{recursive tree process} \citep{aldous:rtp}.
Briefly, a recursive tree process (RTP) is a Galton--Watson tree
in which each individual~\(\mathbf{i}\) has an associated
value~\(X_{\mathbf{i}}\)
that is calculated by applying a function~\(g\)
to
all the values~\(X_{\mathbf{i}1},\dots,X_{\mathbf{i}N_{\mathbf{i}}}\)
associated with the~\(N_{\mathbf{i}}\)
children~\(\mathbf{i}1,\dots,\mathbf{i}N_{\mathbf{i}}\)
of~\(\mathbf{i}\)
and an independent noise~\(\xi_{\mathbf{i}}\)
at~\(\mathbf{i}\).
An RTP is called \emph{endogenous} if the value at the root~\(X_{\treeroot}\)
is measurable with respect to the~\(\sigma\)-algebra
generated by the noise~\(\xi_{\mathbf{i}}\) and number of children~\(N_{\mathbf{i}}\)
for each individual in the tree.
If the local approximations are defined
via the recursion associated with~\(f\),
\ref{itm:gla:e:conv} is closely related to the question
of endogeny of the recursive tree process.

Properties related to \ref{itm:gla:e:conv} have also been called
\emph{long-range independence} by \citet{gamarnik}
and \emph{replica symmetry} by \citet{waestlund}
and have been used to calculate limiting constants for
the behaviour of some combinatorial optimisation problems.

In our statement of the main theorem we will encounter the following
two sequences.
\begin{definition}\label{def:epsrho}
  Let~\((k_{n})_{n \in \naturals} \subseteq \intervaloo{0}{\infty}\)
  be any sequence.

  Recall the definitions of~\(\alpha_{n}\)
  from \autoref{ass:coupling}
  and the definition of~\(\Gamma_{p,n}\)
  and~\(\kappa_{p,n}\)
  from~\eqref{eq:gamman} and~\eqref{eq:kappan}.
  For~\(n, \ell \in \naturals\)
  let
  \[
    \begin{split}
    \varepsilon_{n,\ell}
    &= \frac{\Gamma_{2,n}^{2}}{n}
    +\vartheta \kappa_{1,n} \Gamma_{1,n}
  +\Gamma_{1,n}\vartheta(\Gamma_{2,n}+1)^{\ell}
    \biggl(
      \frac{\Gamma_{3,n}}{n\vartheta}+\kappa_{1,n}+\kappa_{2,n}+
      \frac{2+\Gamma_{1,n}}{k_{n}}+\frac{k_{n}}{n\vartheta}\biggr)\\
  &\quad + \frac{1}{k_{n}}
    +\frac{k_{n}^{2}}{n\vartheta\Gamma_{1,n}}
    +\alpha_{n}
    (
      \Gamma_{1,n}
      + (\Gamma_{2}+1)^{\ell-1}
        (\Gamma_{2,n}+\vartheta\Gamma_{1,n})
    )\\
    &\quad+
       (1+\Gamma_{1,n}\vartheta(\Gamma_{2}+1)^{\ell})
    (d_{\mathrm{TV}}(\mu_{E,n},\mu_{E})
    + d_{\mathrm{TV}}(\mu_{V,n},\mu_{V})
      )
    \end{split}
  \]
  and
  \[
    \rho_{n,\ell}  = \min\set[\bigg]{
        \frac{\vartheta\Gamma_{2,n}+\vartheta\Gamma_{1,n}+1}{n\vartheta}
      (\Gamma_{1,n}+1)^{2}(\Gamma_{2,n}+C)^{2\ell+1}(\Gamma_{3,n}+1)^{2},1
    }.
  \]
\end{definition}
The sequence~\(\varepsilon_{n,\ell}\) arises from the coupling
probability (cf.~\autoref{prop:maincoup}).
The sequence~\(\rho_{n,\ell}\) absorbs the correlation
between neighbourhoods of a fixed collection of vertices
and bounds the probability of certain other
desirable events in our proofs.

As in the discussion of the convergence rate of \autoref{prop:maincoup},
the terms~\(\Gamma_{1,n}\), \(\Gamma_{2,n}\) and~\(\Gamma_{3,n}\)
are bounded in probability
and~\(\alpha_{n}\), \(\kappa_{1,n}\) and~\(\kappa_{2,n}\)
converge to zero in probability.
Hence both~\(\varepsilon_{n,\ell}\) and~\(\rho_{n,\ell}\)
converge to~\(0\)
in probability as~\(n \to \infty\)
for all~\(\ell \in \naturals\)
if~\(k_{n}\) is chosen appropriately, e.g.~\(k_{n} = n^{1/3}\).

In addition to the local approximation in property~GLA
we will assume a much simpler bound for
the effect of the local perturbation.

\begin{assumption}\label{ass:simpbounds}
  Assume that there are real-valued
  functions~\(H_{E} \colon \intervalco{0}{\infty}^{4} \to \reals\)
  and~\(H_{V} \colon \intervalco{0}{\infty}^{2} \to \reals\)
  that satisfy
  \begin{align}
  \label{eq:jbound:e}
  J_{E}
  &= \max\set[\Big]{1, \sup_{n\in\naturals}
    \expe[H_{E}(w_{e},w'_{e},w_{v},w_{u})^{6}]}
  < \infty,\\
  \shortintertext{and}
  \label{eq:jbound:v}
  J_{V}
  &= \max\set[\Big]{1, \sup_{n\in\naturals}
    \expe[H_{V}(w_{v},w'_{v})^{6}]}
  < \infty.
  \end{align}
  Let~\(J = J_{E}+J_{V}+J_{E}J_{V}\).

  Assume further that there exists a non-decreasing
  function~\(h \colon \intervalco {0}{\infty} \to \reals\)
  such that for
  \begin{equation}
    \label{eq:nubound}
    \chi_{n} = \frac{1}{n}\sum_{v \in V_{n}} \zeta_{n}(v)
    \quad\text{with}\quad
    \zeta_{n}(v) = \expe_{n}[h(\setcard{D_{1}(v)}+4)^{4}] < \infty
  \end{equation}
  we have that~\(\chi_{n}\) is bounded in probability.

  Finally, assume that
  \begin{equation}\label{eq:deltahbound:e}
  \abs{\Delta_{e}f}
  \leq \indfunc_{\set{\max\set{X_{e},X'_{e}} = 1}}
  H_{E}(w_{e},w'_{e},w_{v},w_{u})
  \end{equation}
  and
  \begin{equation}\label{eq:deltahbound:v}
  \abs{\Delta_{v}f}
  \leq \vdegbound H_{V}(w_{v},w'_{v}).
  \end{equation}
\end{assumption}

These bounds will allow us to make generous use of the Cauchy--Schwarz
inequality in proofs
especially when we are not on the event where the coupling holds.

Note that the dependence on~\(n\) is only implicit in the terms inside the
expectation in~\eqref{eq:jbound:e} and~\eqref{eq:jbound:v}
because we have~--~as usual~--~dropped the superscript~\((n)\)
for the weights~\(w_{e}\) and~\(w_{v}\).

Property~GLA and the simpler integrability bounds
of \autoref{ass:simpbounds} now finally yield explicit bounds
for the Kolmogorov distance of~\(f(\mathbf{G}_{n})\)
to a normal distribution.
\begin{theorem}\label{thm:fgnconv}
  Suppose~\((f,(\mathbf{G}_{n})_{n\in\naturals})\) satisfies
  property~GLA (\autoref{def:gla})
  for~\(\nu\), \(\mu_{E}\) and~\(\mu_{V}\).
  Assume that \autoref{ass:simpbounds} holds with~\(J\)
  and~\(\chi_{n}\) as defined there.
  Let~\(\sigma_{n}^{2} = \var_{n}(f(\mathbf{G}_{n}))\),
  set
  \[
    Z_{n}
    = \frac{f(\mathbf{G}_{n})-\expe_{n}[f(\mathbf{G}_{n})]}{\sigma_{n}}
  \]
  and let~\(\Phi\) be the cumulative distribution function
  of the standard normal distribution.
  Then we have for all~\(k,n \in \naturals\)
  \begin{equation}\label{eq:clt}
  \begin{split}
  &\sup_{t \in \reals} \abs{\prob_{n}(Z_{n} \leq t)-\Phi(t)}\\
  &\quad\leq C_{0} J^{1/4}
  \biggl[
  \biggl(\frac{n}{\sigma_{n}^{2}}\biggr)^{\!\!1/2}
  (\vartheta^{1/2}+\Gamma_{2,n}+\chi_{n}^{1/2})^{2}
  (
  (M_{n}^{E}\delta_{k})^{1/8}+(M_{n}^{V}\delta^{V}_{k})^{1/8}+\varepsilon_{n,k}^{1/16}+\rho_{n,k}^{1/16}
  )\\
  &\qquad\qqquad\quad+ \biggl(\frac{n}{\sigma_{n}^{2}}\biggr)^{\!3/4}
  \frac{\vartheta \Gamma_{1,n}+\chi_{n}^{1/2}}{n^{1/4}}
  \biggr].
  \end{split}
\end{equation}
\end{theorem}
Under the assumptions of the theorem~\(\Gamma_{1,n}\), \(\Gamma_{2,n}\)
\(\chi_{n}\), \(M^{E}_{n}\) and~\(M^{V}_{n}\) are bounded in probability.
Furthermore~\(\varepsilon_{n,k}\) and~\(\rho_{n,k}\)
converge to zero in probability as~\(n \to \infty\)
for all~\(k \in \naturals\).

If~\(n \sigma_{n}^{-2}\) is bounded in probability we can
make all terms on the right hand-side of~\eqref{eq:clt} arbitrarily small
as follows.
First choose~\(k\) large enough so
that the terms involving the bounded terms
(whose bound is independent of~\(n\))
and~\(\delta^{E}_{k}\) and~\(\delta^{V}_{k}\)
are as small as desired.
Then choose~\(n\) large enough that for this~\(k\)
the terms~\(\varepsilon_{n,k}\) and~\(\rho_{n,k}\)
are as small as desired.
This behaviour of the variance is in general not a given
and will need to be verified separately in applications.

The following corollary replaces property~GLA
with a slightly simpler condition
that is particularly suitable if~\(f\)
has a recursive structure.
\begin{corollary}\label{cor:fgnconv}
  Let~\((\mathbf{G}_{n})_{n\in\naturals}\) be a sequence of weighted
  inhomogeneous graphs and let~\(f\) be a function defined on weighted graphs.
  Suppose that
  \begingroup
  \raggedright
  \begin{enumerate}
  \item the underlying graph sequence~\((G_{n})_{n\in\naturals}\)
    satisfies \autoref{ass:coupling},
  \item there are two probability measures~\(\mu_{E}\) and~\(\mu_{V}\)
    on~\(\intervaloo{0}{\infty}\)
    with~\(d_{\mathrm{TV}}(\mu_{E,n},\mu_{E}) \to 0\)
    and~\(d_{\mathrm{TV}}(\mu_{V,n},\mu_{V}) \to 0\) as~\(n\to\infty\),
  \item \autoref{ass:simpbounds} holds with~\(J\) and~\(\chi_{n}\)
    as defined there and
  \item the effects of perturbations of~\(\mathbf{G}_{n}\) on~\(f\)
    can be approximated locally in the following sense.

    There exist
    functions~\(g^{L}_{k}\) and~\(g^{U}_{k}\) defined on weighted
    rooted graphs for any~\(k \in \naturals\)
    and there exist
    two sequences of
    functions~\(m_{n} \colon V_{n} \to \reals\)
    and~\(\tilde{m}_{n} \colon V_{n}^{2} \to \reals\)
    such that
    \[
      (M_{n})_{n\in\naturals}
      = \Bigl(n^{-1}\sum_{v \in V_{n}} m_{n}(v)\Bigr)_{n\in\naturals}
      \quad\text{and}\quad
      (\tilde{M}_{n})_{n\in\naturals}
      =\Bigl(n^{-2}\sum_{v, u\in V_{n}} \tilde{m}_{n}(v,u)\Bigr)_{n\in\naturals}
    \]
    are bounded in probability
    and furthermore two
    sequences~\((\delta^{E}_{k})_{k \in \naturals}, (\delta_{k}^{V})_{k \in \naturals}\)
    with~\(\delta^{E}_{k} \to 0\) and~\(\delta^{V}_{k} \to 0\) as~\(k\to \infty\)
    such that for all~\(k\in\naturals\) the following conditions
    are satisfied.
  \end{enumerate}

  \begin{enumerate}[label=\textup{(\(\text{GLA}'\)~\arabic*)}, leftmargin=*]
  \item\label{eq:fgmv} For any~\(v\in V_{n}\),
    whenever~\(B_{k}(v,\mathbf{G}_{n})\) is a tree, then
    \begin{equation*}
      g^{L}_{k}(B_{k}(v,\mathbf{G}_{n}))
      \leq f(\mathbf{G}_{n})-f(\mathbf{G}_{n}-v)
      \leq g^{U}_{k}(B_{k}(v,\mathbf{G}_{n})).
    \end{equation*}
  \item\label{eq:gtree} For any~\(v \in V_{n}\)
    if~\(B_{k}(v,\mathbf{G}_{n}) \cong \mathbf{T}\) for some rooted weighted
    tree~\(\mathbf{T}\),
    then
    \begin{equation*}
      g^{L}_{k}(\mathbf{T})
      =g^{L}_{k}(B_{k}(v,\mathbf{G}_{n}))
      \quad\text{and}\quad
      g^{U}_{k}(\mathbf{T})
      =g^{U}_{k}(B_{k}(v,\mathbf{G}_{n})).
    \end{equation*}
  \item\label{eq:gult}
    For any~\(v,u \in V_{n}\)
    if we have~\(\mathbf{T}_{k}(v) \sim
  \mathbf{T}_{k}(W_{v},\nu,\lambda,\mu_{E},\mu_{V})\)
  and~\(\mathbf{\tilde{T}}_{k}(v,u) \sim
  \mathbf{\tilde{T}}(W_{v},W_{u},\nu,\mu_{E},\mu_{V})\),
  then
  \begin{align*}
    \expe_{n}[
      (g^{U}_{k}(\mathbf{T}_{k}(v))
      -g^{L}_{k}(\mathbf{T}_{k}(v)))^{2}
    ]
    \leq m(v)\delta_{k}
  \shortintertext{and}
    \expe_{n}[
      (g^{U}_{k}(\mathbf{\tilde{T}}_{k}(v,u))
      -g^{L}_{k}(\mathbf{\tilde{T}}_{k}(v,u)))^{2}
    ]
    \leq \tilde{m}(v,u) \tilde{\delta}_{k}.
  \end{align*}
  \end{enumerate}
  \endgroup
  Let~\(\sigma_{n}^{2} = \var_{n}(f(\mathbf{G}_{n}))\),
  set
  \[
  Z_{n}
  = \frac{f(\mathbf{G}_{n})-\expe_{n}[f(\mathbf{G}_{n})]}{\sigma_{n}}
  \]
  and let~\(\Phi\) be the cumulative distribution function
  of the standard normal distribution.
  Then we have for all~\(k,n \in \naturals\)
  \[
  \begin{split}
  &\sup_{t \in \reals} \abs{\prob_{n}(Z_{n} \leq t)-\Phi(t)}\\
  &\quad\leq C_{0} J^{1/4}
  \biggl[
  \biggl(\frac{n}{\sigma_{n}^{2}}\biggr)^{\!\!1/2}
  (\vartheta^{1/2}+\Gamma_{2,n}+\chi_{n}^{1/2})^{2}
  (
  (M_{n}\delta_{k})^{1/8}+(\tilde{M}_{n}\tilde{\delta}_{k})^{1/8}
  +\varepsilon_{n,k}^{1/16}+\rho_{n,k}^{1/16}
  )\\
  &\qquad\qqquad\quad+ \biggl(\frac{n}{\sigma_{n}^{2}}\biggr)^{\!3/4}
  \frac{\vartheta \Gamma_{1,n}+\chi_{n}^{1/2}}{n^{1/4}}
  \biggr].
  \end{split}
  \]
\end{corollary}
The three conditions \ref{eq:fgmv} to \ref{eq:gult}
together imply property~GLA,
so we may informally refer to them as \emph{property~\(\text{GLA}'\)}.
Again, the intuition is that \ref{eq:fgmv}
can be used to approximate the effect of the perturbation on~\(f\)
locally with~\(g_{k}^{L}\) with an approximation error at
most~\(g^{U}_{k}-g^{L}_{k}\).
Then \ref{eq:gtree} allows us to estimate this approximation
error on the limiting Galton--Watson tree,
where \ref{eq:gult} ensures that the approximation error
goes to~\(0\) as~\(k \to \infty\).

We will prove \autoref{thm:fgnconv} and \autoref{cor:fgnconv} in \autoref{chap:proof}.

\subsection{Related Work}

These results extend the central limit theorem shown by \citet{cao}.
We were able to include weights on the vertices
and could prove the result in the more general setting
of rank-one inhomogeneous random graphs.
As far as we are aware \citeauthor{cao}'s result is the only general
central limit theorem for combinatorial optimisation problems
in a sparse Erdős--Rényi random graph setting,
but for specific functions central limit theorems
have been established.

The properties of the size of the giant component in (supercritical)
sparse Erdős--Rényi graphs are well studied and
central limit results were obtained by
\citet{stepanov,pittel,pittelwormald,bollobasriordan,barraez}.

For the maximal matching (without weights)
a central limit theorem can be shown for the Erdős--Rényi model with edge
probability~\(p_{n} = \lambda/n\)
\citetext{\citealp{glasgow} \citealp[and earlier][for~\(\lambda > e\)]{kreacic}
  \citealp[for~\(\lambda < 1\)]{pittel},
\citealp[see also][]{aronson,glasgow2023exact}}.
These results are closely related to a central limit theorem
for the value produced by a greedy algorithm
to compute the maximum matching \citep{dyer}.

\citet{barbour} recently proved a general central limit theorem
for \emph{local} graph statistics in the configuration model.
The configuration model generates a random graph with a given
degree sequence.
In fact, conditional on its degrees the inhomogeneous
graph model we considered here (and more general inhomogeneous graph models)
have the same distribution as a configuration model conditioned
on producing no loops or multiple edges \citep[Thm.~7.18]{vdh}.

\citet{rucinski} established a CLT for subgraph counts
in non-sparse Erdős--Rényi settings
(where at least~\(np^{m}_{n}\to\infty\) with~\(m \geq 1\) depending on
the subgraphs).
\citet{maugis} extended this to the analogous setting in
inhomogeneous random graphs with edge probabilities
\[
  p^{(n)}_{uv}
  = \rho_{n} \kappa(W_{u},W_{v})
\]
where~\(W_{v} \overset{\text{i.i.d.}}{\sim} \mathcal{U}\intervalcc{0}{1}\),
\(\rho_{n} \subseteq \intervaloo{0}{1}\)
(with~\(n\rho_{n}^{m} \to \infty\) with~\(m\geq 1\) depending on the subgraphs)
and~\(\kappa\) a bounded kernel.

The first-order (law of large numbers-like) behaviour
of a number of combinatorial optimisation problems, on the contrary,
has been studied extensively in the sparse Erdős--Rényi graph setting
\citep[e.g.][]{karp,gamarnik,bayati}.
Some of the methods that were used to obtain limiting constants in this setting
can in fact be used to verify property~GLA,
so that a first-order result together with our central limit theorem
framework
immediately also proves the second-order behaviour.
Results for more general sparse graphs
do not appear to be as abundant \citep{bord}.

\subsection{Applications}\label{sec:appl}
In this section we will briefly present two applications of
the central limit theorem.
The first is a simple example in which we consider the total sum of (artificial)
edge weights based on the weights at the end vertices.
Contrary to the setup in which we assign edge weights as usual in our
weighted graph model,
this results in edge weights that are not independent,
so that a standard central limit is not immediately applicable.
As a second example we consider maximum weight matching.

\subsubsection{Total sum of dependent edge weights}

In order to whet our appetite here is a simple application
of \autoref{thm:fgnconv}.
Let~\(\mathbf{G}_{n}\) be a sequence of weighted inhomogeneous
random graphs satisfying the assumptions of \autoref{thm:fgnconv}.
We will assume that the connectivity weights are such
that~\(\expe_{n}[(W^{(n)})^{4}]\) converges in probability
to a constant and that~\(\nu\) has fourth moments.
In this example we
do not place independent weights on the edges with distribution~\(\mu_{E}\),
which we ignore from now on.
Instead we will use the weights we put on the vertices
with~\(\mu_{V,n}=\mu_{V}\) to induce artificial
weights on the edges by adding up the weights of
their endpoints.
We will assume that~\(\mu_{V}\) has at least sixth moments.

With the usual notation of~\(\mathbf{X}\) and~\(\mathbf{w}\)
we are interested in the quantity
\[
  N(\mathbf{G}_{n})
  = \sum_{u,v \in V_{n}} (w_{v}+w_{u})X_{vu}
  = \sum_{e = \edge{u}{v} \in V_{n}^{(2)}} (w_{v}+w_{u})X_{e},
\]
i.e.~in twice the total sum of these artificial edge weights.
Since the same~\(w_{v}\) will appear for different
edges~\(e \in V_{n}^{(2)}\),
this is not a sum of independent random variables.
Observe that~\(N(\mathbf{G}_{n})\) can be rewritten as
\[
  N(\mathbf{G}_{n})
  = \sum_{v \in V_{n}} \setcard{D_{1}(v)}w_{v},
\]
but again that this is \emph{not} a sum of independent random variables,
since~\(\setcard{D_{1}(v)}\) is not independent for different~\(v\)
(take the simple example in which we consider a graph with just two
 vertices, if the degree of one of the vertices is~\(1\),
 we know that the degree of the other must also be~\(1\)).
Since we cannot easily rewrite~\(N(\mathbf{G}_{n})\)
as a sum of independent random variables,
we cannot easily apply one of the standard central limit theorems.
Hence, we will appeal to our central limit theorem \autoref{thm:fgnconv}.

We will first identify suitable local approximations for property~GLA.
Since the problem is \enquote{truly local} in the sense that
the effect of a local change can be fully estimated
with local information, this is straightforward.

We observe that for any edge~\(e = \edge{u}{v} \in V_{n}^{(2)}\)
\begin{equation}\label{eq:ne}
  \Delta_{e}N
  = N(\mathbf{G}_{n})-N(\mathbf{G}_{n}^{e})
  = (w_{v}+w_{u})(X_{e}-X'_{e})
\end{equation}
and any vertex~\(v \in V_{n}\)
\begin{equation}\label{eq:nv}
  \Delta_{v} N
  = N(\mathbf{G}_{n})-N(\mathbf{G}_{n}^{v})
  = \setcard{D_{1}(v)}(w_{v}-w'_{v}).
\end{equation}
To shorten notation from now on write~\(B_{k} = B_{k}(v,\mathbf{G}_{n})\),
\(B^{e}_{k}=B_{k}(v,\mathbf{G}^{e}_{n})\)
and~\(B^{v}_{k}=B_{k}(v,\mathbf{G}^{v}_{n})\).
For~\(k \geq 1\) we can let
\[
  \mathrm{LA}^{E,L}(B_{k},B^{e}_{k})
  = \mathrm{LA}^{E,U}(B_{k},B^{e}_{k})
  = (w_{v}+w_{u})(X_{e}-X'_{e})
\]
and
\[
  \mathrm{LA}^{V,L}(B_{k},B^{v}_{k})
  = \mathrm{LA}^{V,U}(B_{k},B^{v}_{k})
  = \setcard{D_{1}(v)}(w_{v}-w'_{v}).
\]
Then
\[
  \mathrm{LA}^{E,L}(B_{k},B^{e}_{k})
  = \Delta_{e}N
  = \mathrm{LA}^{E,U}(B_{k},B^{e}_{k})
\]
and
\[
  \mathrm{LA}^{V,L}(B_{k},B^{v}_{k})
  = \Delta_{v}N
  = \mathrm{LA}^{V,U}(B_{k},B^{v}_{k}),
\]
which immediately verifies \ref{itm:gla:e:delta} and \ref{itm:gla:v:delta}.
The construction of these functions relies only on properties
that are preserved under isomorphisms for weighted graphs,
so \ref{itm:gla:e:tree} and \ref{itm:gla:v:tree} are also satisfied.
Since the upper and lower bound coincide,
\ref{itm:gla:e:conv} and \ref{itm:gla:v:conv}
are trivially satisfied.
Hence, property~GLA (\autoref{def:gla}) holds in our problem.

We turn to the simpler integrability bounds in \autoref{ass:simpbounds}.
From~\eqref{eq:ne} and~\eqref{eq:nv} we obtain
\[
  \abs{\Delta_{e}N}
  \leq \indfunc_{\set{\max\set{X_{e},X'_{e}}=1}}(w_{v}+w_{u})
  \quad
  \text{and}
  \quad
  \abs{\Delta_{v} N}
  \leq \setcard{D_{1}(v)}(w_{v}+w'_{v}).
\]
The sixth moments of the weights are bounded by assumption.
Hence we can choose
\[
  H_{E}(w_{v},w_{u})=w_{v}+w_{u}
  \quad\text{and}\quad
  H_{V}(w_{v},w'_{v}) = w_{v}+w'_{v}
\]
to satisfy~\eqref{eq:jbound:e} and~\eqref{eq:jbound:v}
of \autoref{ass:simpbounds}.
For~\eqref{eq:nubound} set~\(h(x)=x\)
so that we need a bound on the fourth moment of~\(\setcard{D_{1}(v)}\),
which can be found in \autoref{lem:higher:deg}.
We then have that
\[
  \zeta_{n}(v)
  = \expe_{n}[(\setcard{D_{1}(v)}+4)^{4}]
  \leq C (W_{v}+1)^{4}(\Gamma_{2,n}+1)^{4} < \infty
\]
and that
\[
  \chi_{n}
  = \frac{1}{n} \sum_{v \in V_{n}} \zeta_{n}(v)
  \leq \frac{1}{n} \sum_{v \in V} C(W_{v}+1)^{4}(\Gamma_{2,n}+1)^{4}
  = C(\Gamma_{2,n}+1)^{4}\expe_{n}[(W^{(n)}+1)^{4}]
\]
is bounded in probability because we assumed the existence of
fourth moments for~\(W^{(n)}\).

With all assumptions verified we can now apply \autoref{thm:fgnconv}.
As discussed in the remarks after \autoref{thm:fgnconv}
the bound for the Kolmogorov distance of the distribution
of~\(\sigma_{n}^{-1}(N(\mathbf{G}_{n})-\expe_{n}[N(\mathbf{G}_{n})])\)
to a standard normal distribution goes to zero
in probability if~\(n \sigma_{n}^{-2}\) is bounded in probability.
Hence, in order to conclude convergence to a standard normal,
we have to verify that the variance of~\(N(\mathbf{G}_{n})\)
is of sufficiently high order.
A straightforward but tedious calculation, which we will not show here,
verifies that
indeed~\(\var_{n}(N(\mathbf{G}_{n}))\) is of order~\(n\)
so that~\(n \sigma_{n}^{-2}\) is bounded in probability.
This then allows us to conclude the desired convergence.
The convergence rate depends on the rate
of convergence of~\(\nu_{n}\) to~\(\nu\)
and other properties of~\(\nu_{n}\) and~\(\nu\).

\subsubsection{Maximum weight matching}

We will now apply \autoref{cor:fgnconv}
to the maximum weight matching problem
on an inhomogeneous random graph satisfying \autoref{ass:coupling}
for some measure~\(\nu\)
and with~\(\mathrm{Exp}(1)\) edge weights
and no vertex weights.
We will assume that~\(\nu\) has third moments,
so that~\(\sizebias{\nu}\) has second moments.

Furthermore, we will assume that the following
technical condition holds.
\begin{assumption}
Define an operator~\(T\) on the space of
probability distributions on~\(\reals\) by
letting~\(T(\mu)\) be the distribution of
\[
  \max_{i \in [N]}\set{0,\xi_{i}-X_{i}},
\]
where~\(N \sim \mathrm{MPoi}(\sizebias{\nu})\),
\(\xi_{1},\dots,\xi_{N} \overset{\text{i.i.d.}}{\sim} \mathrm{Exp}(1)\)
and~\(X_{i},\dots,X_{N} \overset{\text{i.i.d.}}{\sim} \mu\)
are independent.
We assume that the iterated operator~\(T^{2}\)
(defined as~\(T^{2}(\mu) = T(T(\mu))\))
has a unique fixed point.
\end{assumption}
This condition is in essence what \citet{gamarnik} call \enquote{long-range
  independence}
and verify
in the case~\(N \sim \mathrm{Poi}(c)\)
corresponding to~\(\nu = \delta_{c}\)
in our setting
\citep[Thm.~3]{gamarnik}.

\begin{definition}[Maximum weight matching]
  Consider an edge-weighted graph~\(\mathbf{G}\)
  with vertex set~\(V\) and edge set~\(E\).

  A matching on~\(\mathbf{G}\)
  is a subset of edges~\(M \subseteq E\)
  in which no two edges have a vertex in common.
  In other words for all~\(v \in V\)
  there is at most one~\(u \in V\)
  such that~\(\edge{v}{u} \in E\)
  (we call~\(u\) the vertex matched to~\(v\)).

  The maximum weight matching is a matching~\(M\)
  that maximises the sum of edge weights~\(\sum_{e \in M} w_{e}\).
\end{definition}

Note that a maximum weight matching need not match every vertex
to another vertex.

We follow the strategy used by \citet{cao}
to tackle this problem.
We want to apply \autoref{cor:fgnconv}
and so need to identify suitable functions~\(g_{k}^{L}\)
and~\(g_{k}^{U}\).

As alluded to before,
recursive properties are usually
a very good starting point to verify property GLA or its slightly
simplified cousin in \autoref{cor:fgnconv}.
Let~\(M(\mathbf{G})\) be the weight of the maximum weight matching
on a graph~\(\mathbf{G}\)
with vertex set~\(V\) and edge set~\(E\).
Then for any vertex~\(v \in V\) the weight of the maximum
weight matching~\(M(\mathbf{G})\) satisfies the recursion
\[
  M(\mathbf{G})
  = \max\set[\Big]{M(\mathbf{G}-v),
    \max_{u: \edge{v}{u} \in E} w_{\edge{v}{u}}+M(\mathbf{G}-\set{v,u})}.
\]
Essentially this formula says that
we need to decide between not matching~\(v\)
to any partner vertex, so that the matching is in effect a matching
on~\(\mathbf{G}-v\),
or matching~\(v\) to any of its neighbours~\(u\),
upon which the weight of the matching increases by
the weight of the edge between~\(v\) and~\(u\)
and the remainder of the matching happens on
the graph~\(\mathbf{G}-\set{v,u}\).
Now define
\[
  h(\mathbf{G},v) = M(\mathbf{G})-M(\mathbf{G}-v)
\]
and note that
\[
  h(\mathbf{G},v)
  = \max\set[\Big]{0,
    \max_{u: \edge{v}{u} \in E} w_{\edge{v}{u}}-h(\mathbf{G}-v,u)
  }.
\]
Intuitively, this expression quantifies how
much better it is to match~\(v\)
to one of its neighbours
rather than to leave it unmatched.
If~\(h(\mathbf{G},v) = 0\),
then the weight of the maximum weight matching on~\(\mathbf{G}\)
and~\(\mathbf{G}-v\) are the same, which means that~\(v\)
can remain unmatched in~\(\mathbf{G}\) and we still
attain the maximum possible weight.
If~\(h(\mathbf{G},v)>0\),
then matching~\(v\) to one of its neighbours means that the
matching performs better than the matching that does not match~\(v\).
Hence,~\(v\) should be matched in the maximum weight matching on~\(v\).

Let~\(\mathbf{T}\) be a weighted tree of height at most~\(k\).
Let~\(\treeroot\) be the root of~\(\mathbf{T}\).
Let~\(\mathcal{C}(u)\) be the set of children of the
vertex~\(u\) in~\(\mathbf{T}\).
Denote the edge weights of~\(\mathbf{T}\) by~\(w_{e}\).
Define~\(h_{k}(\placeholder, \mathbf{T}) \colon \mathbf{T} \to \reals\)
by setting~\(h_{k}(u,\mathbf{T})=0\) for all leaves~\(u\) of~\(\mathbf{T}\)
and by the recursion
\[
  h_{k}(u,\mathbf{T})
  = \max\set[\Big]{0,
    \max_{u' \in \mathcal{C}(u)} w_{\edge{u}{u'}}-h_{k}(u',\mathbf{T})
  }
\]
for all non-leaf vertices~\(u\) of~\(\mathbf{T}\).
Note that by this recursion the value~\(h_{k}(u,\mathbf{T})\)
only ever depends on the subtree in~\(\mathbf{T}\)
that is induced by the descendants of~\(u\).

A short induction argument shows that
under the assumption that~\(B_{k}(v,\mathbf{G}_{n})\)
is a tree
\[
  h_{k}(v, B_{k}(v,\mathbf{G}_{n}))
  \leq
  h(\mathbf{G}_{n},v)
  \quad\text{if~\(k\) is even}
\]
and
\[
  h(\mathbf{G}_{n},v)
  \leq h_{k}(v, B_{k}(v,\mathbf{G}_{n}))
  \quad\text{if~\(k\) is odd}.
\]
Thus, if~\(B_{2k+1}(v, \mathbf{G}_{n})\)
is a tree for~\(k \in \naturals\)
(which naturally implies that~\(B_{2k}(v,\mathbf{G}_{n})\) is a tree),
we have
\[
  h_{2k}(v, B_{2k}(v,\mathbf{G}_{n}))
  \leq
  h(\mathbf{G}_{n},v)
  \leq h_{2k+1}(v, B_{2k+1}(v,\mathbf{G}_{n})).
\]
This suggests the following definition for~\(g_{k}^{L}\)
and~\(g_{k}^{U}\):
Let~\(k_{U}\) be the largest odd number less than or equal to~\(k\)
and~\(k_{L} = k_{U}-1\).
Then~\(k_{L} \leq k\) and~\(k_{U} \leq k\)
so that if~\(B_{k}(v,\mathbf{G}_{n})\)
is a tree, we can set
\begin{equation*}
  g^{L}_{k}(B_{k}(v,\mathbf{G}_{n}))
  = h_{k_{L}}(v, B_{k_{L}}(v,\mathbf{G}_{n}))
  \quad\text{and}\quad
  g^{U}_{k}(B_{k}(v,\mathbf{G}_{n}))
  = h_{k_{U}}(v, B_{k_{U}}(v,\mathbf{G}_{n})).
\end{equation*}

\paragraph{Property~\(\text{GLA}'\)}
Immediately this construction ensures
\[
  g^{L}_{k}(B_{k}(v,\mathbf{G}_{n}))
  \leq
    h(\mathbf{G}_{n},v)
  \leq
  g^{U}_{k}(B_{k}(v,\mathbf{G}_{n})),
\]
which verifies \ref{eq:fgmv} in \autoref{cor:fgnconv}.

The construction of~\(g^{L}_{k}\) and~\(g^{U}_{k}\)
also ensures \ref{eq:gtree} of \autoref{cor:fgnconv}
because the definition relies only on
structure that is preserved under isomorphisms
on weighted rooted graphs,
namely edges and weights.

The next step is to verify \ref{eq:gult}.
For the first part
let~\(\mathbf{T}(v) \sim \mathbf{T}(W_{v}, \nu,\mathrm{Exp}(1), \placeholder)\)
and~\(\mathbf{T}_{k}(v)\) be its level-\(k\)~subtree.
Recall that~\(\mathbf{T}(v)\) is a delayed weighted Galton--Watson
tree that can be constructed by joining
together~\(N \sim \mathrm{Poi}(W_{v})\)
independent weighted Galton--Watson trees~\(\mathbf{T}^{(i)}\)
for~\(i\in \set{1,\dots,n}\) with offspring
distribution~\(\mathrm{MPoi}(\sizebias{\nu})\)
via edges~\(\edge{\treeroot}{\treeroot_{i}}\)
with independent edge weights according to~\(\mathrm{Exp}(1)\).
The depth-\(k\) subtree~\(\mathbf{T}_{k}\) of~\(\mathbf{T}\)
is then made up of the~\(N\) subtrees~\(\mathbf{T}^{(i)}_{k-1}\)
of~\(\mathbf{T}^{(i)}\)
joined at the root~\(\treeroot\).
By the construction of~\(g_{k}^{L}\) and~\(g_{k}^{U}\)
we need to analyse
\begin{equation*}
  \expe_{n}[(h_{2k+1}(\treeroot, \mathbf{T}_{2k+1}(v))-
  h_{2k}(\treeroot, \mathbf{T}_{2k}(v)))^{2}]
\end{equation*}
to verify condition~\ref{eq:gult} of \autoref{cor:fgnconv}.
We evaluate one recursion step
in order to separate out the effect of the
different offspring distribution at the root.
Recall that~\(h_{k}(u,\mathbf{T})\)
only depends on the subtree of~\(\mathbf{T}\)
induced by the descendants of~\(u\).
If~\(h_{k}(\placeholder, \mathbf{T}_{k}(v))\)
assigns value~\(0\) to the leaves of~\(\mathbf{T}_{k}(v)\),
then~\(h_{k-1}(\placeholder, \mathbf{T}^{(i)}_{k-1})\)
does the same for the subtree induced by the descendants of~\(\treeroot_{i}\).
In particular,
\begin{align*}
  &\abs{h_{2k+1}(\treeroot, \mathbf{T}_{2k+1}(v))-
  h_{2k}(\treeroot, \mathbf{T}_{2k}(v))}^{2}\\
  &\quad=\abs[\Big]{
  \max\set[\Big]{0,
    \max_{1\leq i \leq N} w_{\edge{\treeroot}{\treeroot_{i}}}
    -h_{2k+1}(\treeroot_{i},\mathbf{T}_{2k+1}(v))
    }
    -
  \max\set[\Big]{0,
    \max_{1\leq i \leq N} w_{\edge{\treeroot}{\treeroot_{i}}}
    -h_{2k}(\treeroot_{i},\mathbf{T}_{2k}(v))
    }}^{2}.\\
  &\quad=\abs[\Big]{
  \max\set[\Big]{0,
    \max_{1\leq i \leq N} w_{\edge{\treeroot}{\treeroot_{i}}}
    -h_{2k}(\treeroot_{i},\mathbf{T}^{(i)}_{2k})
    }
    -
  \max\set[\Big]{0,
    \max_{1\leq i \leq N} w_{\edge{\treeroot}{\treeroot_{i}}}
    -h_{2k-1}(\treeroot_{i},\mathbf{T}^{(i)}_{2k-1})
    }}^{2}.\\
  &\quad\leq \max_{1 \leq i \leq N}\abs{h_{2k}(\treeroot_{i},\mathbf{T}^{(i)}_{2k})-
    h_{2k-1}(\treeroot_{i},\mathbf{T}^{(i)}_{2k-1})}^{2}\\
  &\quad\leq \sum_{i=1}^{N} \abs{h_{2k}(\treeroot_{i},\mathbf{T}^{(i)}_{2k})-
    h_{2k-1}(\treeroot_{i},\mathbf{T}^{(i)}_{2k-1})}^{2}.
\end{align*}
Since all~\(\mathbf{T}^{(i)}\) are independent Galton--Watson trees
with the same offspring distribution that are independent
of~\(N \sim \mathrm{Poi}(W_{v})\),
we have
\begin{align}
  \expe_{n}[(h_{2k+1}(\treeroot, \mathbf{T}_{2k+1})-
  h_{2k}(\treeroot, \mathbf{T}_{2k}))^{2}]
  &\leq \expe_{n}\biggl[
      \sum_{i=1}^{N} \abs{h_{2k}(\treeroot_{i},\mathbf{T}^{(i)}_{2k})-
      h_{2k-1}(\treeroot_{i},\mathbf{T}^{(i)}_{2k-1})}^{2}
    \biggr]\notag\\
  &\leq \expe_{n}[N]
    \expe_{n}[\abs{h_{2k}(\treeroot_{1},\mathbf{T}^{(1)}_{2k})-
    h_{2k-1}(\treeroot_{1},\mathbf{T}^{(1)}_{2k-1})}^{2}]\notag\\
  &\leq W_{v}\expe[\abs{h_{2k}(\treeroot_{1},\mathbf{T}^{(1)}_{2k})-
    h_{2k-1}(\treeroot_{1},\mathbf{T}^{(1)}_{2k-1})}^{2}].
    \label{eq:appl:etaw}
\end{align}
We dropped the conditioning on~\(\mathcal{F}_{n}\)
in the last expectation, because the random
variables inside the expectation do not depend on~\(\mathcal{F}_{n}\)
in any way.

Shorten~\(h_{k}(\treeroot_{1}, \mathbf{T}^{(1)}_{k})\)
to~\(h_{k}(\treeroot_{1})\).
With this notation it is enough to verify that
\begin{equation}\label{eq:hk}
  \delta_{k}
  = \expe[(h_{2k}(\treeroot_{1})-h_{2k-1}(\treeroot_{1}))^{2}]
  \to 0
  \quad\text{as~\(k \to \infty\)},
\end{equation}
because then together with~\(m_{n}(v) = W_{v}\),
for which we have
that~\(M_{n} = n^{-1}\sum_{v\in V_{n}} W_{v} = \vartheta\Gamma_{1}\)
is bounded in probability,
and~\eqref{eq:appl:etaw} we would have
\begin{equation}\label{eq:firstpart}
  \expe_{n}[(h_{2k+1}(\treeroot, \mathbf{T}_{2k+1}(v))-
  h_{2k}(\treeroot, \mathbf{T}_{2k}(v)))^{2}]
  \leq m_{n}(v)\delta_{k}
\end{equation}
as required for the first part of \ref{eq:gult}.

In order to verify~\eqref{eq:hk},
note first that
as before a short induction argument
shows that~\(h_{2k}(\treeroot_{1}) \leq h_{2k-1}(\treeroot_{1})\)
for all~\(k \in \naturals_{+}\).
Furthermore,~\(h_{2k-1}(\treeroot_{1})\) is non-increasing in~\(k\)
and~\(h_{2k}(\treeroot_{1})\) is non-decreasing in~\(k\).
Set
\[
  h^{L} = \lim_{k\to\infty} h_{2k}(\treeroot_{1})
  \quad\text{and}\quad
  h^{U} = \lim_{k\to\infty} h_{2k-1}(\treeroot_{1}).
\]
Then by the monotonicity of the sequences
\[
  h_{2k-1}(\treeroot_{1})-h_{2k}(\treeroot_{1})
  \downto h^{U}-h^{L}
\]
and
\begin{equation*}
  0
  \leq
  h_{2k-1}(\treeroot_{1}) - h_{2k}(\treeroot_{1})
  \leq
  h_{2k-1}(\treeroot_{1})
  \leq
  h_{1}(\treeroot_{1})
  \leq
  \max_{u \in \mathcal{C}(\treeroot_{1})} w_{\edge{\treeroot_{1}}{u}}.
\end{equation*}
Since the right-hand side has finite second moment
(because we assumed that~\(\sizebias{\nu}\) has finite second moment
and the~\(w_{e}\) are exponentially distributed),
we can apply Lebesgue's dominated convergence theorem
and obtain the desired convergence for~\eqref{eq:hk}
if~\(h^{U}-h^{L} = 0\) almost surely.
By definition we have~\(h^{L} \leq h^{U}\),
so the almost sure equality can be concluded
from equality of the expectations.

Hence, \eqref{eq:hk} and with it the first part of~\ref{eq:gult}
follow from the following claim.
\begin{claim}
  We have~\(\expe[h^{L}] = \expe[h^{U}]\).
\end{claim}
\begin{proof}
  Under the technical condition that the
  distributional operator~\(T^{2}\)
  has a unique fixed point,
  the arguments used by \citet{gamarnik}
  to prove their Proposition~1 and Theorem~3
  also apply in our setting,
  which implies that~\(h_{k}(\treeroot)\)
  converges in distribution to some limit~\(H_{\infty}\).
  But this implies that~\(h^{L}\) and~\(h^{U}\) have
  the same distribution,
  namely~\(H_{\infty}\).
  Then~\(\expe[h^{L}]=\expe[h^{U}]\) as claimed.
\end{proof}

This shows the first part of condition \ref{eq:gult} in \autoref{cor:fgnconv}.
For the second part we need to consider
\begin{equation*}
  \expe_{n}[(h_{2k+1}(\treeroot, \mathbf{\tilde{T}}_{2k+1}(v,u))-
  h_{2k}(\treeroot, \mathbf{\tilde{T}}_{2k}(v,u)))^{2}].
\end{equation*}
As explained in \autoref{def:ttilde}
we may assume that~\(\mathbf{\tilde{T}}_{k}(v,u) \sim
\mathbf{\tilde{T}}(W_{v},W_{u},\nu,\mu_{E},\mu_{V})\)
is constructed
from~\((\mathbf{T}_{k}(v),\mathbf{T}_{k-1}(u), \treeroot, \treeroot', w)\).
Again the idea is to unwrap what the recursion implies
for the different subtrees.
We only need to focus on the \enquote{artifical}
edge~\(\edge{\treeroot}{\treeroot'}\) of weight~\(w\)
between~\(\mathbf{T}_{k}(v)\) and~\(\mathbf{T}_{k-1}(u)\).
By the recursive structure of~\(h_{k}\)
\begin{align*}
  h_{2k+1}(\treeroot, \mathbf{\tilde{T}}_{2k+1}(v,u))
  &= \max\set{ h_{2k+1}(\treeroot, \mathbf{T}_{2k+1}(v)) , w-h_{2k}(\treeroot',
  \mathbf{T}_{2k}(u)) }
  \shortintertext{and}
  h_{2k}(\treeroot, \mathbf{\tilde{T}}_{2k}(v,u))
  &= \max\set{ h_{2k}(\treeroot, \mathbf{T}_{2k+1}(v)) , w-h_{2k-1}(\treeroot',
  \mathbf{T}_{2k-1}(u)) }.
\end{align*}
Let
\begin{equation*}
Y_{k} = h_{2k+1}(\treeroot, \mathbf{T}_{2k+1}(v))
-h_{2k}(\treeroot, \mathbf{T}_{2k+1}(v))
  \quad\text{and}\quad
  Y'_{k} = h_{2k-1}(\treeroot',
  \mathbf{T}_{2k-1}(u))-h_{2k}(\treeroot',
  \mathbf{T}_{2k}(u)).
\end{equation*}
Then
\begin{equation*}
  \abs{h_{2k+1}(\treeroot, \mathbf{\tilde{T}}_{2k+1}(v,u))-
  h_{2k}(\treeroot, \mathbf{\tilde{T}}_{2k}(v,u))}^{2}
  \leq \max\set{\abs{Y_{k}}^{2},\abs{Y'_{k}}^{2}}.
\end{equation*}
When we verified the first part of \ref{eq:gult},
we already showed that~\(\expe_{n}[\abs{Y_{k}}^{2}]\)
satisfies~\eqref{eq:firstpart},
i.e.~\(\expe_{n}[\abs{Y_{k}}] \leq W_{v}\delta_{k}\).
The exact same reasoning can be used to show
that~\(\expe_{n}[\abs{Y'_{k}}] \leq W_{u}\delta_{k}\).
Together this shows
\[
  \expe_{n}[\abs{h_{2k+1}(\treeroot, \mathbf{\tilde{T}}_{2k+1}(v,u))-
    h_{2k}(\treeroot, \mathbf{\tilde{T}}_{2k}(v,u))}^{2}]
  \leq (W_{v}+W_{u})\delta_{k},
\]
so that the second part of \ref{eq:gult}
is satisfied with~\(\tilde{m}_{n}(v,u)=W_{v}+W_{u}\)
for which~\(\tilde{M}_{n} = n^{-2}\sum_{v,u \in V_{n}} W_{u}+W_{v}
= 2\vartheta\Gamma_{1,n}\) is bounded in probability.

This verifies the simplified version of property~GLA
from \autoref{cor:fgnconv}.
Hence, we can apply \autoref{cor:fgnconv}
once we have verified \autoref{ass:simpbounds}.

\paragraph{Bounds for \autoref{ass:simpbounds}}
Because there are no vertex weights,
we only need to consider~\eqref{eq:jbound:e}
and~\eqref{eq:deltahbound:e}.
Indeed, for~\eqref{eq:deltahbound:e} we only need to find a bound of the form
\[
  \abs{M(\mathbf{G}_{n})-M(\mathbf{G}_{n}^{e})}
  \leq \indfunc_{\set{\max\set{X_{e},X'_{e}} = 1}} H_{E}(w_{e},w_{e'}).
\]
We identify a suitable bound by considering the cases separately.
First we consider the case that perturbing the edge removes it from the graph,
i.e.~\(X_{e}=1\), but~\(X'_{e}=0\).
If~\(e\) with weight~\(w_{e}\) is part of the maximum weight matching, removing
it from the graph (and therefore from the matching)
can cost the maximum weight matching at most~\(w_{e}\),
because the removal of~\(e\) allows other previously blocked edges
to participate again.
If~\(e\) is not part of the matching, removing it does not change the
weight of the maximum weight matching at all.
In case the perturbation adds a previously nonexistent edge,
i.e.~\(X_{e}=0\), \(X'_{e}=1\),
the weight of the maximum weight matching can increase at most
by the weight~\(w'_{e}\) of this edge.
If the edge is present in both the unperturbed and perturbed
graph and only changes its weight,
the weight of the maximum weight matching changes at most
by the difference of the old and new weight.
In any case the difference is bounded by the maximum of
the old weight~\(w_{e}\) and the new weight~\(w'_{e}\).
In other words we can choose~\(H_{E}(w_{e},w'_{e})=\max\set{w_{e},w'_{e}}\).

Since~\(w_{e},w'_{e} \sim \mathrm{Exp}(1)\),
we immediately have~\(\expe[H_{E}(w_{e},w'_{e})^{6}]\),
which implies~\eqref{eq:jbound:e},
since our weight distribution is the same for all~\(n\).

\paragraph{Variance bound}
As in the previous example, an application of \autoref{cor:fgnconv}
now provides an estimate for the Kolmogorov distance
of the distribution
of~\(\sigma_{n}^{-1} (N(\mathbf{G}_{n})-\expe_{n}[N(\mathbf{G}_{n})])\)
to a standard normal distribution.
In order to obtain a convergence we have
to verify that the variance is at least of order~\(n\).
This again a technical calculations.
For the case~\(\nu = \delta_{c}\) we refer to
the calculations done by
\citet[Lem.~3.2]{cao},
whose approach carries over to~\(\nu \neq \delta_{c}\)
under our technical condition.

\subsection{Outlook}
The main contribution of this thesis was to show that
the framework used by \citet{cao}
to establish the central limit theorem
for the (homogeneous) edge-weighted Erdős--Rényi model
can be extended to more inhomogeneous
graph models and to models with weights on edges and vertices.

The setting we investigate in this thesis still exhibits
a fair amount of uniformity in the limit.
Yet still, the methods used in the proof for the Erdős--Rényi model
had to be adapted not inconsiderably to apply to this case as well.
It would be interesting to investigate which level
of inhomogeneity~--~either in the graphs or the limiting objects~--~these
methods can still support and at which point other methods
need to be considered.

The inhomogeneous random graph models we investigate here
do not exhibit a spatial structure,
but many interesting real-world networks have
inherent spatial and geometric properties that influence the graph structure.
The local limiting behaviour of spatial inhomogeneous random graphs
is known \citep{spatial}
and may differ significantly from the local behaviour
of sparse inhomogeneous random graphs.
It is therefore doubtful that the sparsity/tree-based approach pursued
here is directly suitable to these graphs.
Nevertheless, the methods used here may be applicable to a subclass
of spatial random graphs whose local properties are sufficiently
similar to the sparse inhomogeneous graphs we considered.
Moreover, the general approach of the perturbative Stein's method
and local approximation has been used successfully
in a spatial setting \citep{chatterjee-sen}.

A final direction for future research would be
to consider dynamic versions of the underlying random graph
\citep[see e.g.][]{zhang,dynamic:er}.
If it is possible to show a central limit theorem
for each time-point of the evolving graph,
one might hope for a functional central limit theorem
for the entire process.

\section{Analysis of the structure
  of inhomogeneous random graphs}
\label{chap:structure}

One of the fundamental results in the analysis of
the behaviour of Erdős--Rényi random graphs is
that their connected components can be described with
branching processes \citep[see, e.g.][Chap.~4]{vdh}.
Indeed, the entire local neighbourhood structure of a sparse Erdős--Rényi
random graph can be related to a Galton--Watson tree
with Poisson offspring distribution.
In the realm of inhomogeneous random graphs similar branching
process results have been known since the introduction of the model
by \citet{soderberg} and extensive
study by \citet{bollobas}.
Specifically, the local structure of a wide class of sparse
inhomogeneous random graphs
can be related to a class of multi-type branching processes
\citetext{\citealp{bollobas}[Chap.~3], \citealp{vdh:2}}.
In our setting the limiting object will essentially turn out
to be a single-type branching process with
Poisson offspring distribution.

In this chapter we will analyse the local structure of sparse inhomogeneous
random graphs of the form introduced in \autoref{sec:setting}.
We will start by showing relatively simple
results about the sizes of the neighbourhoods of a fixed vertex
and about the probability that a fixed vertex or edge
is part of the neighbourhood of another vertex in \autoref{sec:nbhdsize}.
In a second step we will find explicit bounds for the correlation
between neighbourhoods of different vertices in \autoref{sec:corr}.

We will then use the graph exploration procedure introduced in
\autoref{sec:graphexpl}
to explicitly couple neighbourhoods in the inhomogeneous
graph to Galton--Watson trees in \autoref{sec:gimt}.
Finally, in \autoref{sec:compcoup} we present several coupling results that are
at first glance more complex, but follow directly from the coupling
established in \autoref{sec:gimt}.

\subsection{Neighbourhood size and path probabilities}
\label{sec:nbhdsize}
In this section we will briefly establish a few results about the size
of the neighbourhood of a vertex and for path probabilities.
Since the strategy of coupling the cluster (i.e~the neighbourhood of a vertex)
to a branching process is well established \citep{bollobas},
these results are by no means surprising and are in principle
known in a much more general setting.
We were not able to locate all of the precise results we need
in the literature, though, so we state all of them
here in a consistent notation.

The estimates in this section only depend on the underlying graph
structure and not on the additional vertex and edge weights
in the weighted graph~\(\mathbf{G}_{n}\).
Hence, all results in this section will be shown for~\(G_{n} = (V_{n},E_{n})\).
Still, all results presented here will still hold if~\(G_{n}\)
is replaced with~\(\mathbf{G}_{n}\).
Again,~\(V_{n}^{(2)} = \sset{\edge{u}{v}}{u,v \in V_{n}}\)
is the set of possible edges.

Recall the definition of the local neighbourhood~\(B_{\ell}(v, G_{n})\)
of a vertex~\(v\) in~\(G_{n}\) up to level~\(\ell \in \naturals\)
(cf.~\autoref{def:nbhd}).
For the rest of this section we will drop the reference to~\(G_{n}\)
and will just write~\(B_{\ell}(v)\) for~\(B_{\ell}(v,G_{n})\).

In a slight abuse of notation we will write
both~\(u \in B_{\ell}(v)\) for a vertex~\(u \in V_{n}\)
to mean that~\(u\) is contained in the vertex set of~\(B_{\ell}(v)\),
which means that there must be a path from~\(v\)
to~\(u\) of length no more than~\(\ell\),
and~\(e \in B_{\ell}(v)\)
to mean that the possible edge~\(e \in V_{n}^{(2)}\)
is contained in the edge set of~\(B_{\ell}(v)\),
which means that~\(e\) is part of a path of length at most~\(\ell\)
from~\(v\) to an arbitrary vertex.

It is convenient to have a notation for the set of vertices
in a local neighbourhood~\(B_{\ell}(v, G)\) of a vertex~\(v\)
and for the set of vertices at a certain level.
\begin{definition}\label{def:dl}
  Let~\(G=(V,E)\) be a graph.
  For a level~\(\ell \in \naturals\) and a vertex~\(v \in V\)
  let~\(S_{\ell}(v,G)\) be the set of vertices in~\(B_{\ell}(v,G)\).
  Furthermore,  set~\(S_{-1}(v,G) = \emptyset\).

  Then let
  \[
    D_{\ell}(v,G) = S_{\ell}(v,G) \setminus S_{\ell-1}(v,G)
    \quad\text{for~\(\ell \in \naturals\)}
  \]
  be the vertices of~\(B_{\ell}(v,G)\) at level~\(\ell\).
\end{definition}

The local neighbourhood can be generalised from the neighbourhood
of a single vertex to the neighbourhood of a set of vertices.
\begin{definition}
  Let~\(G = (V,E)\) be a graph
  and let~\(\mathcal{V} \subseteq V\) be a set of vertices.
  Then we let~\(B_{\ell}(\mathcal{V},G)\)
  be the subgraph of~\(G\)
  that is induced by the union of all paths from~\(B_{\ell}(v,G)\)
  for~\(v \in \mathcal{V}\).
  The set of vertices in~\(B_{\ell}(\mathcal{V},G)\) is given by
  \[
    S_{\ell}(\mathcal{V})
    = \bigsetunion_{v \in \mathcal{V}} S_{\ell}(v,G)
  \]
  and we set
  \[
    D_{\ell}(\mathcal{V},G)
    = S_{\ell}(\mathcal{V},G) \setminus S_{\ell-1}(\mathcal{V},G)
  \]
  (with the convention~\(S_{-1}(\mathcal{V}) = \emptyset\)).
\end{definition}
By this construction we only
have~\(D_{\ell}(\mathcal{V},G) \subseteq \bigsetunion_{v \in \mathcal{V}} D_{\ell}(v,G)\).
Equality is only guaranteed if the~\(S_{\ell-1}(v,G)\) are disjoint.

It will be useful to have an estimate
of the expected number of vertices in a neighbourhood as well
as of their \enquote{total connectivity weight}.
Amongst other things these quantities can be used to estimate the correlation
between neighbourhoods of different vertices.

\begin{definition}[Total~\(p\)-connectivity weight]
  Fix~\(p \geq 0\).
  For any set of vertices~\(\mathcal{U} \subseteq V_{n}\)
  let
  \[
    \setweight{\mathcal{U}}_{p}
    = \sum_{u \in \mathcal{U}} W_{u}^{p}
  \]
  denote the total sum of the~\(p\)-th power of the connectivity weights
  of the vertices in~\(\mathcal{U}\).
  We also say that~\(\setweight{\mathcal{U}}_{p}\)
  is the \emph{total~\(p\)-connectivity weight of~\(\mathcal{U}\)}.
\end{definition}
We sometimes suppress~\(p\) if it is equal to~\(1\)
and write~\(\setweight{\mathcal{U}} = \setweight{\mathcal{U}}_{1}\).
Note also that
the cardinality of a set can be written as its total~\(0\)-weight,
that is to say~\(\setcard{\mathcal{U}} = \setweight{\mathcal{U}}_{0}\).

In a first step we calculate the expectation of the
\enquote{total~\(p\)-connectivity weight}
of the explored graph up to level~\(\ell\).
\begin{lemma}\label{lem:bl:weight:raw}
  Let~\(p \geq 0\).
  For any level~\(\ell \in \naturals\) and vertex~\(v \in V_{n}\)
  we have
  \[
    \expe_{n}[\setweight{S_{\ell}(v)}_{p}]
    \leq  W_{v}^{p} + W_{v}(\Gamma_{2,n}+1)^{\ell-1}\Gamma_{p+1,n}.
  \]
  For~\(p=0\) this bound becomes
  \[
    \expe_{n}[\setcard{S_{\ell}(v)}]
    \leq  1+ W_{v} \Gamma_{1,n} (\Gamma_{2,n}+1)^{\ell-1}.
  \]
  In case~\(p=1\) we also have
  \[
    \expe_{n}[\setweight{S_{\ell}(v)}]
    \leq  W_{v}(\Gamma_{2,n}+1)^{\ell}.
  \]
\end{lemma}
\begin{proof}
  By construction and the fact that the~\(D_{r}(v)\) are disjoint
  \[
    \setweight{S_{\ell}(v)}_{p} = \sum_{r=0}^{\ell} \setweight{D_{r}(v)}_{p}.
  \]
  Recall that~\(D_{0}(v) = \set{v}\)
  so that~\(\setweight{D_{0}(v)}_{p} = W_{v}^{p}\).
  For~\(r \geq 1\) we have that
  \begin{align}
    \expe_{n}[\setweight{D_{r}(v)}_{p}]
    &\leq
      \sum_{\substack{u_{1},\dots,u_{r} \in
      V_{n}\setminus\set{v}\\\text{pairwise different}}}
    \expe_{n}[
    \conn{v}{u_{1}}
    \conn{u_{1}}{u_{2}}
    \dotsm
    \conn{u_{r-1}}{u_{r}}
    W_{u_{r}}^{p}]\notag\\
    &\leq \sum_{u_{1} \in V_{n}}  \frac{W_{v}W_{u_{1}}}{n \vartheta}
      \sum_{u_{2} \in V_{n}}  \frac{W_{u_{1}}W_{u_{2}}}{n \vartheta}
      \dotsm
      \sum_{u_{r} \in V_{n}}  \frac{W_{u_{r-1}}W_{u_{r}}}{n \vartheta}
      W_{u_{r}}^{p}\notag\\
    &\leq W_{v}
      \Bigl(\frac{1}{n\vartheta} \sum_{u \in V_{n}} W_{u}^{2}\Bigr)^{\!\!r-1}
    \Bigl(\frac{1}{n\vartheta} \sum_{u \in V_{n}} W_{u}^{p+1}\Bigr)\notag\\
    &= W_{v}
      \Gamma_{2,n}^{r-1}\Gamma_{p+1,n}.\label{eq:drv}
  \end{align}
  Now sum over~\(r\) to obtain
  \begin{equation}
    \expe_{n}[\setweight{S_{\ell}(v)}_{p}]
    \leq W_{v}^{p} + W_{v}\Gamma_{p+1,n} (\Gamma_{2,n}+1)^{\ell-1}.\label{eq:dvsum}
  \end{equation}
  This proves the first part of the claim.

  The second part of the claim then immediately follows by taking~\(p=0\).

  For the third claim note that in case~\(p=1\) the bound in~\eqref{eq:drv}
  becomes
  \begin{equation*}
    \expe_{n}[\setweight{D_{r}(v)}]
    \leq W_{v} \Gamma_{2,n}^{r}.
  \end{equation*}
  But this bound also holds for~\(r=0\),
  so that the summation in~\eqref{eq:dvsum} becomes
  \[
    \expe_{n}[\setweight{S_{\ell}(v)}]
    \leq W_{v} \sum_{r=0}^{\ell} \Gamma_{2,n}^{r}
    \leq W_{v} (\Gamma_{2,n}+1)^{\ell},
  \]
  as claimed.
\end{proof}
Summing the result gives:
\begin{corollary}\label{cor:blv:weight:raw}
  Let~\(p \geq 0\).
  For any level~\(\ell \in \naturals\)
  and vertex set~\(\mathcal{V} \subseteq V_{n}\)
  we have
  \[
    \expe_{n}[\setweight{S_{\ell}(\mathcal{V})}_{p}]
    \leq  \setweight{\mathcal{V}}_{p} +
    \setweight{\mathcal{V}}(\Gamma_{2,n}+1)^{\ell-1}\Gamma_{p+1,n}
  \]
  and thus also
  \[
    \expe_{n}[\setweight{S_{\ell}(\mathcal{V})}]
    \leq  \setweight{\mathcal{V}}(\Gamma_{2,n}+1)^{\ell}.
  \]
\end{corollary}

We will need a similar result for the excess connectivity weight
of vertices
\begin{definition}[Excess connectivity weight]
  Fix~\(n \in \naturals\).
  Let~\(\mathcal{U} \subseteq V_{n}\)
  a set of vertices in~\(G_{n}\).
  Then we say that
  \[
    \setweight{\mathcal{U}}_{+}
    = \sum_{u \in \mathcal{U}} W_{u} \indfunc_{\set{W_{u} > \sqrt{n\vartheta}}}
  \]
  is the \emph{total sum of excess connectivity weights of~\(\mathcal{U}\)}.
\end{definition}
Replicating the exact same arguments as in \autoref{lem:bl:weight:raw}
we can show a structurally similar bound for~\(\setweight{S_{\ell}(v)}_{+}\)
\citep[for details see][Lem.~3.1.7]{thesis}.

\begin{lemma}\label{lem:bl:excessweight:raw}
  For any level~\(\ell \in \naturals\) and vertex~\(v \in V_{n}\)
  we have
  \[
    \expe_{n}[\setweight{S_{\ell}(v)}_{+}]
    \leq  W_{v}\indfunc_{\set{W_{v} > \sqrt{n\vartheta}}}
    + W_{v}(\Gamma_{2,n}+1)^{\ell-1}\kappa_{2,n}.
  \]
\end{lemma}

For estimates involving the Cauchy--Schwarz inequality
it will also be useful to have a bound on the second moment
of~\(\setweight{S_{\ell}(v)}\).
\begin{lemma}\label{lem:bl:weight:squared:raw}
  For any level~\(\ell \in \naturals\) and vertex~\(v \in V_{n}\)
  we have
  \[
    \begin{split}
    &\expe_{n}[\setweight{S_{\ell}(v)}^{2}_{p}]\\
    &\quad\leq
    W_{v}^{2p} + 2W_{v}^{p+1}\Gamma_{p+1,n}(\Gamma_{2,n}+1)^{\ell-1}\\
    &\qquad+C (W_{v}+1)^{2}(\Gamma_{2,n}+2)^{2\ell-2}
    (\Gamma_{3,n}+1)(\Gamma_{p+1,n}+1)^{2}(\Gamma_{p+2,n}+1)(\Gamma_{2p+1,n}+1).
    \end{split}
  \]
  For~\(p=0\) this estimate can be slightly simplified further to
  \[
    \expe_{n}[\setcard{S_{\ell}(v)}^{2}]
    =\expe_{n}[\setweight{S_{\ell}(v)}^{2}_{0}]
    \leq C(W_{v}+1)^{2} (\Gamma_{1,n}+1)^{2}(\Gamma_{2,n}+2)^{2\ell}(\Gamma_{3,n}+1)
  \]
  and for~\(p=1\) to
  \[
    \expe_{n}[\setweight{S_{\ell}(v)}^{2}]
    \leq C(W_{v}+1)^{2} (\Gamma_{2,n}+2)^{2\ell}(\Gamma_{3,n}+1).
  \]
\end{lemma}
\begin{proof}
  Recall that~\(D_{0}(v) = \set{v}\)
  so that~\(\setweight{D_{0}(v)}_{p} = W_{v}^{p}\)
  and that by construction~\(S_{\ell}(v) = \bigsetunion_{r=0}^{\ell} D_{r}(v)\).
  Then
  \begin{equation}
    \setweight{S_{\ell}(v)}^{2}_{p}
    = W_{v}^{2p} + \sum_{r=1}^{\ell} \setweight{D_{r}(v)}^{2}_{p}
      + 2W_{v}^{p}\sum_{r=1}^{\ell} \setweight{D_{r}(v)}_{p}
      + \sum_{\substack{r,s=1\\r \neq s}}^{\ell} \setweight{D_{r}(v)}_{p}\setweight{D_{s}(v)}_{p}.
      \label{eq:sksquared:tod}
  \end{equation}
  A bound for the first sum
  was already established in~\eqref{eq:dvsum} in the
  proof of \autoref{lem:bl:weight:raw}.
  It thus remains to bound the second and third sum.

  The elements of~\(D_{r}(v)\) are exactly those vertices~\(u\),
  for which there exists a path of length exactly~\(r\)
  from~\(v\) to~\(u\) and no such path of length~\(s< r\).
  Write~\(v \to_{r} u\) if this is the case.
  Then
  \begin{equation}\label{eq:dk:square:sep}
    \setweight{D_{r}(v)}^{2}_{p}
    = \sum_{u} \indfunc_{\set{v \to_{r} u}} W_{u}^{2p}
    + \sum_{u \neq u'}
    \indfunc_{\set{v \to_{r} u}} \indfunc_{\set{v \to_{r} u'}} W_{u}^{p}
    W_{u'}^{p}
  \end{equation}
  and
  \begin{equation}\label{eq:drds:square+mixed}
    \setweight{D_{r}(v)}_{p} \setweight{D_{s}(v)}_{p}
    = \sum_{u} \indfunc_{\set{v \to_{r} u}} \indfunc_{\set{v \to_{s} u}}
    W_{u}^{2p}
    + \sum_{u \neq u'} \indfunc_{\set{v \to_{r} u}} \indfunc_{\set{v
        \to_{s} u'}}
    W_{u}^{p}W_{u'}^{p}.
  \end{equation}

  The first sum in~\eqref{eq:dk:square:sep}
  involves only paths to a single vertex and
  its expectation is easily estimated as follows
  \begin{equation}
    \expe_{n}\Bigl[\sum_{u} \indfunc_{\set{v \to_{r} u}} W_{u}^{2p}\Bigr]
    \leq \sum_{\substack{u_{1},u_{2},\dots,u_{r}\\\text{pairw.~diff.}}}
    \frac{W_{v}W_{u_{1}}}{n\vartheta}
    \frac{W_{u_{1}}W_{u_{2}}}{n\vartheta}
    \dots
    \frac{W_{u_{r-1}}W_{u_{r}}}{n\vartheta}W_{u_{r}}^{2p}
    \leq W_{v}
      \Gamma_{2,n}^{r-1}
      \Gamma_{2p+1,n}.
      \label{eq:dk:square:same}
  \end{equation}

  The first sum in~\eqref{eq:drds:square+mixed} is zero for~\(r \neq s\),
  since the shortest path from~\(v\) to~\(u\)
  cannot have length both~\(r\) and~\(s\).

  The paths from~\(v\) to~\(u\) and~\(u'\) in the second sums
  in~\eqref{eq:dk:square:sep} and~\eqref{eq:drds:square+mixed}
  may share edges and thus
  are not necessarily independent.

  If the shortest paths from~\(v\) to~\(u\) and~\(u'\),
  respectively, are of length~\(r\) and~\(s\) and
  share a vertex,
  then the last common vertex of the two paths must be at the same distance
  from~\(v\) in both paths and we may assume that the two paths
  agree up to that point and do not have any edge in common
  after that point.

  Hence, the event that~\(u\) and~\(u'\)
  (with~\(u \neq u'\))
  can be reached from~\(v\)
  in~\(r\) and~\(s\) steps, respectively, can be estimated by counting
  \enquote{eventually bifurcating paths} that agree from~\(v\)
  up to a last common vertex~\(z\) and have no further common vertices after that.
  We will develop the argument for the case~\(r \neq s\),
  where we may assume~\(r < s\)
  without loss of generality.
  The argument for~\(r=s\) is similar,
  so we will not discuss it in more detail here
  \citep[details can be found in][Lem.~3.1.8]{thesis}.

  In case~\(r<s\)
  the paths either bifurcate immediately at~\(v\)
  or they split at a later step~\(t \in \set{1,\dots,r}\).
  Note that if the paths split at~\(r\),
  the path from~\(v\) to~\(u'\)
  includes the complete path from~\(v\) to~\(u\).
  If we write~\(u_{r} = u\) and~\(u'_{s} = u'\)
  we thus have
  \begin{align*}
    &\expe_{n}\Bigl[\sum_{u \neq u'}
      \indfunc_{\set{v \to_{r} u}} \indfunc_{\set{v \to_{s} u'}} W_{u}^{p}
      W_{u'}^{p}\Bigr]\notag\\
    &\quad\leq
      \sum_{\substack{u_{1}, \dots, u_{r},\\
    u'_{1}, \dots, u'_{s}\\
    \text{pairw.~diff.}}}
    \frac{W_{v}W_{u_{1}}}{n\vartheta}
    \dotsm
    \frac{W_{u_{r-1}}W_{u_{r}}}{n\vartheta}
    W_{u_{r}}^{p}
    \frac{W_{v}W_{u'_{1}}}{n\vartheta}
    \dotsm
    \frac{W_{u'_{s-1}}W_{u'_{s}}}{n\vartheta}
    W_{u'_{s}}^{p}\notag\\
    &\qquad + \sum_{t=1}^{r}
      \;\sum_{\substack{u_{1}, u_{2}, \dots, u_{t},\\
    u_{t+1}, \dots, u_{r},\\
    u'_{t+1}, \dots, u'_{s}\\
    \text{pairw.~diff.}}}
    \frac{W_{v} W_{u_{1}}}{n\vartheta}
    \frac{W_{u_{1}} W_{u_{2}}}{n\vartheta}
    \dotsm
    \frac{W_{u_{t-1}} W_{u_{t}}}{n\vartheta}\notag\\
    &\qqquad
      \frac{W_{u_{t}} W_{u_{t+1}}}{n\vartheta}
      \dotsm
      \frac{W_{u_{r-1}} W_{u_{r}}}{n\vartheta}
      W_{u_{r}}^{p}
      \frac{W_{u_{t}} W_{u'_{t+1}}}{n\vartheta}
      \dotsm      \frac{W_{u'_{s-1}} W_{u'_{s}}}{n\vartheta}
      W_{u'_{s}}^{p}.
  \end{align*}
  Separate the sums over the~\(u\)s and~\(u'\)s,
  count the multiplicity of the respective~\(W_{u}\)s
  and use the definition of~\(\Gamma_{q,n}\)
  to obtain
  the bound
  \begin{align}
    &W_{v}^{2} \Gamma_{2,n}^{r+s-2}
    \Gamma_{p+1,n}^{2}
    +\sum_{t=1}^{r-1} W_{v}
    \Gamma_{2,n}^{r+s-t-3}
    \Gamma_{3,n}
      \Gamma_{p+1,n}^{2}
    +W_{v} \Gamma_{2,n}^{s-2}
      \Gamma_{p+1,n}
      \Gamma_{p+2,n}.\notag
  \intertext{The summation of~\(r+s-t-3\) from~\(t=1\)
  to~\(t=r-1\) can be rewritten as a summation of~\(t\)
  from~\(t= s-2\) to~\(t=r+s-4\), which in turn
  is a summation of~\(s-2+t\) from~\(t=0\) to~\(r-2\).}
    &\quad\leq
      W_{v}^{2} \Gamma_{2,n}^{r+s-2}\Gamma_{p+1,n}^{2}
    +W_{v} \Gamma_{3,n} \Gamma_{p+1,n}^{2}
      \Gamma_{2,n}^{s-2}(\Gamma_{2,n}+1)^{r-2}
      +W_{v} \Gamma_{2,n}^{s-2} \Gamma_{p+1,n}\Gamma_{p+2,n}.\notag
  \intertext{In order to simplify this expression, we
      estimate very generously to factor out common terms}
  &\quad\leq
    W_{v}^{2} \Gamma_{2,n}^{r+s-2}\Gamma_{p+1,n}^{2}
   + W_{v} (\Gamma_{2,n}+1)^{r+s-2}(\Gamma_{3,n}+1)
    (\Gamma_{p+1,n}+1)^{2}(\Gamma_{p+2,n}+1)
    .
    \label{eq:dk:square:diff}
\end{align}
Note that the generous estimate can be slightly simplified in
case~\(p=1\), because then~\(\Gamma_{3,n}\) and~\(\Gamma_{p+2,n}\)
coincide so that after factoring only one of the two terms
needs to be part of the product.

The case~\(r=s\)
can be treated similarly,
but the sum does not include a term for~\(t=r\),
so that
\begin{equation}
  \expe_{n}\Bigl[\sum_{u \neq u'}
  \indfunc_{\set{v \to_{r} u}} \indfunc_{\set{v \to_{r} u'}} W_{u}^{p}
  W_{u'}^{p}\Bigr]
  \leq
    W_{v}^{2} \Gamma_{2,n}^{2r-2}\Gamma_{p+1,n}^{2}
   + W_{v} (\Gamma_{2,n}+1)^{2r-2}(\Gamma_{3,n}+1)
    (\Gamma_{p+1,n}+1)^{2}
    .
    \label{eq:dk:square:diff:rr}
\end{equation}

Now~\eqref{eq:dk:square:sep}
together with~\eqref{eq:dk:square:same}, \eqref{eq:dk:square:diff:rr}
and generous estimates of the involved terms
imply
\begin{equation}
  \expe_{n}[\setweight{D_{r}(v)}_{p}^{2}]
  \leq W_{v}^{2}\Gamma_{2,n}^{2r-2}\Gamma_{p+1,n}^{2}
    + W_{v} (\Gamma_{2,n}+1)^{2r-2}
    (\Gamma_{3,n}+1)(\Gamma_{p+1,n}+1)^{2}(\Gamma_{2p+1,n}+1)
    \label{eq:dksquare}.
\end{equation}
This estimate can be simplified if~\(2p+1\)
should happen to be equal to~\(3\) or \(p+1\).
In those cases the~\((\Gamma_{2p+1,n}+1)\) can be replaced by a constant.

Analogously~\eqref{eq:drds:square+mixed} and~\eqref{eq:dk:square:diff}
as well as similarly generous estimates
imply
\begin{equation}\label{eq:dkdl}
  \expe_{n}[\setweight{D_{r}(v)}_{p} \setweight{D_{s}(v)}_{p}]
  \leq     W_{v}^{2} \Gamma_{2,n}^{r+s-2}\Gamma_{p+1,n}^{2}
   + W_{v} (\Gamma_{2,n}+1)^{r+s-2}(\Gamma_{3,n}+1)
    (\Gamma_{p+1,n}+1)^{2}(\Gamma_{p+2,n}+1)
    .
\end{equation}
As remarked after~\eqref{eq:dk:square:diff}
the term~\((\Gamma_{p+2,n}+1)\) may be dropped for~\(p=1\).

Briefly write
\begin{equation*}
  x_{1} = (\Gamma_{3,n}+1)(\Gamma_{p+1,n}+1)^{2}(\Gamma_{2p+1,n}+1)
  \quad\text{and}\quad
  x_{2} = (\Gamma_{3,n}+1)(\Gamma_{p+1,n}+1)^{2}(\Gamma_{p+2,n}+1)
\end{equation*}
to simplify~\eqref{eq:dksquare} and~\eqref{eq:dkdl}, respectively.
Then
\[
  x_{1}+x_{2}
  \leq C(\Gamma_{3,n}+1)(\Gamma_{p+1,n}+1)^{2}(\Gamma_{p+2,n}+1)(\Gamma_{2p+1,n}+1).
\]
By the comments after~\eqref{eq:dksquare} and~\eqref{eq:dk:square:diff},
the~\((\Gamma_{p+2,n}+1)\) may be dropped if~\(p=1\)
and the~\((\Gamma_{2p+1}+1)\)
may be replaced by a constant if~\(2p+1 \in \set{3,p+1}\).

\begingroup
\allowdisplaybreaks
Now~\eqref{eq:sksquared:tod}, \eqref{eq:dvsum}, \eqref{eq:dksquare},
\eqref{eq:dkdl} and similarly rough estimates of the terms involved
imply
\begin{align*}
  &\expe_{n}[\setweight{S_{\ell}(v)}_{p}^{2}]\\
  &\quad\leq
    W_{v}^{2p} + 2W_{v}^{p+1}\Gamma_{p+1,n}(\Gamma_{2,n}+1)^{\ell-1}\\
  &\qquad + C (W_{v}+1)^{2}(\Gamma_{2,n}+2)^{2\ell-2}
  (\Gamma_{3,n}+1)(\Gamma_{p+1,n}+1)^{2}(\Gamma_{p+2,n}+1)(\Gamma_{2p+1,n}+1)
    .
\end{align*}
This shows the claim for general~\(p\).

Again,
for~\(p=1\) the term~\((\Gamma_{p+2,n}+1)\) may be dropped
and the~\((\Gamma_{2p+1}+1)\)
may be replaced by a constant if~\(2p+1 \in \set{3,p+1}\).
Furthermore, for~\(p = 0\) and~\(p=1\)
the first terms can be absorbed into the last term
by increasing~\(C\) appropriately.
This then proves the slightly simplified results for~\(p=0\) and~\(p=1\).
\endgroup
\end{proof}

We can sum the previous bound to obtain results for~\(S_{\ell}(\mathcal{V})\),
where~\(\mathcal{V}\) is a set of vertices.
In the following we will only need the results for~\(p=0\)
and~\(p=1\), so we only work with the simplified results from
\autoref{lem:bl:weight:squared:raw}
\begin{corollary}\label{cor:bvl:weight:squared:raw}
  For any vertex set~\(\mathcal{V} \subseteq V_{n}\)
  and level~\(\ell \in \naturals\)
  we have
  \[
    \expe_{n}[\setcard{S_{\ell}(\mathcal{V})}^{2}]
    =\expe_{n}[\setweight{S_{\ell}(\mathcal{V})}^{2}_{0}]
    \leq C(\setweight{\mathcal{V}}+\setcard{\mathcal{V}})^{2} (\Gamma_{1,n}+1)^{2}(\Gamma_{2,n}+2)^{2\ell}(\Gamma_{3,n}+1)
  \]
  and
  \[
    \expe_{n}[\setweight{S_{\ell}(\mathcal{V})}^{2}]
    \leq C(\setweight{\mathcal{V}}+\setcard{\mathcal{V}})^{2} (\Gamma_{2,n}+2)^{2\ell}(\Gamma_{3,n}+1).
  \]
\end{corollary}
\begin{proof}
  To simplify notation we will assume that for~\(p=0\)
  and~\(p=1\)
  \autoref{lem:bl:weight:squared:raw}
  gives a bound of the form
  \[
    \expe_{n}[\setweight{S_{\ell}(v)}_{p}^{2}]
    \leq C (W_{v}+1)^{2} x_{n,p},
  \]
  where~\(x_{n,p}\) consists of terms involving~\(\Gamma_{2,n}\), \(\Gamma_{3,n}\),
  \(\Gamma_{p+1,n}\), \(\Gamma_{p+2,n}\) and~\(\Gamma_{2p+1,n}\).
  Note that
  \begin{equation}
    \expe_{n}[\setweight{S_{\ell}(\mathcal{V})}_{p}^{2}]
    \leq \sum_{v \in \mathcal{V}} \expe_{n}[\setweight{S_{\ell}(v)}_{p}^{2}]
    + \sum_{\substack{v,v' \in \mathcal{V}\\v \neq v'}}
    \expe_{n}[\setweight{S_{\ell}(v)}_{p} \setweight{S_{\ell}(v')}_{p}]
    \label{eq:slvsetsquared}.
  \end{equation}
  For the first sum we can apply \autoref{lem:bl:weight:squared:raw}
  and obtain
  \begin{equation*}
    \sum_{v \in \mathcal{V}} \expe_{n}[\setweight{S_{\ell}(v)}_{p}^{2}]
    \leq C x_{p,n}
      \sum_{v\in \mathcal{V}} (W_{v}+1)^{2}
    \leq C  x_{p,n}
      (\setweight{\mathcal{V}}+\setcard{\mathcal{V}})^{2}.
  \end{equation*}
  The terms in the second sum can be estimated using Cauchy--Schwarz
  and \autoref{lem:bl:weight:squared:raw}
  \begin{equation*}
    \expe_{n}[\setweight{S_{\ell}(v)} \setweight{S_{\ell}(v')}]
    \leq Cx_{p,n}(W_{v}+1)(W_{v'}+1).
  \end{equation*}
  Summing over these terms we obtain
  \begin{equation*}
    \sum_{\substack{v,v' \in \mathcal{V}\\v \neq v'}}
    \expe_{n}[\setweight{S_{\ell}(v)} \setweight{S_{\ell}(v')}]
    \leq Cx_{p,n}
      \sum_{\substack{v,v' \in \mathcal{V}\\v \neq v'}}
    (W_{v}+1)(W_{v'}+1)
    \leq Cx_{p,n}
      (\setweight{\mathcal{V}}+\setcard{\mathcal{V}})^{2}.
  \end{equation*}
  Hence, both terms on the right-hand side of~\eqref{eq:slvsetsquared}
  can be estimated with the same bound.
  This finishes the proof.
\end{proof}

In fact the approach from~\eqref{eq:dk:square:sep}
can be generalised to bound (higher) moments of the degree
distribution of a given vertex~\(v\) in~\(G_{n}\).
Recall that with the notation from \autoref{def:dl}
the degree of the vertex~\(v\) in~\(G_{n}\)
can be written as~\(\setcard{D_{1}(v)}\).

In the calculations for higher moments of the degree we will encounter
the Stirling numbers of the second kind
\citep[see e.g.][\S\,24.1.4]{abramowitz-stegun}.
\begin{definition}[{Stirling numbers of the second kind}]
  Let~\(k \in \naturals\) and~\(j \in \set{1,\dots,k}\).
  Then the number of ways of partitioning a set of~\(k\) elements
  into~\(j\) non-empty subsets
  is denoted by~\(\mathcal{S}_{k}^{(j)}\)
  and called the Stirling number of the second kind for~\(k,j\).
\end{definition}

\begin{lemma}\label{lem:higher:deg}
  Fix~\(v \in V_{n}\) and~\(k \in \naturals\).
  Then
  \[
    \expe_{n}[\setcard{D_{1}(v)}^{k}]
    \leq  \sum_{j=1}^{k} \mathcal{S}_{k}^{(j)}
      W_{v}^{j} \Gamma_{1,n}^{j}
    \leq (W_{v}+1)^{k}(\Gamma_{1,n}+k)^{k}.
  \]
\end{lemma}
\begin{proof}

  Note first that
  \[
    \setcard{D_{1}(v)}^{k}
    = \Bigl(\sum_{u \in V_{n}} X_{vu}\Bigr)^{\!\!k}
    = \sum_{u_{1},\dots,u_{k}} X_{vu_{1}} \dotsm X_{vu_{k}}.
  \]
  Since the~\(X_{vu_{\ell}}\) are indicator functions,
    any product of~\(X_{vu_{\ell}}\) with identical~\(u_{\ell}\)s collapses
    to one such indicator.
    We can therefore rewrite the sum over all possible combinations of~\(k\)
    vertices
    into a sum over~\(j\) pairwise different vertices to obtain
  \begin{equation*}
    \setcard{D_{1}(v)}^{k}
    = \sum_{j=1}^{k} \mathcal{S}_{k}^{(j)}
      \sum_{\substack{u_{1},\dots,u_{j}\\\text{pairw.~diff.}}} X_{vu_{1}}\dotsm X_{vu_{j}},
  \end{equation*}
  because each product~\(X_{vu_{1}} \dotsm X_{vu_{j}}\)
  with pairwise different~\(u_{\ell}\)
  is realised by exactly~\(\mathcal{S}_{k}^{(j)}\) products of the
  form~\(X_{vu_{1}} \dotsm X_{vu_{k}}\),
  where some of the~\(u_{\ell}\) may be equal.
  To see this in more detail,
  partition the~\(k\) vertices~\(u_{1},\dots,u_{k}\)
  into subsets such that~\(u_{r}\) and~\(u_{s}\)
  are in the same subset of they are the same.
  The number of nonempty subsets of vertices that this procedure produces is
  exactly equal to the number vertices we need
  to write~\(X_{vu_{1}}\dots X_{vu_{k}}\) as a product
  over pairwise different vertices.
  In particular there are exactly~\(\mathcal{S}_{k}^{(j)}\)
  to obtain a product of~\(j\) pairwise different vertices.

  Taking expectations and using that
  the~\(X_{vu_{\ell}}\) in the sums are independent
  since the~\(u_{\ell}\) are pairwise different, we obtain
  \begin{equation*}
    \expe_{n}[\setcard{D_{1}(v)}^{k}]
    = \sum_{j=1}^{k} \mathcal{S}_{k}^{(j)}
      \sum_{\substack{u_{1},\dots,u_{j}\\\text{pairw.~diff.}}}
    \expe_{n}[X_{vu_{1}}]\dotsm \expe_{n}[X_{vu_{j}}]
    = \sum_{j=1}^{k} \mathcal{S}_{k}^{(j)}
      W_{v}^{j}\Gamma_{1,n}^{j},
  \end{equation*}
  which proves the first part of the claim.

  For the second part note that for~\(k \geq 2\) the Stirling numbers
  of the second kind can be bounded above as follows
  \citep[Thm.~3]{rennie}
  \[
    \mathcal{S}_{k}^{(j)}
    \leq \frac{1}{2} \binom{k}{j} j^{k-j}
    \leq \frac{1}{2} \binom{k}{j} k^{k-j},
  \]
  so that
  \begin{equation*}
    \expe_{n}[\setcard{D_{1}(v)}^{k}]
    \leq \sum_{j=1}^{k} \mathcal{S}_{k}^{(j)} W_{v}^{j}\Gamma_{1,n}^{j}
    \leq \frac{1}{2} \sum_{j=1}^{k} \binom{k}{j} (W_{v}\Gamma_{1,n})^{j} k^{k-j}
    = \frac{1}{2} (W_{v}\Gamma_{1,n}+k)^{k}.
  \end{equation*}
  For~\(k=1\) we
  have~\(\expe_{n}[\setcard{D_{1}(v)}] \leq W_{v}\Gamma_{1,n}
    \leq W_{v}\Gamma_{1,n}+1
  \).
  Thus we get for all~\(k \in \naturals\)
  \begin{equation*}
    \expe_{n}[\setcard{D_{1}(v)}^{k}]
    \leq (W_{v}\Gamma_{1,n}+k)^{k}
    \leq (W_{v}+1)^{k}(\Gamma_{1,n}+k)^{k}
  \end{equation*}
  as claimed.
\end{proof}

We will also need to bound the probability
that there exists a path of given length between two subsets of vertices.
To simplify notation we introduce a shorthand for this event.
\begin{definition}
  Let~\(\mathcal{U}\) and~\(\mathcal{V}\) be two sets of vertices.
  Let~\(\mathcal{R}\) be a third set of vertices that is disjoint
  from~\(\mathcal{U}\) and~\(\mathcal{V}\)

  \begin{itemize}
  \item Write~\(\mathcal{U} \pathbetw_{\ell} \mathcal{V}\)
  if there exists a path of length~\(\ell\)
  between a vertex~\(u \in \mathcal{U}\)
  and a vertex~\(v \in \mathcal{V}\).

  In some cases it might be useful to restrict this event
  to paths that avoid a set of vertices~\(\mathcal{R}\),
  so we write~\(\mathcal{U} \pathbetw_{(\mathcal{R})\ell} \mathcal{V}\)
  if there is a path of length~\(\ell\) between a vertex~\(u \in \mathcal{U}\)
  and a vertex~\(v \in \mathcal{V}\)
  that does not use any vertices from~\(\mathcal{R}\).

  \item Write~\(\mathcal{U} \pathbetw_{\leq \ell} \mathcal{V}\)
  if there exists a path of length at most~\(\ell\)
  between a vertex~\(u \in \mathcal{U}\)
  and a vertex~\(v \in \mathcal{V}\).

  As above we
  write~\(\mathcal{U} \pathbetw_{(\mathcal{R}) \leq \ell} \mathcal{V}\)
  if there is a path between~\(\mathcal{U}\) and~\(\mathcal{V}\)
  that does not use any vertices from~\(\mathcal{R}\).
  \end{itemize}
\end{definition}
As we do so often, we drop the curly brackets if~\(\mathcal{V}\)
or~\(\mathcal{U}\) is a set with a single element and write
for example~\(v \pathbetw_{\ell} u\) for~\(\set{v} \pathbetw_{\ell} \set{u}\).

\begin{lemma}\label{lem:abpath}
  Fix~\(\ell \in \naturals\), \(\ell \geq 1\).
  Let~\(\mathcal{U}, \mathcal{V} \subseteq V_{n}\)
  be two disjoint sets of vertices
  then
  \[
    \prob_{n}(\mathcal{U} \pathbetw_{\ell} \mathcal{V})
      \leq \frac{\setweight{\mathcal{U}}\setweight{\mathcal{V}}}{n \vartheta}
      \Gamma_{2,n}^{\ell-1}.
  \]
\end{lemma}
\begin{proof}
  Recall that in a path no vertex can appear multiple times.
  This ensures that the edges of a path are all distinct and hence independent,
  so that
  \begin{align*}
    \prob_{n}(\mathcal{U} \pathbetw_{\ell} \mathcal{V})
    &\leq \sum_{\substack{u_{0} \in \mathcal{U}\\u_{1},\dots,u_{\ell-1}\\
    u_{\ell} \in \mathcal{V}}}
      \expe_{n}[X_{u_{0}u_{1}} X_{u_{1}u_{2}}
      \dotsm X_{u_{\ell-1}u_{\ell}}]\\
    &\quad\leq \frac{1}{n\vartheta} \sum_{u_{0} \in \mathcal{U}} W_{u_{0}}
      \Bigl(\frac{1}{n\vartheta} \sum_{v} W_{v}^{2}\Bigr)^{\!\!\ell-1}
      \sum_{u_{\ell} \in \mathcal{V}} W_{u_{\ell}}\\
    &\quad\leq \frac{\setweight{\mathcal{U}}\setweight{\mathcal{V}}}{n \vartheta}
      \Gamma_{2,n}^{\ell-1}.
  \end{align*}
  This finishes the proof.
\end{proof}

In particular the probability that there is an edge between two sets
of vertices~\(\mathcal{U}\) and~\(\mathcal{V}\) is bounded above
by~\((n\vartheta)^{-1}\setweight{\mathcal{U}}\setweight{\mathcal{V}}\).
\begin{corollary}\label{cor:abpathupto}
  Fix~\(\ell \in \naturals\), \(\ell \geq 1\).
  Let~\(\mathcal{U}, \mathcal{V} \subseteq V_{n}\)
  be two disjoint sets of vertices,
  then
  \[
      \prob_{n}(\mathcal{U} \pathbetw_{\leq \ell} \mathcal{V})
      \leq \frac{\setweight{\mathcal{U}}\setweight{\mathcal{V}}}{n \vartheta}
        (1+\Gamma_{2,n})^{\ell-1}.
  \]
\end{corollary}
\begin{proof}
  By \autoref{lem:abpath} we have
  \begin{equation*}
    \prob_{n}(\mathcal{U} \pathbetw_{\leq \ell} \mathcal{V})
    \leq \sum_{r=1}^{\ell}
      \prob_{n}(\mathcal{U} \pathbetw_{r} \mathcal{V})
    \leq \sum_{r=1}^{\ell} \frac{\setweight{\mathcal{U}}\setweight{\mathcal{V}}}{n \vartheta}
      \Gamma_{2,n}^{r-1}
    \leq \frac{\setweight{\mathcal{U}}\setweight{\mathcal{V}}}{n \vartheta}
      (1+\Gamma_{2,n})^{\ell-1}
  \end{equation*}
  as claimed.
\end{proof}

The following special case of \autoref{cor:abpathupto} is of particular
interest to us.
\begin{corollary}\label{lem:uinbl}
  For any level~\(\ell \in \naturals\)
  and vertices~\(v,u \in V_{n}\) with~\(u \neq v\)
  we have
  \[
    \prob_{n}(u \in B_{\ell}(v))
    \leq
    \frac{W_{u}W_{v}}{n\vartheta} (\Gamma_{2,n}+1)^{\ell-1}.
  \]
\end{corollary}
\begin{proof}
  The vertex~\(u\) is contained in
  the~\(\ell\)-neighbourhood~\(B_{\ell}(v)\) of~\(v\)
  if and only if there is a path of length at most~\(\ell\)
  from~\(v\) to~\(u\).
  Hence, we apply \autoref{cor:abpathupto} with~\(\mathcal{V}=\set{v}\)
  and~\(\mathcal{U}=\set{u}\)
  to conclude the claim.
\end{proof}

This result can be used to bound the probability that an edge~\(e\)
is contained in~\(B_{\ell}(v)\).
\begin{corollary}\label{lem:einbl}
  For any level~\(\ell \in \naturals\),
  vertex~\(v \in V_{n}\)
  and edge~\(e = \edge{u}{u'}\) whose endpoints are not equal to~\(v\)
  we have
  \[
    \prob_{n}(e \in B_{\ell}(v))
    \leq \frac{W_{v}(W_{u}+W_{u'})}{n\vartheta} (\Gamma_{2,n}+1)^{\ell-1}.
  \]
\end{corollary}
\begin{proof}
  The edge~\(e = \edge{u}{u'}\)
  can only be contained in the~\(\ell\)-neighbourhood
  of~\(v\) if one of its endpoints~\(u\) or~\(u'\)
  can be reached in at most~\(\ell-1\)
  steps from~\(v\).
  Hence
  \begin{equation*}
    \prob_{n}(e \in B_{\ell}(v))
    \leq \prob_{n}(u \in B_{\ell-1}(v)) + \prob_{n}(u' \in B_{\ell-1}(v)).
  \end{equation*}
  The claim now follows from \autoref{lem:uinbl}.
\end{proof}

\autoref{cor:abpathupto} still holds conditional
on knowing the edges emanating from a fixed set of vertices.
In order to formulate this in more detail,
we need some more notation.

\begin{definition}
  For any vertex~\(v \in V_{n}\) let
  \[
    \tilde{S}_{v} = (X_{\edge{v}{u}})_{u \in V_{n}}
  \]
  be the collection of all indicator functions of possible
  edges emanating from~\(v\).
\end{definition}

For an edge~\(e=\edge{u}{v}\) define
\[
  \tilde{S}_{e} = (S_{v},S_{u})
\]
so that~\(\tilde{S}_{e}\)
contains information about all edges incident to~\(e\).

For the proof of the result it is convenient to estimate the set of neighbours
of a vertex set in a way that makes this set independent of
the neighbours of another vertex set.
\begin{definition}\label{def:dignore}
  Let~\(\mathcal{V}, \mathcal{U} \subseteq V_{n}\) be two
  disjoint subsets of vertices.

  Set~\(S^{(\mathcal{U})}_{1}(\mathcal{V}) = S_{1}(\mathcal{V})
  \setminus \mathcal{U}\)
  and~\(D^{(\mathcal{U})}_{1}(\mathcal{V}) = D_{1}(\mathcal{V})
  \setminus \mathcal{U}\)
  so that
  \begin{align*}
    S^{(\mathcal{U})}_{1}(\mathcal{V})
    &= \mathcal{V} \setunion \sset{x \in V_{n} \setminus \mathcal{U}}
    {\text{\(X_{\edge{v}{x}}=1\) for some~\(v \in \mathcal{V}\)}}\\
    D^{(\mathcal{U})}_{1}(\mathcal{V})
    &= \sset{x \in V_{n} \setminus \mathcal{U}}
    {\text{\(X_{\edge{v}{x}}=1\) for some~\(v \in \mathcal{V}\)}}
  \end{align*}
  are independent of~\(X_{vu}\) for all~\(v \in \mathcal{V}\)
  and~\(u \in \mathcal{U}\).
\end{definition}
In particular then~\(D_{1}^{(\mathcal{U})}(\mathcal{V})\)
and~\(D_{1}^{(\mathcal{V})}(\mathcal{U})\) are independent.
The same holds for~\(S_{1}^{(\mathcal{U})}(\mathcal{V})\)
and~\(S_{1}^{(\mathcal{V})}(\mathcal{U})\).

Clearly we also have
\begin{equation*}
  \setcard{S_{1}(v)}
  \leq \setcard{S^{(\mathcal{U})}_{1}(v)} +\setcard{\mathcal{U}}
  \quad\text{and}\quad
  \setweight{S_{1}(v)}
  \leq \setweight{S_{1}^{(\mathcal{U})}(v)} + \setweight{\mathcal{U}}.
\end{equation*}
as well as
\begin{equation*}
  \setcard{D_{1}(v)}
  \leq \setcard{D^{(\mathcal{U})}_{1}(v)} +\setcard{\mathcal{U}}
  \quad\text{and}\quad
  \setweight{D_{1}(v)}
  \leq \setweight{D_{1}^{(\mathcal{U})}(v)} + \setweight{\mathcal{U}}.
\end{equation*}
Given a fixed number of vertices~\(u_{1},\dots,u_{m} \in V_{n}\)
we write~\(D_{1}^{(u_{1},\dots,u_{m})}(v)\) for~\(D_{1}^{(\set{u_{1},\dots,u_{m}})}(v)\);
we also write~\(S_{1}^{(u_{1},\dots,u_{m})}(v)\) for~\(S_{1}^{(\set{u_{1},\dots,u_{m}})}(v)\).

\begin{lemma}\label{lem:nbhdcontains}
Fix three disjoint sets of vertices~\(\mathcal{U}\),
\(\mathcal{V}\) and~\(\mathcal{R}\)
(the sets may be empty).
Set~\(S = (\widetilde{S}_{w})_{w \in \mathcal{U} \setunion \mathcal{V}
  \setunion \mathcal{R}}\)
and~\(\tilde{S} = (\widetilde{S}_{w})_{w \in \mathcal{U} \setunion \mathcal{V}}\)
(note the absence of vertices from~\(\mathcal{R}\) in~\(\tilde{S}\)).

Then there exists a function~\(\xi_{\ell}(\mathcal{U},\mathcal{V},\mathcal{R})\)
that is~\(\sigma(\tilde{S},\mathcal{F}_{n})\)-measurable
and independent of
all edges of the form~\(\edge{u}{u'}\) for~\(u,u' \in \mathcal{U}\),
\(\edge{v}{v'}\) for~\(v,v' \in \mathcal{V}\).
This function satisfies
\[
  \prob_{n}(\mathcal{U} \pathbetw_{(\mathcal{R}) \leq \ell}
    \mathcal{V} \given \tilde{S})
  \leq \xi_{\ell}(\mathcal{U},\mathcal{V},\mathcal{R}),
\]
and
\begin{equation*}
  \expe_{n}[\xi_{\ell}(\mathcal{U},\mathcal{V},\mathcal{R})]
  \leq \min\set[\bigg]{
      \frac{\setweight{\mathcal{U}}\setcard{\mathcal{V}}}{n\vartheta}
    (\Gamma_{2,n}+1)^{\ell-1},1}
\end{equation*}
as well as
\begin{equation*}
  \expe_{n}[\xi_{\ell}(\mathcal{U},\mathcal{V},\mathcal{R})
  f(\setcard{D_{1}(\mathcal{U})}, \setcard{D_{1}(\mathcal{V})} )]
  \leq C \beta_{n}(\mathcal{U},\mathcal{V})
  \min\set[\bigg]{
    \frac{\setweight{\mathcal{U}}\setweight{\mathcal{V}}}{n\vartheta}
    (\Gamma_{2,n}+1)^{\ell-1},1}
\end{equation*}
for all functions~\(f\) that are non-decreasing in both arguments such that
\[
  \expe_{n}[
  f(
    \setcard{D_{1}(\mathcal{U})}+\setcard{\mathcal{V}}+1,
    \setcard{D_{1}(\mathcal{V})}+\setcard{\mathcal{U}}+1
  )]
  \leq \beta_{n}(\mathcal{U},\mathcal{V}).
\]

\end{lemma}

We will not present the proof of this result here.
The rough idea of the calculation is
to isolate the edges emanating from~\(\mathcal{U}\),
\(\mathcal{V}\) and~\(\mathcal{R}\)
and treat the remaining graph that is independent of these
edges as in the proof of \autoref{lem:abpath}
\citep[a full proof of the result can be found in][Lem.~3.1.19]{thesis}.

These path and inclusion probabilities can now be used to bound the probability
that the neighbourhood of a vertex is not tree shaped.
Like the following section this proof is an extension
of an idea by \citet[Lem.~6.7]{cao}.

\begin{lemma}\label{lem:probnottree}
  For every vertex~\(v \in V_{n}\) and level~\(\ell \in \naturals\)
  \[
    \prob_{n}(\text{\(B_{\ell}(v)\) is not a tree})
    \leq C(1+\Gamma_{2,n})^{2\ell+1}(\Gamma_{3,n}+1)\frac{(W_{v}+1)^{2}}{n\vartheta}.
  \]
\end{lemma}
\begin{proof}
  Let~\(A_{\ell}\) be the event that~\(B_{\ell}(v)\) is a tree.
  If~\(B_{\ell-1}(v)\) is a tree,~\(B_{\ell}(v)\) can only fail
  to be a tree
  if there is an edge between two vertices in~\(D_{\ell-1}(v)\)
  or if there is a vertex not in~\(S_{\ell-1}(v)\)
  that is connected to two vertices in~\(D_{\ell-1}(v)\)
  (see \autoref{fig:notree}).
  \begin{figure}[htbp]
    \centering
    \begin{tikzpicture}[x=1.5cm,y=1.5cm,
      mynode/.style={draw, circle, inner sep=2pt, minimum size=1.2em}]

      \path [rounded corners=10pt, fill=green!20]
      (0,.5)--(-1.75,-.75)--(-2.2,-2.4)--(2,-2.4)--(1.75,-.75)
      --cycle;

      \path [rounded corners=10pt, fill=blue!20,]
      (-2,-1.75)--(-2,-2.3)--(1.8,-2.3)--(1.8,-1.75)
      --cycle;

      \node[mynode] (v) at (0,0)  {};

      \node[mynode] (v1) at (-1.5,-1) {};
      \node[mynode] (v2) at (0,-1)  {};
      \node[mynode] (v3) at (1.5,-1)  {};

      \node[mynode] (v11) at (-1.8,-2) {};
      \node[mynode] (v12) at (-1.2,-2)  {};

      \node[mynode] (v111) at (-2.2,-3) {};
      \node[mynode] (v112) at (-1.6,-3) {};

      \node[mynode] (v21) at (-.3,-2) {};
      \node[mynode] (v22) at (.3,-2)  {};

      \node[mynode] (v211) at (-.3,-3) {};

      \node[mynode] (v31) at (1.5,-2)  {};

      \node[mynode, fill=red!20] (evil) at (1,-3)  {};

      \draw (v)--(v1)--(v11)--(v111);
      \draw (v11)--(v112);
      \draw (v1)--(v12);
      \draw (v)--(v2)--(v21)--(v211);
      \draw (v2)--(v22);
      \draw (v)--(v3)--(v31);
      \draw[blue, ultra thick] (v11)--(v12);
      \draw[red, ultra thick] (v22)--(evil)--(v31);
    \end{tikzpicture}
    \caption{There are two ways~\(B_{3}(v)\) can fail to be a tree
      if~\(\mathbox[green]{B_{2}(v)}\) (shown in green) is a tree.
      Either two vertices in~\(\mathbox[blue]{D_{2}(v)}\)
      are connected via an edge (shown in blue)
      or two vertices from~\(\mathbox[blue]{D_{2}(v)}\)
      have an edge each (shown in red) to a vertex not in~\(S_{2}(v)\).}
    \label{fig:notree}
  \end{figure}
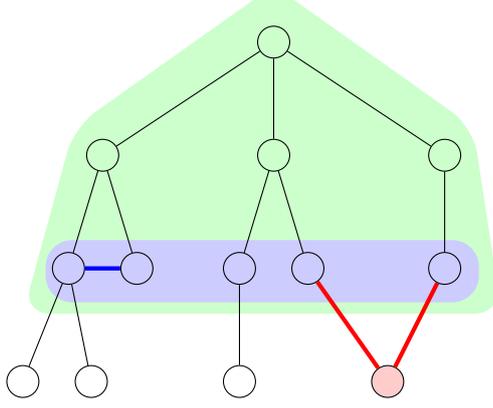
  Thus
  \begin{align*}
  \indfunc_{A_{\ell-1}} \prob_{n}(\setcomplement{A_{\ell}} \given B_{\ell-1}(v))
  &\leq \indfunc_{A_{\ell-1}}
  \sum_{u,u' \in D_{\ell-1}(v)} \Bigl(\prob_{n}(\conn{u}{u'}=1) +
  \sum_{x \notin S_{\ell-1}(v)} \prob_{n}(\conn{u}{x}=1) \prob_{n}(\conn{x}{u'}=1)
  \Bigr)
  \\
  &\leq \indfunc_{A_{\ell-1}}
  \sum_{u,u' \in D_{\ell-1}(v)} \Bigl( \frac{W_{u}W_{u'}}{n\vartheta} +
  \sum_{x \notin S_{\ell-1}(v)} \frac{W_{u}W_{x}}{n\vartheta}
  \frac{W_{x}W_{u'}}{n\vartheta}
  \Bigr)
  \\
  &\leq \indfunc_{A_{\ell-1}} (1+\Gamma_{2,n})
  \frac{\setweight{D_{\ell-1}(v)}^{2}}{n\vartheta}.
  \end{align*}
  Take expectations and conclude that
  \begin{equation*}
  \prob_{n}(A_{\ell-1} \setintersect \setcomplement{A_{\ell}})
  \leq (1+\Gamma_{2,n})
  \frac{\expe_{n}[\setweight{D_{\ell-1}(v)}^{2}]}{n\vartheta}.
  \end{equation*}
  Iteratively, this implies
  \[
  \prob_{n}(\setcomplement{A_{\ell}})
  \leq  \prob_{n}(\setcomplement{A_{\ell}} \setintersect A_{\ell-1})
  +\prob_{n}(\setcomplement{A_{\ell-1}})
  \leq \dots
    \leq (1+\Gamma_{2,n})
    \sum_{r=1}^{\ell-1}
    \frac{\expe[\setweight{D_{r-1}(v)}^{2}]}{n\vartheta}.
  \]
  Since the~\(D_{r-1}(v)\) are disjoint, and their union is
  contained in~\(S_{\ell-1}(v)\)
  the square of the sum of the weights of the~\(D_{r-1}(v)\)
  can be bounded by the square of the sum of weights in~\(S_{\ell-1}(v)\),
  which in turn can be bounded by \autoref{lem:bl:weight:squared:raw}.
  \[
    \prob_{n}(\setcomplement{A_{\ell}})
    \leq (1+\Gamma_{2,n})
    \frac{\expe_{n}[\setweight{S_{\ell-1}(v)}^{2}]}{n\vartheta}\\
    \leq C(1+\Gamma_{2,n})^{2\ell+1}(\Gamma_{3,n}+1)\frac{(W_{v}+1)^{2}}{n\vartheta}.
  \]
  This concludes the proof.
\end{proof}

We obtain a similar bound when we additionally condition
on information about a particular edge in the graph.
\begin{lemma}\label{lem:probnottree:eep}
  For every four distinct vertices~\(v,u,v',u' \in V_{n}\)
  and level~\(\ell \in \naturals\) we have
  \[
    \begin{split}
    \prob_{n}(\text{\(B_{\ell}(v)\) is not a tree}
    \given X_{\edge{v}{u}},X_{\edge{v'}{u'}})
  &\leq C \frac{(W_{v}+1)^{2}}{n\vartheta}
      (\Gamma_{3,n}+1)(\Gamma_{2,n}+1)^{2\ell+1}\\
    &\quad + C \frac{W_{u}W_{v}}{n\vartheta} (\Gamma_{2,n}+1)^{2\ell}
    +C\frac{W_{v}(W_{u'}+W_{v'})}{n\vartheta}
      (\Gamma_{2,n}+1)^{\ell{}}
    \end{split}
  \]
  as well as
    \[
    \prob_{n}(\text{\(B_{\ell}(v)\) is not a tree}
    \given X_{\edge{v}{u}})
  \leq C \frac{(W_{v}+1)^{2}}{n\vartheta}
  (\Gamma_{3,n}+1)(\Gamma_{2,n}+1)^{2\ell+1}
  +C \frac{W_{u}W_{v}}{n\vartheta} (\Gamma_{2,n}+1)^{2\ell}
  \]
  and
  \[
    \prob_{n}(\text{\(B_{\ell}(v)\) is not a tree}
    \given X_{\edge{v'}{u'}})
  \leq C \frac{(W_{v}+1)^{2}}{n\vartheta}
  (\Gamma_{3,n}+1)(\Gamma_{2,n}+1)^{2\ell+1}
  +C \frac{W_{v}(W_{u'}+W_{v'})}{n\vartheta} (\Gamma_{2,n}+1)^{\ell}.
  \]
\end{lemma}
\begin{proof}
  We will only prove the first claim.
  The other two claims follow by a similar, but slightly simpler
  argument.

  Set~\(e = \edge{v}{u}\) and~\(e' = \edge{v'}{u'}\).
  If~\(B_{\ell}(v) = B_{\ell}(v,\mathbf{G}_{n})\) is not a tree,
  then~\(B_{\ell}(v, \mathbf{G}_{n}-e)\) is not a tree
  or~\(B_{\ell-1}(u, \mathbf{G}_{n}-e)\) is not a tree
  or the two (trees)~\(B_{\ell}(v, \mathbf{G}_{n}-e)\)
  and~\(B_{\ell-1}(u, \mathbf{G}_{n}-e)\) intersect
  (see \autoref{fig:notrees}).
  Therefore we have
  \begin{align*}
    \prob_{n}(\text{\(B_{\ell}\) is not a tree} \given X_{e},X_{e'})
    &\leq \prob_{n}(\text{\(B_{\ell}(v,\mathbf{G}_{n}-\set{e})\) is not a tree}
    \given X_{e},X_{e'})\\
    &\quad+\prob_{n}(\text{\(B_{\ell-1}(u,\mathbf{G}_{n}-\set{e})\) is not a tree}
    \given X_{e},X_{e'})\\
    &\quad+\prob_{n}(\text{\(B_{\ell}(v,\mathbf{G}_{n}-\set{e})\)
      and~\(B_{\ell-1}(u,\mathbf{G}_{n}-\set{e})\) intersect}  \given X_{e},X_{e'}).
  \end{align*}

  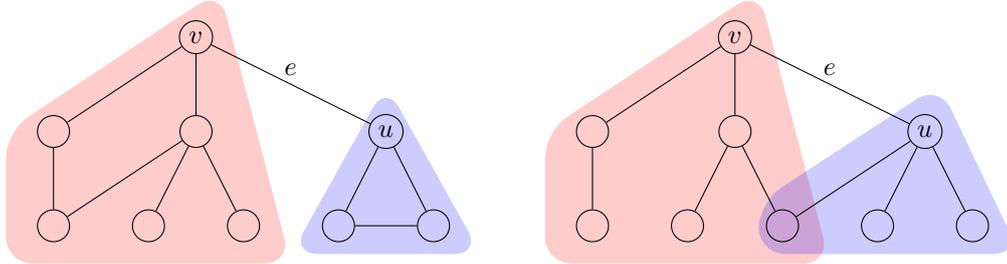
\begin{figure}[htbp]
    \centering
    \subcaptionbox{\(B_{2}(v,\mathbf{G}-e)\)
      and~\(B_{1}(u,\mathbf{G}-e)\)
      contain cycles.\label{sfig:cycle}}[0.45\linewidth]{%
      \centering
    \begin{tikzpicture}[x=1.25cm,y=1.25cm,
      mynode/.style={draw, circle, inner sep=2pt, minimum size=1.2em}]

      \path [rounded corners=10pt, fill=red!20]
      (.25,.5)--(-2,-1)--(-2,-2.4)--(1,-2.4)
      --cycle;
      \path [rounded corners=10pt, fill=blue!20]
      (2,-.5)--(1,-2.3)--(3,-2.3)
      --cycle;

      \node[mynode] (v) at (0,0)  {\(v\)};
      \node[mynode] (v1) at (-1.5,-1)    {};
      \node[mynode] (v2) at (0,-1)  {};
      \node[mynode] (u) at (2,-1)  {\(u\)};
      \node[mynode] (v11) at (-1.5,-2) {};
      \node[mynode] (v21) at (-.5,-2) {};
      \node[mynode] (v22) at ( .5,-2) {};
      \node[mynode] (u1) at (1.5,-2) {};
      \node[mynode] (u2) at (2.5,-2) {};

      \draw (v)--(v1)--(v11);
      \draw (v)--(v2)--(v21);
      \draw (v2)--(v22);
      \draw (v)--node[above]{$e$}(u)--(u1);
      \draw (u)--(u2);
      \draw (u1)--(u2);
      \draw (v11)--(v2);
    \end{tikzpicture}}
    \subcaptionbox{\(B_{2}(v,\mathbf{G}-e)\) and~\(B_{1}(u,\mathbf{G}-e)\)
      intersect.\label{sfig:inter}}[0.45\linewidth]{%
      \centering
      \begin{tikzpicture}[x=1.25cm,y=1.25cm,
        mynode/.style={draw, circle, inner sep=2pt, minimum size=1.2em}]

        \path [rounded corners=10pt, fill=red, fill opacity=0.2]
        (.25,.5)--(-2,-1)--(-2,-2.4)--(1,-2.4)
        --cycle;
        \path [rounded corners=10pt, fill=blue, fill opacity=.2]
        (2.15,-.5)--(.25,-1.75)--(.25,-2.3)--(3,-2.3)
        --cycle;

        \node[mynode] (v) at (0,0)  {\(v\)};
        \node[mynode] (v1) at (-1.5,-1)    {};
        \node[mynode] (v2) at (0,-1)  {};
        \node[mynode] (u) at (2,-1)  {\(u\)};
        \node[mynode] (v11) at (-1.5,-2) {};
        \node[mynode] (v21) at (-.5,-2) {};
        \node[mynode] (v22) at ( .5,-2) {};
        \node[mynode] (u1) at (1.5,-2) {};
        \node[mynode] (u2) at (2.5,-2) {};

        \draw (v)--(v1)--(v11);
        \draw (v)--(v2)--(v21);
        \draw (v2)--(v22);
        \draw (v)--node[above]{$e$}(u)--(u1);
        \draw (u)--(u2);
        \draw (v22)--(u);
      \end{tikzpicture}}
    \caption{Examples of the three ways~\(B_{2}(v,\mathbf{G})\)
      can fail to be a tree.
      Either there is a cycle in~\(\mathbox{B_{2}(v,\mathbf{G}-e)}\),
      there is a cycle in~\(\mathbox[blue]{B_{1}(u,\mathbf{G}-e)}\)
      \subref{sfig:cycle}
      or~\(\mathbox{B_{2}(v,\mathbf{G}-e)}\)
      and~\(\mathbox[blue]{B_{1}(u,\mathbf{G}-e)}\)
      overlap \subref{sfig:inter}.}
    \label{fig:notrees}
  \end{figure}

  The conditioning on~\(X_{e},X_{e'}\) can be removed
  since~\(B_{\ell}(v,\mathbf{G}_{n}-e)\) is isomorphic
  to~\(B_{\ell}(v,\mathbf{G}_{n}-\set{e,e'})\)
  and~\(B_{\ell-1}(u,\mathbf{G}_{n}-e)\)
  is isomorphic to~\(B_{\ell-1}(u,\mathbf{G}_{n}-\set{e,e'})\)
  with high probability
  even conditionally on~\(X_{e},X'_{e}\).
  The events involving only~\(\mathbf{G}_{n}-\set{e,e'}\)
  are then independent of~\(X_{e},X_{e'}\)
  so that
  \begin{equation*}
    \begin{split}
      \prob_{n}(\text{\(B_{\ell}\) is not a tree}\given X_{e},X_{e'})
      &\leq \prob_{n}(\text{\(B_{\ell}(v,\mathbf{G}-\set{e,e'})\) is not a tree})\\
    &\quad+\prob_{n}(\text{\(B_{\ell-1}(u,\mathbf{G}-\set{e,e'})\) is not a tree})\\
    &\quad+\prob_{n}(\text{\(B_{\ell}(v,\mathbf{G}-\set{e,e'})\)
      and~\(B_{\ell-1}(u,\mathbf{G}-\set{e,e'})\) intersect})\\
    &\quad + 2\prob_{n}(B_{\ell}(v,\mathbf{G}_{n}-e) \ncong B_{\ell}(v,\mathbf{G}_{n}-\set{e,e'})
    \given X_{e}, X_{e'})\\
    &\quad + 2\prob_{n}(B_{\ell-1}(u,\mathbf{G}_{n}-e) \ncong B_{\ell-1}(u,\mathbf{G}_{n}-\set{e,e'})
                \given X_{e}, X_{e'})
    \end{split}
  \end{equation*}
  \begingroup
  \allowdisplaybreaks
  For the last two terms
  note
  that~\(B_{\ell}(v,\mathbf{G}_{n}-e) \ncong B_{\ell}(v,\mathbf{G}_{n}-\set{e,e'})\),
  implies~\(e' \in B_{\ell}(v,\mathbf{G}_{n}-e)\),
  since otherwise all paths in~\(B_{\ell}(v,\mathbf{G}_{n}-e)\)
  would avoid~\(e'\)  and would thus already be
  in~\(B_{\ell}(v,\mathbf{G}_{n}-\set{e,e'})\).
    The edge~\(e'\) can only be
      present in~\(B_{\ell}(v,\mathbf{G}_{n}-e)\)
      if at least one of its endpoints~\(v'\) or~\(u'\) can be reached
    in~\(B_{\ell}(v,\mathbf{G}_{n}-\set{e,e'})\), which is independent
    of~\(X_{e}\) and~\(X_{e'}\).
    Together with
    \autoref{lem:uinbl} we obtain
  \begin{align*}
    &\prob_{n}(B_{\ell}(v,\mathbf{G}_{n}-e) \neq B_{\ell}(v,\mathbf{G}_{n}-\set{e,e'})
                \given X_{e}, X_{e'})\\
    &\quad\leq \prob_{n}(e' \in B_{\ell}(v,\mathbf{G}_{n}-e)
      \given X_{e}, X_{e'})\\
    &\quad\leq \prob_{n}(u' \in B_{\ell}(v,\mathbf{G}_{n}-\set{e,e'}))
      + \prob_{n}(v' \in B_{\ell}(v,\mathbf{G}_{n}-\set{e,e'}))\\
    &\quad \leq \frac{W_{v}(W_{u'}+W_{v'})}{n\vartheta} (\Gamma_{2,n}+1)^{\ell}.
  \end{align*}
  The analogous result applies to~\(B_{\ell-1}(u,\mathbf{G}_{n}-e)\)
  and~\(B_{\ell-1}(u,\mathbf{G}_{n}-\set{e,e'})\).
  \endgroup

  Now apply \autoref{lem:probnottree} to bound the (unconditional) probability
  that the neighbourhoods~\(B_{\ell}(v,\mathbf{G}_{n}-\set{e,e'})\)
  and~\(B_{\ell-1}(u,\mathbf{G}_{n}-\set{e,e'})\) are not trees
  and \autoref{lem:uinbl}
  to bound the probability that they intersect.
  Then
  \begin{align*}
    \prob_{n}(\text{\(B_{\ell}\) is not a tree}\given X_{e},X_{e'})
    &\leq C \frac{(W_{v}+1)^{2}}{n\vartheta}
      (\Gamma_{3,n}+1)(\Gamma_{2,n}+1)^{2\ell+1}\notag\\
    &\quad + C \frac{W_{u}W_{v}}{n\vartheta} (\Gamma_{2,n}+1)^{2\ell}
    +C\frac{W_{v}(W_{u'}+W_{v'})}{n\vartheta}
      (\Gamma_{2,n}+1)^{\ell}.
  \end{align*}
  This shows the claim.
\end{proof}

\subsection{Correlation between neighbourhoods}
\label{sec:corr}
In this section we investigate the correlation between different
neighbourhoods in the graph~\(\mathbf{G}_{n}\) more closely.
Before we get into the formal argument, we will briefly recall
\autoref{cor:abpathupto}, which bounds the probability that there is a
path of length at most~\(\ell\)
between two disjoint sets of vertices~\(\mathcal{U}\) and~\(\mathcal{V}\)
by~\((n\vartheta)^{-1}
\setweight{\mathcal{U}}\setweight{\mathcal{V}}(1+\Gamma_{2,n})^{\ell-1}\).
This implies that the probability that
the~\(\ell\)-neighbourhoods~\(B_{\ell}(\mathcal{V}, G_{n})\)
and~\(B_{\ell}(\mathcal{U}, G_{n})\)
(here in the unweighted graph, but the argument is the same for the weighted
graph)
share a vertex is bounded by
\[
  \prob_{n}(\text{\(B_{\ell}(\mathcal{V}, G_{n})\)
    and~\(B_{\ell}(\mathcal{U}, G_{n})\) share a vertex})
  \leq  \frac{\setweight{\mathcal{U}}\setweight{\mathcal{V}}}{n\vartheta}
    (1+\Gamma_{2,n})^{2\ell-2}.
\]
Intuitively, if the neighbourhoods do not share a vertex,
their structure is determined by independent random variables,
which would mean that they are independent.
That argument is made more rigorous in the remainder of this section.
As in the previous section, the results we show here are
by no means surprising,
but results for the exact setup we needed were not readily available
in the literature.

We show that the~\(\ell\)-neighbourhoods
of two disjoint sets of root vertices
are relatively weakly correlated
by constructing slightly altered
independent versions of the~\(\ell+1\)-neighbourhoods
conditionally on the~\(\ell\)-neighbourhoods.
An iterative argument that relates the correlation
of the~\(\ell+1\)-neighbourhoods
to the correlation of the~\(\ell\)-neighbourhoods
then finishes the argument.

The discussion here extends \citeauthor{cao}'s \citep{cao} approach
for edge-weighted Erdős--Rényi graphs to inhomogeneous
random graphs with additional weights on edges and vertices.
Our construction of the altered neighbourhoods
needs to take into account both edge and vertex weights.
We will do that in two separate steps.
In a first step we ignore the vertex weights at level~\(\ell+1\)
(formally, we do this by applying a function~\(\tau_{\ell+1}\)
that removes these weights).

Given the~\(\ell\)-neighbourhoods, the~\(\ell+1\)-neighbourhoods
without vertex weights at level~\(\ell+1\)
can be constructed by adding edges emanating from level~\(\ell\) vertices.
If the edges that are added to the neighbourhoods
are distinct,
the added randomness is independent.
If edges have to be used for neighbourhoods
emanating from both root vertex sets,
they can be
replaced by independent copies
for one of the two sets
to make the added randomness
independent.
Provided that not too many edges have to be rerandomised
in this way,
the resulting objects are close enough to the
original~\(\ell+1\)-neighbourhoods.

\begin{lemma}\label{lem:covcouplmorecomplextilde}
  Fix~\(m, m' \in \naturals\)
  and~\(m+m'\) distinct vertices~\(v_{1},\dots,v_{m}\)
  and~\(v'_{1},\dots,v'_{m'}\).
  Let~\(E_{n} = \sset{(i,j)}{1 \leq i < j \leq n}\).
  Let~\(F \subseteq V_{n} \setunion E_{n} \)
  and~\(F' \subseteq V_{n} \setunion E_{n} \).
  For~\(r \in \naturals\) let
  \begin{align*}
  \mathbf{B}_{r}
  &= (
  B_{r}(v_{1},\mathbf{G}_{n}),
  B_{r}(v_{1},\mathbf{G}_{n}^{F}),
  \dots,
  B_{r}(v_{m},\mathbf{G}_{n}),
  B_{r}(v_{m},\mathbf{G}_{n}^{F})
  ), \\
  \mathbf{B}'_{r}
  &= (
  B_{r}(v'_{1},\mathbf{G}_{n}),
  B_{r}(v'_{1},\mathbf{G}_{n}^{F'}),
  \dots,
  B_{r}(v'_{m'},\mathbf{G}_{n}),
  B_{r}(v'_{m'},\mathbf{G}_{n}^{F'})
  ).
  \end{align*}
  Let~\(\mathbf{S}_{r}\) be the set of vertices
  in~\(\mathbf{B}_{r}\)
  and similarly let~\(\mathbf{S}'_{r}\) be
  the set of vertices in~\(\mathbf{B}'_{r}\).
  Then~\(\mathbf{D}_{r} = \mathbf{S}_{r} \setminus
  \mathbf{S}_{r-1}\)
  and~\(\mathbf{D}'_{r} = \mathbf{S}'_{r} \setminus
  \mathbf{S}'_{r-1}\)
  are the level~\(r\)-vertices of~\(\mathbf{B}_{r}\)
  and~\(\mathbf{B}'_{r}\), respectively.

  Let~\(\mathbf{I}_{r}\) be the event
  that the~\(\mathbf{S}_{r}\) and~\(\mathbf{S}'_{r}\)
  do not intersect.
  Let~\(\tau_{r}\) be the function that takes~\(m+m'\) weighted
  rooted graphs as input and removes the weight of the
  vertices at level~\(r\).

  Fix any level~\(\ell \in \naturals\),
  then
  there is a coupling of~\(\tau_{\ell+1}(\mathbf{B}_{\ell+1})\)
  to~\(\mathbf{\tilde{B}}_{\ell+1}\)
  and of~\(\tau_{\ell+1}(\mathbf{B}_{\ell+1}')\)
  to~\(\mathbf{\tilde{B}}'_{\ell+1}\)
  such that~\(\mathbf{\tilde{B}}_{\ell+1}\)
  and~\(\mathbf{\tilde{B}}'_{\ell+1}\)
  are conditionally independent on~\(\mathbf{I}_{\ell}\)
  given~\(\mathbf{B}_{\ell},\mathbf{B}_{\ell}'\).
  Furthermore,
  on~\(\mathbf{I}_{\ell}\)
  the law of~\(\mathbf{\tilde{B}}_{\ell+1}\)
  given~\(\mathbf{B}_{\ell},\mathbf{B}'_{\ell}\)
  is equal to the law of~\(\tau_{\ell+1}(\mathbf{B}_{\ell+1})\)
  given~\(\mathbf{B}_{\ell}\)
  and the law of~\(\mathbf{\tilde{B}}_{\ell+1}'\),
  given~\(\mathbf{B}_{\ell},\mathbf{B}'_{\ell}\)
  is equal to the law of~\(\tau_{\ell+1}(\mathbf{B}'_{\ell+1})\)
  given~\(\mathbf{B}'_{\ell}\).

  In formulas, for all bounded measurable functions~\(g\) and~\(g'\)
  we have (almost surely)
  \begin{gather*}
  \indfunc_{\mathbf{I}_{\ell}}
  \cov_{n}(g(\mathbf{\tilde{B}}_{\ell+1}),g'(\mathbf{\tilde{B}}'_{\ell+1})
  \given \mathbf{B}_{\ell}, \mathbf{B}'_{\ell}) = 0,\\
  \indfunc_{\mathbf{I}_{\ell}}
  \expe_{n}[g(\mathbf{\tilde{B}}_{\ell+1})
  \given \mathbf{B}_{\ell}, \mathbf{B}'_{\ell}]
  = \indfunc_{\mathbf{I}_{\ell}}
  \expe_{n}[g(\tau_{\ell+1}(\mathbf{B}_{\ell+1})) \given \mathbf{B}_{\ell}],\\
  \indfunc_{\mathbf{I}_{\ell}}
  \expe_{n}[g'(\mathbf{\tilde{B}}'_{\ell+1})
  \given \mathbf{B}_{\ell}, \mathbf{B}'_{\ell}]
  = \indfunc_{\mathbf{I}_{\ell}}
  \expe_{n}[g'(\tau_{\ell+1}(\mathbf{B}'_{\ell+1})) \given \mathbf{B}'_{\ell}].
  \end{gather*}

  Moreover,
  \begin{align}
  \indfunc_{\mathbf{I}_{\ell}}
  \prob_{n}(\mathbf{\tilde{B}}_{\ell+1} \neq \tau_{\ell+1}(\mathbf{B}_{\ell+1})
  \given \mathbf{B}_{\ell}, \mathbf{B}'_{\ell})
  &\leq \indfunc_{\mathbf{I}_{\ell}} C\frac{\setweight{\mathbf{S}_{\ell}}
    \setweight{\mathbf{S}'_{\ell}}}{n\vartheta},
  \label{eq:bvteqbv}\\
  \indfunc_{\mathbf{I}_{\ell}}
  \prob_{n}(\mathbf{\tilde{B}}'_{\ell+1} \neq \tau_{\ell+1}(\mathbf{B}'_{\ell+1})
  \given \mathbf{B}_{\ell}, \mathbf{B}'_{\ell})
  &\leq \indfunc_{\mathbf{I}_{\ell}} C\frac{\setweight{\mathbf{S}_{\ell}}
    \setweight{\mathbf{S}'_{\ell}}}{n\vartheta}.
  \label{eq:buteqbu}
  \end{align}
\end{lemma}
\begin{proof}
  For~\(\mathbf{T}_{1}, \mathbf{T}_{2} \subseteq V_{n}\) denote with
  \[
  X(\mathbf{T}_{1},\mathbf{T}_{2})
  = \sset{(e,X_{e},X'_{e},w_{e},w'_{e})}
  {e=\edge{v}{u}, v \in \mathbf{T}_{1}, u \in \mathbf{T}_{2}}
  \]
  the set of edges connecting~\(\mathbf{T}_{1}\) and~\(\mathbf{T}_{2}\)
  and all the information potentially associated to those edges
  in~\(\mathbf{G}_{n}\), \(\mathbf{G}_{n}^{F}\) and~\(\mathbf{G}_{n}^{F'}\).

  An edge that is new in~\(\mathbf{B}_{\ell+1}\),
  i.e.~an edge that is present in~\(\mathbf{B}_{\ell+1}\),
  but not in~\(\mathbf{B}_{\ell}\),
  must connect a vertex in~\(\mathbf{D}_{\ell}\)
  to a vertex in~\(V_{n} \setminus \mathbf{S}_{\ell-1}\).
  Hence,
  there exists a function~\(\Psi\)
  depending on~\(F\)
  such that
  \begin{equation*}
  \tau_{\ell+1}(\mathbf{B}_{\ell+1})
  = \Psi(\mathbf{B}_{\ell},
  X(\mathbf{D}_{\ell},V_{n} \setminus \mathbf{S}_{\ell-1})).
  \end{equation*}
  Intuitively, this function just identifies the edges that are part
  of~\(\tau_{\ell+1}(\mathbf{B}_{\ell+1})\), but not of~\(\mathbf{B}_{\ell}\)
  and adds them to~\(\mathbf{B}_{\ell}\).
  Note that for an edge~\(e\) between~\(\mathbf{D}_{\ell}\)
  and~\(V_{n} \setminus \mathbf{S}_{\ell-1}\)
  with~\(X_{e}=X'_{e}=0\) the values of~\(w_{e}\) and~\(w'_{e}\)
  do not matter for~\(\Psi\),
  because such an edge cannot be part of~\(\mathbf{B}_{\ell+1}\).
  In other words, the values~\(w_{e}\) and~\(w'_{e}\) influence~\(\Psi\)
  only if~\(X_{e}+X'_{e} \geq 1\).
  We call such an edge with~\(X_{e}+X'_{e} \geq 1\) \emph{relevant}
  for~\(\Psi\).
  There is a similar function~\(\Psi'\) depending on~\(F'\)
  such that
  \begin{equation*}
    \tau_{\ell+1}(\mathbf{B}'_{\ell+1})
    = \Psi'(\mathbf{B}'_{\ell},
    X(\mathbf{D}'_{\ell},V_{n} \setminus \mathbf{S}'_{\ell-1})).
  \end{equation*}

  Define
  \begin{equation*}
    X_{1} = X(\mathbf{D}_{\ell},
              V_{n} \setminus (\mathbf{S}_{\ell-1} \setunion \mathbf{S}'_{\ell})
            ),\quad
    X_{2} = X(\mathbf{D}'_{\ell},
              V_{n} \setminus (\mathbf{S}'_{\ell-1} \setunion \mathbf{S}_{\ell})
    )
    \quad\text{and}\quad
    X_{3} = X(\mathbf{D}_{\ell},\mathbf{D}'_{\ell}).
  \end{equation*}
  Then~\(X_{1}\) contains the information on all potential new edges
  of~\(\mathbf{B}_{\ell+1}\) that do not have an endpoint
  in~\(\mathbf{B}'_{\ell}\).
  Analogously,~\(X_{2}\) contains the information on all potential new edges
  of~\(\mathbf{B}'_{\ell+1}\) that do not have an endpoint
  in~\(\mathbf{B}_{\ell}\).
  Finally,~\(X_{3}\) contains the information on all potential new
  edges of~\(\mathbf{B}_{\ell+1}\) and~\(\mathbf{B}'_{\ell+1}\)
  that connect~\(\mathbf{B}_{\ell}\) and~\(\mathbf{B}'_{\ell}\).
  By construction
  \begin{equation*}
    X_{1} \setunion X_{3}
    = X(\mathbf{D}_{\ell},
      V_{n} \setminus (\mathbf{S}_{\ell-1} \setunion \mathbf{S}'_{\ell-1}))
      \quad\text{and}\quad
      X_{2} \setunion X_{3}
    = X(\mathbf{D}'_{\ell},
      V_{n} \setminus (\mathbf{S}'_{\ell-1} \setunion \mathbf{S}_{\ell-1})).
  \end{equation*}
  On the event~\(\mathbf{I}_{\ell}\)
  the set of edges~\(X(\mathbf{D}_{\ell},V_{n} \setminus
  \mathbf{S}_{\ell-1})\)
  coincides with~\(X_{1} \setunion X_{3}\),
  since there can be no edges between~\(\mathbf{D}_{\ell}\)
  and~\(\mathbf{S}'_{\ell-1}\), because
  that would imply
  that~\(\mathbf{D}_{\ell} \subseteq \mathbf{S}_{\ell}\)
  and~\(\mathbf{S}'_{\ell}\)
  have a nonempty intersection.
  An analogous consideration holds for~\(X_{2} \setunion X_{3}\).
  Hence, on~\(\mathbf{I}_{\ell}\) we have
  \[
  \tau_{\ell+1}(\mathbf{B}_{\ell+1}) = \Psi(\mathbf{B}_{\ell}, X_{1} \setunion
  X_{3})
  \quad\text{and}\quad
  \tau_{\ell+1}(\mathbf{B}'_{\ell+1}) = \Psi'(\mathbf{B}'_{\ell}, X_{2}
  \setunion X_{3}).
  \]

  On~\(\mathbf{I}_{\ell}\)
  the edge collections~\(X_{1}\) and~\(X_{3}\) are conditionally
  independent given~\(\mathbf{B}_{\ell},\mathbf{B}'_{\ell}\).
  Let~\(\widetilde{X}_{1}\) be an independent copy
  of~\(X(\mathbf{D}_{\ell}, \mathbf{S}'_{\ell-1})\)
  and~\(\widetilde{X}_{2}\) an independent copy
  of~\(X(\mathbf{D}'_{\ell}, \mathbf{S}_{\ell-1})\).
  Finally, let~\(\widetilde{X}_{3}\) be an independent copy
  of~\(X_{3}\).
  Then define
  \[
  \mathbf{\tilde{B}}_{\ell+1}
  = \Psi(\mathbf{B}_{\ell}, X_{1} \setunion \widetilde{X}_{1} \setunion
  X_{3})
  \quad\text{and}\quad
  \mathbf{\tilde{B}}'_{\ell+1}
  = \Psi'(\mathbf{B}'_{\ell}, X_{2} \setunion \widetilde{X}_{2}
  \setunion \widetilde{X}_{3}).
  \]
  The thus constructed~\(\mathbf{\tilde{B}}_{\ell+1}\)
  and~\(\mathbf{\tilde{B}}'_{\ell+1}\) are conditionally independent
  on~\(\mathbf{I}_{\ell}\) given~\(\mathbf{B}_{\ell}, \mathbf{B}'_{\ell}\).
  Furthermore, on~\(\mathbf{I}_{\ell}\)
  and conditionally on~\(\mathbf{B}_{\ell},\mathbf{B}'_{\ell}\)
  the law of~\(X_{1} \setunion \widetilde{X}_{1} \setunion X_{3}\)
  is equal to the law of~\(X(\mathbf{D}_{\ell},V_{n} \setminus
  \mathbf{S}_{\ell-1})\)
  given only~\(\mathbf{B}_{\ell}\),
  because~\(\widetilde{X}_{1}\) provides the
  \enquote{missing source of randomness}
  for~\(X_{1} \setunion X_{3}\) when conditioning on
  both~\(\mathbf{B}_{\ell}\) and~\(\mathbf{B}'_{\ell}\),
  where the edges between~\(\mathbf{D}_{\ell}\) and~\(\mathbf{S}'_{\ell-1}\)
  are fixed,
  compared to conditioning~\(\mathbf{B}_{\ell+1}\) only
  on~\(\mathbf{B}_{\ell}\),
  where these edges are random.
  Then on~\(\mathbf{I}_{\ell}\) the law of~\(\mathbf{\tilde{B}}_{\ell+1}\)
  conditionally on~\(\mathbf{B}_{\ell},\mathbf{B}'_{\ell}\)
  is the law of~\(\tau_{\ell+1}(\mathbf{B}_{\ell+1})\)
  given~\(\mathbf{B}_{\ell}\).
  An analogous result holds for~\(\mathbf{\tilde{B}}'_{\ell+1}\).

  It remains to verify that~\(\mathbf{\tilde{B}}_{\ell+1}\)
  differs from~\(\tau_{\ell+1}(\mathbf{B}_{\ell+1})\)
  with small probability on~\(\mathbf{I}_{\ell}\)
  conditionally on~\(\mathbf{B}_{\ell}\)
  and~\(\mathbf{B}'_{\ell}\).
  Write
  \[
  \widetilde{X}_{1}
  = \sset{(e,\widetilde{X}_{e},\widetilde{X}'_{e},
    \widetilde{w}_{e},\widetilde{w}'_{e})}{e = \edge{u}{v},
    v \in \mathbf{D}_{\ell},
    u \in \mathbf{S}'_{\ell-1}}.
  \]
  By construction~\(\mathbf{\tilde{B}}_{\ell+1}\)
  and~\(\tau_{\ell+1}(\mathbf{B}_{\ell+1})\) differ only if
  there is an edge~\(e\) in~\(\widetilde{X}_{1}\)
  that is relevant for~\(\Psi\),
  which can only be the case if~\(\widetilde{X}_{e}=1\)
  or~\(\widetilde{X}'_{e}=1\).
  That is the same as saying that there is a path of length~\(1\)
  between the (fixed) sets of vertices~\(\mathbf{D}_{\ell}\)
  and~\(\mathbf{S}'_{\ell}\) in a graph~\(\mathbf{\widetilde{G}}_{n}\)
  or~\(\mathbf{\widetilde{G}}'_{n}\),
  which are based on~\(\widetilde{X}_{e}\) and~\(\widetilde{X}'_{e}\),
  respectively.
  Hence, by \autoref{lem:abpath}
  \begin{align}
    \indfunc_{\mathbf{I}_{\ell}}
    \prob_{n}(\mathbf{\tilde{B}}_{\ell+1} \neq \tau_{\ell+1}(\mathbf{B}_{\ell+1})
      \given \mathbf{B}_{\ell}, \mathbf{B}'_{\ell})
    &\leq \indfunc_{\mathbf{I}_{\ell}}
      (\prob_{n}(\mathbf{D}_{\ell} \pathbetw^{\mathbf{\widetilde{G}}_{n}}_{1} \mathbf{S}'_{\ell})
      +\prob_{n}(\mathbf{D}_{\ell} \pathbetw^{\mathbf{\widetilde{G}}'_{n}}_{1} \mathbf{S}'_{\ell}))\notag\\
    &\leq \indfunc_{\mathbf{I}_{\ell}}2
    \frac{\setweight{\mathbf{D}_{\ell}}\setweight{\mathbf{S}'_{\ell-1}}}{n
      \vartheta}\notag\\
    &\leq \indfunc_{\mathbf{I}_{\ell}}C
    \frac{\setweight{\mathbf{S}_{\ell}}\setweight{\mathbf{S}'_{\ell}}}{n
      \vartheta}\label{eq:proof:btlbtaul:equal}
    .
  \end{align}

  Similarly we want to show that~\(\mathbf{\tilde{B}}'_{\ell+1}\)
  differs from~\(\tau_{\ell+1}(\mathbf{B}'_{\ell+1})\)
  with small probability on~\(\mathbf{I}_{\ell}\)
  conditionally on~\(\mathbf{B}_{\ell}\)
  and~\(\mathbf{B}'_{\ell}\).
  Write
  \begin{align*}
  \widetilde{X}_{2}
  &= \sset{(e,\widetilde{X}_{e},\widetilde{X}'_{e},
    \widetilde{w}_{e},\widetilde{w}'_{e})}{e = \edge{u}{v},
    v \in \mathbf{D}'_{\ell},
    u \in \mathbf{S}_{\ell-1}}
  \shortintertext{and}
  \widetilde{X}_{3}
  &= \sset{(e,\widetilde{X}_{e},\widetilde{X}'_{e},
    \widetilde{w}_{e},\widetilde{w}'_{e})}{e = \edge{u}{v},
    v \in \mathbf{D}_{\ell},
    u \in \mathbf{D}'_{\ell}}.
  \end{align*}
  By construction~\(\mathbf{\tilde{B}}'_{\ell+1}\)
  and~\(\tau_{\ell+1}(\mathbf{B}'_{\ell+1})\) differ
  only if
  \begin{itemize}
    \item there is an edge~\(e\) in~\(\widetilde{X}_{2}\)
    that is relevant for~\(\Psi'\) or
    \item there is an edge~\(e\) in~\(X_{3}\) that differs in a relevant
      way between~\(X_{3}\) and~\(\widetilde{X}_{3}\),

      where an edge~\(e\) \emph{differs in a relevant way
      between~\(X_{3}\) and~\(\widetilde{X}_{3}\)}
      if~\((X_{e},X'_{e},w_{e},w'_{e})\)
      differs from
      \((\widetilde{X}_{e},\widetilde{X}'_{e},\widetilde{w}_{e},\widetilde{w}'_{e})\),
      unless all of~\(X_{e}\), \(X'_{e}\), \(\widetilde{X}_{e}\)
      and~\(\widetilde{X}'_{e}\)
      are equal to zero,
      because that would mean that the edge is not relevant for~\(\Psi'\).
      In particular an edge can only differ in a relevant way if at least
      one
      of~\(X_{e}\), \(X'_{e}\), \(\widetilde{X}_{e}\)
      or~\(\widetilde{X}'_{e}\) is equal to one.
  \end{itemize}

  Hence, as in~\eqref{eq:proof:btlbtaul:equal} \autoref{lem:abpath} shows
  \begin{equation*}
  \indfunc_{\mathbf{I}_{\ell}}
  \prob_{n}(\mathbf{\tilde{B}}_{\ell+1} \neq \tau_{\ell+1}(\mathbf{B}_{\ell+1})
  \given \mathbf{B}_{\ell}, \mathbf{B}'_{\ell})
  \leq \indfunc_{\mathbf{I}_{\ell}}C
  \frac{\setweight{\mathbf{D}_{\ell}}\setweight{\mathbf{S}'_{\ell-1}}}{n
  \vartheta}
  +\indfunc_{\mathbf{I}_{\ell}}C
  \frac{\setweight{\mathbf{D}_{\ell}}\setweight{\mathbf{D}'_{\ell}}}{n
  \vartheta}
  \leq \indfunc_{\mathbf{I}_{\ell}}C
  \frac{\setweight{\mathbf{S}_{\ell}}\setweight{\mathbf{S}'_{\ell}}}{n
  \vartheta}
  .
  \end{equation*}
  This finishes the proof.
\end{proof}

We now add the missing weights to the vertices at level~\(\ell+1\).
Again the new weights are independent if no vertex appears for both
sets of root vertices.
If a vertex is needed for both root vertex sets,
its weight can be rerandomised to still obtain
independent random variables for both sets.
As long as the number of rerandomised vertex weights is not too large,
the independent versions of the neighbourhoods differ from
the original~\(\ell+1\)-neighbourhoods with small enough
probability.
\begin{lemma}\label{lem:covcouplmorecomplex}
  Let~\(\mathbf{B}_{r}\),~\(\mathbf{B}'_{r}\),
  \(\mathbf{S}_{r}\), \(\mathbf{S}'_{r}\),
  \(\mathbf{D}_{r}\), \(\mathbf{D}'_{r}\),
  \(\mathbf{I}_{r}\) and~\(\tau_{r}\) be
  as in \autoref{lem:covcouplmorecomplextilde}.

  Fix any level~\(\ell \in \naturals\), then
  there is a coupling of~\(\mathbf{\bar{B}}_{\ell+1}\)
  with~\(\mathbf{B}_{\ell+1}\)
  and of~\(\mathbf{B}_{\ell+1}'\)
  with~\(\mathbf{\bar{B}}'_{\ell+1}\)
  such that~\(\mathbf{\bar{B}}_{\ell+1}\) and~\(\mathbf{\bar{B}}'_{\ell+1}\)
  are conditionally independent
  given~\(\mathbf{B}_{\ell},\mathbf{B}_{\ell}'\)
  on~\(\mathbf{I}_{\ell}\).
  Furthermore,
  the law of~\(\mathbf{\bar{B}}_{\ell+1}\)
  given~\(\mathbf{B}_{\ell},\mathbf{B}'_{\ell}\)
  is equal to the law of~\(\mathbf{B}_{\ell+1}\) given~\(\mathbf{B}_{\ell}\)
  and analogously for~\(\mathbf{\bar{B}}_{\ell+1}'\).

  In formulas, for all functions~\(g\) and~\(g'\)
  we have almost surely
  \begin{gather*}
  \indfunc_{\mathbf{I}_{\ell}}
  \cov_{n}(g(\mathbf{\bar{B}}_{\ell+1}),g'(\mathbf{\bar{B}}'_{\ell+1})
  \given \mathbf{B}_{\ell}, \mathbf{B}'_{\ell}) = 0,\\
  \indfunc_{\mathbf{I}_{\ell}}
  \expe_{n}[g(\mathbf{\bar{B}}_{\ell+1}) \given \mathbf{B}_{\ell},
  \mathbf{B}'_{\ell}]
  = \indfunc_{\mathbf{I}_{\ell}}
  \expe_{n}[g(\mathbf{B}_{\ell+1}) \given \mathbf{B}_{\ell}],\\
  \indfunc_{\mathbf{I}_{\ell}}
  \expe_{n}[g'(\mathbf{\bar{B}}_{\ell+1}) \given \mathbf{B}_{\ell},
  \mathbf{B}'_{\ell}]
  = \indfunc_{\mathbf{I}_{\ell}}
  \expe_{n}[g'(\mathbf{B}'_{\ell+1}) \given \mathbf{B}'_{\ell}].
  \end{gather*}

  Moreover,
  \begin{equation*}
  \begin{split}
  \indfunc_{\mathbf{I}_{\ell}}
  \prob_{n}(\mathbf{\bar{B}}_{\ell+1} \neq \mathbf{B}_{\ell+1}
  \given \mathbf{B}_{\ell}, \mathbf{B}'_{\ell})
  &\leq \indfunc_{\mathbf{I}_{\ell}} C (1+\Gamma_{2,n})
  \frac{\setweight{\mathbf{S}_{\ell}}\setweight{\mathbf{S}'_{\ell}}}
  {n\vartheta},\\
  \indfunc_{\mathbf{I}_{\ell}}
  \prob_{n}(\mathbf{\bar{B}}'_{\ell+1} \neq \mathbf{B}'_{\ell+1}
  \given \mathbf{B}_{\ell}, \mathbf{B}'_{\ell})
  &\leq \indfunc_{\mathbf{I}_{\ell}} C (1+\Gamma_{2,n})
  \frac{\setweight{\mathbf{S}_{\ell}}\setweight{\mathbf{S}'_{\ell}}}
  {n\vartheta}.
  \end{split}
  \end{equation*}
\end{lemma}

\begin{proof}
  Let~\(\mathbf{B}_{\ell+1}\), \(\mathbf{B}_{\ell+1}'\),
  \(\mathbf{\tilde{B}}_{\ell+1}\) and~\(\mathbf{\tilde{B}}'_{\ell+1})\)
  be as in \autoref{lem:covcouplmorecomplextilde}.
  Let~\(\mathbf{\tilde{S}}_{\ell}\) be the union of the vertex sets of the
  constituent graphs of~\(\mathbf{\tilde{B}}_{\ell}\)
  and similarly~\(\mathbf{\tilde{S}}'_{\ell}\)
  be the corresponding vertex set of~\(\mathbf{\tilde{B}}'_{\ell}\).
  Define~\(\mathbf{\tilde{D}}_{\ell}
  = \mathbf{\tilde{S}}_{\ell} \setminus \mathbf{\tilde{S}}_{\ell-1}\)
  and~\(\mathbf{\tilde{D}}'_{\ell}
  = \mathbf{\tilde{S}}'_{\ell} \setminus \mathbf{\tilde{S}}'_{\ell-1}\).

  Construct~\(\mathbf{\bar{B}}_{\ell+1}\) from~\(\mathbf{\tilde{B}}_{\ell+1}\)
  by adding the remaining vertex weights
  at level~\(\ell+1\) as follows
  \begin{align*}
  (\bar{w}_{u},\bar{w}'_{u})
  &=
  \begin{cases}
  (\tilde{w}_{u},\tilde{w}'_{u}) & u \in \mathbf{\tilde{S}}'_{\ell}\\
  (w_{u},w'_{u}) &  u \notin \mathbf{\tilde{S}}'_{\ell}\\
  \end{cases}
  \quad
  u \in \mathbf{\tilde{D}}_{\ell+1}
  \shortintertext{and}
  (\bar{w}_{u},\bar{w}'_{u})
  &=
  \begin{cases}
  (\tilde{w}_{u},\tilde{w}'_{u})
  & u \in \mathbf{\tilde{S}}_{\ell} \setunion \mathbf{\tilde{D}}_{\ell+1}\\
  (w_{u},w'_{u})
  &  u \notin \mathbf{\tilde{S}}_{\ell}  \setunion
  \mathbf{\tilde{D}}_{\ell+1}\\
  \end{cases}
  \quad
  u \in \mathbf{\tilde{D}}'_{\ell+1},
  \end{align*}
  where~\((\tilde{W},\tilde{W}')\) is an i.i.d.\ copy of~\((W,W')\).

  Let~\(\widebar{\Psi}\) be the function such that
  \begin{equation}\label{eq:bblpsi}
  \mathbf{\bar{B}}_{\ell+1}
  = \widebar{\Psi}(\mathbf{\tilde{B}}_{\ell+1},\mathbf{\tilde{B}}'_{\ell+1},
  \mathbf{\tilde{W}}),
  \end{equation}
  where~\(\mathbf{\tilde{W}}\) is the collection of random
  variables~\((w_{u},w'_{u},\tilde{w}_{u},\tilde{w}'_{u})_{u \in V_{n}
    \setminus (\mathbf{\tilde{S}}_{\ell} \setunion
    \mathbf{\tilde{S}}'_{\ell})}\).
  The function~\(\widebar{\Psi}\) endows~\(\mathbf{\tilde{B}}_{\ell+1}\)
  with weights on the~\(\ell+1\)-level vertices from~\(\mathbf{\tilde{W}}\)
  and chooses~\((w_{u},w'_{u})\) or~\((\tilde{w}_{u},\tilde{w}'_{u})\)
  according to the status of~\(u\) in~\(\mathbf{\tilde{B}}'_{\ell+1}\).
  Because the alternatives~\((w_{u},w'_{u})\)
  and~\((\tilde{w}_{u},\tilde{w}'_{u})\)
  have the same distribution and are both independent
  of~\(\mathbf{\tilde{B}}_{\ell+1}\),
  the realisation of~\(\mathbf{\tilde{B}}'_{\ell+1}\) does
  not matter for the \emph{distribution} of the resulting object.
  This implies that for all realisations~\(\mathbf{b}'\)
  of~\(\mathbf{\tilde{B}}'_{\ell}\)
  we have
  \begin{equation}\label{eq:bblwobbpell}
  \mathbf{\bar{B}}_{\ell+1}
  \eqdist
  \widebar{\Psi}(\mathbf{\tilde{B}}_{\ell+1},\mathbf{b}',
  \mathbf{W}),
  \end{equation}
  where~\(\mathbf{W}\) is a collection of independent random variables
  with the same distribution
  as~\((w_{u},w'_{u},\tilde{w}_{u},\tilde{w}'_{u})\).
  In particular this also holds if~\(\mathbf{b}'\) is empty.
  This equality in distribution also holds conditional
  on~\(\mathbf{\tilde{B}}_{\ell+1}\) and~\(\mathbf{\tilde{B}}'_{\ell+1}\).
  For the same reasons, the function also satisfies
  the following distributional equality
  \begin{equation}\label{eq:blpsi}
  \mathbf{B}_{\ell+1}
  \eqdist
  \widebar{\Psi}(\tau_{\ell+1}(\mathbf{B}_{\ell+1}),\emptyset,
  \mathbf{W}).
  \end{equation}

  The construction of~\(\mathbf{\bar{B}}_{\ell+1}\)
  and~\(\mathbf{\bar{B}}_{\ell+1}\) ensures that on~\(\mathbf{I}_{\ell}\)
  each vertex weight occurs only in
  one of~\(\mathbf{\bar{B}}_{\ell+1}\) or~\(\mathbf{\bar{B}}'_{\ell+1}\)
  and the decision where it occurs
  is deterministic given~\(\mathbf{\tilde{B}}_{\ell+1}\)
  and~\(\mathbf{\tilde{B}}'_{\ell+1}\).
  Thus~\(\mathbf{\bar{B}}_{\ell+1}\) and~\(\mathbf{\bar{B}}'_{\ell+1}\)
  are conditionally independent on~\(\mathbf{I}_{\ell}\)
  given~\(\mathbf{\tilde{B}}_{\ell+1}\)
  and~\(\mathbf{\tilde{B}}'_{\ell+1}\).
  In particular
  \begin{equation}
  \indfunc_{\mathbf{I}_{\ell}}
  \cov_{n}(g(\mathbf{\bar{B}}_{\ell+1}), g'(\mathbf{\bar{B}}'_{\ell+1})
  \given \mathbf{\tilde{B}}_{\ell+1}, \mathbf{\tilde{B}}'_{\ell+1})
  = 0.
  \label{eq:covgbgub0}
  \end{equation}

  Furthermore, the observations~\eqref{eq:bblpsi},
  \eqref{eq:bblwobbpell} and~\eqref{eq:blpsi}
  about~\(\widebar{\Psi}\) show that
  we have for all functions~\(g\)
  almost surely that
  \begin{align}
  \indfunc_{\mathbf{I}_{\ell}}
  \expe_{n}[g(\mathbf{\bar{B}}_{\ell+1})
  \given \mathbf{B}_{\ell},\mathbf{B}'_{\ell}]
  &= \indfunc_{\mathbf{I}_{\ell}}
  \expe_{n}[
  \expe_{n}[g(\widebar{\Psi}(\mathbf{\tilde{B}}_{\ell+1},\mathbf{\tilde{B}}'_{\ell+1},
  \mathbf{\tilde{W}}))
  \given \mathbf{\tilde{B}}_{\ell+1},\mathbf{\tilde{B}}'_{\ell+1}]
  \given \mathbf{B}_{\ell},\mathbf{B}'_{\ell}]\notag\\
  &= \indfunc_{\mathbf{I}_{\ell}}
  \expe_{n}[
  g(\widebar{\Psi}(\tau_{\ell+1}(\mathbf{B}_{\ell+1}),\emptyset,
  \mathbf{W}))
  \given \mathbf{B}_{\ell}]\notag\\
  &= \indfunc_{\mathbf{I}_{\ell}}
  \expe_{n}[
  g(\mathbf{B}_{\ell+1})
  \given \mathbf{B}_{\ell}]\label{eq:expgbvbbvt}
  .
  \end{align}
  Analogous we also have
  \begin{equation}
  \indfunc_{\mathbf{I}_{\ell}}
  \expe_{n}[g'(\mathbf{\bar{B}}'_{\ell+1})
  \given \mathbf{\tilde{B}}_{\ell+1},\mathbf{\tilde{B}}'_{\ell+1}]
  =\indfunc_{\mathbf{I}_{\ell}}
  \expe_{n}[
  g'(\mathbf{B}'_{\ell+1})
  \given \mathbf{B}'_{\ell}].\label{eq:expgbubbut}
  \end{equation}

  It remains to show that the probability that~\(\mathbf{B}_{\ell+1}\)
  and~\(\mathbf{\bar{B}}_{\ell+1}\) differ can be controlled as claimed.
  By construction
  \(\mathbf{B}_{\ell+1}\) and~\(\mathbf{\bar{B}}_{\ell+1}\)
  differ only if the underlying edge
  structures~\(\tau(\mathbf{B}_{\ell+1})\)
  and~\(\mathbf{\tilde{B}}_{\ell+1}\) differ
  or if the underlying edge structures are the same,
  but the vertex weights differ in a relevant way due to
  rerandomisation.
  Vertex weights have to be rerandomised if~\(\mathbf{\tilde{D}}_{\ell+1}\)
  has a non-empty intersection with~\(\mathbf{\tilde{S}}'_{\ell}\).
  Hence,
  \[
    \begin{split}
      &\indfunc_{\mathbf{I}_{\ell}}
      \prob_{n}(\mathbf{B}_{\ell+1} \neq \mathbf{\bar{B}}_{\ell+1}
      \given \mathbf{B}_{\ell},\mathbf{B}'_{\ell})\\
      &\quad\leq \indfunc_{\mathbf{I}_{\ell}}
      \prob_{n}(\mathbf{\tilde{B}}_{\ell+1} \neq \tau_{\ell+1}(\mathbf{B}_{\ell+1})
      \given \mathbf{B}_{\ell},\mathbf{B}'_{\ell})
      + \indfunc_{\mathbf{I}_{\ell}}
      \prob_{n}(\mathbf{\tilde{B}}'_{\ell+1} \neq \tau_{\ell+1}(\mathbf{B}'_{\ell+1})
      \given \mathbf{B}_{\ell},\mathbf{B}'_{\ell})\\
      &\qquad + \indfunc_{\mathbf{I}_{\ell}}
      \prob_{n}(\mathbf{\tilde{B}}_{\ell+1} =
      \tau_{\ell+1}(\mathbf{B}_{\ell+1}),
      \mathbf{\tilde{B}}'_{\ell+1} = \tau_{\ell+1}(\mathbf{B}'_{\ell+1}),
      \mathbf{\tilde{D}}_{\ell+1} \setintersect \mathbf{\tilde{S}}'_{\ell}
      \neq \emptyset
      \given \mathbf{B}_{\ell}, \mathbf{B}'_{\ell}).
    \end{split}
  \]
  The first and second term can be estimated by~\eqref{eq:bvteqbv}
    and~\eqref{eq:buteqbu}
    from \autoref{lem:covcouplmorecomplextilde}.
    In the third term we can replace~\(\mathbf{\tilde{D}}_{\ell+1}\)
    with~\(\mathbf{D}_{\ell+1}\)
    and~\(\mathbf{\tilde{S}}'_{\ell}\) with~\(\mathbf{S}'_{\ell}\)
    because the edge structures of~\(\mathbf{\tilde{B}}_{\ell+1}\)
    and~\(\mathbf{B}_{\ell+1}\) are the same,
    then the probability that~\(\mathbf{D}_{\ell+1}\)
    and~\(\mathbf{S}'_{\ell}\)
    intersect is given by
    the probability that there is an edge between~\(\mathbf{D}_{\ell}\)
    and~\(\mathbf{S}'_{\ell}\) so that by \autoref{lem:abpath}
  \begin{equation*}
  \indfunc_{\mathbf{I}_{\ell}}
      \prob_{n}(\mathbf{B}_{\ell+1} \neq \mathbf{\bar{B}}_{\ell+1}
      \given \mathbf{B}_{\ell},\mathbf{B}'_{\ell})
  \leq \indfunc_{\mathbf{I}_{\ell}}
  C \frac{\setweight{\mathbf{S}_{\ell}} \setweight{\mathbf{S}'_{\ell}}}
  {n\vartheta}
  + \indfunc_{\mathbf{I}_{\ell}}
  C \frac{\setweight{\mathbf{D}_{\ell}} \setweight{\mathbf{S}'_{\ell}}}
  {n\vartheta}
  \leq \indfunc_{\mathbf{I}_{\ell}}
  C \frac{\setweight{\mathbf{S}_{\ell}} \setweight{\mathbf{S}'_{\ell}}}
  {n\vartheta}.
  \end{equation*}

  The probability that~\(\mathbf{\bar{B}}'_{\ell+1}\)
  and~\(\mathbf{B}'_{\ell+1}\) differ
  can be estimated similarly, taking into account that rerandomisation
  of vertex weights happens additionally if~\(\mathbf{\tilde{D}}_{\ell+1}\)
  and~\(\mathbf{\tilde{D}}'_{\ell+1}\) have non-empty intersection,
  which is the case if there is a path consisting of two edges
  connecting~\(\mathbf{D}_{\ell}\) with~\(\mathbf{D}'_{\ell}\).
  Then \autoref{lem:covcouplmorecomplextilde} and \autoref{lem:abpath} imply
  \begin{align*}
  \indfunc_{\mathbf{I}_{\ell}}\prob_{n}(\mathbf{B}'_{\ell+1} \neq
  \mathbf{\bar{B}}'_{\ell+1}
  \given \mathbf{B}_{\ell},\mathbf{B}'_{\ell})
  &\leq \indfunc_{\mathbf{I}_{\ell}}
  C \frac{\setweight{\mathbf{S}_{\ell}}
  \setweight{\mathbf{S}'_{\ell}}}{n\vartheta}
  + \indfunc_{\mathbf{I}_{\ell}}
  C \frac{\setweight{\mathbf{D}'_{\ell}}
  \setweight{\mathbf{S}_{\ell}}}{n\vartheta}
  + \indfunc_{\mathbf{I}_{\ell}}
  C \frac{\setweight{\mathbf{D}'_{\ell}}
  \setweight{\mathbf{D}_{\ell}}}{n\vartheta}
  \Gamma_{2,n}\\
  &\leq \indfunc_{\mathbf{I}_{\ell}} C (1+\Gamma_{2,n})
  \frac{\setweight{\mathbf{S}_{\ell}}
  \setweight{\mathbf{S}'_{\ell}}}{n\vartheta}.
  \end{align*}
  This completes the proof.
\end{proof}

Thanks to the previous constructions we can bound
the covariance between~\(\mathbf{B}_{\ell+1}\)
and~\(\mathbf{B}'_{\ell+1}\)
by a term involving the covariance between~\(\mathbf{B}_{\ell}\)
and~\(\mathbf{B}'_{\ell}\)
and an error term.
\begin{lemma}\label{lem:nbhdcov:iter}
  Let~\(g\) and~\(g'\) be measurable functions that are bounded
  by~\(1\) in absolute value.
  Then
  \[
  \begin{split}
  &\cov_{n}(g(\mathbf{B}_{\ell+1}),g'(\mathbf{B}'_{\ell+1}))\\
  &\quad\leq C\Bigl(\prob_{n}(\setcomplement{\mathbf{I}_{\ell}}) +
  (1+\Gamma_{2,n})
  \frac{\expe_{n}[\setweight{\mathbf{S}_{\ell}}\setweight{\mathbf{S}'_{\ell}}]}
  {n\vartheta}\Bigr)
  + \cov_{n}(
  \expe_{n}[g(\mathbf{B}_{\ell+1}) \given \mathbf{B}_{\ell}],
  \expe_{n}[g'(\mathbf{B}'_{\ell+1}) \given \mathbf{B}'_{\ell}]).
  \end{split}
  \]
\end{lemma}
\begin{proof}
  Since~\(g\) and~\(g'\) are bounded by one we have
  \begin{equation}\label{eq:gvgui+ic}
    \cov_{n}(g(\mathbf{B}_{\ell+1}),g'(\mathbf{B}'_{\ell+1}))
    \leq \cov_{n}(\indfunc_{\mathbf{I}_{\ell}}
    g(\mathbf{B}_{\ell+1}),
    \indfunc_{\mathbf{I}_{\ell}}g'(\mathbf{B}_{\ell+1}))
    +C\prob_{n}(\setcomplement{\mathbf{I}_{\ell}}).
  \end{equation}
  On~\(\mathbf{I}_{\ell}\)
  we approximate~\(\mathbf{B}_{\ell+1}\) with~\(\mathbf{\bar{B}}_{\ell+1}\)
  and~\(\mathbf{B}'_{\ell+1}\) with~\(\mathbf{\bar{B}}'_{\ell+1}\)
  The covariance on the right-hand side of~\eqref{eq:gvgui+ic}
  can thus be split into the covariance of~\(g(\mathbf{B}_{\ell+1})\)
  and~\(g'(\mathbf{B}'_{\ell+1})\)
  and three other covariances involving~\(g(\mathbf{B}_{\ell+1})-g(\mathbf{\bar{B}}_{\ell+1})\) or~\(g'(\mathbf{B}'_{\ell+1})
  -g'(\mathbf{\bar{B}}'_{\ell+1})\)

  By the law of total covariance
  and using the fact that~\(\indfunc_{\mathbf{I}_{\ell}}\)
  is \((\mathbf{B}_{\ell}, \mathbf{B}'_{\ell})\)-measurable
  we have
  \[
    \begin{split}
      \cov_{n}(
      \indfunc_{\mathbf{I}_{\ell}} g(\mathbf{\bar{B}}_{\ell+1}),
      \indfunc_{\mathbf{I}_{\ell}} g'(\mathbf{\bar{B}}'_{\ell+1})
      ) &= \expe_{n}[\indfunc_{\mathbf{I}_{\ell}}
      \cov_{n}(g(\mathbf{\bar{B}}_{\ell+1}), g'(\mathbf{\bar{B}}'_{\ell+1})
      \given \mathbf{B}_{\ell},\mathbf{B}'_{\ell})
      ]\notag\\
      &\quad + \cov_{n}(
      \indfunc_{\mathbf{I}_{\ell}} \expe_{n}[g(\mathbf{\bar{B}}_{\ell+1})
      \given \mathbf{B}_{\ell}, \mathbf{B}'_{\ell}],
      \indfunc_{\mathbf{I}_{\ell}} \expe_{n}[g'(\mathbf{\bar{B}}'_{\ell+1})
      \given \mathbf{B}_{\ell}, \mathbf{B}'_{\ell}]).\notag\\
    \end{split}
  \]
  By \autoref{lem:covcouplmorecomplex} the first expectation
  vanishes
  and the conditional expectations in the covariance can be rewritten
  based
  on~\(\mathbf{B}_{\ell+1}\) and~\(\mathbf{B}'_{\ell+1}\) so that we obtain
  \begin{equation}
    \cov_{n}(
    \indfunc_{\mathbf{I}_{\ell}} g(\mathbf{\bar{B}}_{\ell+1}),
    \indfunc_{\mathbf{I}_{\ell}} g'(\mathbf{\bar{B}}'_{\ell+1})
    )
    = \cov_{n}(
    \indfunc_{\mathbf{I}_{\ell}}
    \expe_{n}[g(\mathbf{B}_{\ell+1}) \given \mathbf{B}_{\ell}],
    \indfunc_{\mathbf{I}_{\ell}}
    \expe_{n}[g'(\mathbf{B}'_{\ell+1}) \given \mathbf{B}'_{\ell}]
    ).\label{eq:covbbtob}
  \end{equation}
  The indicator in the covariance can be dropped at the cost of
  adding~\(C \prob_{n}(\setcomplement{\mathbf{I}_{\ell}})\).
  Hence, \eqref{eq:covbbtob} implies
  \begin{equation}\label{eq:covigt}
    \cov_{n}(\indfunc_{\mathbf{I}_{\ell}} g(\mathbf{\bar{B}}_{\ell+1}),
    \indfunc_{\mathbf{I}_{\ell}} g'(\mathbf{\bar{B}}_{\ell+1}))
    \leq \cov_{n}(
    \expe_{n}[g(\mathbf{B}_{\ell+1}) \given \mathbf{B}_{\ell}],
    \expe_{n}[g'(\mathbf{B}_{\ell+1}) \given \mathbf{B}'_{\ell}]
    )
    +C \prob_{n}(\setcomplement{\mathbf{I}_{\ell}}).
  \end{equation}

  For the other terms
  involving~\(g(\mathbf{B}_{\ell+1})-g(\mathbf{\bar{B}}_{\ell+1})\)
  or~\(g'(\mathbf{B}'_{\ell+1})-g'(\mathbf{\bar{B}}'_{\ell+1})\)
  the triangle inequality implies
  \begin{equation}\label{eq:covggtg}
  \begin{split}
  &\abs{\cov_{n}(
    \indfunc_{\mathbf{I}_{\ell}} (g(\mathbf{B}_{\ell+1})
    -g_{v}(\mathbf{\bar{B}}_{\ell+1})),
    \indfunc_{\mathbf{I}_{\ell}} g'(\mathbf{\bar{B}}'_{\ell+1})
    )}\\
  &\leq \abs{
    \expe_{n}[\indfunc_{\mathbf{I}_{\ell}}
    (g(\mathbf{B}_{\ell+1})-g(\mathbf{\bar{B}}_{\ell+1}))
    g'(\mathbf{\bar{B}}'_{\ell+1})]}
  + \abs{\expe_{n}[\indfunc_{\mathbf{I}_{\ell}}
    (g(\mathbf{B}_{\ell+1})-g(\mathbf{\bar{B}}_{\ell+1}))]
    \expe_{n}[\indfunc_{\mathbf{I}_{\ell}}
    g'(\mathbf{\bar{B}}'_{\ell+1})]}.
  \end{split}
  \end{equation}
  Consider the first term and use that~\(g'\) is bounded by~\(1\),
  use the tower property to condition on~\(\mathbf{B}_{\ell},
    \mathbf{B}'_{\ell}\)
    then apply \autoref{lem:covcouplmorecomplex}
  \begin{align*}
  \abs{
    \expe_{n}[
    \indfunc_{\mathbf{I}_{\ell}}(g(\mathbf{B}_{\ell+1})-g(\mathbf{\bar{B}}_{\ell+1}))
    g'(\mathbf{\bar{B}}'_{\ell+1})]}
  &\leq
  \expe_{n}[\indfunc_{\mathbf{I}_{\ell}}
  \abs{g(\mathbf{B}_{\ell+1})-g(\mathbf{\bar{B}}_{\ell+1})}]\\
  &\leq 2  \expe_{n}[
  \indfunc_{\mathbf{I}_{\ell}}
  \prob_{n}(\mathbf{B}_{\ell+1} \neq \mathbf{\bar{B}}_{\ell+1}
  \given \mathbf{B}_{\ell}, \mathbf{B}'_{\ell})
  ]\\
  &\leq C (1+\Gamma_{2,n})
  \frac{\expe_{n}[\setweight{\mathbf{S}_{\ell}}
    \setweight{\mathbf{S}'_{\ell}}]}
  {n\vartheta}.
  \end{align*}
  The second term in~\eqref{eq:covggtg} can be bounded similarly.
  Thus
  \begin{equation}\label{eq:covggti}
  \abs{\cov_{n}(
    \indfunc_{\mathbf{I}_{\ell}} (g(\mathbf{B}_{\ell+1})
    -g(\mathbf{\bar{B}}_{\ell+1})),
    \indfunc_{\mathbf{I}_{\ell}} g'(\mathbf{\bar{B}}_{\ell+1}')
    )}
  \leq C(1+\Gamma_{2,n})
  \frac{\expe_{n}[\setweight{\mathbf{S}_{\ell}}
    \setweight{\mathbf{S}'_{\ell}}]}
  {n\vartheta}.
  \end{equation}
  The remaining terms
  can be bounded analogously.

  Combine~\eqref{eq:gvgui+ic}, \eqref{eq:covigt},
  and~\eqref{eq:covggti} and the analogous results for the remaining
  terms
  to obtain the claimed bound.

\end{proof}

The following lemma establishes simple bounds for the
error terms from \autoref{lem:nbhdcov:iter}.
\begin{lemma}\label{lem:covgvgu:compl:rest}
  For~\(\ell \in \naturals\) we have
  \begin{align}
  \prob_{n}(\setcomplement{\mathbf{I}_{\ell}})
  &\leq \frac{\sum_{i,i'} W_{v_{i}}W_{v'_{i'}}}{n \vartheta}
  2^{2\ell}(1+\Gamma_{2,n})^{2\ell}
  \shortintertext{and}
  \expe_{n}[\setweight{\mathbf{S}_{\ell}}\setweight{\mathbf{S}'_{\ell}}]
    &\leq C
      \sum_{i,i'} (W_{v_{i}}+1)(W_{v'_{i'}}+1)
      (\Gamma_{3,n}+1) (\Gamma_{2,n}+2)^{2\ell}.
  \end{align}
\end{lemma}
\begin{proof}
  Fix~\(i \in [m], i' \in  [m']\).
  If one of~\(B_{\ell}(v_{i},\mathbf{G}_{n})\)
  or~\(B_{\ell}(v_{i},\mathbf{G}^{F}_{n})\)
  intersects~\(B_{\ell}(v'_{i'},\mathbf{G}_{n})\)
  of~\(B_{\ell}(v'_{i'},\mathbf{G}^{F'}_{n})\),
  then there exists a path of length~\(2\ell\)
  from~\(v_{i}\)
  to~\(v'_{i'}\)
  in a graph~\(\widebar{\mathbf{G}}_{n}\)
  where an edge~\(e\) is present if~\(X_{e}\)
  or an independent copy~\(X'_{e}\) is equal to~\(1\).
  In particular the edge probability for~\(e=\edge{u}{v}\)
  in~\(\widebar{\mathbf{G}}_{n}\)
  can be bounded by~\(2W_{u}W_{v}/(n\vartheta)\).
  Hence, the calculations for
  \autoref{lem:uinbl} imply that the
  probability of intersection is bounded by
  \begin{equation*}
  \frac{W_{v_{i}}W_{v'_{i'}}}{n\vartheta} 2^{2\ell}(\Gamma_{2,n}+1)^{2\ell}.
  \end{equation*}
  Now sum over~\(i \in [m]\) and~\(i' \in [m']\) to obtain the first claim.

  For the second inequality
  let~\(S_{\ell}(v_{i})\) be the vertex set
  of~\(B_{\ell}(v_{i}, \mathbf{G}_{n})\)
  and~\(S_{\ell}^{F}(v_{i})\) the vertex set of~\(B_{\ell}(v_{i},
  \mathbf{G}_{n}^{F})\),
  similarly
  let~\(S_{\ell}(v'_{i})\) be the vertex set of
  of~\(B_{\ell}(v'_{i}, \mathbf{G}_{n})\)
  and~\(S_{\ell}^{F'}(v'_{i})\) the vertex set
  of~\(B_{\ell}(v'_{i}, \mathbf{G}_{n}^{F})\).
  Then
  Cauchy--Schwarz and \autoref{lem:bl:weight:squared:raw} imply
  \begin{equation*}
  \expe_{n}[\setweight{S_{\ell}(v_{i})}\setweight{S_{\ell}(v'_{i'})}]
  \leq C(W_{v_{i}}+1)(W_{v'_{i'}}+1) (\Gamma_{3,n}+1) (\Gamma_{2,n}+2)^{2\ell}
  \end{equation*}
  and the same bound for~\(S^{F}_{\ell}(v_{i})\) instead
  of~\(S_{\ell}(v_{i})\)
  or~\(S^{F'}_{\ell}(v'_{i'})\) instead of~\(S_{\ell}(v'_{i'})\).
  Now use that
  \[
  \setweight{\mathbf{S}_{\ell}}
  \leq \sum_{i=1}^{m} \setweight{S_{\ell}(v_{i})}
  +\sum_{i=1}^{m} \setweight{S^{F}_{\ell}(v_{i})}
  \]
  and the analogous bound for~\(\setweight{\mathbf{S}'_{\ell}}\)
  to obtain the second claim.
\end{proof}

Together the previous results establish a bound
of order~\(n^{-1}\) for the covariance
between~\(\mathbf{B}_{\ell}\) and~\(\mathbf{B}'_{\ell}\).
\begin{lemma}\label{lem:covgbgpbp}
  For any~\(\ell \in \naturals\)
  \[
      \sup_{g,g'}\cov_{n}(g(\mathbf{B}_{\ell}),g'(\mathbf{B}'_{\ell}))\\
      \leq
      \min\set[\Bigg]{\frac{\sum_{i,i'}
          (W_{v_{i}}+1)(W_{v'_{i'}}+1)}{n\vartheta}
        (\Gamma_{3,n}+1)(\Gamma_{2,n}+C)^{2\ell+1},1},
  \]
  where the supremum is taken over all functions~\(g\) and~\(g'\)
  that are bounded by~\(1\).
  If we set~\(\mathcal{V} = \set{v_{1},\dots,v_{m}}\)
  and~\(\mathcal{V}' = \set{v'_{1},\dots,v'_{m'}}\),
  the result can be rewritten as
  \[
      \sup_{g,g'}\cov_{n}(g(\mathbf{B}_{\ell}),g'(\mathbf{B}'_{\ell}))\\
      \leq
      \min\set[\Bigg]{
        \frac{(\setweight{\mathcal{V}}+\setcard{\mathcal{V}})
        (\setweight{\mathcal{V}'}+\setcard{\mathcal{V}'})}{n\vartheta}
        (\Gamma_{3,n}+1)(\Gamma_{2,n}+C)^{2\ell+1},1}.
  \]
\end{lemma}
\begin{proof}
  Apply first \autoref{lem:nbhdcov:iter}
  then use \autoref{lem:covgvgu:compl:rest} to estimate the
  non-covariance \enquote{error terms}.
  This gives the bound
  \begin{align*}
  &\cov_{n}(g(\mathbf{B}_{\ell+1}),g'(\mathbf{B}'_{\ell+1}))\\
    &\quad\leq C\biggl(\prob_{n}(\setcomplement{\mathbf{I}_{\ell}}) +
  (1+\Gamma_{2,n})
  \frac{\expe_{n}[\setweight{\mathbf{S}_{\ell}}\setweight{\mathbf{S}'_{\ell}}]}
  {n\vartheta}\biggr)
  + \cov_{n}(
  \expe_{n}[g(\mathbf{B}_{\ell+1})
  \given \mathbf{B}_{\ell}],
  \expe_{n}[g'(\mathbf{B}'_{\ell+1})
  \given \mathbf{B}'_{\ell}])\\
  &\quad\leq C\biggl(
  \frac{\sum_{i,i'} W_{v_{i}}W_{v'_{i'}}}{n \vartheta}
  2^{2\ell}(1+\Gamma_{2,n})^{2\ell}
  + \frac{\sum_{i,i'} (W_{v_{i}}+1)(W_{v'_{i'}}+1)}{n\vartheta}
  (\Gamma_{2,n}+1)^{2\ell}(\Gamma_{3,n}+1)
  \biggr)\\
  &\qquad+ \cov_{n}(
  \expe_{n}[g(\mathbf{B}_{\ell+1}) \given \mathbf{B}_{\ell}],
    \expe_{n}[g'(\mathbf{B}'_{\ell+1}) \given \mathbf{B}'_{\ell}]).
  \end{align*}
  Since~\(\expe[g(\mathbf{B}_{\ell+1})
    \given \mathbf{B}_{\ell}]\)
    can be written as~\(\bar{g}(\mathbf{B}_{\ell})\),
    where~\(\bar{g}_{v}\) is a measurable function that is bounded
    by~\(1\),
    and similarly for~\(g'\) with a function~\(\bar{g}'\)
    the term can be rewritten as a covariance
    of functions applied to~\(\mathbf{B}_{\ell}\)
    and~\(\mathbf{B}'_{\ell}\).
  \begin{align*}
  &\cov_{n}(g(\mathbf{B}_{\ell+1}),g'(\mathbf{B}'_{\ell+1}))\\
  &\quad\leq     C\biggl(
  \frac{\sum_{i,i'} W_{v_{i}}W_{v'_{i'}}}{n \vartheta}
  2^{2\ell}(1+\Gamma_{2,n})^{2\ell}
  + \frac{\sum_{i,i'} (W_{v_{i}}+1)(W_{v'_{i'}}+1)}{n\vartheta}
  (\Gamma_{2,n}+1)^{2\ell}(\Gamma_{3,n}+1)
  \biggr)\\
  &\qquad   +\cov_{n}(\bar{g}(\mathbf{B}_{\ell}),
  \bar{g}'(\mathbf{B}'_{\ell}))\\
  &\quad\leq \frac{\sum_{i,i'}(W_{v_{i}}+1)(W_{v'_{i'}}+1)}{n\vartheta}
  (\Gamma_{3,n}+1)(\Gamma_{2,n}+C)^{2\ell+1}
  +\cov_{n}(\bar{g}(\mathbf{B}_{\ell}),
  \bar{g}'(\mathbf{B}_{\ell})).
  \end{align*}
  The claim follows by
  taking the supremum over all measurable bounded functions
  (first on the right-hand side and then on the left-hand side)
  and iteration.
\end{proof}

\subsection{Graph exploration}\label{sec:graphexpl}

We now define a procedure that allows
us to explore the neighbourhood of a vertex in a graph.
This procedure can be applied to arbitrary graphs, so for the
remainder of this section, we shall not restrict ourselves to
the sparse inhomogeneous graph setting
and will work on a general graph~\(G = (V,E)\).
The presentation of the graph exploration
in this section is based on the formulation in lecture notes by
\citet[§~3.5.1]{bordenave}.
The approach is also discussed by \citet[§~4.1]{vdh}
who draws on work by \citet[§~10.5]{alon}.
In those discussions, however, the focus is on the cardinality of
the connected component of a vertex~\(v\)
and not on the complete neighbourhood structure.

Fix a graph~\(G = (V,E)\) with vertex set~\(V\)
and edge set~\(E \subseteq \sset{\edge{u}{v}}{u,v \in V}\).
The idea of the exploration on~\(G\) is to discover which
of the connections that could possibly be present in a graph
with vertex set~\(V\) are actually present in~\(G\).
To this end let~\(V^{(2)} = \sset{\edge{u}{v}}{u,v \in V}\) be
the set of edges in the complete graph on~\(V\)
and call elements of~\(V^{(2)}\) \emph{possible edges of~\(G\)}.
The graph~\(G\) is then completely determined by
the edge indicators~\((\indfunc_{E}(e))_{e \in V^{(2)}}\)
that tell us whether a possible edge~\(e \in V^{(2)}\)
is present in the edge set~\(E\) of~\(G\).
Because we explore~\(G\) by visiting vertices,
it is slightly more convenient to think of these edge indicators
as being indexed by pairs of vertices
\[
  X_{uv} = \indfunc_{E}(\edge{u}{v})
  \quad\text{for~\(u,v \in V\)}.
\]
To make notation a bit easier we will also define~\(X_{vv} = 0\)
for all~\(v \in V\).

Formally, the graph exploration of~\(G\) started in a vertex~\(v_{0} \in V\)
is given by a sequence of vertices~\(v_{0},v_{1},v_{2},v_{3},\dots \in V\)
along with sets~\(C_{j}\), \(A_{j}\) and~\(U_{j}\)
as well as a function~\(\varphi \colon S \to G\),
where~\(S \subseteq \ulamharris\) is a subtree of the Ulam--Harris tree.

\begin{algorithm}\label{alg:explore}
Start with a fixed vertex~\(v_{0}\) in~\(G\)
and set
\(C_{-1} = \emptyset\), \(A_{-1} = \set{v_{0}}\),
\(U_{-1} = V \setminus \set{v_{0}}\)
and on the Ulam--Harris side with~\(\mathbf{i}_{0} = \treeroot\).
Set~\(\varphi(\mathbf{i}_{0}) = \varphi(\treeroot) = v_{0}\).

For~\(j \in \set{0,1,\dots}\)
given~\(C_{j-1}\), \(A_{j-1}\) and~\(U_{j-1}\)
let~\(v_{j} = \varphi(\mathbf{i}_{j})\) be the smallest element
in~\(A_{j-1}\)
(\enquote{smallest} in the sense that its preimage~\(\mathbf{i}_{j}\)
under~\(\varphi\) is minimal in the order~\(\prec\) on the Ulam--Harris
tree).
Define~\(I_{j} = \sset{u \in U_{j-1}}{X_{v_{j}u}=1}\)
and let
\begin{equation*}
C_{j} = C_{j-1} \setunion \set{v_{j}},\quad
A_{j} = A_{j-1} \setminus \set{v_{j}} \setunion I_{j} \quad\text{and}\quad
U_{j} = U_{j-1} \setminus I_{j}.
\end{equation*}
Then set~\(N_{\mathbf{i}_{j}} = \setcard{I_{j}}\),
enumerate the elements of~\(I_{j}\)
as~\(\set{u_{1},\dots,u_{N_{\mathbf{i}_{j}}}}\)
(if we want the exploration to always yield the same results,
 we need to impose an order on the vertices in the set,
 in our applications we can always assume that~\(V = V_{n} = [n]\)
 and use the natural order on~\(\naturals\))
and extend~\(\varphi\) by setting~\(\varphi((\mathbf{i}_{j},1))=u_{1}, \dots,
\varphi((\mathbf{i}_{j},N_{\mathbf{i}_{j}}))=u_{N_{\mathbf{i}_{j}}}\)
so that the image of~\(\varphi\) now also covers all of~\(I_{j}\).

The exploration stops if~\(A_{j} = \emptyset\).
\end{algorithm}
The set~\(C_{j}\) can be seen as the set of explored vertices for which all
neighbours have been seen,
\(A_{j}\) is the set of active vertices
(i.e.~vertices that have been seen by the exploration,
but whose neighbours may not all have been seen yet)
and~\(U_{j}\) the set of unexplored vertices.

Knowledge of the exploration process up to step~\(j\)
in the form
of~\(C_{\ell}\), \(A_{\ell}\), \(U_{\ell}\)
for all~\(\ell \in \set{-1,\dots,j}\)
and~\(\varphi^{-1}\) on the set~\(C_{j} \setunion A_{j}\)
does not give us a complete picture of the explored graph,
because the exploration does not explicitly keep track
of edges between~\(v_{j}\) and~\(A_{j-1}\).

If we rely not only on the sets~\(C_{j}\), \(A_{j}\), \(U_{j}\)
and the function~\(\varphi\),
but instead keep track of all edges between~\(v_{j}\)
and~\(A_{j-1} \setunion U_{j-1}\), we can recover the entire subgraph
structure
and see each edge in the subgraph only once.
This observation motivates the following definition.
\begin{definition}\label{def:nei:explrel}
  Consider the exploration of the neighbourhood of a vertex~\(v_{0}\)
  in a graph~\(G\) as defined in \autoref{alg:explore}.
  For any~\(j\) so that~\(v_{j}\) is well-defined we call
  the neighbours of~\(v_{j}\)
  that are in~\(A_{j-1} \setunion U_{j-1}\)
  \emph{exploration-relevant neighbours}.
\end{definition}
We can recover all neighbours
of~\(v_{j}\) even if we only ever keep track of exploration-relevant
neighbours.

Let~\(\mathcal{G}_{j}\) be the~\(\sigma\)-algebra generated by the
edge indicators along the exploration sequence~\(X_{v_{\ell},u}\)
for~\(\ell \in \set{0,\dots,j}\)
and~\(u \in V\),
i.e.
\[
\mathcal{G}_{j}
= \sigma(X_{e})_{e \in \mathcal{E}'_{j}},
\quad\text{where}\quad
\mathcal{E}'_{j} = \sset{\edge{v_{\ell}}{u} \in V^{(2)}}
    {\ell \in \set{0,\dots,j}, u \in V \setminus \set{v_{\ell}}}.
\]
Clearly the exploration process as recorded by~\(C_{\ell}\),
\(A_{\ell}\),
\(U_{\ell}\) for~\(\ell \in \set{0,\dots,j}\)
and~\(\varphi^{-1}\) on the set~\(C_{j} \setunion A_{j}\)
is measurable with respect to~\(\mathcal{G}_{j}\).

Observe that given~\(\mathcal{G}_{j}\) the selection of~\(v_{j+1}\)
from~\(A_{j}\) is deterministic,
because we only need to know the preimages of the vertices in~\(A_{j}\)
under~\(\varphi\) in order to pick~\(v_{j+1}\).
This implies that~\(v_{j+1}\) is~\(\mathcal{G}_{j}\)-measurable.

On the other hand,~\(X_{v_{j+1},u}\) for~\(u \in A_{j} \setunion
U_{j}\)
are independent of~\(\mathcal{G}_{j}\),
since the relevant possible edges are not included
in~\(\mathcal{E}'_{j}\).
To see this, note that~\(\mathcal{E}'_{j}\)
only contains edges with at least one endpoint
in~\(C_{j}\).
Since~\(v_{j+1} \notin C_{j}\),
it follows that~\(u\) would have to be in~\(C_{j}\)
for~\(\edge{v_{j+1}}{u}\) to be contained in~\(\mathcal{E}'_{j}\).
But~\(u \in A_{j} \setunion U_{j}\) by definition,
which is disjoint with~\(C_{j}\) by construction.

Note further that the preimages of all vertices from~\(I_{j}\),
i.e., the vertices that are added to~\(A_{j}\) in step~\(j\),
are larger in the order~\(\prec\) on the Ulam--Harris tree than
the preimages of vertices in~\(C_{j-1} \setunion A_{j-1}\).
This implies that those vertices are only up for
selection
once all vertices from~\(A_{j-1}\) have been fully explored.
In particular the sequence~\(v_{0},v_{1},\dots\) contains
the vertices in exactly the order in which they were added to
the active set (i.e., removed from the unexplored set).
This means that the vertex
sequences~\((v_{\varphi(\mathbf{i})})_{\mathbf{i} \in S}\)
(traversed in the order given by~\(\prec\))
and~\((v_{j})_{j \in [\setcard{S}]}\) are exactly the same.

\subsection{Neighbourhood coupling to the
limiting tree}\label{sec:gimt}

We now construct a coupling between the neighbourhood of
a vertex~\(v\) in the unweighted inhomogeneous random graph~\(G_{n}\)
satisfying \autoref{ass:coupling}
and a Galton--Watson tree.
We slightly modify the approach by \citet{olveracravioto}
by combining it with the exploration as described by \citet{bordenave}.
As mentioned before we do not work under the minimal moment assumptions by
\citet{olveracravioto}, instead we will assume
that second and third moments of the connectivity weight distributions
exist.
This allows us to simplify some arguments and prove much more explicit
bounds for the coupling probabilities.
Couplings like this with explicit error bounds are
interesting in the context of the objective method
\citep{aldoussteele}.
A related coupling was used by \citet{fraiman}
to define and analyse stochastic recursions
(similar in principle to the RDE and RTP we briefly mentioned before)
on directed random graphs.

The coupling is found in two steps.
In a first step the neighbourhood is coupled to an intermediate tree
in which the connectivity weights are still dependent on~\(\mathcal{F}_{n}\).
The intermediate tree is then coupled to the desired limiting object
in a second step.

\begin{definition}\label{def:imt}
  Fix a vertex~\(v \in V_{n}\)
  and conditionally on~\(\mathcal{F}_{n}\)
  define the intermediate tree~\(\imt{T}(v)\)
  via a sequence of random
  variables~\(\sset{(\imt{W}_{\mathbf{i}},\imt{N}_{\mathbf{i}})}
  {\mathbf{i} \in \ulamharris}\),
  where~\(\imt{W}_{\mathbf{i}}\) is the type of individual~\(\mathbf{i}\)
  and~\(\imt{N}_{\mathbf{i}}\) is its number of children.
  The distribution of~\(\sset{(\imt{W}_{\mathbf{i}},\imt{N}_{\mathbf{i}})}
  {\mathbf{i} \in \ulamharris}\)
  satisfies
  \begin{itemize}
  \item \(\imt{W}_{\treeroot} = W_{v}\)
    and~\(\imt{N}_{\treeroot} \sim \mathrm{Poi}\bigl(\frac{W_{v} \Lambda_{n}}{n\vartheta}\bigr)\),
  \item all other (non-root) individuals~\(\mathbf{i} \neq \treeroot\)
    have independent types and numbers of
    children~\((\imt{W}_{\mathbf{i}},\imt{N}_{\mathbf{i}})\)
    with distribution
    \[
      \prob_{n}((\imt{W}_{\mathbf{i}},\imt{N}_{\mathbf{i}}) \in \placeholder)
      = \sum_{i=1}^{n} \frac{W_{i}}{\Lambda_{n}}
      \prob((W_{i},D_{i}) \in \placeholder \given W_{i}),
    \]
    where~\(D_{i}\)
    is Poisson distributed
    with mean~\(\Lambda_{n} W_{i} / (n\vartheta)\)
    given~\(W_{i}\).
  \end{itemize}

  The tree structure on~\(\imt{T}(v)\) is then obtained
  recursively from~\(\imt{A}_{0} = \set{\treeroot}\)
  and
  \[
    \imt{A}_{k}
    = \sset{(\mathbf{i},j)}
    {\mathbf{i} \in \imt{A}_{k-1}, 1 \leq j \leq \imt{N}_{\mathbf{i}}}
    \quad\text{for~\(k \in \naturals\), \(k \geq 1\)}.
  \]
\end{definition}

Intuitively, this defines~\(\imt{T}(v)\)
as a multi-type Galton--Watson process with~\(n\) types
corresponding to the vertices of~\(G_{n}\).
Technically, we have defined the tree so that
the type of a vertex is its weight~\(\imt{W}_{\mathbf{i}}\).
If the weights are all different, we can immediately
infer which vertex~\(v\) gave rise to this weight
and call~\(v\) the type.
If some of the~\(W_{v}\) are the same,
we may assume that we sample~\(\imt{W}_{\mathbf{i}}\)
by partitioning the unit interval into~\(n\)
subintervals of length~\(W_{i}/\Lambda_{n}\),
drawing a uniform random variable from~\(\intervaloo{0}{1}\)
and choosing~\(\imt{W}_{\mathbf{i}}\)
equal to the~\(W_{i}\) into whose interval the uniform
random variable falls.

We now use the exploration introduced in \autoref{sec:graphexpl}
to explore the neighbourhood of~\(v\) in~\(G_{n}\)
and build the intermediate tree~\(\imt{T}(v)\) at the same
time coupling the two in the process.
Broadly speaking the graph exploration is driven by Bernoulli
random variables, which we couple to Poisson
random variables in order to arrive at the Poisson
structure of~\(\imt{T}(v)\).

To simplify notation, set
\[
  p'_{vu} = \frac{W_{v} W_{u}}{n\vartheta},
\]
such that~\(p_{vu} = p'_{vu} \imin 1\).

Let~\(X_{vu} \sim \mathrm{Bin}(1,p_{vu})\)
be the edge indicators in~\(G_{n}\).
Couple~\(Z_{vu} \sim \mathrm{Poi}(p'_{vu})\)
to~\(X_{vu}\)
(this coupling is constructed more explicitly in \autoref{lem:couplexz}).
Let~\(Z^{*}_{vu}\) be i.i.d.~copies of~\(Z_{vu}\) that are independent
of
both~\(Z_{vu}\) and~\(X_{vu}\).

Start the exploration in~\(G_{n}\)
in the vertex~\(v_{0} = v\) (set~\(C_{-1} = \emptyset\),
\(A_{-1} = \set{v}\), \(U_{-1} = V_{n} \setminus \set{v}\)).
In~\(\imt{T}(v)\) give~\(\treeroot\) the type~\(v\)
and let~\(\mathbf{i}_{0} = \treeroot\) be the first individual
we visit
(we keep track of the types of individuals we see
via~\(\imt{C}_{-1} = \emptyset\),
\(\imt{A}_{-1} = \set{v}\), \(\imt{U}_{-1} = V_{n} \setminus \set{v}\)).

In~\(G_{n}\)
the exploration process
is governed by the random variables~\(X_{v_{j}u}\)
for~\(u \in A_{j-1} \setunion U_{j-1}\).
In particular given~\(C_{j-1}\), \(A_{j-1}\), \(U_{j-1}\),
we can select~\(v_{j}\)
and obtain~\(I_{j}\) by collecting those
vertices~\(u \in U_{j-1}\) for which~\(X_{v_{j}u} = 1\).
We also record those vertices in~\(A_{j-1}\)
that satisfy~\(X_{v_{j}u}=1\), which then allows us to
recover all exploration-relevant neighbours
and thus the complete graph structure.
(Assume that we order elements in~\(I_{j}\) by their vertex label.)

In~\(\imt{T}(v)\) we
have the analogous sets of types~\(\imt{C}_{j-1}\), \(\imt{A}_{j-1}\),
\(\imt{U}_{j-1}\)
and obtain the children of the type-\(v_{j}\) vertex~\(\mathbf{i}_{j}\)
by collecting~\(Z_{v_{j}u}\) children of type~\(u\)
for~\(u \in \imt{U}_{j-1} \setunion \imt{A}_{j-1}\)
and~\(Z^{*}_{v_{j}u}\) children of type~\(u\) for~\(u \in \imt{C}_{j}\).
(As written, this procedure might suggest a certain ordering for the children
that depends on the types in~\(\imt{U}_{j-1}\), \(\imt{A}_{j-1}\)
and~\(\imt{C}_{j-1}\).
The label of the children should not carry any information
about its type, so we relabel the children of~\(\mathbf{i}_{j}\)
randomly and add them to~\(\imt{T}(v_{0})\) as~\((\mathbf{i}_{j},t)\).)
Let~\(\imt{I}_{j}\) be the set of types of children of~\(v_{j}\).

Assuming that~\(C_{j-1} = \imt{C}_{j-1}\),
\(A_{j-1}=\imt{A}_{j-1}\) and~\(U_{j-1} = \imt{U}_{j-1}\),
the exploration-relevant neighbours of~\(v_{j}\) in~\(G\)
can be identified with the children of the type-\(v_{j}\)
vertex in~\(\imt{T}(v_{0})\) if
\begin{enumerate}
  \item\label{itm:coupl:xz}
    \(X_{v_{j}u} = Z_{v_{j}u}\)
    for all~\(u \in U_{j-1} \setunion A_{j-1}\),
  \item\label{itm:coupl:za0}
    \(Z_{v_{j}u} = 0\)
    for all~\(u \in A_{j-1}\) and
  \item\label{itm:coupl:zc0}
    \(Z^{*}_{v_{j}u} = 0\) for all~\(u \in C_{j-1}\).
\end{enumerate}

Conditions~\ref{itm:coupl:za0} and~\ref{itm:coupl:zc0}
imply that~\(\imt{I}_{j}\) only contains types from~\(U_{j-1}\).
Together with condition~\ref{itm:coupl:xz}
this implies~\(I_{j} = \imt{I}_{j}\),
so that we can conclude
that also~\(C_{j} = \imt{C}_{j}\),
\(A_{j}=\imt{A}_{j}\) and~\(U_{j} = \imt{U}_{j}\).
Then condition~\ref{itm:coupl:xz} ensures that~\(u\)
is an exploration-relevant neighbour of~\(v_{j}\)
if and only if the type-\(v_{j}\) vertex
in~\(\imt{T}(v_{0})\) has a unique child of type~\(u\).
But this implies that the subgraph of~\(G\)
induced by~\(C_{j}\) has the same graph structure
as the~\(\imt{T}(v_{0})\) constructed so far.

In order to continue exploring~\(G_{n}\) and building~\(\imt{T}(v)\)
at the same time, we now reorder the children of~\(\mathbf{i}_{j}\)
so that their order matches the order in~\(I_{j}\).
(This does not change the graph structure modulo graph isomorphism,
which is all we are concerned with.
It just ensures that the next exploration step continues
with an individual of the correct type.)

We now verify that the procedure to generate children of~\(\mathbf{i}_{j}\)
actually yields~\(\imt{T}(v_{0})\) as defined in \autoref{def:imt}.

\begin{lemma}
  The tree generated via this coupling procedure
  has the distribution of the intermediate tree
  as defined in \autoref{def:imt}.
\end{lemma}

\begin{proof}
  The distribution of the tree in \autoref{def:imt} is fully characterised
  by the offspring distribution for each individual.

  From the definition it follows that
  the number of children of
  an individual
  of type~\(u\)
  is Poisson distributed
  with mean~\(\Lambda_{n}W_{u} (n\vartheta)^{-1}\).
  The types of these children can be obtained by a thinning with
  the type distribution for non-root individuals:
  Each child has type~\(W_{i}\) independently
  with probability~\(W_{i}\Lambda_{n}^{-1}\)
  (i.e.~the type is chosen according to~\(\sizebias{\nu}_{n}\)).
  This implies that the numbers of children of
  type~\(i\) of a type-\(u\) individual
  are independently Poisson distributed with parameter
  \[
    \frac{W_{i}}{\Lambda_{n}}
    \frac{\Lambda_{n}W_{u}}{n\vartheta}
    = \frac{W_{u} W_{i}}{n\vartheta}
    = p'_{ui}.
  \]
  This coincides exactly with the offspring distribution
  induced by the coupling procedure,
  where a type-\(u\) individual has~\(\mathrm{Poi}(p'_{ui})\)
  many children of type~\(i\)
  (depending on the status of the type~\(i\) in the exploration so far
   this is either the random variable~\(Z_{ui}\) or~\(Z^{*}_{ui}\)).
\end{proof}

In order to estimate how long the coupling procedure can continue
to produce isomorphic structures,
we we first estimate the probability that
for a fixed~\(v\)
the coupled random variables~\(X_{vu}\) and~\(Z_{vu}\) differ for any~\(u\).
\begin{lemma}\label{lem:couplexz}
  Let~\(v \in V_{n}\), then
  we can couple~\(X_{vu} \sim \mathrm{Bin}(1,p_{vu})\)
  and~\(Z_{vu} \sim \mathrm{Poi}(p'_{vu})\)
  for~\(u \in V_{n} \setminus \set{v}\)
  such that for any~\(J \subseteq V_{n} \setminus \set{v}\)
  \begin{equation*}
    \prob_{n}
    (
    \max_{u \in J}
    \abs{X_{vu}-Z_{vu}} \geq 1)
    \leq\sum_{u \in V_{n} \setminus \set{v}}
    ((p'_{vu})^{2}
    +p'_{vu}\indfunc_{\set{p'_{vu} \geq 1}}).
  \end{equation*}
\end{lemma}
\begin{proof}
  In order to couple~\(X_{vu} \sim \mathrm{Bin}(1,p_{vu})\)
  and~\(Z_{vu} \sim \mathrm{Poi}(p'_{vu})\)
  consider auxiliary random variables~\(Y_{vu} \sim
  \mathrm{Poi}(p'_{vu})\).
  First we couple the Bernoulli random variable~\(X_{vu}\)
  to a Poisson random variable~\(Y_{vu}\)
  with the same parameter~\(p_{uv}\)
  such that
  \begin{equation*}
    \prob_{n}(X_{vu} \neq Y_{vu})
    \leq p_{vu}^{2}
    \leq (p'_{vu})^{2}.
  \end{equation*}
  In a second step we couple
  the two Poisson random variables~\(Y_{vu}\) and~\(Z_{vu}\)
  with parameters~\(p_{vu}\) and~\(p'_{vu}\)
  such that
  \begin{equation*}
    \prob_{n}(Y_{vu} \neq Z_{vu})
    = p'_{vu} - p'_{vu} \imin 1
    \leq p'_{vu}\indfunc_{\set{p'_{vu} \geq 1}}.
  \end{equation*}
  Summing over the vertices in~\(J \subseteq V_{n} \setminus \set{v}\)
  we thus obtain
  \begin{equation*}
    \prob_{n}
      (
      \max_{u \in J}
      \abs{X_{vu}-Z_{vu}} \geq 1)
    \leq \sum_{u \in J} (
      \prob_{n}(X_{vu} \neq Y_{vu})
      +\prob_{n}(Y_{vu} \neq Z_{vu}))
    \leq\!\!\sum_{u \in V_{n} \setminus \set{v}}
      ((p'_{vu})^{2}
      +p'_{vu}\indfunc_{\set{p'_{vu} \geq 1}})
  \end{equation*}
  as claimed.
\end{proof}

The main result of this section is
\begin{prop}\label{prop:blttl}
  For any vertex~\(v \in V_{n}\) of~\(G_{n} = (V_{n},E_{n})\)
  and any level~\(\ell \in \naturals\)
  it is possible to couple the neighbourhood of a vertex~\(v\)
  to an intermediate tree~\(\imt{T}(v)\)
  such that for any sequence~\((k_{n})_{n\in\naturals} \subseteq \intervaloo{0}{\infty}\)
  \[
    \prob_{n}(B_{\ell}(v) \ncong \imt{T}_{\ell}(v))
    \leq \expe_{n}[\setweight{S_{\ell}(v)}_{2} ]
    \frac{\Gamma_{2,n}}{n\vartheta}
    + \expe_{n}[\setweight{S_{\ell}(v)}_{+}]\Gamma_{1,n}
    + \expe_{n}[\setweight{S_{\ell}(v)}]
    \Bigl(\kappa_{1,n}+\frac{1}{k_{n}}+\frac{k_{n}}{n\vartheta}\Biggr),
  \]
  where~\(\imt{T}_{\ell}(v)\) is~\(\imt{T}(v)\)
  truncated at level~\(\ell\).
\end{prop}

\begin{proof}
  Recall that~\(S_{\ell}(v)\) is the set of vertices
  in the~\(\ell\)-neighbourhood~\(B_{\ell}(v)\) of~\(v\) in~\(G_{n}\)
  and that we set~\(D_{\ell}(v) = S_{\ell}(v) \setminus S_{\ell-1}(v)\).
  Explore~\(S_{\ell}(v)\)
  with the exploration defined in \autoref{alg:explore}.

  Let~\(\mathcal{G}_{j}\) be the~\(\sigma\)-algebra generated
  by the edge indicators~\(X_{e}\)
  in the exploration process on~\(G_{n}\) up to step~\(j\)
  along with the coupled Poisson random variables~\(Z_{e}\)
  (cf.~\autoref{lem:couplexz}) and an independent copy~\(Z^{*}_{e}\)
  of~\(Z_{e}\)
  \[
    \mathcal{G}_{j}
    = \sigma(X_{e},Z_{e},Z^{*}_{e})_{e \in \mathcal{E}'_{j}}.
  \]
  The Poisson random variables~\(Z_{e}\) and~\(Z^{*}_{e}\)
  will be used to construct~\(\imt{T}(v)\).

  For any~\(u \in  S_{\ell}(v)\)
  let~\(C(u)\), \(A(u)\) and~\(U(u)\)
  be the sets~\(C_{j-1}\), \(A_{j-1}\) and~\(U_{j-1}\), respectively,
  when~\(u\)'s neighbours are being explored,
  i.e.~when~\(u = v_{j}\) for some~\(j \in \naturals\).
  Additionally we also define~\(\mathcal{G}(u) = \mathcal{G}_{j-1}\)
  and~\(\mathcal{G}^{+}(u) = \mathcal{G}_{j}\).

  With this setup the step in which~\(u\)'s unexplored neighbours
  are explored is measurable with respect to~\(\mathcal{G}^{+}(u)\),
  but conditionally independent of~\(\mathcal{G}(u)\)
  given the sets~\(C(u)\), \(A(u)\) and~\(U(u)\).

  We say that the coupling between
  the neighbourhoods of~\(v\) in~\(G_{n}\) and~\(\imt{T}(v)\)
  breaks in level~\(\ell\) if there is
  a~\(u \in D_{\ell-1}(v)\), i.e.~a vertex at level~\(\ell-1\),
  satisfying
  \begin{enumerate}
  \item \(X_{uu'} \neq Z_{uu'}\) for some~\(u' \in U(u) \setunion
    A(u)\),
  \item \(Z_{uu'}\neq 0\) for some~\(u' \in A(u)\) or
  \item \(Z^{*}_{uu'} \neq 0\) for some~\(u' \in C(u)\).
  \end{enumerate}
  For a fixed vertex~\(u \in D_{\ell-1}(v)\) let~\(N(u)\) event that at
  least one
  of these three conditions is true for~\(u\).
  Let~\(M_{\ell}\) be the event that the coupling \emph{holds} up to
  level~\(\ell\),
  i.e.~that the coupling has not yet broken up to level~\(\ell\).
  Note that if the coupling holds up to level~\(\ell\),
  the graphs are isomorphic as rooted graphs up to level~\(\ell\).

  Let~\(\mathcal{G}'_{\ell}\) be the~\(\sigma\)-algebra generated by the
  exploration
  process for all vertices up to level~\(\ell-1\)
  \begin{equation}\label{eq:salglvlv}
    \mathcal{G}'_{\ell}
    = \sigma\Bigl(\bigsetunion_{u \in S_{\ell-1}(v)}
    \mathcal{G}^{+}(u)\Bigr).
  \end{equation}
  This~\(\sigma\)-algebra already contains information about the vertices
  at level~\(\ell\) in~\(D_{\ell}(v)\)
  since
  \[
    D_{\ell}(v)
    = V_{n} \setminus \Bigl(
    \bigsetunion_{u \in S_{\ell-1}(v)} C(u)
    \setunion \bigsetintersect_{u \in S_{\ell-1}(v)} U(u)
    \Bigr),
  \]
  such that all of~\(B_{\ell}(v)\) is~\(\mathcal{G}'_{\ell}\)-measurable.
  But the edges going from~\(D_{\ell}(v)\)
  to the as of yet not fully explored vertices~\(V_{n} \setminus
  S_{\ell-1}(v)\)
  are independent of~\(\mathcal{G}'_{\ell}\).
  Clearly,~\(M_{\ell-1}\) is~\(\mathcal{G}'_{\ell-1}\)-measurable.

  Additionally we have~\(\mathcal{G}(u) \supseteq \mathcal{G}'_{\ell-1}\)
  for all~\(u \in D_{\ell-1}(v)\).

  Ultimately we want to estimate the probability
  \(
  \prob_{n}(M_{\ell}^{c})
  \).
  We will do this by noting that
  \(M_{\ell}^{c} = M_{\ell-1}^{c} \setunion (M_{\ell-1} \setintersect
  M_{\ell}^{c})\)
  so that
  \begin{equation}
    \prob_{n}(M_{\ell}^{c})
    \leq \prob_{n}(\setweight{S_{\ell}(v)} > k_{n})
      + \sum_{j=1}^{\ell} \prob_{n}(M_{j-1} \setintersect M_{j}^{c},
      \setweight{S_{j}(v)} \leq k_{n}).\label{eq:breakuptolvl:sumat}
  \end{equation}
  Conditionally on~\(\mathcal{G}'_{\ell-1}\)
  the summands of the second sum can be split further by noting
  that~\(M_{\ell-1} \setintersect M_{\ell}^{c}\)
  can be written as a union over~\(N(u)\) for~\(u \in D_{\ell-1}(v)\).
  Additionally,
  we have that~\(\setweight{C(u) \setunion A(u)} \leq
  \setweight{S_{\ell}(v)}\),
  because all vertices in~\(C(u) \setunion A(u)\)
  must be elements of~\(S_{\ell}(v)\)
  since they are neighbours of vertices in~\(S_{\ell-1}(v)\).
  Hence,
  \begin{equation}\label{eq:breakatlvl}
      \prob_{n}(\setweight{S_{\ell}(v)} \leq k_{n},
      M_{\ell-1} \setintersect M_{\ell}^{c}\given \mathcal{G}'_{\ell-1})
      \leq
      \indfunc_{M_{\ell-1}}
      \sum_{u \in D_{\ell-1}(v)}
      \prob_{n}(\setweight{C(u) \setunion A(u)} \leq k_{n},
      N(u)\given \mathcal{G}'_{\ell-1}).
  \end{equation}
  For the individual summands for~\(u \in D_{\ell-1}(v)\)
  recall the definition of~\(N(u)\)
  and additionally condition
  on~\(\mathcal{G}(u) \supseteq \mathcal{G}'_{\ell-1}\)
  such that
  \begin{align*}
    &\prob_{n}(\setweight{C(u) \setunion A(u)} \leq k_{n},
      N(u) \given \mathcal{G}'_{\ell-1})\notag\\
    &\quad\leq
      \expe_{n}\Bigl[
      \prob_{n}\Bigl(
      \max_{u' \in U(u) \setunion A(u) \setminus \set{u}}
      \abs{X_{uu'}-Z_{uu'}} \geq 1
      \given[\Big] \mathcal{G}(u)\Bigr)
      \given[\Big] \mathcal{G}'_{\ell-1}\Bigr]\notag\\
    &\qquad  +
      \expe_{n}\Bigl[
      \indfunc_{\set{\setweight{C(u) \setunion A(u)} \leq k_{n}}}
      \prob_{n}\Bigl(\sum_{u' \in A(u)} Z_{uu'}
      + \sum_{u' \in C(u)} Z^{*}_{uu'} \geq 1
      \given[\Big] \mathcal{G}(u)\Bigr)
      \given[\Big] \mathcal{G}'_{\ell-1}\Bigr]
  \end{align*}
  \autoref{lem:couplexz}
    still holds conditionally on~\(\mathcal{G}(u)\)
    for the~\(\mathcal{G}(u)\)-measurable
    set~\(J = U(u) \setunion A(u)  \setminus \set{u}\),
    because~\(X_{uu'}\) and~\(Z_{uu'}\) are independent
    of~\(\mathcal{G}(u)\),
    furthermore
    by independence of the Poisson random variables~\(Z_{uu'}\)
    and~\(Z^{*}_{uu'}\) from each other (even conditional
    on~\(\mathcal{G}(u)\))
    and the previous steps in the exploration the
    sum of~\(\sum_{u' \in A(u)} Z_{uu'}\) and~\(\sum_{u' \in C(u)}
    Z^{*}_{uu'}\)
    has distribution \(\mathrm{Poi}(\sum_{u' \in A(u) \setunion C(u)}
    p'_{uu'})\).
  Hence, we have
  \begin{align}
    &\prob_{n}(\setweight{C(u) \setunion A(u)} \leq k_{n},
      N(u) \given \mathcal{G}'_{\ell-1})\notag\\
    &\quad\leq
      \expe_{n}\Bigl[
      \sum_{u'\in V_{n} \setminus \set{u}} ((p'_{uu'})^{2}
      +p'_{uu'}\indfunc_{\set{p'_{uu'} \geq 1}})
     + \indfunc_{\set{\setweight{C(u) \setunion A(u)} \leq k_{n}}}
      (1-e^{-\sum_{u'\in A(u) \setunion C(u)} p'_{uu'}})
      \given[\Big] \mathcal{G}'_{\ell-1}\Bigr]\notag\\
      &\quad\leq
      \expe_{n}\Bigl[
      \sum_{u' \in V_{n} \setminus \set{u}}
      ((p'_{uu'})^{2}+p'_{uu'}\indfunc_{\set{p'_{uu'} \geq 1}})
      \indfunc_{\set{\setweight{C(u) \setunion A(u)} \leq k_{n}}}
      \sum_{u' \in A(u) \setunion C(u)} p'_{uu'}
      \given[\Big]\mathcal{G}'_{\ell-1}\Bigr],
    \label{eq:coupbreak:twosums}
  \end{align}
  where the last inequality follows from~\(1-e^{-x} \leq x\) for~\(x >
  -1\).

  For the first inner sum in~\eqref{eq:coupbreak:twosums} we get
  \begin{align*}
    &\expe_{n}\Bigl[
      \sum_{u' \in V_{n} \setminus \set{u}}
      ((p'_{uu'})^{2} +p'_{uu'}\indfunc_{\set{p'_{uu'} \geq 1}})
      \given[\Big] \mathcal{G}'_{j-1}\Bigr]\notag\\
    &\quad\leq \expe_{n}\Bigl[
      \sum_{u' \in V_{n} \setminus \set{u}} \frac{W_{u}^{2}W^{2}_{u'}}{\vartheta^{2} n^{2}}
      +\sum_{u' \in V_{n} \setminus \set{u}} \frac{W_{u}W_{u'}}{n \vartheta}
      \indfunc_{\set{W_{u}W_{u'} \geq n\vartheta}}
      \given[\Big] \mathcal{G}'_{\ell-1}\Bigr].
\end{align*}
    The indicator function can be split by
    by noting that if~\(W_{u}W_{u'} \geq n\vartheta\), then
    \(W_{u} \geq \sqrt{n\vartheta}\) or \(W_{u'} \geq \sqrt{n\vartheta}\),
    in particular~\(\indfunc_{\set{W_{u}W_{u'} \geq n\vartheta}}
    \leq \indfunc_{\set{W_{u} \geq \sqrt{n\vartheta}}}
      +\indfunc_{\set{W_{u'} \geq \sqrt{n\vartheta}}}\)
    so that the second sum splits into two sums with
    different indicators each.
    We additionally factor out all terms that do not depend on~\(u'\)
    and
    recall the definitions of~\(\Gamma_{p,n}\)
    and~\(\kappa_{p,n}\) to obtain
    \begin{align}
          &\expe_{n}\Bigl[
      \sum_{u' \in V_{n} \setminus \set{u}}
      ((p'_{uu'})^{2} +p'_{uu'}\indfunc_{\set{p'_{uu'} \geq 1}})
      \given[\Big] \mathcal{G}'_{j-1}\Bigr]\notag\\
    &\quad\leq \expe_{n}\Bigl[
      \frac{W_{u}^{2}}{n\vartheta} \frac{1}{n\vartheta}\sum_{u \in V_{n}} W_{u'}^{2}
    +W_{u}\indfunc_{\set{W_{u} \geq \sqrt{n\vartheta}}}
      \frac{1}{n \vartheta}
      \sum_{u' \in V_{n}} W_{u'}
      +W_{u}
      \frac{1}{n \vartheta}
      \sum_{u' \in V_{n}} W_{u'} \indfunc_{\set{W_{u'} \geq \sqrt{n\vartheta}}}
     \given[\Big] \mathcal{G}'_{\ell-1}\Bigr].\notag\\
    &\quad\leq \frac{\Gamma_{2,n}}{n\vartheta}
      \expe_{n}[W_{u}^{2} \given\mathcal{G}'_{\ell-1}]
      + \Gamma_{1,n}
      \expe_{n}[W_{u} \indfunc_{\set{W_{u} > \sqrt{n\vartheta}}}
      \given\mathcal{G}'_{\ell-1}]
      + \kappa_{1,n}
      \expe_{n}[W_{u} \given\mathcal{G}'_{\ell-1}].
      \label{eq:pqqbsum}
  \end{align}

  For the second inner sum in~\eqref{eq:coupbreak:twosums} we find
  \begin{align}
    \expe_{n}\Bigl[
      \indfunc_{\set{\setweight{C(u) \setunion A(u)} \leq k_{n}}}
      \sum_{u' \in A(u) \setunion C(u)} p'_{uu'}
      \given[\Big] \mathcal{G}'_{\ell-1}\Bigr]
    &\leq
      \expe_{n}\Bigl[
      \indfunc_{\set{\setweight{C(u) \setunion A(u)} \leq k_{n}}}
      \frac{W_{u}}{n\vartheta}
      \setweight{A(u) \setunion C(u)}
      \given[\Big] \mathcal{G}'_{\ell-1}\Bigr]\notag\\
    &\leq \frac{\expe_{n}[W_{u} \given \mathcal{G}'_{\ell-1}]}{n\vartheta}
      k_{n}.\label{eq:acsum}
  \end{align}

  Then~\eqref{eq:coupbreak:twosums} together with~\eqref{eq:pqqbsum}
  and~\eqref{eq:acsum}
  implies
  \begin{align}
      &\prob_{n}(\setweight{C(u) \setunion A(u)} \leq k_{n},
      N(u) \given[] \mathcal{G}'_{\ell-1})\notag\\
    \begin{split}
      &\quad\leq
        \expe_{n}[W_{u}^{2} \given\mathcal{G}'_{\ell-1}]
        \frac{\Gamma_{2,n}}{n\vartheta}
      +\Gamma_{1,n}
      \expe_{n}[W_{u} \indfunc_{\set{W_{u} > \sqrt{n\vartheta}}}
      \given\mathcal{G}'_{\ell-1}]
        + \expe_{n}[W_{u} \given\mathcal{G}'_{\ell-1}]
        \Bigl(\kappa_{1,n}+\frac{k_{n}}{n\vartheta}\Bigr)
      .\label{eq:breakstep:end}
    \end{split}
  \end{align}

  Hence, by~\eqref{eq:breakatlvl},
  \eqref{eq:breakstep:end} and~\eqref{eq:coupbreak:twosums}
  \begin{align*}
     & \sum_{j=1}^{\ell}
      \prob_{n}(M_{j-1} \setintersect M_{j}^{c}, \setweight{S_{j}(v)}\leq
      k_{n})\notag\\
    &\leq
      \sum_{j=1}^{\ell}
      \expe_{n}[
      \prob_{n}(\setweight{C(u) \setunion A(u)} \leq k_{n},
      N(u) \given \mathcal{G}'_{j-1})
      ]\notag\\
      &\leq\sum_{j=1}^{\ell} \expe_{n}\Biggl[
      \sum_{u \in D_{j-1}(v)}\!\!
      \expe_{n}[W_{u}^{2} \given\mathcal{G}'_{\ell-1}]
        \frac{\Gamma_{2,n}}{n\vartheta}
              +\Gamma_{1,n}
      \expe_{n}[W_{u} \indfunc_{\set{W_{u} > \sqrt{n\vartheta}}}
      \given\mathcal{G}'_{\ell-1}]
     + \expe_{n}[W_{u} \given\mathcal{G}'_{\ell-1}]
        \Bigl(\kappa_{1,n}+\frac{k_{n}}{n\vartheta}\Bigr)\Biggr]\!.
  \end{align*}
  Recall the definitions of~\(\setweight{\placeholder}_{p}\)
    and~\(\setweight{\placeholder}_{+}\).
    Then use that the disjoint union of the~\(D_{j-1}(v)\)
    from~\(j=1\) to~\(\ell\) is equal to~\(S_{\ell-1}(v)\)
    to estimate the sums over~\(W_{u}^{2}\) and~\(W_{u}\)
    with~\(\setweight{S_{\ell}(v)}_{2}\) and~\(\setweight{S_{\ell}(v)}\),
    respectively.
    We obtain the bound
    \begin{equation}   \label{eq:breaklevell:uncoupled}
      \begin{split}
     & \sum_{j=1}^{\ell}
      \prob_{n}(M_{j-1} \setintersect M_{j}^{c}, \setweight{S_{j}(v)}\leq
      k_{n})\\
     &\quad\leq
      \expe_{n}[\setweight{S_{\ell}(v)}_{2} ]
       \frac{\Gamma_{2,n}}{n\vartheta}
       + \expe_{n}[\setweight{S_{\ell}(v)}_{+}]
         \Gamma_{1,n}
        + \expe_{n}[\setweight{S_{\ell}(v)}]
        \Bigl(\kappa_{1,n}+\frac{k_{n}}{n\vartheta}\Bigr).
 \end{split}
  \end{equation}

  Hence, by~\eqref{eq:breakuptolvl:sumat},
  \eqref{eq:breaklevell:uncoupled}
  and Markov's inequality the probability that the coupling
  breaks can be bounded as follows
  \begin{align*}
    \prob_{n}(B_{\ell}(v) \ncong \imt{T}_{\ell}(v))
    &\leq
      \prob_{n}(\setweight{S_{\ell}(v)} > k_{n})
      +\sum_{j=1}^{\ell}
      \prob_{n}(M_{j-1} \setintersect M_{j}^{c}, \setweight{S_{j}(v)}\leq
      k_{n})\notag\\
    &\leq
      \expe_{n}\setweight{S_{\ell}(v)}_{2}]
        \frac{\Gamma_{2,n}}{n\vartheta}
      +\expe_{n}[\setweight{S_{\ell}(v)}_{+}]
         \Gamma_{1,n}
     + \expe_{n}[\setweight{S_{\ell}(v)}]
      \Bigl(\kappa_{1,n}+\frac{1}{k_{n}}+\frac{k_{n}}{n\vartheta}\Bigr).
  \end{align*}
  This concludes the proof.
\end{proof}

The coupling in \autoref{prop:blttl} holds for a single~\(\ell\)-neighbourhood,
but we would like to be able to couple the neighbourhoods
of several distinct vertices to independent intermediate trees.
The neighbourhoods~\(B_{\ell}(v)\)
of~\(m = \setcard{\mathcal{V}}\) distinct vertices~\(v \in \mathcal{V}\)
to trees~\(T_{\ell}(v)\) are not guaranteed
to be independent,
because the same vertex and with it the same edges
may appear in several neighbourhoods.
By construction, the same holds for the coupled intermediate trees,
since the edges in those trees are coupled to the original edges
in~\(G_{n}\).
In \autoref{sec:corr} we have, however, already
shown that the~\(\ell\)-neighbourhoods are
asymptotically only weakly correlated.
In much the same vein we can show that the intermediate trees
can be altered (with sufficiently small probability)
to make them independent.
This then implies that we may assume that the coupled intermediate trees
are independent at only a small additional penalty to the
coupling probability.

\begin{prop}\label{prop:manycouplings}
  Let~\(\mathcal{V} \subseteq V_{n}\) be a set of vertices
  from~\(G_{n} = (V_{n},E_{n})\).
  Then for all~\(\ell \in \naturals\)
  the neighbourhoods~\(B_{\ell}(v)\) around~\(v \in \mathcal{V}\)
  can be coupled to independent intermediate trees~\(\imt{T}(v)\)
  such that
  \begin{align*}
    &\prob_{n}\biggl(
      \bigsetunion_{v \in \mathcal{V}} \set[\big]{
      B_{\ell}(v) \ncong \imt{T}_{\ell}(v)
      }
      \biggr)\\
    &\leq \setweight{\mathcal{V}}_{2} \frac{\Gamma_{2,n}}{n\vartheta}
    +\setweight{\mathcal{V}}(\Gamma_{2,n}+1)^{\ell}
    \biggl(
      \frac{\Gamma_{3,n}}{n\vartheta}+\kappa_{1,n}+\kappa_{2,n}+
      \frac{2+\Gamma_{1,n}}{k_{n}}+\frac{k_{n}}{n\vartheta}\biggr)
  +\setweight{V}_{+}\Gamma_{1,n} + \setcard{\mathcal{V}}\frac{1}{k_{n}}
    +\frac{k_{n}^{2}}{n\vartheta\Gamma_{1,n}}
  \end{align*}
  for all sequences~\((k_{n})_{n \in\naturals} \subseteq \intervaloo{0}{\infty}\).
\end{prop}

\begin{proof}

Let~\(\imt{S}_{\ell}(v)\) the set of individuals
in the intermediate tree~\(\imt{T}(v)\) up to level~\(\ell\)
and~\(\imt{D}_{r}(v) = \imt{S}_{r}(v)\setminus \imt{S}_{r-1}(v)\).
As in \autoref{sec:nbhdsize} we calculate the expected number of individuals
in~\(\imt{S}_{\ell}(v)\) and their total weight.

Note that in the tree
\[
  \expe_{n}[\imt{N}_{\treeroot}] = \frac{W_{v} \Lambda_{n}}{n\vartheta}
  \leq W_{v}\Gamma_{1,n}
\]
and for~\(\mathbf{i} \neq \treeroot\)
\[
\expe_{n}[\imt{N}_{\mathbf{i}}]
= \sum_{i=1}^{n} \frac{W_{i}}{\Lambda_{n}}
\expe_{n}[D_{i}]
= \sum_{i=1}^{n} \frac{W_{i}}{\Lambda_{n}}
\frac{\Lambda_{n} W_{i}}{n\vartheta}
\leq \frac{\expe_{n}[(W^{(n)})^{2}]}{\vartheta}
= \Gamma_{2,n},
\]
where~\(D_{i} \sim \mathrm{Poi}(\Lambda_{n}W_{i}/(\vartheta
n))\)
given~\(W_{i}\).
Then by construction
\[
\setcard{\imt{D}_{0}(v)} = 1
\quad\text{and}\quad
\expe_{n}[\setcard{\imt{D}_{1}(v)}] = \expe_{n}[\imt{N}_{\treeroot}] \leq
W_{v}\Gamma_{1,n}
\]
and furthermore by standard arguments for Galton--Watson trees
\[
\expe_{n}[\setcard{\imt{D}_{r}(v)}] \leq W_{v} \Gamma_{1,n}\Gamma_{2,n}^{r-1}.
\]
Hence,
\begin{equation}\label{eq:imt:l:size}
\expe_{n}[\setcard{\imt{S}_{\ell}(v)}] \leq 1+
W_{v}\Gamma_{1,n}(\Gamma_{2,n}+1)^{\ell-1},
\end{equation}
which coincides with the bound \autoref{lem:bl:weight:raw}
for the analogous quantity in~\(G_{n}\).

Similarly, by a standard argument for multi-type Galton--Watson
processes
\begin{equation}\label{eq:imt:l:weight}
\expe_{n}[\setweight{\imt{S}_{\ell}(v)}] \leq
W_{v}(\Gamma_{2,n}+1)^{\ell},
\end{equation}
which coincides with the bound from \autoref{lem:bl:weight:raw}
for~\(G_{n}\).

For a fixed set of vertices~\(\mathcal{V}\)
we can use \autoref{prop:blttl} to couple the neighbourhood
of~\(v\) in~\(G_{n}\) to an intermediate tree~\(\imt{T}'(v)\)
up to level~\(\ell\)
of each~\(v \in \mathcal{V}\).
Together these couplings satisfy
\begin{align*}
     &\prob_{n}\biggl(\bigsetunion_{v \in \mathcal{V}} \set[]{B_{\ell}(v)
       \ncong \imt{T}'_{\ell}(v)}\biggr)\notag\\
      &\quad
        \leq \sum_{v \in \mathcal{V}} \Bigl(
        \expe_{n}[\setweight{S_{\ell}(v)}_{2}]
        \frac{\Gamma_{2,n}}{n\vartheta}
        + \expe_{n}[\setweight{S_{\ell}(v)}_{+}]\Gamma_{1,n}
     + \expe_{n}[\setweight{S_{\ell}(v)}]
        \Bigl(\kappa_{1,n}+\frac{1}{k_{n}}+\frac{k_{n}}{n\vartheta}\Bigr)\Bigr).
  \end{align*}
  Use \autoref{lem:bl:weight:raw} to bound the
  expectations of~\(\setweight{S_{\ell}(v)}\) and~\(\setweight{S_{\ell}(v)}_{2}\)
  and \autoref{lem:bl:excessweight:raw} to bound
  the expectation of~\(\setweight{S_{\ell}(v)}_{+}\)
  to further estimate the term by
  \begin{align}
         &\prob_{n}\biggl(\bigsetunion_{v \in \mathcal{V}} \set[]{B_{\ell}(v)
       \ncong \imt{T}'_{\ell}(v)}\biggr)\notag\\
      &\quad\leq
      \sum_{v \in \mathcal{V}} (W_{v}^{2}+W_{v}(\Gamma_{2,n}+1)^{\ell-1}\Gamma_{3,n}
      )
      \frac{\Gamma_{2,n}}{n\vartheta}
      +\sum_{v \in \mathcal{V}} W_{v}\indfunc_{\set{W_{v} > \sqrt{n\vartheta}}}
      +W_{v}(\Gamma_{2,n}+1)^{\ell-1}\kappa_{2,n}\notag\\
     &\qquad +
      \sum_{v \in \mathcal{V}} W_{v}(\Gamma_{2,n}+1)^{\ell}
      \Bigl(\kappa_{1,n}+\frac{1}{k_{n}}+\frac{k_{n}}{n\vartheta}\Bigr)\notag\\
     &\quad\leq \setweight{\mathcal{V}}_{2} \frac{\Gamma_{2,n}}{n\vartheta}
       +\setweight{\mathcal{V}}_{+} \Gamma_{1,n}
     + \setweight{\mathcal{V}} (\Gamma_{2,n}+1)^{\ell}\Bigl(
      \frac{\Gamma_{3,n}}{n\vartheta}+\kappa_{1,n}+\kappa_{2,n}+
      \frac{1}{k_{n}}+\frac{k_{n}}{n\vartheta}\Bigr).
      \label{eq:manytreescoup}
\end{align}

The definition of these coupled trees does not ensure that
the trees are independent,
because the same (coupled) individual may appear in several trees.
But given a family of intermediate trees the following procedure
can generate independent trees~\((\imt{T}_{\ell}(v))_{v \in
\mathcal{V}}\).

We construct level~\(r \in \set{0,\dots,\ell}\)
of all trees~\((\imt{T}_{\ell}(v))_{v \in \mathcal{V}}\) in the same step.
During our procedure we need to keep track of the individuals that
we have seen so far.

Start by setting level~\(0\) of each tree~\(\imt{T}_{\ell}(v)\)
to just~\(v\).

Assuming that we have already explored all vertices at level~\(r-1\)
in all trees,
we now use breadth-first search to completely explore level~\(r\)
of each~\(\imt{T}'_{\ell}(v)\).
Whenever we encounter an individual in~\(\imt{T}'_{\ell}(v)\)
that has not been seen before,
it is copied over to the appropriate~\(\imt{T}_{\ell}(v)\)
and added to the set of individuals that have been seen.
If the individual has been seen before,
an independent Galton--Watson tree
of appropriate height
with offspring distribution~\(\sizebias{\nu}\)
is added in its position to~\(\imt{T}_{\ell}(v)\).

This procedure terminates with independent intermediate trees
and the probability that~\(\imt{T}'_{\ell}(v) \neq
\imt{T}_{\ell}(v)\)
for any~\(v \in \mathcal{V}\) can be estimated by the probability
that any individual added during the process was seen before.
Let~\(\imt{S}(v)\)
be the set of individuals in the intermediate
tree~\(\imt{T}'_{\ell}(v)\) for~\(v \in \mathcal{V}\)
and set~\(\imt{S}_{\ell}(\mathcal{V}) = \bigsetunion_{v \in \mathcal{V}}
\imt{S}(v)\).
Then~\(\setcard{\imt{S}_{\ell}(\mathcal{V})}\) is the total number of individuals
in all trees and~\(\setweight{\imt{S}_{\ell}(\mathcal{V})}\) their total
connectivity weight.
The type of a non-root individual has distribution~\(\sizebias{\nu}\).
Hence, the probability that during the breadth-first search
a particular non-root individual has a type
that has been seen before is bounded above
by (cf.~\eqref{eq:nunnunhat})
\[
\sizebias{\nu}_{n}(
\sset{W_{i}}
{i \in \imt{S}_{\ell}(\mathcal{V})}
)
\leq \frac{\setweight[\big]{\imt{S}_{\ell}(\mathcal{V})}}{\Lambda_{n}}.
\]
If we have that~\(\setcard{\imt{S}_{\ell}(\mathcal{V})} \leq k_{n}\)
and~\(\setweight{\imt{S}_{\ell}(\mathcal{V})} \leq k_{n}\)
then
the number of vertices whose type was already seen is dominated by
a binomial distribution with parameters~\(k_{n}\)
and~\(k_{n}\Lambda_{n}^{-1}\).
Let
\(
Z \sim \mathrm{Bin}(k_{n}, k_{n}\Lambda_{n}^{-1})
\)
such that by Markov's inequality
\[
\prob_{n}(Z \geq 1) \leq \frac{k_{n}^{2}}{\Lambda_{n}}.
\]
With~\eqref{eq:imt:l:weight}, \eqref{eq:imt:l:size}
and~\(\sum_{v\in\mathcal{V}} W_{v} = \setweight{\mathcal{V}}\)
it follows that
\begin{align}
  &\prob_{n}\Bigl(\bigsetunion_{v \in \mathcal{V}}
    \set[\big]{\imt{T}'_{\ell}(v)
    \ncong \imt{T}_{\ell}(v) }\Bigr)\notag\\
  &\quad\leq
    \prob_{n}(Z \geq 1, \setcard{\imt{S}_{\ell}(\mathcal{V})}\leq k_{n},
    \setweight{\imt{S}_{\ell}(\mathcal{V})} \leq k_{n})
    +\prob_{n}(\setcard{\imt{S}_{\ell}(\mathcal{V})} > k_{n})
    +\prob_{n}(\setweight{\imt{S}_{\ell}(\mathcal{V})} > k_{n})\notag\\
  &\quad\leq \frac{k_{n}^{2}}{\Lambda_{n}}
    + \frac{\sum_{v \in \mathcal{V}}\expe_{n}[
    \setcard{\imt{S}_{\ell}(v)}]}{k_{n}}
    + \frac{\sum_{v \in \mathcal{V}} \expe_{n}[
    \setweight{\imt{S}_{\ell}(v)}]}{k_{n}}.\notag\\
   &\quad\leq
    \frac{k_{n}^{2}}{\Lambda_{n}}
    + \frac{\setcard{\mathcal{V}}
      +\setweight{\mathcal{V}}\Gamma_{1,n}(\Gamma_{2,n}+1)^{\ell-1}
    +\setweight{\mathcal{V}}(\Gamma_{2,n}+1)^{\ell}}{k_{n}}
    \label{eq:manytreesindep}
\end{align}

This shows that we can turn a family of intermediate
trees~\((\imt{T}'(v))_{v \in \mathcal{V}}\)
into independent intermediate trees~\((\imt{T}(v))_{v \in \mathcal{V}}\)
at small cost by replacing a repeated individual and all its descendants
by independent draws from~\(\sizebias{\nu}_{n}\).

In particular we can switch the trees~\(\imt{T}'(v)\)
coupled to the neighbourhoods
of~\(v\) to independent trees~\(\imt{T}(v)\).
By~\eqref{eq:manytreescoup} and~\eqref{eq:manytreesindep}
the couplings between~\(B_{\ell}(v)\) and the independent~\(\imt{T}(v)\)
then satisfy
\begin{align*}
  \prob_{n}\Bigl(
    \bigsetunion_{v \in \mathcal{V}} \set[\big]{
    B_{\ell}(v) \ncong \imt{T}_{\ell}(v)
    }
    \Bigr)
  &\leq
    \prob_{n}\Bigl(\bigsetunion_{v \in \mathcal{V}} \set[\big]{B_{\ell}(v)
    \neq \imt{T}'_{\ell}(v)}\Bigr)
    +\prob_{n}\Bigl(\bigsetunion_{v \in \mathcal{V}}
    \set[\big]{\imt{T}'_{\ell}(v)
    \ncong \imt{T}_{\ell}(v) }\Bigr)\\
  &\leq
    \setweight{\mathcal{V}}_{2} \frac{\Gamma_{2,n}}{n\vartheta}
    + \setweight{\mathcal{V}}_{+}\Gamma_{1,n}
    + \setweight{\mathcal{V}} (\Gamma_{2,n}+1)^{\ell}
    \Bigl(
      \frac{\Gamma_{3,n}}{n\vartheta}+\kappa_{1,n}+\kappa_{2,n}+
      \frac{1}{k_{n}}+\frac{k_{n}}{n\vartheta}\Bigr)\\
   &\quad +\frac{k_{n}^{2}}{\Lambda_{n}}
     + \frac{\setcard{\mathcal{V}}
     +\setweight{\mathcal{V}}\Gamma_{1,n}(\Gamma_{2,n}+1)^{\ell-1}
     +\setweight{\mathcal{V}}(\Gamma_{2,n}+1)^{\ell}}{k_{n}}.
\end{align*}
Recall that~\(\Lambda_{n} = n\vartheta\Gamma_{1,n}\),
then collect the terms for~\(\setweight{\mathcal{V}}_{2}\),
  \(\setweight{\mathcal{V}}\) and~\(\setcard{\mathcal{V}}\)
  and estimate terms very generously to obtain the claimed bound.
\end{proof}

The construction of the intermediate trees relies
heavily on the connectivity weights of vertices in~\(G_{n}\).
Since the empirical distribution of those connectivity weights converges
to a limiting distribution by assumption we now define
(limiting) trees that draw from this limiting distribution
and show that those trees can be coupled to
the intermediate trees.

\begin{definition}\label{def:tree}
  Fix a vertex~\(v \in V_{n}\)
  and define a tree~\(\mathcal{T}(v)\)
  via a sequence of random
  variables~\(\sset{(W_{\mathbf{i}},N_{\mathbf{i}})}
  {\mathbf{i} \in \ulamharris}\),
  where~\(W_{\mathbf{i}}\) is the type of individual~\(\mathbf{i}\)
  and~\(N_{\mathbf{i}}\) is its number of children.
  The distribution of~\(\sset{(\imt{W}_{\mathbf{i}},\imt{N}_{\mathbf{i}})}
  {\mathbf{i} \in \ulamharris}\)
  satisfies
  \begin{itemize}
  \item \(W_{\treeroot} = W_{v}\)
    and~\(N_{\treeroot} \sim \mathrm{Poi}(W_{v})\),
  \item all other (non-root) individuals~\(\mathbf{i} \neq \treeroot\)
    have independent types and numbers of
    children~\((W_{\mathbf{i}},N_{\mathbf{i}})\)
    with distribution
    \[
      \prob((W_{\mathbf{i}},N_{\mathbf{i}}) \in \placeholder)
      = \prob((\sizebias{W},N) \in \placeholder),
    \]
    where~\(\sizebias{W} \sim \sizebias{\nu}\)
    and~\(N \sim \mathrm{MPoi}(\sizebias{W})\).
  \end{itemize}

  The tree structure on~\(\mathcal{T}(v)\) is then obtained
  recursively from~\(\mathcal{A}_{0} = \set{\treeroot}\)
  and
  \[
    \mathcal{A}_{k}
    = \sset{(\mathbf{i},j)}
    {\mathbf{i} \in \mathcal{A}_{k-1}, 1 \leq j \leq N_{\mathbf{i}}}
    \quad\text{for~\(k \in \naturals\), \(k \geq 1\)}.
  \]
\end{definition}
Ignoring the root, which differs from all other individuals,
the structure of this tree is given by a single-type branching
process with a mixed Poisson offspring distribution.

Note also that a underlying structure of the tree~\(\mathcal{T}(v)\)
constructed as described in \autoref{def:tree}
has exactly the distribution~\(T(W_{v},\nu)\)
defined in \autoref{def:limtree:noweight}.

We now show that~\(\imt{T}(v)\)
can be coupled to~\(\mathcal{T}(v)\).
This coupling relies on the coupling between~\(\nu_{n}\)
and~\(\nu\) that can be obtained because
the Wasserstein distance between the two measures
is bounded by~\(\alpha_{n}\).
Furthermore, the Poisson random variables can easily be coupled
once the vertex attributes are known.

\begin{lemma}\label{lem:treecoup}
  Let~\((G_{n})_{n \in \naturals}\) be a sequence of rank-one inhomogeneous
  random graphs that satisfies \autoref{ass:coupling}.
  Fix any~\(n\in naturals\).
  Let~\(\mathcal{V} \subseteq V_{n}\) be a set of vertices
  from~\(G_{n} = (V_{n},E_{n})\).
  We can couple the intermediate trees~\((\imt{T}(v))_{v \in \mathcal{V}}\)
  to the Poisson trees~\((\mathcal{T}(v))_{v\in\mathcal{V}}\) defined
  as in \autoref{def:tree} such that
  \[
    \prob_{n}\Bigl(
    \bigsetunion_{v \in \mathcal{V}}
    \set[\big]{ \imt{T}_{\ell}(v) \ncong \mathcal{T}_{\ell}(v)   }
    \Bigr)
    \leq \setweight{\mathcal{V}}\alpha_{n}
    \Bigl(
      \frac{1}{\vartheta}
      + (\Gamma_{2}+1)^{\ell-1}
        \Bigl(\frac{\Gamma_{2,n}}{\vartheta\Gamma_{1,n}}+1\Bigr)
    \Bigr).
\]
\end{lemma}

\begin{proof}
  We couple the types and number of children for each individual separately.
  Since we can bound the expected number of individuals
  in the relevant trees, we can then give a bound
  for the probability that the trees have a different structure.

Clearly~\(\imt{W}_{\treeroot} = W_{\treeroot} = W_{v}\) by construction,
so the type at the root is the same in both trees.
We can couple the numbers of children of the root
such that
\begin{equation*}
  \prob_{n}(\imt{N}_{\treeroot} \neq N_{\treeroot})
  \leq \expe_{n}\Bigl[
    \abs[\Big]{\frac{\Lambda_{n}}{n\vartheta}-1}W_{v}
    \Bigr]
  = W_{v} \abs{\Gamma_{1,n}-1},
\end{equation*}
where we used~\(\Lambda_{n} = n\vartheta\Gamma_{1,n}\)
and~\(\mathcal{F}_{n}\)-measurability of all involved terms
in the last line.

For~\(\mathbf{i} \neq \treeroot\)
first couple~\(\imt{W}_{\mathbf{i}} \sim \sizebias{\nu}_{n}\)
to~\(W_{\mathbf{i}} \sim \sizebias{\nu}\)
optimally according to the Wasserstein distance,
i.e.
\[
  \expe_{n}[\abs{\imt{W}_{\mathbf{i}}-W_{\mathbf{i}}}]
  \leq \alpha_{n},
\]
where~\(\alpha_{n}\) is~\(\mathcal{F}_{n}\)-measurable
and~\(\alpha_{n} \overset{\prob}{\to} 0\).
Then couple~\(\imt{N}_{\mathbf{i}}\) and~\(N_{\mathbf{i}}\)
such that
\begin{equation*}
  \prob_{n}(\imt{N}_{\mathbf{i}} \neq N_{\mathbf{i}})
  \leq \expe_{n}\Bigl[
      \abs[\Big]{\frac{\Lambda_{n}}{n\vartheta}-1}
      \imt{W}_{\mathbf{i}}
    \Bigr]
    + \expe_{n}[\abs{\imt{W}_{\mathbf{i}}-W_{\mathbf{i}}}]\\
  \leq \abs{\Gamma_{1,n}-1} \expe_{n}[\imt{W}_{\mathbf{i}}]+\alpha_{n}.
\end{equation*}
Since
\[
  \expe_{n}[\imt{W}_{\mathbf{i}}]
  = \sum_{i=1}^{n} W_{i} \frac{W_{i}}{\Lambda_{n}}
  = \sum_{i=1}^{n} \frac{W_{i}^{2}}{n\vartheta\Gamma_{1,n}}
  = \frac{1}{\Gamma_{1,n}} \Gamma_{2,n}
\]
this implies
\begin{equation*}
  \prob_{n}(\imt{N}_{\mathbf{i}} \neq N_{\mathbf{i}})
  \leq \abs{\Gamma_{1,n}-1} \frac{\Gamma_{2,n}}{\Gamma_{1,n}}
  +\alpha_{n}.
\end{equation*}

Recall that~\(\Gamma_{1,n}=\vartheta^{-1} \expe_{n}[W^{(n)}]\)
and assume that~\(W^{(n)} \sim \nu_{n}\) is coupled
optimally to~\(W \sim \nu\) according to the Wasserstein distance.
Then we have~\(\expe_{n}[W]=\vartheta\) and
\[
  \abs{\Gamma_{1,n}-1}
  = \vartheta^{-1} \abs{\expe_{n}[W^{(n)}]-\expe_{n}[W]}
  \leq \vartheta^{-1} \expe_{n}[\abs{W^{(n)}-W}]
  \leq \vartheta^{-1} \alpha_{n}.
\]

The tree structure of~\(\imt{T}(v)\)
and~\(\mathcal{T}(v)\)
is determined only by~\(\imt{N}_{\mathbf{i}}\)
and~\(N_{\mathbf{i}}\),
respectively.
In particular~\(\imt{T}_{\ell}(v)\) and~\(\mathcal{T}_{\ell}(v)\)
can only disagree if there is an individual~\(\mathbf{i}\)
at level~\(\ell-1\)
in~\(\mathcal{T}_{\ell}(v)\) whose number of children~\(N_{\mathbf{i}}\)
is different from~\(\imt{N}_{\mathbf{i}}\).
If we can control the number of vertices in~\(\mathcal{T}_{\ell}(v)\),
a simple union bound
and the fact that the distribution of~\(\imt{N}_{\mathbf{i}}\)
  and~\(N_{\mathbf{i}}\) is the same for all~\(\mathbf{i} \neq \treeroot\)
can be used to bound the probability
that the coupling generates different tree structures.
As in the proof of \autoref{prop:blttl} we will work conditionally
on the previous level.

Let~\(\mathcal{S}_{r}(v)\) be the set of individuals
in~\(\mathcal{T}_{r}(v)\)
and~\(\mathcal{D}_{r}(v) = \mathcal{S}_{r}(v) \setminus \mathcal{S}_{r-1}(v)\)
the individuals at level~\(r \in \naturals\).
Let~\(\mathcal{G}_{\ell}\) be the~\(\sigma\)-algebra generated
by~\(\imt{W}_{\mathbf{i}}, W_{\mathbf{i}}, \imt{N}_{\mathbf{i}},
N_{\mathbf{i}}\) up to level~\(\ell\).

For~\(\ell=1\)
we have
\[
  \prob_{n}(\imt{T}_{1}(v) \ncong \mathcal{T}_{1}(v))
  \leq \prob_{n}(\imt{N}_{\treeroot} \neq N_{\treeroot})
  \leq W_{v}\abs{\Gamma_{1,n}-1}
  \leq W_{v} \vartheta^{-1}\alpha_{n}.
\]
For~\(\ell \geq 2\) the coupling breaks if
any of the individuals at level~\(\ell-1\)
creates a different number of children.
Since all individuals apart from~\(\treeroot\)
have the same distribution that is furthermore
independent of other individuals
(in particular independent from those in a lower level),
that probability of generating different tree structures via this
coupling is bounded by
\begin{equation*}
  \indfunc_{\set{\imt{T}_{\ell-1}(v) \cong \mathcal{T}_{\ell-1}(v)}}
    \prob_{n}(\imt{T}_{\ell}(v) \ncong \mathcal{T}_{\ell}(v)
    \given \mathcal{G}_{\ell-1})
  \leq
    \sum_{\mathbf{i} \in \mathcal{D}_{\ell}(v)}
    \prob_{n}(\imt{N}_{\mathbf{i}} \neq N_{\mathbf{i}})
  \leq \setcard{\mathcal{D}_{\ell-1}(v)}
    \alpha_{n} \biggl(\frac{\Gamma_{2,n}}{\vartheta \Gamma_{1,n}}+1\biggr).
\end{equation*}

Now sum over the probabilities that the tree structure is different
for the first time at a specific level to find
\begin{align*}
  \prob_{n}(\imt{T}_{\ell}(v) \ncong \mathcal{T}_{\ell}(v))
  &\leq \prob_{n}(\imt{T}_{1}(v) \ncong \mathcal{T}_{1}(v))
  + \sum_{r=2}^{\ell}\expe_{n}\bigl[
    \indfunc_{\set{\imt{T}_{r-1}(v) \cong \mathcal{T}_{r-1}(v)}}
    \prob_{n}(\imt{T}_{r}(v) \ncong \mathcal{T}_{r}(v)
    \given \mathcal{G}_{r-1})
    \bigr]\\
  &\leq W_{v}\vartheta^{-1} \alpha_{n}
    + \expe_{n}[\setcard{S_{\ell-1}(v)}-1]\alpha_{n}
    \Bigl(\frac{\Gamma_{2,n}}{\vartheta \Gamma_{1,n}}+1\Bigr)\\
  &\leq W_{v}\vartheta^{-1}\alpha_{n}
    + W_{v}(\Gamma_{2}+1)^{\ell-1}
    \alpha_{n}\Bigl(\frac{\Gamma_{2,n}}{\vartheta \Gamma_{1,n}}+1\Bigr).
\end{align*}
In the second to last step we used
that~\(\expe[N_{\treeroot}] = W_{v}\)
and~\(\expe[N_{\mathbf{i}}] = \expe[\sizebias{W}]
= \expe[W^{2}]/\expe[W] = \Gamma_{2}\)
for~\(\mathbf{i} \neq \treeroot\)
to conclude
\[
  \expe[\setcard{\mathcal{S}_{\ell}(v)}]
  \leq 1+W_{v}(\Gamma_{2}+1)^{\ell}.
\]

Sum these bounds over~\(v \in \mathcal{V}\) to finish the proof.
\end{proof}

With this lemma it is now possible to couple the neighbourhoods
to the limiting trees
as claimed in \autoref{prop:maincoup}.
We restate the proposition in the notation of this section.
\begin{repprop}{prop:maincoup}
  Let~\((G_{n})_{n \in \naturals}\) be a sequence of
  rank-one inhomogeneous random graphs satisfying \autoref{ass:coupling}.
  Fix any~\(n \in \naturals\).
  Let~\(\mathcal{V} \subseteq V_{n}\) be a subset of vertices
  of~\(G_{n}  = (V_{n},E_{n})\).
  Then for all~\(\ell \in \naturals\)
  the neighbourhoods around~\(B_{\ell}(v)\)
  can be coupled to independent limiting trees~\(\mathcal{T}(v)\)
  as defined in \autoref{def:tree}
  such that
  \begin{align*}
    &\prob_{n}\Bigl(
    \bigsetunion_{v \in \mathcal{V}} \set[\big]{
    B_{\ell}(v) \ncong \mathcal{T}_{\ell}(v)
    }
    \Bigr)\\
  &\quad\leq
    \setweight{\mathcal{V}}_{2} \frac{\Gamma_{2,n}}{n\vartheta}
    +\setweight{\mathcal{V}}_{+}\Gamma_{1,n}
    +\setweight{\mathcal{V}}(\Gamma_{2,n}+1)^{\ell}
    \Bigl(
      \frac{\Gamma_{3,n}}{n\vartheta}+\kappa_{1,n}+\kappa_{2,n}+
      \frac{2+\Gamma_{1,n}}{k_{n}}+\frac{k_{n}}{n\vartheta}\Bigr)\\
  &\qquad + \setcard{\mathcal{V}}\frac{1}{k_{n}}
    +\frac{k_{n}^{2}}{n\vartheta\Gamma_{1,n}}
    +\setweight{\mathcal{V}}\alpha_{n}
    \Bigl(
      \frac{1}{\vartheta}
      + (\Gamma_{2}+1)^{\ell-1}
        \Bigl(\frac{\Gamma_{2,n}}{\vartheta\Gamma_{1,n}}+1\Bigr)
    \Bigr)
  \end{align*}
  for all sequences~\((k_{n})_{n \in\naturals} \subseteq \intervaloo{0}{\infty}\).
\end{repprop}

\begin{proof}
  This is a direct consequence of \autoref{prop:manycouplings}
  and \autoref{lem:treecoup}.
\end{proof}

\subsection{More complex couplings}
\label{sec:compcoup}

This section collects a few additional more specialised coupling results
for the weighted inhomogeneous random graph~\(\mathbf{G}_{n}\)
that follow directly from \autoref{prop:maincoup}
for the unweighted model~\(G_{n}\).
These results are extensions of lemmas shown by \citet[§~6]{cao}
for Erdős--Rényi graphs.

The following is just a reformulation of \autoref{prop:maincoup}.
\begin{lemma}\label{lem:nbhdcoupl:noweights}
  Fix~\(\ell \in \naturals\) and let~\(\mathcal{V} \subseteq V_{n}\)
  be a set of vertices vertices.
  Then there is a
  coupling~\(((B_{\ell}(v,G_{n}),
  \mathcal{T}_{\ell}(v))_{v \in \mathcal{V}})\)
  such
  that the~\(\mathcal{T}_{\ell}(v) \sim T_{\ell}(W_{v},\nu)\)
  are independent limiting trees
  with
  \begin{equation*}
  \prob_{n}(\text{\(B_{\ell}(v,G_{n}) \cong \mathcal{T}_{\ell}(v)\)
    for all~\(v \in \mathcal{V}\)})
  \geq 1-\eta_{n,\ell}(\mathcal{V}),
  \end{equation*}
  where
  \[
    \begin{split}
    \eta_{n,\ell}(\mathcal{V})
      &=
    \setweight{\mathcal{V}}_{2} \frac{\Gamma_{2,n}}{n\vartheta}
    +\setweight{\mathcal{V}}_{+}\Gamma_{1,n}
    +\setweight{\mathcal{V}}(\Gamma_{2,n}+1)^{\ell}
    \Bigl(
      \frac{\Gamma_{3,n}}{n\vartheta}+\kappa_{1,n}+\kappa_{2,n}+
      \frac{2+\Gamma_{1,n}}{k_{n}}+\frac{k_{n}}{n\vartheta}\Bigr)\\
  &\quad + \setcard{\mathcal{V}}\frac{1}{k_{n}}
    +\frac{k_{n}^{2}}{n\vartheta\Gamma_{1,n}}
    +\setweight{\mathcal{V}}\alpha_{n}
    \Bigl(
      \frac{1}{\vartheta}
      + (\Gamma_{2}+1)^{\ell-1}
        \Bigl(\frac{\Gamma_{2,n}}{\vartheta\Gamma_{1,n}}+1\Bigr)
    \Bigr).
  \end{split}
  \]
\end{lemma}

\begin{lemma}\label{lem:bvbucoupl}
  Fix~\(\ell \in \naturals\) and let~\(\mathcal{V} \subseteq V_{n}\)
  be a set of vertices vertices.
  Then there is a
  coupling~\(((B_{\ell}(v,\mathbf{G}_{n}),
  \mathbf{T}_{\ell}(v))_{v \in \mathcal{V}})\)
  such
  that the~\(\mathbf{T}_{\ell}(v)\)
  are independent with
  distribution~\(\mathbf{T}_{\ell}(W_{v},\nu,\mu_{E},\mu_{V})\)
  and
  \begin{equation*}
    \prob_{n}(\text{\(B_{\ell}(v,\mathbf{G}_{n})
        \cong \mathbf{T}_{\ell}(v)\)
        for all~\(v \in \mathcal{V}\)})
      \geq 1-\varepsilon_{n,\ell}(\mathcal{V}),
  \end{equation*}
  where
  \[
    \varepsilon_{n,\ell}(\mathcal{V})
    = \eta_{n,\ell}(\mathcal{V})
    +  (\setcard{\mathcal{V}}+\setweight{\mathcal{V}}(\Gamma_{2}+1)^{\ell})
    (d_{\mathrm{TV}}(\mu_{E,n},\mu_{E})
    + d_{\mathrm{TV}}(\mu_{V,n},\mu_{V})
    )
  \]
  with~\(\eta_{n,\ell}(\mathcal{V})\) as in \autoref{lem:nbhdcoupl:noweights}.
\end{lemma}
\begin{proof}
  Apply \autoref{lem:nbhdcoupl:noweights} to
  couple~\(((B_{\ell}(v,G_{n}),
  \mathcal{T}_{\ell}(v))_{v \in \mathcal{V}})\)
  such that the~\(\mathcal{T}_{\ell}(v)\) are independent
  with distribution~\(T_{\ell}(W_{v},\nu)\)
  and
  \[
  \prob_{n}(A)
  =
  \prob_{n}(\text{\(B_{\ell}(v,G_{n}) \cong \mathcal{T}_{\ell}(v)\)
    for all~\(v \in \mathcal{V}\)})
  \geq 1-\eta_{n,\ell}(\mathcal{V}),
  \]
  where~\(A\) is the event that~\(B_{\ell}(v,G_{n}) \cong \mathcal{T}_{\ell}(v)\)
  for all~\(v \in \mathcal{V}\).
  This provides the coupling of the underlying graph structure.
  It remains to also couple the edge and vertex weights.

  For each edge~\(e \in V_{n}^{(2)}\) couple~\(w_{e}\),
  the weight in~\(\mathbf{G}_{n}\),
  to~\(\tilde{w}_{e}\),
  such that~\(\prob_{n}(w_{e} \neq \tilde{w}_{e}) =
  d_{\mathrm{TV}}(\mu_{E,n},\mu_{E})\).
  Similarly, for each each vertex~\(v \in V_{n}\) introduce a
  coupling~\((w_{v},\tilde{w}_{v})\) such
  that~\(\prob_{n}(w_{v} \neq \tilde{w}_{v}) = d_{\mathrm{TV}}(\mu_{V,n},\mu_{V})\).

  Let~\(E^{E}\) be the event that there is an edge~\(e\)
  in any of the~\(B_{\ell}(v,\mathbf{G}_{n})\)
  such that~\(w_{e} \neq \tilde{w}_{e}\).
  Similarly let~\(E^{V}\) be the event that there is
  a vertex~\(u\) in any of the~\(B_{\ell}(v,\mathbf{G}_{n})\)
  such that~\(w_{v} \neq \tilde{w}_{v}\).
  On the event~\(A\) we have that~\(B_{\ell}(v,G_{n})\)
  and~\(\mathcal{T}_{\ell}(v)\) are isomorphic,
  so that we can reformulate~\(E^{E}\) and~\(E^{V}\)
  in terms of edges or vertices in~\(\mathcal{T}_{\ell}(v)\).
  Let~\(\mathcal{S}_{\ell}(v)\) be the set of vertices
  in~\(\mathcal{T}_{\ell}(v)\).
  Since a tree with~\(n\) vertices has~\(n-1\) edges
  and the~\(\mathcal{T}_{\ell}(v)\) are trees,
  the number of edges
  and vertices that are relevant for~\(E^{E}\) and~\(E^{V}\)
  can be bounded
  by~\(\sum_{v \in \mathcal{V}} \setcard{\mathcal{S}_{\ell}(v)}\).

  In particular the same calculation
  as in \autoref{lem:treecoup} implies
  \begin{align*}
  \prob_{n}(A \setintersect E^{E})
    \leq
      \sum_{v \in \mathcal{V}} \expe_{n}[\setcard{\mathcal{S}_{\ell}(v)}]
      d_{\mathrm{TV}}(\mu_{E,n},\mu_{E})
     \leq
      (\setcard{\mathcal{V}}+\setweight{\mathcal{V}}(\Gamma_{2}+1)^{\ell})
      d_{\mathrm{TV}}(\mu_{E,n},\mu_{E})
  \shortintertext{and}
  \prob_{n}(A  \setintersect E^{V})
  \leq
    \sum_{v \in \mathcal{V}} \expe_{n}[\setcard{\mathcal{S}_{\ell}(v)}]
    d_{\mathrm{TV}}(\mu_{V,n},\mu_{V})
  \leq
    (\setcard{\mathcal{V}}+\setweight{\mathcal{V}}(\Gamma_{2}+1)^{\ell})
    d_{\mathrm{TV}}(\mu_{V,n},\mu_{V}).
  \end{align*}

  On the set~\(A\) couple~\(B_{\ell}(v,\mathbf{G}_{n})\)
  to~\(\mathbf{T}_{\ell}(v)\)
  by assigning edge weight~\(\tilde{w}_{e}\)
  to the edge isomorphic to~\(e\)
  and vertex weight~\(\tilde{w}_{v}\) to the vertex isomorphic to~\(v\)
  in~\(\mathcal{T}_{\ell}(v)\)
  resulting
  in~\(\mathbf{T}_{\ell}(v) \sim
  \mathbf{T}_{\ell}(W_{v},\nu,\mu_{E},\mu_{V})\).
  This coupling satisfies
  \begin{align*}
  &\prob_{n}(\text{%
    \(B_{\ell}(v,\mathbf{G}_{n}) \cong \mathbf{T}_{\ell}(v)\)
     for all~\(v \in \mathcal{V}\)})\\
  &\quad\geq \prob_{n}(A)-\prob_{n}(A \setintersect (E^{E} \setunion E^{V}))\\
  &\quad\geq 1-\eta_{n,\ell}(\mathcal{V})
  -(\setcard{\mathcal{V}}+\setweight{\mathcal{V}}(\Gamma_{2}+1)^{\ell})
    (d_{\mathrm{TV}}(\mu_{E,n},\mu_{E})+d_{\mathrm{TV}}(\mu_{V,n},\mu_{V})).
  \end{align*}
  This proves the claim.
\end{proof}

Indeed the coupled tree structure of the neighbourhood of a vertex~\(v\)
can be manipulated slightly to be independent of
a vertex~\(e'=\edge{u'}{v'}\)
which does not emanate from~\(v\).
In essence we rerandomise the relevant edge and
bound the probability that it actually occurs in the neighbourhood.
\begin{lemma}\label{lem:bkvindepofep}
  Let~\(v \in V_{n}\) be a vertex in~\(\mathbf{G}_{n}\)
  and let~\(e'=\edge{u'}{v'} \in V_{n}^{(2)}\)
  be an edge with endpoints distinct from~\(v\).
  Given a coupling
  \[
  (
  B_{\ell}(v,\mathbf{G}_{n}),
  \mathbf{T}_{\ell}(v)
  )
  \]
  of~\(B_{\ell}(v,\mathbf{G}_{n})\)
  with~\(\mathbf{T}_{\ell}(v) \sim \mathbf{T}_{\ell}(W_{v},\nu,\mu_{E},\mu_{V})\)
  that satisfies
  \[
  \varepsilon_{n,\ell}(\set{v})
  \geq 1-\prob_{n}(
  B_{\ell}(v,\mathbf{G}_{n}) \cong \mathbf{T}_{\ell}(v)
  )
  \]
  it is is possible to
  couple~\((B_{\ell}(v,\mathbf{G}_{n}), \mathbf{\widetilde{T}}_{\ell}(v))\)
  such that~\(\mathbf{\widetilde{T}}_{\ell}(v)\) is independent
  of~\(Y_{e'}\),
  where~\(Y_{e'} = (X_{e'},X'_{e'})\),
  and
  \[
  \prob_{n}(
  B_{\ell}(v,\mathbf{G}_{n})
  \cong \mathbf{\widetilde{T}}_{\ell}(v))
  \given
  Y_{e'}
  )\geq
  1-\Bigl(
  \varepsilon_{n,\ell}( \set{v})
  + C\frac{W_{v}(W_{v'}+W_{u'})}{n\vartheta}
  (\Gamma_{2,n}+1)^{\ell+1}
  \Bigr).
  \]
\end{lemma}

\begin{proof}
  Let~\(X''\) be a copy of~\(X\)
  that is independent of everything else,
  in particular independent of~\((X,W)\) and~\((X',W')\).
  Let~\(\mathbf{G}_{n}''\) be the weighted graph obtained
  from~\(\mathbf{G}_{n}\) by replacing~\(X_{e'}\)
  with~\(X''_{e'}\).
  Based on the initial coupling,
  couple~\((B_{\ell}(v,\mathbf{G}_{n}''),\mathbf{\widetilde{T}}_{\ell}(v))\),
  where~\(\mathbf{\widetilde{T}}_{\ell}(v)
  \sim \mathbf{T}_{\ell}(W_{v},\nu,\mu_{E},\mu_{V})\).
  By construction~\(B_{\ell}(v,\mathbf{G}_{n}'')\)
  is independent of~\(Y_{e'}\),
  so we may pick this coupling in a way
  such that~\(\mathbf{\widetilde{T}}_{\ell}(v)\)
  is independent of~\(Y_{e'}\) as well.

  Moreover,
  \[
  \prob_{n}(B_{\ell}(v,\mathbf{G}_{n}) \cong \mathbf{\widetilde{T}}_{\ell}(v)
  \given Y_{e'})
  \geq
  1-
  (
  \prob_{n}(B_{\ell}(v,\mathbf{G}_{n}) \notcong B_{\ell}(v,\mathbf{G}''_{n})
  \given Y_{e'})
  + \prob_{n}(B_{\ell}(v,\mathbf{G}''_{n}) \notcong
  \mathbf{\widetilde{T}}_{\ell}(v) \given Y_{e'})
  )
  \]
  By construction~\(B_{\ell}(v,\mathbf{G}_{n}'')\)
  and~\(\mathbf{\widetilde{T}}_{\ell}(v)\)
  are independent of~\(Y_{e'}\),
  so by assumption
  \begin{equation*}
  \prob_{n}(B_{\ell}(v,\mathbf{G}_{n}'') \notcong
  \mathbf{\widetilde{T}}_{\ell}(v)
  \given Y_{e'})
  =\prob_{n}(B_{\ell}(v,\mathbf{G}_{n}'') \notcong
  \mathbf{\widetilde{T}}_{\ell}(v))
  \leq \varepsilon_{n,\ell}( \set{v}).
  \end{equation*}
  Finally,~\(\mathbf{G}_{n}''\) is independent of~\(Y_{e'}\)
  and differs from~\(\mathbf{G}_{n}\) only in~\(e'\).
  Hence,~\(B_{\ell}(v,\mathbf{G}_{n}'')\)
  can differ from~\(B_{\ell}(v,\mathbf{G}_{n})\)
  only if~\(e'\) is present in
  one and not the other.
  Thus a rough estimate yields
  \[
  \prob_{n}(B_{\ell}(v,\mathbf{G}_{n}'') \ncong B_{\ell}(v,\mathbf{G}_{n})
  \given Y_{e'})
  \leq \prob_{n}(e' \in B_{\ell}(v,\mathbf{G}_{n}''))
  + \prob_{n}(e' \in B_{\ell}(v,\mathbf{G}_{n}) \given X_{e'}).
  \]
  By \autoref{lem:einbl} the first term can be bounded as follows:
  \begin{equation*}
  \prob_{n}(e' \in B_{\ell}(v,\mathbf{G}_{n}''))
  \leq \frac{W_{v}(W_{u'}+W_{v'})}{n\vartheta}
  (\Gamma_{2,n}+1)^{\ell}.
  \end{equation*}

  For the second probability note that
  since~\(e'\) can only be present in the neighbourhood
  if one of its vertices~\(u'\) or~\(v'\) is
  present in the neighbourhood in its own right,
  i.e.~when~\(e'\) itself is ignored,
  we have by \autoref{lem:uinbl} that
  \begin{align*}
  \prob_{n}(e' \in B_{\ell}(v,\mathbf{G}_{n}) \given X_{e'})
  &\leq \prob_{n}(v' \in B_{\ell}(v,\mathbf{G}_{n} - e'))
  + \prob_{n}(u' \in B_{\ell}(v,\mathbf{G}_{n} - e'))\\
  &\leq \frac{W_{v}(W_{v'}+W_{u'})}{n\vartheta}
  (\Gamma_{2,n}+1)^{\ell}.
  \end{align*}

  Together the last inequalities show the claim.
\end{proof}

Similarly we can find a coupling
for the effect of flipping the edge~\(e=\edge{v}{u}\) between~\(v\) and~\(u\)
and another edge~\(e' = \edge{v'}{u'}\)
that does not have any vertex in common with~\(e\)
on the neighbourhood of~\(v\).
\begin{lemma}\label{lem:couplegngne}
  Fix two vertices~\(v\) and~\(u\) and set~\(e = \edge{v}{u}\).
  Let~\(e' = \edge{v'}{u'}\) be another edge with vertices distinct
  from~\(u\) and~\(v\).
  Given a coupling
  \[
  (
  B_{\ell}(v,\mathbf{G}_{n}),
  B_{\ell}(u,\mathbf{G}_{n}),
  \mathbf{T}_{\ell}(v),
  \mathbf{T}_{\ell}(u)
  )
  \]
  with
  independent~\(\mathbf{T}_{\ell}(v)\sim
  \mathbf{T}_{\ell}(W_{v},\nu,\mu_{E},\mu_{V})\)
  and~\(\mathbf{T}_{\ell}(u)
  \sim \mathbf{T}_{\ell}(W_{u},\nu,\mu_{E},\mu_{V})\)
  that satisfies
  \[
  \varepsilon_{n,\ell}( \set{u,v})
  \geq 1-\prob_{n}(
  B_{\ell}(v,\mathbf{G}_{n}) \cong \mathbf{T}_{\ell}(v),
  B_{\ell}(u,\mathbf{G}_{n}) \cong \mathbf{T}_{\ell}(u)
  )
  \]
  it is possible to
  couple~\((B_{\ell}(v,\mathbf{G}_{n}), B_{\ell}(v,\mathbf{G}_{n}^{e}),
  \mathbf{\widetilde{T}}_{\ell}(v), \mathbf{\widetilde{T}}^{e}_{\ell}(v))\)
  such that
  \begin{align*}
  (\mathbf{\widetilde{T}}_{\ell}(v), \mathbf{\widetilde{T}}^{e}_{\ell}(v))
  \given (Y_{e} = (1,0), Y_{e'})
  &\eqdist (\mathbf{\tilde{T}}_{\ell},\mathbf{T}_{\ell}),\\
  (\mathbf{\widetilde{T}}_{\ell}(v), \mathbf{\widetilde{T}}^{e}_{\ell}(v))
  \given (Y_{e} = (0,1), Y_{e'})
  &\eqdist (\mathbf{T}_{\ell},\mathbf{\tilde{T}}_{\ell}),
  \end{align*}
  where~\(Y_{e} = (X_{e},X'_{e})\), \(Y_{e'} = (X_{e'},X'_{e'})\)
  and~\(\mathbf{T}_{\ell}
  \sim \mathbf{T}_{\ell}(W_{v},\nu,\mu_{E},\mu_{V})\),
  \(\mathbf{\tilde{T}}_{\ell}
  \sim \mathbf{\tilde{T}}_{\ell}(W_{v},W_{u},\nu,\mu_{E},\mu_{V})\).
  Furthermore,
  \[
  \begin{split}
  &\prob_{n}(
  (B_{\ell}(v,\mathbf{G}_{n}),B_{\ell}(v,\mathbf{G}_{n}^{e})
  \cong (\mathbf{\widetilde{T}}_{\ell}(v),
  \mathbf{\widetilde{T}}^{e}_{\ell}(v)))
  \given
  Y_{e},Y_{e'}
  )\\
  &\quad\geq
  1-\Bigl(
  \varepsilon_{n,\ell}( \set{u,v})
  + 2d_{\mathrm{TV}}(\mu_{E,n},\mu_{E})
  + C\frac{W_{u}W_{v}+(W_{u}+W_{v})(W_{u'}+W_{v'})}{n\vartheta}
  (\Gamma_{2,n}+1)^{2\ell}
  \Bigr).
  \end{split}
  \]
\end{lemma}

\begin{proof}
  Let~\((W'',X'')\) be a copy of~\((W,X)\)
  that is independent of everything else.
  Let~\(\mathbf{G}_{n}''\) be the weighted graph obtained
  from~\(\mathbf{G}_{n}\) by replacing~\(X_{e}\)
  with~\(X''_{e}\)
  and~\(X_{e'}\) with~\(X''_{e'}\).
  Based on the initial coupling,
  couple~\(B_{\ell}(v,\mathbf{G}_{n}'')\)
  with~\(\mathbf{T}''_{\ell}(v)\)
  and~\(B_{k-1}(u,\mathbf{G}_{n}'')\)
  with~\(\mathbf{T}''_{\ell-1}(u)\),
  where~\(\mathbf{T}''_{\ell}(v)\)
  and~\(\mathbf{T}''_{\ell-1}(u)\) are independent
  with
  distributions~\(\mathbf{T}''_{\ell}(v) \sim
  \mathbf{T}_{\ell}(W_{v},\nu,\mu_{E},\mu_{V})\)
  and~\(\mathbf{T}''_{\ell-1}(u) \sim
  \mathbf{T}_{\ell}(W_{u},\nu,\mu_{E},\mu_{V})\).
  By construction~\(B_{\ell}(v,\mathbf{G}_{n}'')\)
  and~\(B_{\ell-1}(u,\mathbf{G}_{n}'')\) are
  independent of~\(Y_{e},Y_{e'}\),
  so we may pick this coupling in a way
  that~\(\mathbf{T}''_{\ell}(v)\) and~\(\mathbf{T}''_{\ell-1}(v)\)
  are independent of~\(Y_{e},Y_{e'}\) as well.

  For brevity write~\(B_{\ell} = B_{\ell}(v,\mathbf{G}_{n})\)
  and~\(B_{\ell}' = B_{\ell}(v,\mathbf{G}_{n}^{e})\).
  These neighbourhoods can be constructed from the smaller
  neighbourhoods of~\(v\) and~\(u\)
  on~\(\mathbf{G}_{n}-e\) if we take into account~\(X_{e}\)
  and~\(w_{e}\)
  or~\(X'_{e}\) and~\(w'_{e}\) as required.
  That is to say there is a function~\(\Psi\) that describes the
  procedure of possibly \enquote{gluing together}
  the smaller neighbourhoods  (see \autoref{fig:glue}) such that
  \begin{align*}
  B_{\ell}
  &= \Psi(
  B_{\ell}(v,\mathbf{G}_{n}-e),B_{\ell-1}(u,\mathbf{G}_{n}-e),
  X_{e}, w_{e}
  )\\
  B'_{\ell}
  &= \Psi(
  B_{\ell}(v,\mathbf{G}_{n}-e),B_{\ell-1}(u,\mathbf{G}_{n}-e),
  X'_{e}, w'_{e}
  ).
  \end{align*}

  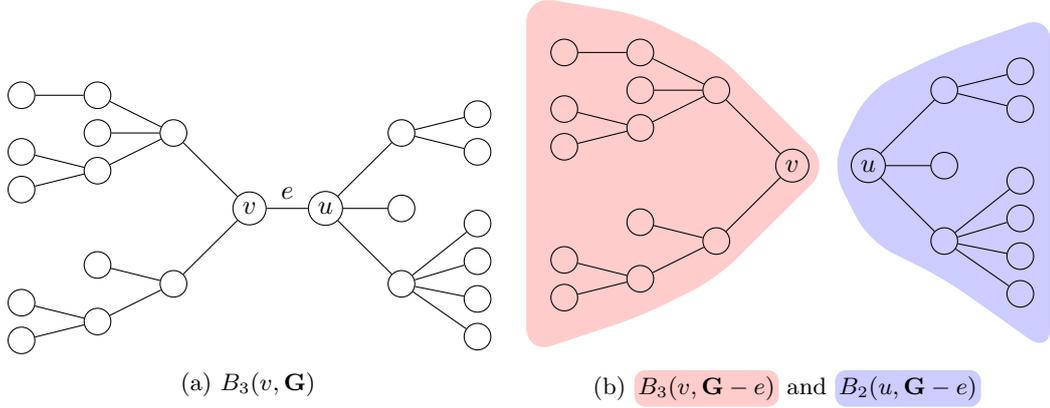
\begin{figure}[htbp]
    \centering
    \subcaptionbox{\(B_{3}(v,\mathbf{G})\)}[0.45\linewidth]{%
      \centering
      \begin{tikzpicture}[x=1cm,y=1cm,
        mynode/.style={draw, circle, inner sep=2pt, minimum size=1em}]

        \node[mynode] (v) at (0,0)    {\(v\)};
        \node[mynode] (u) at (1,0)    {\(u\)};

        \node[mynode] (v1) at (-1,1)  {};
        \node[mynode] (v2) at (-1,-1)  {};

        \node[mynode] (v11) at (-2,1.5)  {};
        \node[mynode] (v12) at (-2,1)  {};
        \node[mynode] (v13) at (-2,.5)  {};

        \node[mynode] (v111) at (-3,1.5)  {};

        \node[mynode] (v131) at (-3,.75)  {};
        \node[mynode] (v132) at (-3,.25)  {};

        \node[mynode] (v21) at (-2,-.75)  {};
        \node[mynode] (v22) at (-2,-1.5)  {};

        \node[mynode] (v221) at (-3,-1.25)  {};
        \node[mynode] (v222) at (-3,-1.75)  {};

        \node[mynode] (u1) at (2,1)  {};
        \node[mynode] (u2) at (2,0)  {};
        \node[mynode] (u3) at (2,-1)  {};

        \node[mynode] (u11) at (3,1.25)  {};
        \node[mynode] (u12) at (3,0.75)  {};

        \node[mynode] (u31) at (3,-.2)  {};
        \node[mynode] (u32) at (3,-.7)  {};
        \node[mynode] (u33) at (3,-1.2)  {};
        \node[mynode] (u34) at (3,-1.7)  {};

        \draw (v)-- node[above] {\(e\)} (u);
        \draw (v)--(v1)--(v11)--(v111);
        \draw (v1)--(v12);
        \draw (v1)--(v13)--(v131);
        \draw (v13)--(v132);
        \draw (v)--(v2)--(v21);
        \draw (v2)--(v22)--(v221);
        \draw (v22)--(v222);
        \draw (u)--(u1)--(u11);
        \draw (u1)--(u12);
        \draw (u)--(u2);
        \draw (u)--(u3)--(u31);
        \draw (u3)--(u32);
        \draw (u3)--(u33);
        \draw (u3)--(u34);
      \end{tikzpicture}}
    \subcaptionbox{\(\tcboxmath[colframe=red!20, colback=red!20,
    boxsep=0pt,
    left=0pt,right=0pt,top=1pt,bottom=1pt]{B_{3}(v,\mathbf{G}-e)}\)
    and~\(\tcboxmath[colframe=blue!20, colback=blue!20,
    boxsep=0pt,
    left=0pt,right=0pt,top=1pt,bottom=1pt]{B_{2}(u,\mathbf{G}-e)}\)}[0.45\linewidth]{%
      \centering
      \begin{tikzpicture}[x=1cm,y=1cm,
        mynode/.style={draw, circle, inner sep=2pt, minimum size=1em}]

        \path [rounded corners=10pt, fill=red!20]
        (0.5,0)--(-1,1.5)--(-2,2)--(-3.5,2.25)--(-3.5,-2.5)--(-2,-2)--(-1,-1.5)--cycle;

        \path [rounded corners=10pt, fill=blue!20]
        (0.5,0)--(1,1)--(2,1.5)--(3.4,2)--(3.4,-2.5)--(2,-1.5)--(1,-1)--cycle;

        \node[mynode] (v) at (0,0)    {\(v\)};
        \node[mynode] (u) at (1,0)    {\(u\)};

        \node[mynode] (v1) at (-1,1)  {};
        \node[mynode] (v2) at (-1,-1)  {};

        \node[mynode] (v11) at (-2,1.5)  {};
        \node[mynode] (v12) at (-2,1)  {};
        \node[mynode] (v13) at (-2,.5)  {};

        \node[mynode] (v111) at (-3,1.5)  {};

        \node[mynode] (v131) at (-3,.75)  {};
        \node[mynode] (v132) at (-3,.25)  {};

        \node[mynode] (v21) at (-2,-.75)  {};
        \node[mynode] (v22) at (-2,-1.5)  {};

        \node[mynode] (v221) at (-3,-1.25)  {};
        \node[mynode] (v222) at (-3,-1.75)  {};

        \node[mynode] (u1) at (2,1)  {};
        \node[mynode] (u2) at (2,0)  {};
        \node[mynode] (u3) at (2,-1)  {};

        \node[mynode] (u11) at (3,1.25)  {};
        \node[mynode] (u12) at (3,0.75)  {};

        \node[mynode] (u31) at (3,-.2)  {};
        \node[mynode] (u32) at (3,-.7)  {};
        \node[mynode] (u33) at (3,-1.2)  {};
        \node[mynode] (u34) at (3,-1.7)  {};

        \draw (v)--(v1)--(v11)--(v111);
        \draw (v1)--(v12);
        \draw (v1)--(v13)--(v131);
        \draw (v13)--(v132);
        \draw (v)--(v2)--(v21);
        \draw (v2)--(v22)--(v221);
        \draw (v22)--(v222);
        \draw (u)--(u1)--(u11);
        \draw (u1)--(u12);
        \draw (u)--(u2);
        \draw (u)--(u3)--(u31);
        \draw (u3)--(u32);
        \draw (u3)--(u33);
        \draw (u3)--(u34);
      \end{tikzpicture}}
    \caption{Illustration of the gluing procedure.
      The neighbourhood~\(B_{3}(v,\mathbf{G})\)
      can be obtained by combining
      the two neighbourhoods~\(B_{3}(v,\mathbf{G}-e)\)
      and~\(B_{2}(u,\mathbf{G}-e)\)
      that cannot use~\(e\)
      and information about~\(e\).}
    \label{fig:glue}
  \end{figure}

  Let~\((w_{e},w'_{e},\widetilde{w}_{e},\widetilde{w}'_{e})\)
  be a coupling independent
  of everything else that
  satisfies~\(w_{e},w'_{e} \sim \mu_{E,n}\),
  \(\widetilde{w}_{e},\widetilde{w}'_{e} \sim \mu_{E}\)
  as well as
  \[
  \prob_{n}(w_{e} \neq \widetilde{w}_{e})
  \leq d_{\mathrm{TV}}(\mu_{E,n},\mu_{E})
  \quad\text{and}\quad
  \prob_{n}(w'_{e} \neq \widetilde{w}'_{e})
  \leq d_{\mathrm{TV}}(\mu_{E,n},\mu_{E}).
  \]
  Then define
  \begin{equation*}
  \mathbf{\widetilde{T}}_{\ell}(v)
  =
    \Psi(\mathbf{T}''_{\ell}(v),\mathbf{T}''_{\ell-1}(u),X_{e},\widetilde{w}_{e})
    \quad\text{and}\quad
  \mathbf{\widetilde{T}}^{e}_{\ell}(v)
  =
  \Psi(\mathbf{T}''_{\ell}(v),\mathbf{T}''_{\ell-1}(u),X'_{e},\widetilde{w}'_{e}).
  \end{equation*}
  Conditionally on~\(Y_{e},Y_{e'}\)
  the
  pair~\((\mathbf{\widetilde{T}}_{\ell}(v),\mathbf{\widetilde{T}}^{e}_{\ell}(v))\)
  has the desired distribution.

  Let~\(E_{0}\) be the event that~\(B_{\ell}(v,\mathbf{G}_{n}'')\)
  and~\(B_{\ell-1}(u,\mathbf{G}_{n}'')\)
  share a vertex.
  On the complement of~\(E_{0}\) it is possible
  to join~\(B_{\ell}(v,\mathbf{G}_{n}'')\)
  and~\(B_{\ell-1}(u,\mathbf{G}_{n}'')\)
  at the edge~\(e\) to obtain a tree
  provided both~\(B_{\ell}(v,\mathbf{G}_{n}'')\)
  and~\(B_{\ell-1}(u,\mathbf{G}_{n}'')\)
  are trees.
  Then
  \begin{equation}\label{eq:g''t}
    \begin{split}
      &\prob_{n}(
      (B_{\ell},B'_{\ell})
      \cong
      (\mathbf{\widetilde{T}}_{\ell}(v),\mathbf{\widetilde{T}}^{e}_{\ell}(v))
      \given Y_{e},Y_{e'}
      )\\
      &\quad\geq
      \prob_{n}(B_{\ell}(v,\mathbf{G}_{n}'') \cong \mathbf{T}''_{\ell}(v),
      B_{\ell-1}(u,\mathbf{G}_{n}'') \cong \mathbf{T}''_{\ell-1}(u))\\
      &\qquad -\prob_{n}(E_{0})\\
      &\qquad - \prob_{n}(B_{\ell}(v,\mathbf{G}_{n}'') \neq
      B_{\ell}(v,\mathbf{G}_{n}-e)
      \given Y_{e},Y_{e'})\\
      &\qquad - \prob_{n}(B_{\ell-1}(u,\mathbf{G}_{n}'') \neq
      B_{\ell-1}(u,\mathbf{G}_{n}-e)
      \given Y_{e},Y_{e'})\\
      &\qquad - \prob_{n}(\widetilde{w}_{e} \neq w_{e})
      -\prob_{n}(\widetilde{w}'_{e} \neq w'_{e}).
    \end{split}
  \end{equation}

  By assumption
  \[
  \prob_{n}(B_{\ell}(v,\mathbf{G}_{n}'') \cong \mathbf{T}_{\ell}(v),
  B_{\ell-1}(u,\mathbf{G}_{n}'') \cong \mathbf{T}_{\ell-1}(u))
  \geq 1-\varepsilon_{n,\ell}( \set{u,v}).
  \]

  Furthermore,~\(B_{\ell}(v,\mathbf{G}''_{n})\)
  and~\(B_{\ell-1}(u,\mathbf{G}''_{n})\)
  share a vertex only if there is a path from~\(v\) to~\(u\)
  of length at most~\(2\ell-1\).
  Hence, \autoref{cor:abpathupto} implies
  \begin{equation*}
  \prob_{n}(E_{0})
  \leq \prob_{n}(v \pathbetw_{2\ell-1} v)
  \leq \frac{W_{v}W_{u}}{n\vartheta}
  (\Gamma_{2,n}+1)^{2\ell-1}.
  \end{equation*}

  Finally, \(\mathbf{G}_{n}''\) is independent of~\(Y_{e},Y_{e'}\)
  and differs from~\(\mathbf{G}_{n}\) only in~\(e\) and~\(e'\).
  Hence, the neighbourhood~\(B_{\ell}(v,\mathbf{G}_{n}'')\) can differ
  from~\(B_{\ell}(v,\mathbf{G}_{n}-e)\)
  only if~\(e\) is present in the former or if~\(e'\) is present in
  one and not the other.
  Thus
  \begin{equation*}
  \begin{split}
  &\prob_{n}(B_{\ell}(v,\mathbf{G}_{n}'') \neq B_{\ell}(v,\mathbf{G}_{n}-e)
  \given Y_{e},Y_{e'})\\
  &\quad\leq \prob_{n}(e \in B_{\ell}(v,\mathbf{G}_{n}''))
   + \prob_{n}(e' \in B_{\ell}(v,\mathbf{G}_{n}''))
   + \prob_{n}(e' \in B_{\ell}(v,\mathbf{G}_{n}-e) \given X_{e'}).
  \end{split}
  \end{equation*}

  For the first term note that~\(e\) is present
  in~\(B_{\ell}(v,\mathbf{G}''_{n})\)
  if and only if~\(v\) is connected to~\(u\) via~\(e\),
  i.e.
  \[
  \prob_{n}(e \in B_{\ell}(v,\mathbf{G}_{n}''))
  = \expe_{n}[X''_{e}]
  = \frac{W_{u}W_{v}}{n \vartheta}.
  \]
  For the second term we we can apply \autoref{lem:einbl}
  \begin{equation*}
  \prob_{n}(e' \in B_{\ell}(v,\mathbf{G}_{n}''))
  \leq \frac{W_{v}(W_{u'}+W_{v'})}{n\vartheta}
  (\Gamma_{2,n}+1)^{\ell}.
  \end{equation*}
  The third term contains a conditioning which can be removed as follows.
  Since~\(e'\) can only be present in the neighbourhood
  if one of its vertices~\(u'\) or~\(v'\) is
  present in the neighbourhood in its own right,
  i.e.~when~\(e'\) itself is ignored,
  we can drop the conditioning on~\(X_{e'}\),
  so that by \autoref{lem:uinbl}
  \begin{align*}
  \prob_{n}(e' \in B_{\ell}(v,\mathbf{G}_{n}-e) \given X_{e'})
  &\leq \prob_{n}(v' \in B_{\ell}(v,\mathbf{G}_{n} -
  \set{e,e'}))
  + \prob_{n}(u' \in B_{\ell}(v,\mathbf{G}_{n} - \set{e,e'}))\\
  &\leq \prob_{n}(v' \in B_{k}(v,\mathbf{G}_{n}))
  + \prob_{n}(u' \in B_{\ell}(v,\mathbf{G}_{n}))\\
  &\leq \frac{W_{v}(W_{u'}+W_{v'})}{n \vartheta}
  (\Gamma_{2,n}+1)^{\ell}.
  \end{align*}

  The remaining terms can be treated similarly.
  Putting everything together, this shows the claim.
\end{proof}

The result from \autoref{lem:couplegngne} can be slightly simplified
by dropping the conditioning on~\(e'\).
\begin{lemma}\label{lem:couplegngn}
  Fix two vertices~\(v\) and~\(u\) and set~\(e = \edge{v}{u}\).
  Given a coupling
  \[
  (
  B_{\ell}(v,\mathbf{G}_{n}),
  B_{\ell}(u,\mathbf{G}_{n}),
  \mathbf{T}_{\ell}(v),
  \mathbf{T}_{\ell}(u)
  )
  \]
  with
  independent~\(\mathbf{T}_{\ell}(v)\sim
  \mathbf{T}_{\ell}(W_{v},\nu,\mu_{E},\mu_{V})\)
  and~\(\mathbf{T}_{\ell}(u)
  \sim \mathbf{T}_{\ell}(W_{u},\nu,\mu_{E},\mu_{V})\)
  that satisfies
  \[
  \varepsilon_{n,\ell}( \set{u,v})
  \geq 1-\prob_{n}(
  B_{\ell}(v,\mathbf{G}_{n}) \cong \mathbf{T}_{\ell}(v),
  B_{\ell}(u,\mathbf{G}_{n}) \cong \mathbf{T}_{\ell}(u)
  )
  \]
  it is possible to
  couple~\((B_{\ell}(v,\mathbf{G}_{n}), B_{\ell}(v,\mathbf{G}_{n}^{e}),
  \mathbf{\widetilde{T}}_{\ell}(v), \mathbf{\widetilde{T}}^{e}_{\ell}(v))\)
  such that
  \begin{align*}
  (\mathbf{\widetilde{T}}_{\ell}(v), \mathbf{\widetilde{T}}^{e}_{\ell}(v))
  \given Y_{e} = (1,0)
  &\eqdist (\mathbf{\tilde{T}}_{\ell},\mathbf{T}_{\ell}),\\
  (\mathbf{\widetilde{T}}_{\ell}(v), \mathbf{\widetilde{T}}^{e}_{\ell}(v))
  \given Y_{e} = (0,1)
  &\eqdist (\mathbf{T}_{\ell},\mathbf{\tilde{T}}_{\ell}),
  \end{align*}
  where~\(Y_{e} = (X_{e},X'_{e})\)
  as well as~\(\mathbf{T}_{\ell}
  \sim \mathbf{T}_{\ell}(W_{v},\nu,\mu_{E},\mu_{V})\)
  and~\(\mathbf{\tilde{T}}_{\ell}
  \sim \mathbf{\tilde{T}}_{\ell}(W_{v},W_{u},\nu,\mu_{E},\mu_{V})\).
  Furthermore,
  \[
  \begin{split}
  &\prob_{n}(
  (B_{\ell}(v,\mathbf{G}_{n}),B_{\ell}(v,\mathbf{G}_{n}^{e})
  \cong (\mathbf{\widetilde{T}}_{\ell}(v),
  \mathbf{\widetilde{T}}^{e}_{\ell}(v)))
  \given
  Y_{e}
  )\\
  &\quad\geq
  1-\Bigl(
  \varepsilon_{n,\ell}(\set{u,v})
  + 2d_{\mathrm{TV}}(\mu_{E,n},\mu_{E})+ C\frac{W_{u}W_{v}}{n\vartheta}
  (\Gamma_{2,n}+1)^{2\ell}
  \Bigr).
  \end{split}
  \]
\end{lemma}

\begin{proof}
  Replicate the steps of the proof of \autoref{lem:couplegngn},
  but do not modify the graph at~\(e'\)
  and do not condition on~\(Y_{e'}\).

  In particular the probabilities in~\eqref{eq:g''t}
  are only conditioned on~\(Y_{e}\) and not on~\(Y_{e'}\).
  The relevant conditional probabilities on the right-hand side
  can then be estimated as
  \begin{equation*}
  \prob_{n}(B_{\ell}(v,\mathbf{G}_{n}'') \neq B_{\ell}(v,\mathbf{G}_{n}-e)
  \given Y_{e})
  \leq \prob_{n}(e \in B_{\ell}(v,\mathbf{G}_{n}''))
  \leq \expe_{n}[X''_{e}]
  \leq \frac{W_{u}W_{v}}{n \vartheta}
  \end{equation*}
  and
  \begin{equation*}
  \prob_{n}(B_{\ell-1}(u,\mathbf{G}_{n}'') \neq B_{\ell-1}(u,\mathbf{G}_{n}-e)
  \given Y_{e})
  \leq \prob_{n}(e \in B_{\ell-1}(u,\mathbf{G}_{n}''))
  \leq \expe_{n}[X''_{e}]
  \leq \frac{W_{u}W_{v}}{n \vartheta}.
  \end{equation*}

  Putting these results together we obtain the claimed bound.
\end{proof}

\section{Proof of the central limit theorem}
\label{chap:proof}

The proof of \autoref{thm:fgnconv} is based on the perturbative Stein's method.
We will follow the approach used by \citet[§~5]{cao}
for the case of the (homogeneous) Erdős--Rényi graph.
The same  general strategy of combining
the perturbative Stein's method
with local approximations of the problem
was also used by \citeauthor{chatterjee-sen}
\citetext{\citealp{chatterjee-sen}, \citealp[see also][§~4]{chatterjee}}
to show a central limit theorem for the minimal spanning tree
in Euclidean space~\(\reals^{d}\) and the lattice~\(\integers^{d}\)
for dimensions~\(d \geq 2\).

We will build our proof around
the following reformulation of
\citeauthor{chatterjee}'s perturbative Stein's method
\citep[Cor.~3.2]{chatterjee}.
Recall that~\(\sigma_{n}^{2} = \var_{n}(f(\mathbf{G}_{n}))\).
\begin{prop}\label{lem:steinappl}
  If there exists a function~\(f \colon V_{n}^{(2)} \setunion V_{n}
  \to \reals\)
  such that for all~\(x, x' \in V_{n}^{(2)} \setunion V_{n}\) we have
  \[
    \sigma_{n}^{-4}\cov(\Delta_{x}f \Delta_{x}f^{F},\Delta_{x'}f \Delta_{x'}f^{F'})
    \leq c(x,x')
  \]
  for all~\(F \subseteq (E_{n} \setunion V_{n}) \setminus\set{x}\),
  \(F' \subseteq (E_{n} \setunion V_{n}) \setminus \set{x'}\).
  Then we have
  \begin{equation}\label{eq:znwdist}
    \sup_{t \in\reals}\abs{\prob_{n}(Z_{n} \leq t)-\Phi(t)}
    \leq \sqrt{2} \biggl( \sum_{x,x' \in E_{n} \setunion V_{n}} c(x,x')
    \biggr)^{\!\!1/4}
    +
    \sigma_{n}^{-3/2}
    \biggl(
    \sum_{x\in E_{n} \setunion V_{n}} \expe\bigl[\abs{\Delta_{x}f}^{3}\bigr]
    \biggr)^{\!\!1/2}\!.
  \end{equation}
\end{prop}

\begin{proof}
  Apply \citep[Cor.~3.2]{chatterjee}
  to~\(\mathbf{G}_{n}\)
  interpreted as a sequence of independent random
  variables~\(((X_{e},w_{e})_{e\in E_{n}}, (w_{v})_{v \in V_{n}})\)
  and the function~\(\sigma_{n}^{-1}(f(x)-\expe[f(\mathbf{G}_{n})])\).
\end{proof}

\autoref{thm:fgnconv} can now be proved by finding a suitable function~\(c\)
and bounding the two sums in~\eqref{eq:znwdist}.
In particular we need to identify bounds for~\(c(x,x')\)
for all possible combinations of edges and vertices.
This is what we set out to do in the remainder of this chapter,
which can broadly be split into three parts.
The first part is dedicated to identifying straightforward
bounds for~\(c\) and~\(\Delta_{x}f\).
The second part of the chapter examines bounds for~\(c\),
where we make use of the fact that edges and vertices that are not
incident to each other are weakly correlated.
Finally, we show that the sum of the bounds we identified are of the order claimed in the theorem.

The second sum on the right-hand side of~\eqref{eq:znwdist}
does not involve~\(c\) and is easily bounded
after separation into vertex and edge sums
\[
  \sum_{x\in E_{n} \setunion V_{n}} \expe_{n}[\abs{\Delta_{x}f}^{3}]
  =\sum_{e\in E_{n}} \expe_{n}[\abs{\Delta_{e}f}^{3}]
  + \sum_{v \in V_{n}} \expe_{n}[\abs{\Delta_{v}f}^{3}].
\]

\begin{lemma}\label{lem:sumdelthree:e}
  Under the conditions of \autoref{thm:fgnconv}
  \[
    \sum_{e\in E_{n}} \expe_{n}[\abs{\Delta_{e}f}^{3}]
    \leq 2 n \vartheta \Gamma_{1,n}^{2} J_{E}^{1/2}.
  \]
\end{lemma}
\begin{proof}
  By the estimate~\eqref{eq:deltahbound:e} for~\(\abs{\Delta_{e}f}\),
  independence of edge indicators~\(\mathbf{X}\) and~\(\mathbf{X}'\)
  from the weights~\(\mathbf{w}\) and~\(\mathbf{w}'\)
  and the moment bound~\eqref{eq:jbound:e} for~\(H_{E}\)
  we have
  \begin{align*}
    \expe_{n}[\abs{\Delta_{e}f}^{3}]
    &\leq \expe_{n}[
      \indfunc_{\set{\max\set{X_{e},X'_{e}} = 1}} \abs{H_{E}(w_{e},w'_{e},w_{v},w_{u})}^{3}
      ]\\
    &\leq ( \prob_{n}(X_{e}=1)+\prob_{n}(X'_{e}=1))
      \expe_{n}[\abs{H_{E}(w_{e},w'_{e},w_{v},w_{u})}^{6}]^{\!1/2}\\
    &\leq 2 \frac{W_{u}W_{v}}{n\vartheta} J_{E}^{1/2}.
  \end{align*}
  Hence, summing over all edges or combinations of two vertices
  we obtain
  \begin{equation*}
    \sum_{e\in E_{n}} \expe_{n}[\abs{\Delta_{e}f}^{3}]
    \leq \sum_{u,v \in V_{n}} 2 \frac{W_{u}W_{v}}{n\vartheta} J_{E}^{1/2}
    \leq 2 n\vartheta \Gamma_{1,n}^{2} J_{E}^{1/2}.
  \end{equation*}
  The claim follows.
\end{proof}

The vertex sum can be treated analogously.
\begin{lemma}\label{lem:sumdelthree:v}
  Under the conditions of \autoref{thm:fgnconv}
  \[
    \sum_{v\in V_{n}} \expe_{n}[\abs{\Delta_{v}f}^{3}]
    \leq J_{V}^{1/2} \sum_{v \in V_{n}} \zeta(v)^{3/4}.
  \]
\end{lemma}
\begin{proof}
  By the bound~\eqref{eq:deltahbound:v} on~\(\abs{\Delta_{v}f}\),
  independence of edge indicators and weights
  together with the bounds~\eqref{eq:jbound:v} and~\eqref{eq:nubound}
  for moments of~\(H_{V}\) and~\(\vdegbound\)
  we have
  \begin{equation*}
    \expe_{n}[\abs{\Delta_{v}f}^{3}]
    \leq \expe_{n}[\vdegbound^{3}]
      \expe_{n}[H_{V}(w_{v},w'_{v})^{3}]
    \leq \zeta_{n}(v)^{3/4} J_{V}^{1/2} .
  \end{equation*}
  Summing over~\(v\) we obtain
  \[
    \sum_{v\in V_{n}} \expe_{n}[\abs{\Delta_{v}f}^{3}]
    \leq J_{V}^{1/2} \sum_{v \in V_{n}}\zeta_{n}(v)^{3/4}.
  \]
  The claim follows.
\end{proof}

It remains to identify a suitable function~\(c\)
bounding the covariances of the randomised derivatives
such that the sum over all vertex-edge combinations
has a rate that still allows for the desired convergence.

The~\(O(n^{2})\) \enquote{diagonal terms}~\(c(e,e)\)
are relatively easy to bound;
and a bound of order~\(O(\sigma_{n}^{-4} n^{-1})\)
is sufficient to ensure the desired
convergence.

\begin{lemma}\label{lem:ceebd:same:all}
  Let~\(u,v,v'\) be three different vertices
  and set~\(e= \edge{u}{v}\) and~\(e' = \edge{u}{v'}\).
  Then we may take
  \begin{alignat*}{8}
    c(e,e)
    &= \sigma_{n}^{-4} C J_{E}^{2/3} \frac{W_{u}W_{v}}{n\vartheta},
    &\quad&&
    c(e,e')
    &= \sigma_{n}^{-4} C W_{u}^{2} \frac{W_{v}}{n\vartheta}
      \frac{W_{v'}}{n\vartheta} J_{E}^{2/3},\\
    c(v,v)
    &= \sigma_{n}^{-4} CJ_{V}^{2/3} \zeta_{n}(v),
    &\quad&&
    c(e,v)
    &= \sigma_{n}^{-4} C \frac{W_{u}W_{v}}{n\vartheta}
      J_{E}^{1/3} J_{V}^{1/3} \zeta_{n}(v)^{1/2}.
  \end{alignat*}
\end{lemma}
\begin{proof}
  We will present the proof only for~\(c(e,v)\),
  the other proofs are similar
  \citep[detailed proofs can be found in][Lems.~4.2.3\,--\,4.2.6]{thesis}.
  It is enough to verify
  \[
    \cov_{n}(\Delta_{e}f\Delta_{e}f^{F},\Delta_{v}f\Delta_{v}f^{F'})
    \leq C \frac{W_{u}W_{v}}{n\vartheta}
      J_{E}^{1/3} J_{V}^{1/3} \zeta_{n}(v)^{1/2}.
  \]
  The claim then follows.

  By the bounds~\eqref{eq:deltahbound:e} and~\eqref{eq:deltahbound:v}
  for~\(\Delta_{e}f\) and~\(\Delta_{v} f\) we have
  \begin{align*}
    \expe_{n}[\abs{
      \Delta_{e}f \Delta_{e}f^{F} \Delta_{v}f \Delta_{v}f^{F'}}]
    &\leq\expe_{n}[\abs{
      \indfunc_{\set{\max\set{X_{e},X'_{e}}=1}}
      H_{E}(w_{e},w'_{e},w_{v},w_{u})
      H_{E}(w_{e},w'_{e},w^{F}_{v},w^{F}_{u})\\
     &\quad\qquad  \vdegbound H_{V}(w_{v},w'_{v})
       \vdegbound[F'] H_{V}(w_{v},w^{\prime}_{v})}].
  \end{align*}
  Recall \autoref{def:dignore}
  and estimate~\(\vdegbound\) against~\(h(\setcard{D_{1}^{(u)}(v)+1})\),
  which ignores the edge~\(e=\edge{u}{v}\),
  and similarly for~\(\vdegbound[F']\).
  Collect terms,
  then separate independent terms.
  With~\eqref{eq:jbound:e}, \eqref{eq:jbound:v}, \eqref{eq:nubound}
  and judicious application of Cauchy--Schwarz
  we obtain the bound
  \[
    \expe_{n}[\abs{
      \Delta_{e}f \Delta_{e}f^{F} \Delta_{v}f \Delta_{v}f^{F'}}]
    \leq C \frac{W_{v}W_{u}}{n\vartheta} \zeta_{n}(v)^{1/2}
    J_{E}^{1/3} J_{V}^{1/3}.
  \]
  Similarly we have
  \begin{equation*}
    \expe_{n}[\abs{\Delta_{e}f \Delta_{e}f^{F}}]
    \leq  2 \frac{W_{u}W_{v}}{n\vartheta} J_{E}^{1/3}.
  \end{equation*}
  and
  \begin{equation*}
    \expe_{n}[\abs{\Delta_{v}f \Delta_{v}f^{F}}]
    \leq  J_{V}^{1/3} \zeta_{n}(v)^{1/2}.
  \end{equation*}

  Putting these bounds together we obtain
  \begin{equation*}
    \cov_{n}(\Delta_{e}f\Delta_{e}f^{F},\Delta_{v'}f\Delta_{v'}f^{F'})
     \leq C \frac{W_{u}W_{v}}{n\vartheta}
      J_{E}^{1/3} J_{V}^{1/3} \zeta_{n}(v)^{1/2}.
  \end{equation*}
  This proves the claim.
\end{proof}

We will now find bounds for~\(c\)
for the remaining cases in which the vertices and edges
that are involved are not incident to each other.
The mainly Cauchy--Schwarz-based approach of the previous
proofs would not give the desired convergence rates here.
We rely on property~GLA (\autoref{def:gla})
and the coupling to the limiting Galton--Watson tree
(\autoref{def:limtree:weights})
to bound the effect of a local change on the function
by a local quantity.
Property GLA and the coupling to the limiting tree
structure ensure that the approximation error
of using the local quantity goes to zero.
The sparsity of the underlying graph ensures that
local quantities are only very weakly correlated
(cf.~\autoref{lem:covgbgpbp}).

In our calculations we will use the following
multi-purpose error term that absorbs
the probability of various additional coupling
events and the covariance bound
for local neighbourhoods.
\begin{definition}\label{def:rho}
  Let~\(n,k \in \naturals\).
  For any set of vertices~\(\mathcal{V} \subseteq V_{n}\)
  set
  \[
    \rho_{n,k}(\mathcal{V})
    = \min\set[\bigg]{
        \frac{(\setweight{\mathcal{V}}+\setcard{V})^{2}}{n\vartheta}
        (\Gamma_{1,n}+1)^{2}(\Gamma_{2,n}+C)^{2k+1}(\Gamma_{3,n}+1)^{2},1
    }.
  \]
  Recall the definition of~\(\varepsilon_{n,k}(\mathcal{V})\)
  from \autoref{lem:nbhdcoupl:noweights} and \autoref{lem:bvbucoupl}
  \begin{align*}
      \varepsilon_{n,k}(\mathcal{V})
    &= \setweight{\mathcal{V}}_{2} \frac{\Gamma_{2,n}}{n\vartheta}
    +\setweight{\mathcal{V}}_{+}\Gamma_{1,n}
    +\setweight{\mathcal{V}}(\Gamma_{2,n}+1)^{k}
    \biggl(
      \frac{\Gamma_{3,n}}{n\vartheta}+\kappa_{1,n}+\kappa_{2,n}+
      \frac{2+\Gamma_{1,n}}{k_{n}}+\frac{k_{n}}{n\vartheta}\biggr)\\
  &\quad + \setcard{\mathcal{V}}\frac{1}{k_{n}}
    +\frac{k_{n}^{2}}{n\vartheta\Gamma_{1,n}}
    +\setweight{\mathcal{V}}\alpha_{n}
    \Bigl(
      \frac{1}{\vartheta}
      + (\Gamma_{2}+1)^{k-1}
        \Bigl(\frac{\Gamma_{2,n}}{\vartheta\Gamma_{1,n}}+1\Bigr)
    \Bigr)\\
    &\quad+
       (\setcard{\mathcal{V}}+\setweight{\mathcal{V}}(\Gamma_{2}+1)^{k})
    (d_{\mathrm{TV}}(\mu_{E,n},\mu_{E})
    + d_{\mathrm{TV}}(\mu_{V,n},\mu_{V})
    ).
  \end{align*}
\end{definition}

The remainder of this subsection is dedicated to establishing the
following proposition.
\begin{prop}\label{lem:ceebd:alldiff}
  Let~\(u,v,u',v' \in V_{n}\) be four distinct vertices
  and set~\(e = \edge{u}{v}\) and~\(e' = \edge{u'}{v'}\).
  Then we may take
  \begin{align*}
    \begin{split}
    c(e,e')
    &= \sigma_{n}^{-4} CJ_{E} \frac{W_{u}W_{v}}{n\vartheta}
    \frac{W_{u'}W_{v'}}{n\vartheta}
    (
      (m^{E}_{n}(v,u)\delta_{k}^{E})^{1/2}+
      (m^{E}_{n}(v',u')\delta_{k}^{E})^{1/2}\\
    &\qqquad +\varepsilon_{n,k}(\set{u,v,u',v'})^{1/4}
       +\rho_{n,k}(\set{u,v,u',v'})^{1/4}
    ),
  \end{split}\\
    c(v,v') &= \sigma_{n}^{-4}  CJ_{V} \zeta_{n}(v)\zeta_{n}(v')
    ((m^{V}_{n}(v)\delta_{k}^{V})^{1/2}+(m^{V}_{n}(v')\delta_{k}^{V})^{1/2}+
   \varepsilon_{n,k}(\set{v,v'})^{1/4}
              +\rho_{n,k}(\set{v,v'})^{1/4})\\
    \shortintertext{and}
    \begin{split}
      c(e,v') &= \sigma_{n}^{-4} CJ_{E}J_{V} \frac{W_{u}W_{v}}{n\vartheta}
      \zeta_{n}(v')^{1/2}
      ((m_{n}^{E}(v,u)\delta_{k}^{E})^{1/2}+(m_{n}^{V}(v')\delta_{k}^{V})^{1/2}\\
   &\qqquad +\varepsilon_{n,k}(\set{u,v,v'})^{1/4}
    +\rho_{n,k}(\set{u,v,v'})^{1/4}).
  \end{split}
\end{align*}
\end{prop}
We will only prove the result for~\(c(e,v')\) in detail here.
The proof for~\(c(e,e')\) and~\(c(v,v')\)
follow the same general structure
\citep[see][Prop.~4.3.2 in \S\,4.3.1 and Prop.~4.3.7 in \S\,4.3.2]{thesis}.

In order to make use of the coupling between
the local neighbourhood and our limiting tree,
we need to introduce some notation.

Let~\(E_{0}(e)\) be the event that both~\(B_{k}(v,\mathbf{G}_{n})\)
and~\(B_{k}(v,\mathbf{G}_{n}^{e})\) are trees
and let
\[
  A_{e} = \set{(X_{e},X'_{e})=(1,0)\text{ or }(X_{e},X'_{e})=(0,1)}
\]
be the event that~\(e\) is switched by going from~\(X_{e}\) to~\(X'_{e}\).
Let~\(E_{0}(v')\) be the event that~\(B_{k}(v',\mathbf{G}_{n})\)
and thus also~\(B_{k}(v',\mathbf{G}_{n}^{v'})\) is a tree.
Define
\[
  \tilde{L}_{k}^{E}(e)
  = \indfunc_{E_{0}(e)} \indfunc_{A_{e}}
  \mathrm{LA}_{k}^{E,L}(
    B_{k}(v,\mathbf{G}_{n}), B_{k}(v,\mathbf{G}_{n}^{e})
  )
  \;\text{and}\;
  \tilde{L}_{k}^{V}(v')
  = \indfunc_{E_{0}(v')}
  \mathrm{LA}_{k}^{V,L}(B_{k}(v',\mathbf{G}_{n}),B_{k}(v',\mathbf{G}_{n}^{v}))
\]
Let~\(Q(\placeholder, y) \colon \reals \to \intervalcc{-y}{y}\)
be the function truncating~\(x\) to level~\(y \geq 0\), i.e.~\(
  Q(x,y)
  = \max\set{\min\set{x,y},-y}\).
For brevity let
\begin{equation*}
  \tilde{H}_{E}(e) = H_{E}(w_{e},w'_{e},w_{v},w_{u}),
  \quad
  \tilde{H}_{E}(F \setunion e) = H_{E}(w_{e},w'_{e},w^{F}_{v},w^{F}_{u}).
\end{equation*}
Then define a truncated version of~\(\tilde{L}_{k}^{E}(e)\)
and~\(\tilde{L}_{k}^{V}(v')\)
\[
  L_{k}^{E}(e) = Q(\tilde{L}_{k}^{E}(e),\tilde{H}_{E}(e))
  \quad\text{and}\quad
  L_{k}^{V}(v') = Q(\tilde{L}_{k}^{V}(v'),\vdegbound H_{V}(w_{v},w'_{v}))
\]
This will be our local approximation of~\(\Delta_{e}f\) and~\(\Delta_{v'}f\),
respectively, (hence~\enquote{\(L\)}).
The construction ensures
\begin{equation}
  \label{eq:lbound}
  \abs{L_{k}^{E}(e)} \leq \indfunc_{A_{e}} \tilde{H}_{E}(e)
  \quad\text{and}\quad
  \abs{L_{k}^{V}(v')} \leq \vdegbound \tilde{H}_{V}(w_{v},w'_{v})
\end{equation}
Let~\(R_{k}^{E}(e)\) be the difference of~\(L_{k}^{E}(e)\) to~\(\Delta_{e}f\)
and~\(R_{k}^{V}(v')\) be the difference of~\(L_{k}^{V}(v')\) to~\(\Delta_{v'}f\)
(it is the remainder to~\(\Delta_{e}f\), hence~\enquote{\(R\)}),
i.e.~set
\[
  R_{k}^{E}(e) = \Delta_{e}f - L_{k}^{E}(e)
  \quad\text{and}\quad
  R_{k}^{V}(v') = \Delta_{v'}f - L_{k}^{V}(v').
\]
Let~\(\tilde{A}_{e} = \set{\max\set{X_{e},X'_{e}}=1}\),
so that~\(\tilde{A}_{e} = A_{e} \setunion \set{X_{e}=X'_{e}=1}\).
Then by~\eqref{eq:deltahbound:e}, \eqref{eq:deltahbound:v} and~\eqref{eq:lbound}
\begin{equation}
  \label{eq:rbound}
    \abs{R_{k}^{E}(e)}
    \leq 2 \indfunc_{\tilde{A}_{e}} \tilde{H}_{E}(e)
    \quad\text{and}\quad
    \abs{R_{k}^{V}(v')}
    \leq 2 \vdegbound \tilde{H}_{V}(w_{v},w'_{v}).
\end{equation}

For~\(F \subseteq (V_{n}^{(2)} \setunion V_{n}) \setminus \set{e}\)
define~\(L_{k}^{E}(F \setunion e)\)
with~\(B_{k}(v,\mathbf{G}_{n}^{F})\) and~\(B_{k}(v,\mathbf{G}_{n}^{F \setunion e})\)
instead of~\(B_{k}(v,\mathbf{G}_{n})\) and~\(B_{k}(v,\mathbf{G}_{n}^{e})\).
Analogous to~\(R_{k}^{E}(e)\) define the remainder~\(R_{k}^{E}(F \setunion e)\)
\[
  R_{k}^{E}(F \setunion e) = \Delta_{e}f^{F} - L_{k}^{E}(F \setunion e).
\]
For~\(F' \subseteq (V_{n}^{(2)} \setunion V_{n}) \setminus \set{v'}\)
we also define~\(L_{k}^{V}(F' \setunion v')\)
and~\(R_{k}^{V}(F' \setunion v')\)
based on~\(B_{k}(v',\mathbf{G}^{F})\)
and~\(B_{k}(v',\mathbf{G}_{n}^{F' \setunion v'})\)
instead of~\(B_{k}(v',\mathbf{G})\) and~\(B_{k}(v',\mathbf{G}_{n}^{v})\).

With these definitions we want to bound
\begin{equation}
  \label{eq:covrl:ev}
  \begin{split}
    &\cov_{n}(\Delta_{e}f\Delta_{e}f^{F}, \Delta_{v'}f\Delta_{v'}f^{F'})\\
    &\quad=\cov_{n}(
      (R_{k}^{E}(e)+L_{k}^{E}(e))
      (R_{k}^{E}(F \setunion  e)+L_{k}^{E}(F \setunion e)),\\
    &\qqquad\qquad  (R_{k}^{V}(v')+L_{k}^{V}(v'))
      (R_{k}^{V}(F' \setunion v')+L_{k}^{V}(F' \setunion v'))
    ).
  \end{split}
\end{equation}
Expand the expression on the right into sixteen terms of the form
\[
  \cov_{n}(U_{1}^{E}(e)U_{2}^{E}(F \setunion e),
            U_{3}^{V}(v')U_{4}^{V}(F' \setunion v')),
\]
where~\(U^{E}_{i}\) for~\(i \in \set{1,2}\)
may be either~\(L^{E}_{k}\) or~\(R^{E}_{k}\)
and~\(U^{V}_{i}\) for~\(i \in \set{3,4}\)
may be either~\(L^{V}_{k}\) or~\(R^{V}_{k}\)

Recall that~\(R^{E}_{k}\) corresponds to the difference of~\(\Delta_{e}f\)
and~\(L^{E}_{k}\).
In other words,~\(R^{E}_{k}\) is the error of
approximating~\(\Delta_{e}f\) locally.
Hence, the covariances involving at least one~\(R^{E}_{k}\)
term can be bounded by coupling the neighbourhood to
the limiting Galton--Watson tree
and appealing to \ref{itm:gla:e:conv}.
\begin{lemma}\label{lem:covr:ev:ru}
  \[
    \begin{split}
    &\cov_{n}(U_{1}^{E}(e)U_{2}^{E}(F \setunion e),
         U_{3}^{V}(v')U_{4}^{V}(F' \setunion  v'))\\
    &\quad\leq CJ_{E}J_{V} \frac{W_{u}W_{v}}{n\vartheta} \zeta_{n}(v')^{1/2}
    ((m_{n}^{E}(v,u)\delta_{k}^{E})^{1/2}+\varepsilon_{n,k}(\set{u,v,v'})^{1/4}
    +\rho_{n,k}(\set{u,v,v'})^{1/4})
  \end{split}
  \]
  if~\(U^{E}_{1}\) or~\(U^{E}_{2}\) is an~\(R^{E}_{k}\)-term.
\end{lemma}
\begin{proof}
  Assume that the covariance is of the form
  \[
    \cov_{n}(R_{k}^{E}(e)U_{2}^{E}(F \setunion e),
         U_{3}^{V}(v')U_{4}^{V}(F' \setunion  v'))
  \]
  the remaining cases are analogous.

  Set
  \[
    \tilde{H}_{E}
    = \max\set{\tilde{H}_{E}(e), \tilde{H}_{E}(F \setunion e)}.
  \]
  By~\eqref{eq:lbound} and~\eqref{eq:rbound}
  and noting that~\(\tilde{A}_{e} = A_{e} \setunion \set{X_{e}=X'_{e}=1}\)
  we have
  \begin{equation}
    \abs{R_{k}^{E}(e)U_{2}^{E}(F \setunion e)
               U_{3}^{V}(v')U_{4}^{V}(F' \setunion  v')}
    \leq C\indfunc_{A_{e}}
      \tilde{H}_{E} \tilde{D}(v')^{2}\tilde{H}_{V}^{2}
      \abs{R_{k}^{E}(e)}
      +C\indfunc_{\set{X_{e}=X'_{e}=1}}
      \tilde{H}_{E}^{2} \tilde{D}(v')^{2}\tilde{H}_{V}^{2}.\label{eq:reev}
  \end{equation}

  The expectation of the
  second term in~\eqref{eq:reev} can be bounded using independence
  and then the Cauchy--Schwarz inequality
  \begin{align*}
    C\expe_{n}[\indfunc_{\set{X_{e}=X'_{e}=1}}
      \tilde{H}_{E}^{2} \tilde{D}(v')^{2}\tilde{H}_{V}^{2}]
    &\leq C\prob_{n}(X_{e}=X'_{e}=1) \expe_{n}[\tilde{H}_{E}^{2}]
    \expe_{n}[\tilde{D}(v')^{2}\tilde{H}_{V}^{2}]\\
    &\leq C\min\set[\bigg]{\frac{W_{v}W_{u}}{n\vartheta},1}^{\!\!2}
      J_{E}^{1/3}J_{V}^{1/3} \zeta_{n}(v')^{1/2}\\
    &\leq J_{E}^{1/3}J_{V}^{1/3} \zeta_{n}(v')^{1/2}
      \frac{W_{u}W_{v}}{n\vartheta} \rho_{n,k}(\set{u,v,v'}).
  \end{align*}

  Let~\(Y_{e} = (X_{e},X'_{e})\).
  Then the indicator functions of the first term in~\eqref{eq:reev}
  are~\(Y_{e},Y_{e'}\)-measurable.
  The conditional expectation of the remainder of the term
  can be bounded with the Cauchy--Schwarz inequality
  so that by independence of~\(\mathbf{w}\), \(\mathbf{w}'\)
  and~\(\mathbf{X}\), \(\mathbf{X}'\),
  we get
  \begin{equation*}
    \expe_{n}[\indfunc_{A_{e}}
      \tilde{H}_{E} \tilde{D}(v')^{2}\tilde{H}_{V}^{2}
    \abs{R_{k}^{E}(e)} \given Y_{e}]
    \leq \indfunc_{A_{e}}
      J_{E}^{1/6}J_{V}^{1/3} \zeta_{n}(v')^{1/2}
    \expe_{n}[R_{k}^{E}(e)^{2} \given Y_{e}]^{1/2}.
  \end{equation*}
  It remains to bound the second moment of~\(R_{k}^{E}(e)\)
  conditional on~\(Y_{e}\).
  Recall that~\(R_{k}^{E}(e) = \Delta_{e}f-L_{k}^{E}(e)\).
  Write~\(B_{k} = B_{k}(v,\mathbf{G}_{n})\)
  and~\(B'_{k} = B_{k}(v,\mathbf{G}_{n}^{e})\).
  Then \ref{itm:gla:e:delta} and~\eqref{eq:deltahbound:e}
  ensure that on~\(E_{0}(e)\) we have
  \[
    0
    \leq  \Delta_{e}f-L_{k}^{E}(e)
    \leq  \mathrm{LA}^{E,U}_{k}(B_{k},B'_{k})
         -\mathrm{LA}^{E,L}_{k}(B_{k},B'_{k}).
  \]
  In particular~\(\abs{R_{k}^{E}(e)}
    \leq \mathrm{LA}^{E,U}_{k}(B_{k},B'_{k})
    -\mathrm{LA}^{E,L}_{k}(B_{k},B'_{k})\)
  on~\(E_{0}(e')\) and so
  together with~\eqref{eq:rbound}
  we have
  \begin{equation*}
    \indfunc_{A_{e}}\abs{R_{k}^{E}(e)}
    \leq \indfunc_{A_{e}}\indfunc_{E_{0}(e)} \indfunc_{E_{1}}
    (\mathrm{LA}_{k}^{E,U}(B_{k},B'_{k})-\mathrm{LA}_{k}^{E,L}(B_{k},B'_{k}))
      +\indfunc_{A_{e}}
      C \tilde{H}_{E} \indfunc_{\setcomplement{E_{0}(e)} \setunion \setcomplement{E_{1}}}
  \end{equation*}
  for an arbitrary set~\(E_{1}\) that will be chosen later.
  Taking conditional expectations, applying Cauchy--Schwarz on the second
  term
  and using the moment bound~\eqref{eq:jbound:e} for~\(H_{E}\)
  and then~\((x+y)^{1/2} \leq x^{1/2}+y^{1/2}\)
  we see
   \begin{equation}
    \begin{split}
    \indfunc_{A_{e}}\expe_{n}[(R_{k}^{E}(e))^{2} \given Y_{e}]
    &\leq
      \expe_{n}[\indfunc_{A_{e}}
      \indfunc_{E_{0}(e)} \indfunc_{E_{1}}  (\mathrm{LA}_{k}^{E,U}(B_{k},B'_{k})
      -\mathrm{LA}_{k}^{E,L}(B_{k},B'_{k}))^{2} \given Y_{e}]\\
      &\quad+ \indfunc_{A_{e}}
      CJ_{E}^{1/3} (\prob_{n}(\setcomplement{E_{0}(e)}
          \given Y_{e})^{1/2}
      +\prob_{n}(\setcomplement{E_{1}} \given Y_{e})^{1/2}).
    \end{split}\label{eq:reksquarecond:ev}
  \end{equation}

  Apply \autoref{lem:probnottree:eep} and recall the definition
  of~\(\rho_{n,k}(\set{u,v})\) to obtain
  \[
    \prob_{n}(E_{0}(e)^{c} \given Y_{e})
    \leq C \rho_{n,k}(\set{u,v})
    \leq C \rho_{n,k}(\set{u,v,v'}).
  \]
  For the second probability in the second term of~\eqref{eq:reksquarecond:ev}
  we use \autoref{lem:couplegngn} and~\autoref{lem:bvbucoupl} to
  couple~\((B_{k},B'_{k},\mathbf{T},\mathbf{T}')\)
  and set~\(E_{1} = \set{(B_{k},B'_{k}) \cong (\mathbf{T},\mathbf{T}')}\)
  such that
  \begin{equation*}
      \prob_{n}(E_{1} \given Y_{e})
      \geq
      1-\Bigl(
      \varepsilon_{n,k}(\set{u,v})
      + 2d_{\mathrm{TV}}(\mu_{E,n},\mu_{E})
      + C\frac{W_{u}W_{v}}{n\vartheta}
      (\Gamma_{2,n}+1)^{2k}
      \Bigr).
  \end{equation*}
  Absorb~\(2d_{\mathrm{TV}}(\mu_{E,n},\mu_{E})\)
  into~\(\varepsilon_{n,k}(\set{u,v,u',v})\)
  and recall the definition of~\(\rho_{n,k}(\set{u,v,u',v'})\)
  to conclude
  \begin{equation}\label{eq:e1c}
    \prob_{n}(E_{1}^{c} \given Y_{e})
    \leq C (\varepsilon_{n,k}(\set{u,v,v'})+\rho_{n,k}(\set{u,v,v'})).
  \end{equation}

  \autoref{lem:couplegngn} and \ref{itm:gla:e:conv}
  then give that
  \[
    \expe_{n}[\indfunc_{A_{e}}
    \indfunc_{E_{0}(e)} \indfunc_{E_{1}}  (\mathrm{LA}_{k}^{E,U}(B_{k},B'_{k})
    -\mathrm{LA}_{k}^{E,L}(B_{k},B'_{k}))^{2} \given Y_{e}]
    \leq m^{E}_{n}(v,u)\delta_{k}^{E}.
  \]

  Hence, the expectation of the first part of~\eqref{eq:reev} can be bounded by
  \begin{equation*}
    \begin{split}
    &\expe_{n}[C\indfunc_{A_{e}}
      \tilde{H}_{E}\tilde{H}_{V}^{2}
      \abs{R_{k}^{E}(e)}]\\
      &\quad\leq  \frac{W_{u}W_{v}}{n\vartheta} CJ_{E}J_{V} \zeta_{n}(v')^{1/2}
      ((m_{n}^{E}(v,u)\delta_{k}^{E})^{1/2}+\varepsilon_{n,k}(\set{u,v,v'})^{1/4}
      +\rho_{n,k}(\set{u,v,v'})^{1/4}).
    \end{split}
  \end{equation*}

  Putting this together we obtain
  \[
    \begin{split}
      &\expe_{n}[\abs{
      R_{k}^{E}(e)U_{2}^{E}(F \setunion e)
      U_{3}^{V}(v')U_{4}^{V}(F' \setunion  v')}]\\
    &\quad\leq CJ_{E}J_{V} \frac{W_{u}W_{v}}{n\vartheta} \zeta_{n}(v')^{1/2}
    ((m_{n}^{E}(v,u)\delta_{k}^{E})^{1/2}+\varepsilon_{n,k}(\set{u,v,v'})^{1/4}
    +\rho_{n,k}(\set{u,v,v'})^{1/4}).
  \end{split}
  \]
  The other two terms in the covariance can be bounded similarly,
  which proves the claim.
\end{proof}

The terms involving an~\(R^{V}_{k}\) can be bounded similarly by
appealing to \ref{itm:gla:v:conv}.
\begin{lemma}\label{lem:covr:ev:ur}
  \[
    \begin{split}
    &\cov_{n}(U_{1}^{E}(e)U_{2}^{E}(F \setunion e),
         U_{3}^{V}(v')U_{4}^{V}(F' \setunion  v'))\\
         &\quad\leq C J_{V}J_{E} \frac{W_{u}W_{v}}{n\vartheta} \zeta_{n}(v')^{1/2}
         ((m^{V}_{n}(v)\delta_{k}^{V})^{1/2}+\varepsilon_{n,k}(\set{u,v,v'})^{1/4}
         +\rho_{n,k}(\set{u,v,v'})^{1/4})
    \end{split}
  \]
  if~\(U^{V}_{3}\) or~\(U^{V}_{4}\) is an~\(R^{V}_{k}\)-term.
\end{lemma}
\begin{proof}
  Assume that the covariance is of the form
  \[
    \cov_{n}(U_{1}^{E}(e)U_{2}^{E}(F \setunion e),
         R_{k}^{V}(v')U_{4}^{V}(F' \setunion  v'))
  \]
  the remaining cases are analogous.

  By~\eqref{eq:lbound}, \eqref{eq:rbound},
  measurability, independence and Cauchy--Schwarz
  we have
  \begin{align}
    \expe_{n}[\abs{U_{1}^{E}(e)U_{2}^{E}(F \setunion e)
         R_{k}^{V}(v')U_{4}^{V}(F' \setunion  v')}
     \given Y_{e}
      ]
    &\leq C\indfunc_{\tilde{A}_{e}}
      \expe_{n}[\tilde{H}_{E}^{2} \tilde{D}(v') H_{V}(w_{v'},w'_{v'})
      \abs{R_{k}^{V}(v')}
      \given Y_{e}
      ]\notag\\
    &\leq C \indfunc_{\tilde{A}_{e}} J_{E}^{1/3}J_{V}^{1/6} \zeta_{n}(v')^{1/4}
      \expe_{n}[\abs{R_{k}^{V}(v')}^{2}
       \given Y_{e'}
      ]^{1/2}.\label{eq:rvpsq}
  \end{align}
  Let~\(B_{k} = B_{k}(v',\mathbf{G}_{n})\)
  and~\(B'_{k} = B_{k}(v',\mathbf{G}_{n}^{v})\).
  Use \autoref{lem:bvbucoupl}
  to couple~\(B_{k}(v',\mathbf{G}_{n})\)
  and~\(\mathbf{T} \sim \mathbf{T}_{k}(W_{v'},\nu,\mu_{E},\mu_{V})\).
  Exchange the weight of the root of~\(\mathbf{T}\)
  for a random variable~\(\tilde{w}_{v'}\) with distribution~\(\mu_{V}\) coupled
  to the weight of~\(v'\) in~\(\mathbf{G}_{n}^{v'}\)
  such that~\(\tilde{w}_{v'} \neq w_{v'}\)
  with probability at most~\(d_{\mathrm{TV}}(\mu_{V,n},\mu_{V})\)
  and call the resulting weighted tree~\(\mathbf{\tilde{T}}\).

  Let~\(E_{1}\) be the event that~\(B_{k} \cong \mathbf{T}\)
  and~\(B'_{k} \cong \mathbf{\tilde{T}}\).
  Then
  \begin{align*}
    \abs{R_{k}^{V}(v')}
    &\leq \mathrm{LA}_{k}^{V,U}(B_{k},B'_{k})-\mathrm{LA}_{k}^{V,L}(B_{k},B'_{k})\\
    &\leq \indfunc_{E_{0}(v')}\indfunc_{E_{1}}(\mathrm{LA}_{k}^{V,U}(B_{k},B'_{k})
          -\mathrm{LA}_{k}^{V,L}(B_{k},B'_{k}))
      + \indfunc_{E_{0}(v')^{c} \setunion E_{1}^{c}} C\tilde{D}(v')\tilde{H}_{V}.
  \end{align*}
  Square this inequality, take conditional expectations,
  use independence and the Cauchy--Schwarz inequality to find
  \begin{align}
    \expe[(R_{k}^{V}(v'))^{2}
    \given Y_{e}
    ]
    &\leq \expe_{n}[\indfunc_{E_{0}(v')}\indfunc_{E_{1}}
      (\mathrm{LA}_{k}^{V,U}(B_{k},B'_{k})-\mathrm{LA}_{k}^{V,L}(B_{k},B'_{k}))^{2}
       \given Y_{e}
      ]\notag\\
    &\quad  + C \expe_{n}[\tilde{D}(v')^{2}\tilde{H}_{V}^{2}
      \indfunc_{E_{0}(v')^{c} \setunion E_{1}^{c}}
       \given Y_{e'}
      ]\notag\\
    \begin{split}
      &\leq \expe_{n}[
      \indfunc_{E_{0}(v')}\indfunc_{E_{1}}
      (\mathrm{LA}_{k}^{V,U}(B_{k},B'_{k})-\mathrm{LA}_{k}^{V,L}(B_{k},B'_{k}))^{2}
      \given Y_{e}
      ]\\
    &\quad  +C
      \zeta_{n}(v')^{1/2} J_{V}^{1/3}
      (\prob_{n}(E_{0}(v')^{c}
      \given Y_{e}
      )
      + \prob_{n}(E_{1}^{c}
      \given Y_{e}
      ))^{1/2}.
    \end{split}\label{eq:rvps:alone}
  \end{align}

  By \autoref{lem:probnottree:eep} and the definition of~\(\rho_{n,k}\) we have
  \begin{equation}
    \prob_{n}(E_{0}(v')^{c} \given Y_{e})
    \leq \rho_{n,k}(\set{u,v,v'}).\label{eq:e0c:ev}
  \end{equation}
  By construction of the coupling~\(B_{k} \cong \mathbf{T}\)
  implies~\(B'_{k} \cong \bar{\mathbf{T}}\)
  unless~\(\tilde{w}_{v} \neq w_{v}\).
  Then by \autoref{lem:bkvindepofep}
  \begin{align*}
    \prob_{n}(\setcomplement{E_{1}} \given Y_{e})
    &\leq \prob_{n}(B_{k} \ncong \mathbf{T} \given Y_{e})
      +\prob_{n}(\tilde{w}_{v} \neq w_{v})\notag\\
    &\leq   \varepsilon_{n,k}(\set{v'})
  + C\frac{W_{v'}(W_{v}+W_{u})}{n\vartheta}
      (\Gamma_{2,n}+1)^{k+1} + d_{\mathrm{TV}}(\mu_{V,n},\mu_{V}).
   \end{align*}
   Absorb \(d_{\mathrm{TV}}(\mu_{V,n},\mu_{V})\) into~\(
   \varepsilon_{n,k}(\set{v'})
   \leq \varepsilon_{n,k}(\set{u,v,v'})\)
   to obtain
  \begin{equation}
     \prob_{n}(\setcomplement{E_{1}} \given Y_{e})
     \leq C (\varepsilon_{n,k}(\set{u,v,v'})
     +\rho_{n,k}(\set{u,v,v'})).\label{eq:e1c:ev}
   \end{equation}
   In fact the construction of \autoref{lem:bkvindepofep} allows
   us to assume that~\(\mathbf{T}\) does not depend on~\(Y_{e}\) at all.

  On~\(E_{1}\) the neighbourhoods~\(B_{k}\) and~\(B'_{k}\)
  can be replaced with~\(\mathbf{T}\) and~\(\mathbf{\bar{T}}\),
  which are independent of~\(Y_{e}\).
  Then by \ref{itm:gla:v:conv}
  \begin{align}
    \expe_{n}[
      \indfunc_{E_{0}(v')}\indfunc_{E_{1}}
    (\mathrm{LA}_{k}^{V,U}(B_{k},B'_{k})-\mathrm{LA}_{k}^{V,L}(B_{k},B'_{k}))^{2}
    \given Y_{e}
      ]
    &\leq \expe_{n}[
      (\mathrm{LA}_{k}^{V,U}(\mathbf{T},\mathbf{\bar{T}})
    -\mathrm{LA}_{k}^{V,L}(\mathbf{T},\mathbf{\bar{T}}))^{2}]\notag\\
    &\leq m^{V}_{n}(v')\delta_{k}^{V}.\label{eq:vedelta}
  \end{align}

  Putting~\eqref{eq:rvpsq}, \eqref{eq:rvps:alone}, \eqref{eq:e0c:ev},
  \eqref{eq:e1c:ev} and~\eqref{eq:vedelta}
  together we have
  \[
    \begin{split}
    &\expe_{n}[\abs{
      U_{1}^{E}(e)U_{2}^{E}(F \setunion e)
      R_{k}^{V}(v')U_{4}^{V}(F' \setunion  v')}]\\
    &\quad\leq CJ_{E}J_{V} \frac{W_{u}W_{v}}{n\vartheta} \zeta_{n}(v')^{1/2}
    ((m_{n}^{V}(v')\delta_{k}^{V})^{1/2}
    +\varepsilon_{n,k}(\set{u,v,v'})^{1/4}+\rho_{n,k}(\set{u,v,v'})^{1/4}
    ).
  \end{split}
  \]
  The other two terms in the covariance can be bounded similarly,
  which proves the claim.
\end{proof}

The last of the sixteen covariance terms does not
involve any~\(R^{E}_{k}\) and~\(R_{k}^{V}\) terms
and thus has to be bounded differently.
The key idea here is that~\(L^{E}_{k}(e)\) is a function of the
neighbourhood around~\(v\) and~\(L^{V}_{k}(v')\) a function
of the neighbourhood around~\(v'\).
By sparsity
the neighbourhood around~\(v\) should only be very weakly correlated
to the neighbourhood around~\(v'\) (cf.~\autoref{sec:corr}),
so that~\(L^{E}_{k}(e)\) and~\(L^{V}_{k}(v')\) are also only weakly correlated.
The formal proof proceeds by constructing independent approximations
to~\(L^{E}_{k}(e)\) and~\(L^{V}_{k}(v')\).

\begin{lemma}\label{lem:cov:ve:llll}
  We have
  \begin{equation*}
    \cov_{n}(L_{k}^{E}(e)L_{k}^{E}(F \setunion e),
         L_{k}^{V}(v')L_{k}^{V}(F' \setunion v'))
          \leq C J_{V}^{1/3} J_{E}^{1/3} \frac{W_{u}W_{v}}{n\vartheta}
         \zeta_{n}(v')^{1/2}
       \rho_{n,k}(\set{u,v,v'}).
  \end{equation*}
\end{lemma}

\begin{proof}
  For any vertex~\(v \in V_{n}\) let
  \[
    S_{v} = ((X_{\edge{v}{x}},X'_{\edge{v}{x}},w_{\edge{v}{x}},w'_{\edge{v}{x}})
    _{x \in V_{n}},
    w_{v},w'_{v}).
  \]
  be the collection of random variables of~\(\mathbf{X}\), \(\mathbf{X}'\),
  \(\mathbf{w}\) and~\(\mathbf{w}'\)
  at~\(v\) and edges emanating from~\(v\).
  For an edge~\(e=\edge{u}{v}\) define
  \[
    S_{e} = (S_{v},S_{u})
  \]
  the collection of random variables associated with the two end
  vertices~\(u\) and~\(v\) of the edge~\(e\).
  Let~\(U_{v} = (X_{\edge{v}{x}},X'_{\edge{v}{x}})_{x \in V_{n}}\)
  be the collection of random
  variables that describe the presence or absence of edges emanating from~\(v\).

  Set
  \[
    X^{E}(e) = L_{k}^{E}(e)L_{k}^{E}(F \setunion e)
    \quad\text{and}\quad
    X^{V}(v') = L_{k}^{V}(v')L_{k}^{V}(F' \setunion v')
    .
  \]
  Note that by construction~\(X^{E}(e) = \indfunc_{A_{e}}X^{E}(e)\),
  so that we can add or remove~\(\indfunc_{A_{e}}\) from terms
  involving~\(X^{E}(e)\) at will.

  Apply the law of total variance to obtain
  \begin{align*}
  \cov_{n}(L_{k}^{E}(e)L_{k}^{E}(F \setunion e),
         L_{k}^{V}(v')L_{k}^{V}(F' \setunion v'))
    &=\cov_{n}(X^{E}(e),X^{V}(v'))\\
    &=
    \cov_{n}(\expe_{n}[X^{E}(e) \given S_{e},S_{v'}], \expe_{n}[X^{V}(v')
    \given S_{e},S_{v'}])\\
    &\quad+ \expe_{n}[\cov_{n}(X^{E}(e), X^{V}(v') \given S_{e},S_{v'})]\\
    &=
      \cov_{n}(\indfunc_{A_{e}}\expe_{n}[X^{E}(e) \given S_{e},S_{v'}],
      \expe_{n}[X^{V}(v') \given S_{e},S_{v'}])\\
    &\quad+ \expe_{n}[\indfunc_{A_{e}}
      \cov_{n}(X^{E}(e), X^{V}(v') \given S_{e},S_{v'})].
\end{align*}
  The claim will follow from \autoref{lem:covxexv} and~\autoref{lem:exp:xexv}.
\end{proof}

\begin{lemma}\label{lem:covxexv}
  We have
  \[
      \cov_{n}(\indfunc_{A_{e}}\expe_{n}[X^{E}(e) \given S_{e},S_{v'}],
        \expe_{n}[X^{V}(v') \given S_{e},S_{v'}])
        \leq C J_{V}^{1/3} J_{E}^{1/3}
        \frac{W_{u}W_{v}}{n\vartheta} \zeta_{n}(v')^{1/2}
        \rho_{n,k}(\set{u,v,v'}).
  \]
\end{lemma}
\begin{proof}
  The random variable~\(X^{E}(e)\) can be written as a function of
  \[
    (
      B_{k}(v,\mathbf{G}_{n}),
      B_{k}(v,\mathbf{G}_{n}^{e}),
      B_{k}(v,\mathbf{G}_{n}^{F}),
      B_{k}(v,\mathbf{G}_{n}^{F \setunion e}).
    )
  \]
  Define an approximation~\(\tilde{X}^{E}(e)\) of~\(X^{E}(e)\)
  as the same function applied
  to
  \[
    (
      B_{k}(v,\mathbf{G}_{n}-v'),
      B_{k}(v,\mathbf{G}_{n}^{e}-v'),
      B_{k}(v,\mathbf{G}_{n}^{F}-v'),
      B_{k}(v,\mathbf{G}_{n}^{F \setunion e}-v')
    ).
  \]
  Similarly,~\(X^{V}(v')\) can be written as a function of
  \[
    (
      B_{k}(v',\mathbf{G}_{n}),
      B_{k}(v',\mathbf{G}_{n}^{v'}),
      B_{k}(v',\mathbf{G}_{n}^{F'}),
      B_{k}(v',\mathbf{G}_{n}^{F' \setunion v'})
    )
  \]
  and can be approximated by~\(\tilde{X}^{V}(v')\) that is  defined
  as the same function, but applied to
  \[
      (
        B_{k}(v',\mathbf{G}_{n}-\set{u,v}),
        B_{k}(v',\mathbf{G}_{n}^{v'}-\set{u,v}),
        B_{k}(v',\mathbf{G}_{n}^{F'}-\set{u,v}),
        B_{k}(v',\mathbf{G}_{n}^{F' \setunion v'}-\set{u,v})
      ).
  \]
  By construction~\(\tilde{X}^{E}(e)\) is independent of~\(S_{v'}\)
  and~\(\tilde{X}^{V}(v')\) is independent of~\(S_{e}\).
  Set
  \[
    Z^{E}(e) = \expe_{n}[X^{E}(e) \given S_{e},S_{v'}]
    \quad\text{and}\quad
    \tilde{Z}^{E}(e) = \expe_{n}[\tilde{X}^{E}(e) \given S_{e},S_{v'}]
    = \expe_{n}[\tilde{X}^{E}(e) \given S_{e}]
  \]
  and analogously
  \[
    Z^{V}(v') = \expe_{n}[X^{V}(v') \given S_{e},S_{v'}]
    \quad\text{and}\quad
    \tilde{Z}^{V}(v') = \expe_{n}[\tilde{X}^{V}(v') \given S_{e},S_{v'}]
    = \expe_{n}[\tilde{X}^{V}(v') \given S_{v'}].
  \]
  The bounds from~\eqref{eq:lbound} and~\eqref{eq:rbound}
  imply
  \[
    \abs{X^{E}(e)}, \abs{\tilde{X}^{E}(e)}
    \leq \indfunc_{A_{e}} \tilde{H}_{E}^{2}
  \quad
  \text{and}
  \quad
    \abs{X^{V}(v')}, \abs{\tilde{X}^{V}(v')}
    \leq \tilde{D}(v')^{2}\tilde{H}_{V}^{2}.
  \]
  The conditional versions then satisfy
  \begin{equation*}
    \abs{Z^{E}(e)}
    = \abs{\expe_{n}[X^{E}(e) \given S_{e},S_{v'}]}
    \leq \expe_{n}[\indfunc_{A_{e}} \tilde{H}_{E}^{2} \given S_{e},S_{v'}]
    \leq \indfunc_{A_{e}} \tilde{H}_{E}^{2}
  \end{equation*}
  by measurability of~\(\tilde{H}_{E}\) and~\(A_{e}\)
  and in exactly the same way
  \[
    \abs{\tilde{Z}^{E}(e)} \leq \indfunc_{A_{e}} \tilde{H}_{E}^{2}
  \]
  as well as
  \begin{equation}\label{eq:zvbound}
    \abs{Z^{V}(v')}
    = \abs{\expe_{n}[X^{V}(v') \given S_{e},S_{v'}]}
    \leq \expe_{n}[\tilde{D}(v')^{2}\tilde{H}_{V}^{2} \given S_{e},S_{v'}]
    \leq \tilde{D}(v')^{2}\tilde{H}_{V}^{2}
  \end{equation}
  by measurability of~\(\tilde{H}_{V}\) and~\(\tilde{D}(v')\)
  and
  \[
    \abs{\tilde{Z}^{V}(v')} \leq \tilde{D}(v')^{2}\tilde{H}_{V}^{2}.
  \]

  Split the relevant covariance
  \begin{equation*}
    \begin{split}
      \cov_{n}(\indfunc_{A_{e}} Z^{E}(e),Z^{V}(v'))
      &=
      \cov_{n}(\indfunc_{A_{e}}\tilde{Z}^{E}(e),\tilde{Z}^{V}(v'))\\
      &\quad
      +\cov_{n}(
        \indfunc_{A_{e}}(Z^{E}(e)-\tilde{Z}^{E}(e)),
        \tilde{Z}^{V}(v')
      )\\
      &\quad+\cov_{n}( \indfunc_{A_{e}}\tilde{Z}^{E}(e),
           Z^{V}(v')-\tilde{Z}^{V}(v'))\\
      &\quad+\cov_{n}(\indfunc_{A_{e}}(Z^{E}(e)-\tilde{Z}^{E}(e)),
              (Z^{V}(v')-\tilde{Z}^{V}(v'))).
    \end{split}
  \end{equation*}

  Since~\(\indfunc_{A_{e}} \tilde{Z}^{E}(e)\) is a function of~\(S_{e}\)
  that is completely independent of~\(S_{v'}\)
  and~\(\tilde{Z}^{V}(v')\) is a function of~\(S_{v'}\)
  that is completely independent of~\(S_{e}\),
  the two approximations are independent and thus the first term vanishes.

  Now bound the three remaining covariances.
  We will only show the argument for
  \[
    \cov_{n}(
        \indfunc_{A_{e}}(Z^{E}(e)-\tilde{Z}^{E}(e)),
        \tilde{Z}^{V}(v')
      ),
  \]
  the argument for the other two covariances is similar.

  By~\eqref{eq:zvbound}
  \[
    \abs{\indfunc_{A_{e}}(Z^{E}(e)-\tilde{Z}^{E}(e)) Z^{V}(v')}
    \leq \indfunc_{A_{e}} \tilde{D}(v')^{2} \tilde{H}_{V}^{2}
      \abs{Z^{E}(e)-\tilde{Z}^{E}(e)}.
  \]
  By construction
  \begin{align*}
    \abs{Z^{E}(e)-\tilde{Z}^{E}(e)}
    &\leq \expe_{n}[\abs{X^{E}(e)-\tilde{X}^{E}(e)}
      \given S_{e},S_{v'}]\\
    &\leq 2 \tilde{H}_{E}^{2}
      \prob_{n}(X^{E}(e) \neq \tilde{X}^{E}(e) \given S_{e},S_{v'}).
  \end{align*}

  Similar to previous proofs~\(X^{E}(e) \neq \tilde{X}^{E}(e)\)
  implies that~\(B_{k}(v,\mathbf{G}-v') \neq B_{k}(v,\mathbf{G})\)
  for at least one~\(\mathbf{G}\)
  of~\(\mathbf{G} = \mathbf{G}_{n},\mathbf{G}_{n}^{e},\mathbf{G}_{n}^{F},
  \mathbf{G}_{n}^{F \setunion e}\).
  Then apply \autoref{lem:nbhdcontains}
  \begin{equation*}
    \prob_{n}(B_{k}(v,\mathbf{G}_{n}-v') \neq B_{k}(v,\mathbf{G}_{n})
      \given S_{e},S_{v'})
    \leq \prob_{n}(v' \in  B_{k}(v,\mathbf{G}_{n}) \given S_{e},S_{v'})
    \leq \xi_{k}(\set{v'},\set{u,v},\emptyset).
  \end{equation*}
  Sum these probabilities for all four graphs to~\(\Xi\) with
  \[
    \prob_{n}(X^{E}(e) \neq \tilde{X}^{E}(e) \given S_{e},S_{v'})
    \leq \Xi
  \]
  and
  \[
    \expe_{n}[\tilde{D}(v')^{2}\Xi]
    \leq C \zeta_{n}(v')^{1/2} \min\set[\bigg]{
      \frac{W_{v'}(W_{u}+W_{v})}{n\vartheta}(\Gamma_{2,n}+1)^{k},1}.
  \]

  Take the expectation and use independence to find
  \begin{align*}
    \expe_{n}[\abs{\indfunc_{A_{e}} (Z^{E}(e)-\tilde{Z}^{E}(e)) Z^{V}(v')}]
    &\leq C \expe_{n}[\indfunc_{A_{e}}
      \tilde{D}(v')^{2}\tilde{H}_{V}^{2}
      \tilde{H}_{E}^{2}
      \Xi
      ]\\
    &\leq
      C \frac{W_{u}W_{v}}{n\vartheta}
      J_{V}^{1/3} J_{E}^{1/3} \zeta_{n}(v')^{1/2}
      \rho_{n,k}(\set{u,v,v'}).
  \end{align*}
  The other terms in the covariance can be bounded similarly.
  Hence,
  \[
     \cov_{n}(
        \indfunc_{A_{e}}(Z^{E}(e)-\tilde{Z}^{E}(e)),
        \tilde{Z}^{V}(v')
      )
     \leq C \frac{W_{u}W_{v}}{n\vartheta}
      J_{V}^{1/3} J_{E}^{1/3} \zeta_{n}(v')^{1/2}
      \rho_{n,k}(\set{u,v,v'}).
  \]

  The same bound holds for the other covariances.
  This finishes the proof.
\end{proof}

\begin{lemma}\label{lem:exp:xexv}
  \[
      \expe_{n}[\indfunc_{A_{e}} \cov_{n}(X^{E}(e),X^{V}(v') \given S_{e},S_{v'})]\\
      \leq C  J_{E}^{1/3}J_{V}^{1/3} \frac{W_{u}W_{v}}{n\vartheta} \zeta_{n}(v')^{1/2}
      \rho_{n,k}(\set{u,v,v'}).
  \]
\end{lemma}
\begin{proof}
  Define
    \begin{align*}
    \mathbf{B}_{k}(-\set{u,v})
    &= (B_{k-1}(w,\mathbf{G}_{n}-\set{u,v}),
             B_{k-1}(w,\mathbf{G}_{n}^{F}-\set{u,v}))_{w \in V_{n}},\\
    \mathbf{B}_{k}(-v')
    &= (B_{k-1}(w,\mathbf{G}_{n}-\set{v'}),
             B_{k-1}(w,\mathbf{G}_{n}^{F'}-\set{v'}))_{w \in V_{n}},\\
    \mathbf{B}_{k}(-)
    &= (B_{k-1}(w,\mathbf{G}_{n}-\set{u,v,v'}),
            B_{k-1}(w,\mathbf{G}_{n}^{F}-\set{u,v,v'}))_{w \in V_{n}},\\
    \mathbf{B}'_{k}(-)
    &= (B_{k-1}(w,\mathbf{G}_{n}-\set{u,v,v'}),
         B_{k-1}(w,\mathbf{G}_{n}^{F'}-\set{u,v,v'}))_{w \in V_{n}}.
  \end{align*}
  With this notation~\(X^{E}(e)\) can be written as a function
  of~\((\mathbf{B}_{k}(-\set{u,v}),S_{e})\), because
  the entire~\(k\)-neighbourhood
  of~\(u\) and~\(v\) can easily be obtained
  from~\(S_{e}\) and~\(\mathbf{B}_{k}(-\set{u,v})\).

  We can define an approximation~\(\tilde{X}^{E}(e)\) of~\(X^{E}(e)\)
  as the same function applied to~\((\mathbf{B}_{k}(-),S_{e})\).
  Formally the function is defined taking all
  of~\(\mathbf{B}_{k}(-\set{u,v})\) or~\(\mathbf{B}_{k}(-)\)
  as argument, but given~\(S_{e}\)
  it can be recast into a form that only uses neighbourhoods of
  vertices~\(w\) that are connected
  to~\(u\) or~\(v\) by an edge
  in the original graph~\(\mathbf{G}_{n}\)
  or in~\(\mathbf{G}_{n}^{F}\).
  We will call such vertices~\(w\)s~--~and by extension their
  neighbourhoods~--~\emph{relevant}.
  See  \autoref{fig:uvnbhd} for an illustration.

  \begin{figure}[htbp]
    \centering
    \begin{tikzpicture}[y=1.5cm,x=2cm,
      mynode/.style={draw, circle, inner sep=2pt, minimum size=1em}]
      \node[mynode] (v) at (0,0)    {\(v\)};
      \node[mynode] (u) at (1,0)    {\(u\)};
      \node[mynode] (v1) at (-1,1.5)  {};
      \node[mynode] (v2) at (-1,-1.5)  {};
      \node[mynode, gray] (v3) at (-1,0)  {};
      \node[mynode, gray] (v31) at (-2,.3) {};
      \node[mynode, gray] (v32) at (-2,-.3) {};
      \node[mynode] (v11) at (-2,2)  {};
      \node[mynode] (v12) at (-2,1.5)  {};
      \node[mynode] (v13) at (-2,1)  {};
      \node[mynode] (v21) at (-2,-1.25)  {};
      \node[mynode] (v22) at (-2,-2)  {};
      \node[mynode] (u1) at (2,1)  {};
      \node[mynode] (u2) at (2,0)  {};
      \node[mynode, gray] (u3) at (2,-1)  {};
      \node[mynode] (u11) at (3,1.25)  {};
      \node[mynode] (u12) at (3,0.75)  {};

      \node[mynode, gray] (u31) at (3,-.2)  {};
      \node[mynode, gray] (u32) at (3,-.7)  {};
      \node[mynode, gray] (u33) at (3,-1.2)  {};
      \node[mynode, gray] (u34) at (3,-1.7)  {};

      \draw[dashed] (v)--(u);
      \draw[dashed] (v)--(v1);
      \draw (v1)--(v11);
      \draw (v1)--(v12);
      \draw (v1)--(v13);
      \draw[dashed] (v)--(v2);
      \draw (v2)--(v21);
      \draw (v2)--(v22);
      \draw[dashed] (u)--(u1);
      \draw (u1)--(u11);
      \draw (u1)--(u12);
      \draw[dashed] (u)--(u2);
      \draw[gray] (u3)--(u31);
      \draw[gray] (u3)--(u32);
      \draw[gray] (u3)--(u33);
      \draw[gray] (u3)--(u34);
      \draw[gray] (v3)--(v31);
      \draw[gray] (v3)--(v32);
    \end{tikzpicture}
    \caption{The neighbourhood around~\(v\) in a graph~\(G\)
      can be obtained from the information on edges incident to~\(u\)
      and~\(v\) (dashed)
      and the neighbourhoods (in~\(G-\set{v,u}\))
      of vertices that connect to~\(v\) or~\(u\) via dashed edges
      (solid black).
      Neighbourhoods (in~\(G-\set{v,u}\))
      of vertices that do not connect to~\(v\)
      or~\(u\) are irrelevant (shown in grey).}
    \label{fig:uvnbhd}
  \end{figure}
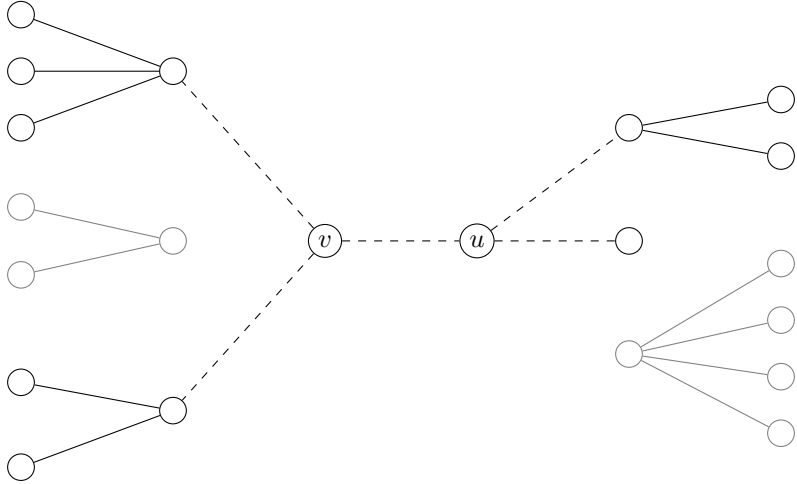

  Similarly~\(X^{V}(v')\) can be written as a function
  of~\((\mathbf{B}_{k}(-v'),S_{v'})\)
  which
  is completely determined by the neighbourhoods~\(B_{k-1}(w,\mathbf{G}_{n})\)
  of vertices~\(w\) which are connected to~\(v'\)
  in~\(\mathbf{G}_{n}\) or~\(\mathbf{G}_{n}^{F}\) via an edge,
  again those vertices are~\emph{relevant}.
  Define an approximation~\(\tilde{X}^{V}(v')\) of~\(X^{V}(v')\)
  as the same function applied to~\((\mathbf{B}_{k}(-),S_{v'})\).

  On the event~\(X^{E}(e) \neq \tilde{X}^{E}(e) \)
  there is at least one relevant vertex~\(w\) connected to~\(u\) or~\(v\)
  such that~\(B_{k-1}(w, \mathbf{G}_{n}-\set{u,v})\) differs
  from~\(B_{k-1}(w, \mathbf{G}_{n}-\set{u,v,v'})\)
  or~\(B_{k-1}(w, \mathbf{G}_{n}^{F}-\set{u,v})\) differs
  from~\(B_{k-1}(w, \mathbf{G}_{n}^{F}-\set{u,v,v'})\).
  That is to say there is a path of length at most~\(k-1\)
  from~\(w\) to~\(v'\) that avoids~\(u\) and~\(v\).
  Since~\(w\) is directly connected to~\(u\) or~\(v\) by an edge,
  it follows that there is a path from~\(u\) or~\(v\)
  to~\(v'\) of no more than~\(k\) steps.
  Thus by \autoref{lem:nbhdcontains}
  \begin{align*}
    \prob_{n}(\tilde{X}^{E}(e) \neq X^{E}(e) \given S_{e},S_{v'})
    &\leq \prob_{n}(\set{u,v} \pathbetw_{\leq k} v')
      +\prob_{n}(\set{u,v} \pathbetw^{F}_{\leq k} v')\\
    &\leq \xi_{k}(\set{u,v},\set{v'},\emptyset)
      +\xi^{F}_{k}(\set{u,v},\set{v'},\emptyset)
  \end{align*}
  Call the right-hand side of the last equation~\(\Xi\)
  and note that
  \[
    \expe_{n}[\tilde{D}(v')^{2}\Xi]
    \leq C \zeta_{n}(v')^{1/2} \frac{(W_{u}+W_{v})W_{v'}}{n\vartheta}
    (\Gamma_{2,n}+1)^{k}.
  \]

  Similarly~\(X^{V}(v')\) differs from~\(\tilde{X}^{V}(v')\)
  only if there is a vertex~\(w\) connected to~\(v'\) by an edge
  such that~\(B_{k-1}(w,\mathbf{G}_{n}-\set{v'})\)
  and~\(B_{k-1}(w,\mathbf{G}_{n}-\set{u,v,v'})\)
  or~\(B_{k-1}(w,\mathbf{G}_{n}^{F'}-\set{v'})\)
  and~\(B_{k-1}(w,\mathbf{G}_{n}^{F'}-\set{u,v,v'})\)
  differ.
  This implies that there is a path in~\(\mathbf{G}_{n}\)
  or~\(\mathbf{G}_{n}^{F'}\) of length at most~\(k-1\)
  from~\(w\) to~\(u\) or~\(v\) that avoids~\(v'\).
  Hence, there is a
  path of length at most~\(k\)
  in~\(\mathbf{G}_{n}\) or~\(\mathbf{G}_{n}^{F'}\)
  from~\(v'\) to~\(u\) or~\(v\).
  Thus by \autoref{lem:nbhdcontains}
  \begin{align*}
    \prob_{n}(\tilde{X}^{V}(v') \neq X^{V}(v') \given S_{e},S_{v'})
    &\leq \prob_{n}(v' \pathbetw_{\leq k} \set{u,v})
      +\prob_{n}(v' \pathbetw^{F}_{\leq k} \set{u,v} )\\
    &\leq \xi_{k}(\set{v'},\set{u,v},\emptyset)
      +\xi^{F}_{k}(\set{v'},\set{u,v},\emptyset)
  \end{align*}
  Call the right-hand side of the last equation~\(\Xi'\)
  and note that
  \[
    \expe_{n}[\tilde{D}(v')^{2}\Xi']
    \leq C \zeta_{n}(v')^{1/2}
    \min\set[\bigg]{\frac{(W_{u}+W_{v})W_{v'}}{n\vartheta}
    (\Gamma_{2,n}+1)^{k},1}.
  \]

  Expand the covariance of interest
  \begin{equation*}
    \begin{split}
      \cov_{n}(X^{E}(e),X^{V}(v') \given S_{e},S_{v'})
      &=
      \cov_{n}(\tilde{X}^{E}(e),\tilde{X}^{V}(v') \given S_{e},S_{v'})\\
      &\quad
      +\cov_{n}(
        X^{E}(e)-\tilde{X}^{E}(e),
        \tilde{X}^{V}(v') \given S_{e},S_{v'}
      )\\
      &\quad+\cov_{n}(\tilde{X}^{E}(e),
            X^{V}(v')-\tilde{X}^{V}(v') \given S_{e},S_{v'})\\
      &\quad+\cov_{n}(X^{E}(e)-\tilde{X}^{E}(e),
            X^{V}(v')-\tilde{X}^{V}(v') \given S_{e},S_{v'}).
    \end{split}
  \end{equation*}

  As previously, bound the covariances involving an approximation error
  by estimating the approximation error by the bounds on~\(X^{E}\)
  and~\(X^{V}\) times the probability that the approximation is different
  from the original function.
  We obtain
  \[
    \abs{
      \cov_{n}(X^{E}(e)-\tilde{X}^{E}(e),X^{V}(v') \given S_{e},S_{v'})
    }
    \leq C \tilde{H}_{E}^{2}\tilde{D}(v')^{2} \tilde{H}_{V}^{2}\Xi
  \]
  and similar results for the other terms.
  Thus
  \[
    \abs{
      \cov_{n}(X^{E}(e),X^{V}(v') \given S_{e},S_{v'})
      -\cov_{n}(\tilde{X}^{E}(e),\tilde{X}^{V}(v') \given S_{e},S_{v'})
    }
    \leq C \tilde{H}_{E}^{2} \tilde{D}(v')^{2} \tilde{H}_{V}^{2}
      (\Xi+\Xi')
  \]
  and therefore also
  \begin{equation}
    \begin{split}
    &\expe_{n}[\indfunc_{A_{e}}\abs{
      \cov_{n}(X^{E}(e),X^{V}(v') \given S_{e},S_{v'})
      -\cov_{n}(\tilde{X}^{E}(e),\tilde{X}^{V}(v') \given S_{e},S_{v'})
    }]\\
    &\quad\leq C J_{E}^{1/3} J_{V}^{1/3}
      \zeta_{n}(v')^{1/2} \rho_{n,k}(\set{u,v,v'}).
    \end{split}
  \end{equation}

  It remains to
  bound~\(\cov_{n}(\tilde{X}^{E}(e),\tilde{X}^{V}(v') \given S_{e},S_{v'})\).
  Recall that the construction ensured that~\(\mathbf{B}_{k}(-)\)
  and~\(\mathbf{B}'_{k}(-)\)
  are independent of~\(S_{e}\) and~\(S_{v'}\).
  Hide the dependence of~\(\tilde{X}^{E}(e)\) on~\(S_{e}\)
  and the dependence of~\(\tilde{X}^{V}(v')\) on~\(S_{v'}\)
  in functions~\(\Psi\) and~\(\Psi'\), respectively,
  so that
  \[
    \cov_{n}(\tilde{X}^{E}(e),\tilde{X}^{V}(v') \given S_{e},S_{v'})
    = \cov_{n}(
      \Psi(\mathbf{B}_{k}(-)),\Psi'(\mathbf{B}_{k}(-))
    ).
  \]
  By definition~\(\abs{\Psi} \leq \tilde{H}_{E}^{2}\)
  and~\(\abs{\Psi'} \leq \tilde{D}(v')^{2}\tilde{H}_{V}^{2}\).

  Recall that~\(\Psi\) and~\(\Psi'\) depend not on
  all~\(w\)-neighbourhoods for~\(w \in V_{n}\),
  but only on the neighbourhoods of relevant~\(w\)s,
  i.e~those~\(w\) for which there is an edge to~\(u\) or~\(v\)
  and~\(v'\), respectively.
  This data is known conditioned on~\(S_{e}\) and~\(S_{v'}\).
  More formally, the relevant vertices for~\(\Psi\) are contained in
  \begin{align*}
    D &= D_{1}(v) \setunion D_{1}(u)
    \setunion D_{1}^{F}(v) \setunion D^{F}_{1}(u),
  \intertext{and the vertices relevant for~\(\Psi'\) are contained in}
    D' &= D_{1}(v')  \setunion D_{1}^{F'}(v').
  \end{align*}
  While~\(D\) is not independent of~\(A_{e}\) and~\(D'\),
  \begin{align*}
    \widebar{D}
    &= S_{1}^{(u,v')}(v) \setunion S_{1}^{(v,v')}(u)
      \setunion S_{1}^{(u,v'),F}(v) \setunion S^{(v,v'),F}_{1}(u)
      \setunion \set{v'}
  \intertext{is independent of~\(A_{e}\) and of}
    \widebar{D}'
    &= S_{1}^{(u,v)}(v') \setunion S_{1}^{(u,v),F'}\!(v') \setunion \set{u,v}.
  \end{align*}
  Furthermore, we clearly have~\(D \subseteq \widebar{D}\)
  and~\(D' \subseteq \widebar{D}'\).
  Independence and \autoref{cor:bvl:weight:squared:raw} imply
  \[
    \expe_{n}[(\setweight{\widebar{D}}+\setcard{\widebar{D}})^{2}
    (\setweight{\widebar{D}'}+\setcard{\widebar{D}'})^{2}]
    \leq C(W_{u}+W_{v}+W_{v'}+2)^{4}(\Gamma_{1,n}+1)^{4}(\Gamma_{2,n}+1)^{4}(\Gamma_{3,n}+1)^{2}.
  \]
  By \autoref{lem:covgbgpbp} and~\(D \subseteq \widebar{D}\),
  \(D' \subseteq \widebar{D}'\)
  we have
  \[
    \begin{split}
    &\cov_{n}(\Psi(\mathbf{B}_{k}(-)),\Psi'(\mathbf{B}_{k}(-)))\\
    &\quad\leq C \tilde{H}_{E}^{2} \tilde{D}(v')^{2}\tilde{H}_{V}
    \min\set[\bigg]{\frac{(\setweight{\widebar{D}}+\setcard{\widebar{D}})
        (\setweight{\widebar{D}'}+\setcard{\widebar{D}'})}{n\vartheta}
      (\Gamma_{3,n}+1)(\Gamma_{2,n}+1)^{2k+1}, 1}.
  \end{split}
  \]
  Take expectations,
  use independence
  and recall the definition of~\(\rho_{n,k}(\set{u,v,v'})\) to obtain
  \begin{equation*}
    \expe_{n}[\indfunc_{A_{e}}
    \cov_{n}(\Psi(\mathbf{B}_{k}(-)),\Psi'(\mathbf{B}_{k}(-)))]
    \leq C J_{V}^{1/3} \frac{W_{u}W_{v}}{n\vartheta} \zeta_{n}(v')^{1/2}
    \rho_{n,k}(\set{u,v,v'}).
  \end{equation*}

  Put these covariance estimates together to conclude the claim.
\end{proof}

\begin{proof}[\proofname\ for the~\(c(e,v')\)-term in \autoref{lem:ceebd:alldiff}]
  Recall~\eqref{eq:covrl:ev}
  \begin{equation*}
    \begin{split}
    \cov_{n}(
      &(R_{k}^{E}(e)+L_{k}^{E}(e))
      (R_{k}^{E}(F \setunion  e)+L_{k}^{E}(F \setunion e)),\\
      &(R_{k}^{V}(v')+L_{k}^{V}(v'))
      (R_{k}^{V}(F' \setunion v')+L_{k}^{V}(F' \setunion v'))
    ).
  \end{split}
  \end{equation*}
  Expand this covariance into sixteen terms,
  then apply \autoref{lem:covr:ev:ru},
  \autoref{lem:covr:ev:ur} and~\autoref{lem:cov:ve:llll}
  to these terms as appropriate.
\end{proof}

We are almost ready to prove our main result.

\begin{lemma}\label{lem:epssum}\label{lem:rhosum}
  Recall~\(\eta_{n,\ell}(\mathcal{V})\),
  \(\varepsilon_{n,\ell}(\mathcal{V})\)
  and~\(\rho_{n,\ell}(\mathcal{V})\)
  from \autoref{lem:nbhdcoupl:noweights}, \autoref{lem:bvbucoupl}
  and~\autoref{def:rho},
  respectively.
  Then
  \[
    \sum_{\substack{u_{1},\dots,u_{m} \in V_{n}\\\text{pairw. diff.}}}
    \varepsilon_{n,\ell}(\set{u_{1},\dots,u_{m}})
    \leq mn^{m} \varepsilon_{n,\ell}
    \quad\text{and}\quad
    \sum_{\substack{u_{1},\dots,u_{m} \in V_{n}\\\text{pairw. diff.}}}
    \rho_{n,\ell}(\set{u_{1},\dots,u_{m}})
    \leq mn^{m} \rho_{n,\ell},
  \]
  where~\(\varepsilon_{n,\ell}\) and~\(\rho_{n,\ell}\)
  are defined as in \autoref{def:epsrho}.

  For any~\(r > 1\) we also have
  \[
    \frac{1}{n^{m}}
    \sum_{\substack{u_{1},\dots,u_{m} \in V_{n}\\\text{pairw. diff.}}}
    \varepsilon_{n,\ell}(\set{u_{1},\dots,u_{m}})^{1/r}
    \leq m^{1/r}\varepsilon_{n,\ell}^{1/r}
  \]
  and
  \[
    \frac{1}{n^{m}}\sum_{\substack{u_{1},\dots,u_{m} \in V_{n}\\\text{pairw. diff.}}}
    \rho_{n,\ell}(\set{u_{1},\dots,u_{m}})^{1/r}
    \leq m^{1/r} \rho_{n,\ell}^{1/r}.
  \]
\end{lemma}
\begin{proof}
  We will only show the calculations for~\(\varepsilon_{n,\ell}\)
  \citep[the calculations for~\(\rho_{n,\ell}\)
  can be found in][Lem.~4.4.2]{thesis}.

  We have
  \begin{align*}
    \varepsilon_{n,\ell}(\mathcal{V})
    &= \setweight{\mathcal{V}}_{2} \frac{\Gamma_{2,n}}{n\vartheta}
    +\setweight{\mathcal{V}}_{+}\Gamma_{1,n}\\
  &\quad  +\setweight{\mathcal{V}}(\Gamma_{2,n}+1)^{\ell}
    \biggl(
      \frac{\Gamma_{3,n}}{n\vartheta}+\kappa_{1,n}+\kappa_{2,n}+
      \frac{2+\Gamma_{1,n}}{k_{n}}+\frac{k_{n}}{n\vartheta}\biggr)\\
  &\quad + \setcard{\mathcal{V}}\frac{1}{k_{n}}
    +\frac{k_{n}^{2}}{n\vartheta\Gamma_{1,n}}
    +\setweight{\mathcal{V}}\alpha_{n}
    \biggl(
      \frac{1}{\vartheta}
      + (\Gamma_{2}+1)^{\ell-1}
        \biggl(\frac{\Gamma_{2,n}}{\vartheta\Gamma_{1,n}}+1\biggr)
    \biggr)\\
    &\quad+
       (\setcard{\mathcal{V}}+\setweight{\mathcal{V}}(\Gamma_{2}+1)^{\ell})
    (d_{\mathrm{TV}}(\mu_{E,n},\mu_{E})
    + d_{\mathrm{TV}}(\mu_{V,n},\mu_{V})
    ).
  \end{align*}
  In order to understand the sum over~\(\varepsilon_{n,\ell}(\mathcal{V})\)
  it is therefore enough to understand the sum over~\(\setcard{\mathcal{V}}\),
  \(\setweight{\mathcal{V}}\), \(\setweight{\mathcal{V}}_{+}\)
  and~\(\setweight{\mathcal{V}}_{2}\).
  Recall that~\(\setcard{\mathcal{V}} = \setweight{\mathcal{V}}_{0}\).

  For~\(p \geq 0\) we have
  \begin{equation*}
    \sum_{\substack{u_{1},\dots,u_{m} \in V_{n}\\\text{pairw. diff.}}}
      \setweight{\set{u_{1},\dots,u_{m}}}_{p}
    = \sum_{\substack{u_{1},\dots,u_{m} \in V_{n}\\\text{pairw. diff.}}}
    \sum_{i=1}^{m}W_{u_{i}}^{p}
    \leq m \sum_{u_{1},\dots,u_{m} \in V_{n}} W_{u_{1}}^{p}
    = mn^{m} \vartheta\Gamma_{p,n}.
  \end{equation*}
  In exactly the same way we also obtain
  \begin{equation*}
    \sum_{\substack{u_{1},\dots,u_{m} \in V_{n}\\\text{pairw. diff.}}}
      \setweight{\set{u_{1},\dots,u_{m}}}_{+}
    = \sum_{\substack{u_{1},\dots,u_{m} \in V_{n}\\\text{pairw. diff.}}}
    \sum_{i=1}^{m}W_{u_{i}} \indfunc_{\set{W_{i} > \sqrt{n\vartheta}}}
    \leq mn^{m} \vartheta\kappa_{1,n}.
  \end{equation*}

  Hence
  we can estimate the sum of~\(\varepsilon_{n,\ell}(\set{u_{1},\dots,u_{m}})\)
  over all subsets
  by formally replacing~\(\setweight{\mathcal{V}}_{p}\)
  in~\(\varepsilon_{n,\ell}(\mathcal{V})\)
  with~\(mn^{m}\vartheta \Gamma_{p,n}\)
  and by replacing~\(\setweight{\mathcal{V}}_{+}\)
  with~\(mn^{m}\vartheta \kappa_{1,n}\).
  We obtain
  \[
    \sum_{\substack{u_{1},\dots,u_{m} \in V_{n}\\\text{pairw. diff.}}}
    \varepsilon_{n,\ell}(\set{u_{1},\dots,u_{m}})
    \leq mn^{m} \varepsilon_{n,\ell}
  \]
  This finishes the proof of the first part of the claim.

  The second part of the claim follows from Hölder's inequality.
\end{proof}

Finally, we can proceed to prove the main result.
\begin{proof}[\proofname\ of \autoref{thm:fgnconv}]
  In the previous sections
  we identified a
  function~\(c\)
  defined on both vertices and edges
  with
  \[
    \sigma_{n}^{-4}
    \cov(\Delta_{x}f\Delta_{x}f^{F},\Delta_{x'}f\Delta_{x'}f^{F'})
    \leq c(x,x')
  \]
  for all~\(x,x' \in V_{n} \setunion V_{n}^{(2)}\)
  and~\(F \subseteq (V_{n} \setunion V_{n}^{(2)})\setminus \set{x}\),
  \(F' \subseteq (V_{n} \setunion V_{n}^{(2)})\setminus \set{x'}\).
  Apply \autoref{lem:steinappl} to obtain that
  \begin{equation*}
    \sup_{t \in\reals}\abs{\prob_{n}(Z_{n} \leq t)-\Phi(t)}
    \leq \sqrt{2} \biggl( \sum_{x,x' \in V_{n} \setunion V_{n}^{(2)}}
      c(x,x') \biggr)^{\!\!1/4}
    + \biggl(
        \sigma_{n}^{-3} \sum_{x\in V_{n} \setunion E_{n}}
          \expe[\abs{\Delta_{x}f}^{3}]
    \biggr)^{\!\!1/2}\!.
  \end{equation*}

  All that is left to do is to calculate these sums.
  We start with the first sum over~\(c(x,x)\)
  and split the sum over all the different cases we handled
  in \autoref{lem:ceebd:same:all} and \autoref{lem:ceebd:alldiff}.

  For the terms from
  \autoref{lem:ceebd:same:all}
  we apply Cauchy--Schwarz to separate
  the sums over different vertices and then apply
  the definition of~\(\Gamma_{p,n}\)
  and~\(\chi_{n}\) from~\eqref{eq:nubound}
  to obtain
  \begin{equation}
    \sum_{e \in E_{n}} c(e,e)
    \leq \sigma_{n}^{-4} C J_{E}^{2/3}
      \sum_{e \in E_{n}} \frac{W_{u}W_{v}}{n\vartheta}
    \leq \sigma_{n}^{-4} C J_{E}^{2/3} n\vartheta
      \biggl(\frac{1}{n\vartheta}\sum_{v \in V_{n}}W_{v} \biggr)^{2}
    \leq \frac{n}{\sigma_{n}^{4}} CJ_{E}^{2/3} \vartheta\Gamma_{1,n}^{2}
      \label{eq:ceess}
  \end{equation}
  and similarly
  \begin{equation}
    \sum_{v \in V_{n}} c(v,v)
    \leq \sigma_{n}^{-4} C J_{V}^{2/3}
      \sum_{v \in V_{n}} \zeta_{n}(v)
    \leq \frac{n}{\sigma_{n}^{4}} CJ_{V}^{2/3} \chi_{n},
      \label{eq:cvvss}
  \end{equation}
  as well as
  \begin{align}
    \sum_{v,u \in V_{n}} c(\edge{v}{u},v)
    &\leq \sigma_{n}^{-4} C J_{E}^{1/3}J_{V}^{1/3}
      \sum_{v,u \in V_{n}} \frac{W_{u}W_{v}}{n\vartheta} \zeta_{n}(v)^{1/2}\notag\\
     &\leq \frac{n}{\sigma_{n}^{4}} C J_{E}^{1/3}J_{V}^{1/3}
      \biggl(\frac{1}{n\vartheta} \sum_{u \in V_{n}} W_{u}\biggr)
       \biggl(\frac{1}{n} \sum_{v \in V_{n}} W_{v}^{2}\biggr)^{1/2}
       \biggl(\frac{1}{n} \sum_{v \in V_{n}} \zeta_{n}(v)\biggr)^{1/2}\notag\\
    &\leq \frac{n}{\sigma_{n}^{4}} C J_{E}^{1/3}J_{V}^{1/3}
      \Gamma_{1,n} \vartheta^{1/2} \Gamma_{2,n}^{1/2} \chi_{n}^{1/2}
      \label{eq:cevs}
  \end{align}
  and
  \begin{align}
    \sum_{u,v,v' \in V_{n}} c(\edge{u}{v},\edge{u}{v'})
    &\leq \sigma_{n}^{-4} C J_{E}^{2/3}
      \sum_{u,v,v' \in V_{n}} W_{u}^{2} \frac{W_{v}}{n\vartheta}
      \frac{W_{v'}}{n\vartheta}\notag\\
    &\leq \frac{n}{\sigma_{n}^{4}} CJ_{E}^{2/3}
      \vartheta \biggl(\frac{1}{n\vartheta}\sum_{u} W_{u}^{2}\biggr)
      \biggl(\frac{1}{n\vartheta}\sum_{v} W_{v}\biggr)
      \biggl(\frac{1}{n\vartheta}\sum_{v'} W_{v'}\biggr)\notag\\
    &\leq   \frac{n}{\sigma_{n}^{4}} CJ_{E}^{2/3} \vartheta \Gamma_{2,n}
      \Gamma_{1,n}^{2}.\label{eq:cuvuvp}
  \end{align}

  The bounds~\eqref{eq:ceess}, \eqref{eq:cvvss},
  \eqref{eq:cuvuvp} and~\eqref{eq:cevs}
  can in turn be bounded above by
  \begin{align}
    \frac{n}{\sigma_{n}^{4}} C(J_{E}+J_{V}+J_{E}J_{V})
      (\vartheta^{1/2}+\chi_{n}^{1/2})^{2} (\Gamma_{1,n}+1)^{2}(\Gamma_{2,n}+1)
    &\leq \frac{n^{2}}{\sigma_{n}^{4}} CJ (\vartheta^{1/2}+\chi_{n}^{1/2})^{2}
      \rho_{n,k}\notag\\
    & \leq \frac{n^{2}}{\sigma_{n}^{4}} CJ (\vartheta^{1/2}+
      \Gamma_{2,n}+ \chi_{n}^{1/2})^{2}
      \rho_{n,k}.\label{eq:ceasybounds}
  \end{align}

  For the more complex bound from \autoref{lem:ceebd:alldiff}
  we apply Cauchy--Schwarz to separate the connectivity weights,
  the~\(\delta\), \(\rho\) and~\(\varepsilon\) terms in the sum.
  Then apply \autoref{lem:rhosum} and \autoref{lem:epssum} and the definition
  of~\(m^{E}_{n}\) and~\(M^{E}_{n}\) to obtain
  \begin{align}
    \sum_{\substack{e,e' \in E_{n}\\e \setintersect e' = \emptyset}} c(e,e')\notag
    &\leq \sigma_{n}^{-4} CJ_{E}
    \sum_{\substack{u,v,u',v' \in E_{n}\\\text{pairwise different}}}
    \frac{W_{u}W_{v}}{n\vartheta}\frac{W_{u'}W_{v'}}{n\vartheta}
    ((\delta_{k}^{E}(v,u)+\delta_{k}^{E}(v',u'))^{1/2}\notag\\
    &\qquad+\varepsilon_{n,k}(\set{u,v,u',v'})^{1/4}
      +\rho_{n,k}(\set{u,v,u',v'})^{1/4})\notag\\
    &\leq \frac{n^{2}}{\sigma_{n}^{4}} CJ_{E}
      \Bigl(\frac{1}{n^{4}\vartheta^{4}}
      \sum_{u,v,u',v'} W_{u}^{2}W_{v}^{2}W_{u'}^{2}W_{v'}^{2}\Bigr)^{1/2}
    \Bigl(\frac{1}{n^{4}}
      \sum_{u,v,u',v'} m_{n}^{E}(v,u)\delta_{k}^{E}+m_{n}^{E}(v',u')\delta_{k}^{E}
      \notag\\
   &\qquad +\varepsilon_{n,k}(\set{u,v,u',v'})^{1/2}
      +\rho_{n,k}(\set{u,v,u',v'})^{1/2}\Bigr)^{1/2}\notag\\
    &\leq \frac{n^{2}}{\sigma_{n}^{4}} CJ_{E} \Gamma_{2,n}^{2}
      \Bigl(\delta_{k}^{E}\frac{2}{n^{2}}\sum_{u,v}m_{n}^{E}(v,u)
      + \frac{1}{n^{4}}
      \sum_{u,v,u',v'} \varepsilon_{n,k}(\set{u,v,u',v'})^{1/2}\notag\\
    &\qquad  + \frac{1}{n^{4}}\sum_{u,v,u',v'} \rho_{n,k}(\set{u,v,u',v'})^{1/2}
      \Bigr)^{1/2}\notag\\
    &\leq \frac{n^{2}}{\sigma_{n}^{4}} CJ_{E}\Gamma_{2,n}^{2}
      ((M^{E}_{n}\delta^{E}_{k})^{1/2}+\varepsilon_{n,k}^{1/4}+\rho_{k,n}^{1/4}).
      \label{eq:cee:diff}
  \end{align}
  Similarly we have
  \begin{align}
    \sum_{\substack{v,v' \in V_{n}\\v \neq v'}} c(v,v')
    &\leq \sigma_{n}^{-4}  CJ_{V}
      \sum_{\substack{v,v' \in V_{n}\\v \neq v'}}
    \zeta_{n}(v)^{1/2}\zeta_{n}(v')^{1/2}
    ((m^{V}_{n}(v)\delta_{k}^{V})^{1/2}+(m^{V}_{n}(v')\delta_{k}^{V})^{1/2}+\notag\\
   &\qquad \varepsilon_{n,k}(\set{v,v'})^{1/4}
     +\rho_{n,k}(\set{v,v'})^{1/4})\notag\\
    & \leq \frac{n^{2}}{\sigma_{n}^{4}} C J_{V}
      \Bigl(
        \frac{1}{n^{2}}\sum_{v, v' \in V_{n}}\zeta_{n}(v)\zeta_{n}(v')
      \Bigr)^{1/2}\notag\\
    &\qquad\Bigl(
      \frac{1}{n^{2}}
      \sum_{v, v' \in V_{n}}  (m^{V}_{n}(v)\delta^{V}_{k})^{1/2}+
      (m^{V}_{n}(v')\delta^{V}_{k})^{1/2}
      +\varepsilon_{n,k}(\set{v,v'})^{1/2}+\rho_{n,k}(\set{v,v'})^{1/2}
      \Bigr)^{1/2}\notag\\
    &\leq \frac{n^{2}}{\sigma_{n}^{4}} CJ_{V} \chi_{n}
      ((M^{V}_{n}\delta^{V}_{k})^{1/2}+\varepsilon_{n,k}^{1/4}+\rho_{k,n}^{1/4})
      \label{eq:cvv:diff}
  \end{align}
  and
  \begin{align}
    \sum_{\substack{u,v,v' \in V_{n}\\\text{pairw. diff.}}} c(\edge{u}{v},v')
    &\leq \sigma_{n}^{-4} CJ_{E}J_{V}
      \sum_{u,v,v'}
    \frac{W_{u}W_{v}}{n\vartheta} \zeta_{n}(v')^{1/2}
      ((m^{E}_{n}(v,u)\delta^{E}_{k})^{1/2}
      +(m^{V}_{n}(v)\delta^{V}_{k})^{1/2}\notag\\
   &\qquad +\varepsilon_{n,k}(\set{u,v,v'})^{1/4}
    +\rho_{n,k}(\set{u,v,v'})^{1/4})\notag\\
    &\leq \frac{n^{2}}{\sigma_{n}^{4}} C J_{E}J_{V}
      \Bigl(
      \frac{1}{n^{3}\vartheta^{2}}
      \sum_{u, v, v' \in V_{n}}W_{u}^{2}W_{v}^{2}\zeta_{n}(v')
      \Bigr)^{1/2}
      \Bigl(
      \frac{1}{n^{3}}
      \sum_{u, v, v' \in V_{n}}
      (m^{E}_{n}(v,u)\delta^{E}_{k})^{1/2}
      \notag\\
    &\qquad
      +(m^{V}_{n}(v')\delta^{V}_{k})^{1/2}
      +\varepsilon_{n,k}(\set{u,v,v'})^{1/2}+\rho_{n,k}(\set{u,v,v'})^{1/2}
      \Bigr)^{1/2}\notag\\
    &\leq \frac{n^{2}}{\sigma_{n}^{4}} C J_{E}J_{V} \Gamma_{2,n} \chi_{n}^{1/2}
      ((M^{E}_{n}\delta^{E}_{k})^{1/2}+(M^{V}_{n}\delta^{V}_{k})^{1/2}
      +\varepsilon_{n,k}^{1/4}+\rho_{k,n}^{1/4}).
      \label{eq:cev:diff}
  \end{align}

  The bounds on the right-hand side of~\eqref{eq:cee:diff},
  \eqref{eq:cvv:diff} and~\eqref{eq:cev:diff}
  can all be estimated by
  \begin{equation}
    \begin{split}
    &\frac{n^{2}}{\sigma_{n}^{4}} C (J_{E}+J_{V}+J_{E}J_{V})
    (\Gamma_{2,n}+\chi_{n}^{1/2})^{2}
    ((M^{E}_{n}\delta^{E}_{k})^{1/2}+(M^{V}_{n}\delta^{V}_{k})^{1/2}
      +\varepsilon_{n,k}^{1/4}+\rho_{k,n}^{1/4})\\
    &\quad\leq\frac{n^{2}}{\sigma_{n}^{4}} C J
    (\vartheta^{1/2}+\Gamma_{2,n}+\chi_{n}^{1/2})^{2}
    ((M^{E}_{n}\delta^{E}_{k})^{1/2}+(M^{V}_{n}\delta^{V}_{k})^{1/2}
      +\varepsilon_{n,k}^{1/4}+\rho_{k,n}^{1/4}).
    \end{split}
      \label{eq:chardbounds}
  \end{equation}

  Putting all these terms together and using the simplified bounds
  from~\eqref{eq:ceasybounds} and~\eqref{eq:chardbounds} we obtain
  \begin{equation*}
    \Bigl(\sum_{x,x'} c(x,x')\Bigr)^{\!1/4}
    \leq CJ^{1/4} \Bigl(\frac{n}{\sigma_{n}^{2}}\Bigr)^{\!1/2}
      (\vartheta^{1/2}+\Gamma_{2,n}+\chi_{n}^{1/2})^{2}
           ((M^{E}_{n}\delta^{E}_{k})^{1/8}+(M^{V}_{n}\delta^{V}_{k})^{1/8}
    +\varepsilon_{n,k}^{1/16}+ \rho_{n,k}^{1/16}).
  \end{equation*}

  Furthermore, by \autoref{lem:sumdelthree:e}, \autoref{lem:sumdelthree:v}
  and from~\(\chi_{n} = n^{-1}\sum_{v \in V_{n}} \zeta_{n}(v)\)
  \begin{align*}
    \Bigl(
      \sigma_{n}^{-3} \sum_{x} \expe[\abs{\Delta_{x}f}^{3}]
    \Bigr)^{\!1/2}
    &=\Bigl(
      \sigma_{n}^{-3} \sum_{e} \expe[\abs{\Delta_{x}f}^{3}]
      +\sigma_{n}^{-3} \sum_{v} \expe[\abs{\Delta_{x}f}^{3}]
    \Bigr)^{\!1/2}\\
    &\leq  (\sigma_{n}^{-3} J_{E}^{1/2}2n\vartheta \Gamma_{1,n}^{2}
      + \sigma_{n}^{-3} J_{V}^{1/2} n \chi_{n})^{\!1/2}\\
    &\leq (J_{E}+J_{V}) \Bigl(\frac{n}{\sigma_{n}^{2}}\Bigr)^{\!3/4}
    \frac{\vartheta\Gamma_{1,n}+\chi_{n}^{1/2}}{n^{1/4}}.
  \end{align*}

  Together these two terms give the required bound.
\end{proof}

With \autoref{thm:fgnconv} shown, we can prove \autoref{cor:fgnconv}.

\begin{proof}[\proofname\ of \autoref{cor:fgnconv}]
  Let~\(e = \edge{u}{v}\).
  When~\(B_{k}(v,\mathbf{G}_{n})\)
  and~\(B_{k}(v,\mathbf{G}_{n}^{e})\)
  are trees, define
  \begin{align*}
    \mathrm{LA}^{E,L}_{k}(B_{k}(v,\mathbf{G}_{n}),B_{k}(v,\mathbf{G}_{n}^{e}))
  &= g^{L}_{k}(B_{k}(v,\mathbf{G}_{n}))-g^{U}_{k}(B_{k}(v,\mathbf{G}_{n}^{e}))\\
  \shortintertext{and}
    \mathrm{LA}^{E,U}_{k}(B_{k}(v,\mathbf{G}_{n}),B_{k}(v,\mathbf{G}_{n}^{e}))
  &= g^{U}_{k}(B_{k}(v,\mathbf{G}_{n}))-g^{L}_{k}(B_{k}(v,\mathbf{G}_{n}^{e})).
  \end{align*}
  When~\(B_{k}(v,\mathbf{G}_{n})\)
  and~\(B_{k}(v,\mathbf{G}_{n}^{v})\)
  are trees, define
  \begin{align*}
    \mathrm{LA}^{V,L}_{k}(B_{k}(v,\mathbf{G}_{n}),B_{k}(v,\mathbf{G}^{v}_{n}))
  &= g^{L}_{k}(B_{k}(v,\mathbf{G}_{n}))-g^{U}_{k}(B_{k}(v,\mathbf{G}^{v}_{n}))\\
  \shortintertext{and}
    \mathrm{LA}^{V,U}_{k}(B_{k}(v,\mathbf{G}_{n}),B_{k}(v,\mathbf{G}^{v}_{n}))
  &= g^{U}_{k}(B_{k}(v,\mathbf{G}_{n}))-g^{L}_{k}(B_{k}(v,\mathbf{G}^{v}_{n})).
  \end{align*}

  We now verify that property~GLA holds
  for this choice of functions~\(\mathrm{LA}^{E,L}_{k}\),
  \(\mathrm{LA}^{E,U}_{k}\),
  \(\mathrm{LA}^{V,L}_{k}\),
  \(\mathrm{LA}^{V,U}_{k}\)
  and then apply \autoref{thm:fgnconv}.
  We will only verify \ref{itm:gla:e:delta},
  \ref{itm:gla:e:tree}
  and \ref{itm:gla:e:conv}
  for the edge perturbation.
  The proof for the vertex perturbation
  \ref{itm:gla:v:delta},
  \ref{itm:gla:v:tree}
  and \ref{itm:gla:v:conv}
  is analogous.

  \ref{itm:gla:e:delta} follows from \ref{eq:fgmv}.
  We
  use that~\(\mathbf{G}_{n}^{e}-v = \mathbf{G}_{n}-v\),
  since the vertex~\(v\) and all edges incident to~\(v\)
  are not present in those graphs, so that rerandomisation at~\(e\),
  which is incident to~\(v\)
  do not have any effect.
  Hence,
  \begin{align*}
  \mathrm{LA}^{E,L}_{k}(B_{k}(v,\mathbf{G}_{n}),B_{k}(v,\mathbf{G}^{e}_{n}))
  &= g^{L}_{k}(B_{k}(v,\mathbf{G}_{n}))-g^{U}_{k}(B_{k}(v,\mathbf{G}^{e}_{n}))\\
  &\leq (f(\mathbf{G}_{n})-f(\mathbf{G}_{n}-v))
  -(f(\mathbf{G}_{n}^{e})-f(\mathbf{G}_{n}^{e}-v))\\
  &= f(\mathbf{G}_{n})-f(\mathbf{G}_{n}^{e})\\
  &= \Delta_{e}f
  \end{align*}
  and similarly
  \[
    \mathrm{LA}^{E,U}_{k}(B_{k}(v,\mathbf{G}_{n}),B_{k}(v,\mathbf{G}^{e}_{n}))
    \geq \Delta_{e}f.
  \]

  \ref{itm:gla:e:tree} follows directly from \ref{eq:gtree}.

  For~\ref{itm:gla:e:conv} observe that
  \begin{align*}
  \abs[\big]{\mathrm{LA}^{E,U}(\mathbf{T},\mathbf{T}')
    -\mathrm{LA}^{E,L}(\mathbf{T},\mathbf{T}')}
  &\leq \abs[\big]{(g^{U}_{k}(\mathbf{T})-g^{L}_{k}(\mathbf{T}'))
    -(g^{L}_{k}(\mathbf{T})-g^{U}_{k}(\mathbf{T}))}\\
  &\leq\abs{g^{U}_{k}(\mathbf{T})-g^{L}_{k}(\mathbf{T})}
  +\abs{g^{U}_{k}(\mathbf{T}')-g^{L}_{k}(\mathbf{T}')}.
  \end{align*}
  Thus \((x+y)^{2} \leq 2x^{2}+2y^{2}\) and \ref{eq:gult}
  yield
  \begin{align*}
  \expe_{n}[\abs{\mathrm{LA}^{E,U}(\mathbf{T},\mathbf{T}')
    -\mathrm{LA}^{E,L}(\mathbf{T},\mathbf{T}')}^{2}]
  &\leq 2\expe_{n}[
  \abs{g^{U}_{k}(\mathbf{T})-g^{L}_{k}(\mathbf{T})}^{2}
  ]
  +2\expe_{n}[
  \abs{g^{U}_{k}(\mathbf{T}')-g^{L}_{k}(\mathbf{T}')}^{2}
    ]\\
    &\leq 2m_{n}(v)\delta_{k}+2\tilde{m}_{n}(v,u)\tilde{\delta}_{k}.
  \end{align*}
  Set~\(\delta^{E}_{k} = \delta_{k}+\tilde{\delta}_{k}\)
  and~\(m^{E}_{n}(v,u) = 2(m_{n}(v)+\tilde{m}(v,u))\).
  Then the previous term can be bounded by \(m^{E}_{n}(v,u) \delta^{E}_{k}\).
  Now
  \[
    M_{n}^{E}
    = \frac{1}{n^{2}} \sum_{v,u \in V_{n}} m_{n}^{E}(v,u)
    = \frac{2}{n^{2}} \sum_{v,u \in V_{n}} (m_{n}(v)+\tilde{m}(v,u))
  \]
  is bounded in probability by \ref{eq:gult}.
  Also by \ref{eq:gult}
  we have~\(\delta^{E}_{k} \to 0\) as~\(k \to \infty\),
  which finally verifies~\ref{itm:gla:e:conv}.

  Now the claim follows with an application of \autoref{thm:fgnconv}.
\end{proof}


\end{document}